\documentclass[a4paper,11pt,english,reqno]{amsart}

\usepackage{stmaryrd,upgreek}
\usepackage{bm}

\usepackage[utf8]{inputenc} 
\usepackage[T1]{fontenc}
\usepackage[english]{babel}
\usepackage[autolanguage]{numprint} 
\usepackage[colorlinks=true, citecolor=teal, linkcolor=purple]{hyperref} 
\usepackage{graphicx} 
\usepackage{verbatim} 
\usepackage[all]{xy}
\usepackage{dsfont}
\usepackage{amsmath,amssymb,amsthm} 
\usepackage{lmodern}
\usepackage{nicematrix}

\renewcommand\epsilon\varespilon 
\renewcommand\phi\varphi

\newcommand\NN{\mathbb{N}} 
\newcommand\ZZ{\mathbb{Z}} 
 
\newcommand\RR{\mathbb{R}} 
 
\newcommand\PP{\mathbb{P}}
\newcommand\EE{\mathbb{E}}

\newcommand\tr{\mathrm{tr}}
\newcommand\Car{{\bf 1}}

\newcommand\veps{\varepsilon}
\newcommand\Var{\mathrm{Var}}

\theoremstyle{definition} 
\newtheorem{Def}{Definition}[section] 
\newtheorem{Rem}[Def]{Remark} 
\theoremstyle{plain} 

\newtheorem{Pro}[Def]{Proposition} 
\newtheorem{proposition}[Def]{Proposition} 
\newtheorem{Lem}[Def]{Lemma}  
\newtheorem{lemma}[Def]{Lemma}

\newtheorem{The}[Def]{Theorem} 

\newtheorem{theorem}[Def]{Theorem}

\newtheorem{Cor}[Def]{Corollary} 
\newtheorem{corollary}[Def]{Corollary} 
\newtheorem{assumption}[Def]{Assumptions}

\numberwithin{equation}{section}

\usepackage[normalem]{ulem} 
\usepackage[a4paper, margin=2.4cm]{geometry}
\usepackage{graphicx,xcolor}
\usepackage{mathrsfs}

\newcommand\cF{\mathcal{F}}

\usepackage{nicematrix}

\renewcommand\epsilon{\varepsilon}

\newtheorem{ftheorem}{Théorème}[section]

\newtheorem{fdefinition}[ftheorem]{Définition}

\newcommand{\corF}{\textcolor{blue}}

\numberwithin{equation}{section}

\title[Maximum of the characteristic polynomial of random Jacobi matrices]{Maximum of the characteristic polynomial of random Jacobi matrices}

\author[F.\ Augeri]{Fanny Augeri$^{\mathsection}$}
\address{$^{\mathsection}$Laboratoire de Probabilit\'{e}s, Statistique et Mod\'{e}lisation (LPSM), \newline \indent
Universit\'{e} de Paris, 75205 Paris Cedex 13, France}
\email{augeri@lpsm.paris}

\author[O.\ Zeitouni]{Ofer Zeitouni$^{\#}$}
\address{$^{\#}$Department of Mathematics, Weizmann Institute of Science, 
 POB 26, Rehovot 76100, \newline\indent Israel, and
  Courant Institute, New York University,
  251 Mercer St, New York, NY 10012, \newline \indent
  USA}
 \email{ofer.zeitouni@weizmann.ac.il}

\date{}

\begin{document}

\begin{abstract} We compute the second order asymptotics of the maximum of the absolute value of the log-characteristic polynomial of random Jacobi matrices whose coefficients satisfy some exponential integrability condition. In particular, by the triadiagonal representation of Dumitriu and Eldelman of Gaussian $\beta$ Ensembles, this result partially confirms the Fydorov-Simm conjecture.   
 \end{abstract}
 \maketitle
 \tableofcontents

\section{Introduction and main results}
In the present work, we study the maximum of the characteristic polynomial of the following random Jacobi matrix

\begin{equation}
  \label{eq-Jacobi}
  J_{n}=\begin{pmatrix}b_n& a_{n-1}&0&\cdots&0\\
    a_{n-1}&b_{n-1}&a_{n-2}&\cdots&0\\
    \vdots&\ddots& \ddots&\ddots&\vdots\\
    0&\cdots&a_{2}&b_{2}&a_1\\
    0&\cdots&\cdots&a_1& b_1
  \end{pmatrix}.
  \vspace{10pt}
\end{equation}
where the coefficients of $J_n$ satisfy the following assumptions.
\begin{assumption}\label{ass}
$(a_k)_{k\geq 1}$ and $(b_k)_{k\geq 1}$ are two independent sequences of independent random variables 
 such that
\begin{equation} \label{moments1}  \EE(a_k^2) = k+O(1), \ \mathrm{Var}(a_k^2) = kv+O(1), \  \EE(b_k) = 0, \ \mathrm{Var}(b_k) = v+O\Big(\frac{1}{k}\Big),  \quad k\geq 1,\end{equation}
 where $v$ is some positive constant. 
 Further,
 there exists
$\mathfrak{h}_0>0$ such that
\begin{equation} \label{boundlaplace1} \  \sup_{k\geq 1} \EE \big(e^{\mathfrak{h}_0  |b_k| }\big)<+\infty,  \ \sup_{k\geq 1} \EE \big(e^{\frac{\mathfrak{h}_0}{ \sqrt{k}}|a_k^2-\EE(a_k^2)|}\big) <+\infty.\end{equation}

\end{assumption} 

We denote by  $p_n$ the characteristic polynomial of the scaled Jacobi matrix $J_n/\sqrt{n}$ defined by $p_n(z) =  \mathrm{det}(z\mathrm{I}_n - J_n/\sqrt{n})$ for any $z\in \RR$, where $\mathrm{I}_n$ denotes the identity matrix of size $n\times n$. Our main result reads as follows.

\begin{theorem}\label{maintheo}Let  $\eta>0$ and denote by $I_\eta := \{z \in \RR : \eta \leq |z| \leq 2-\eta\}$. 
%In probability, 
\[ \frac{\max_{z \in I_\eta} \big( \log |p_n(z)| - n \big(\frac{z^2}{4} - \frac12\big) \big) - \sqrt{v}\log n }{\log \log n} \underset{n\to+\infty}{\longrightarrow} - \frac{3\sqrt{v}}{4}, \quad \mbox{\rm in probability.}\]
\end{theorem}
 We believe similar results hold for the imaginary part of $\log p_n(z)$, which is  an eigenvalue-counting function. However, we chose not to present those here in order to keep this already long paper at a reasonable length. Similarly, the restriction to $z\in I_\eta$ is done for convenience and probably can be extended to the full interval $[-2,2]$.

As a particular case of Theorem \ref{maintheo}, we obtain the second order asymptotics of the maximum of the absolute value of the logarithmic potential of {\em Gaussian $\beta$ Ensembles}. The {\em Gaussian $\beta$ Ensemble}, where $\beta>0$, is the law on real $n$-particles configurations given by
\[ \PP_{n,\beta} = C_{n,\beta} \prod_{1\leq i< j\leq n} |\lambda_i-\lambda_j|^\beta e^{-\frac{\beta}{4} \sum_{i=1}^n \lambda_i^2} \prod_{i=1}^n d\lambda_i,\]
where $d\lambda_i$ denotes the Lebesgue measure on $\RR$ and $C_{n,\beta}>0$ is a normalisation constant. When $\beta = 1$, respectively $\beta =2$, $\beta=4$, the Gaussian $\beta$ Ensemble
coincides with the joint law of the eigenvalues of the Gaussian Orthogonal Ensemble (GOE), respectively the Gaussian Unitary Ensemble (GUE)  and the Gaussian Symplectic ensemble (GSE). For general $\beta>0$, there is no natural dense random matrix model whose spectrum is distributed according to $\PP_{n,\beta}$, but Edelman and Dimitriu \cite{DE} found that $\PP_{n,\beta}$ can be interpreted as the joint density of the eigenvalues of a certain tridiagonal model.  
More precisely, if $b_k$ is a centered Gaussian with variance $2/\beta$ and $\sqrt{\beta} a_k$ has a $\chi$-distribution of degree $\beta k$, then the corresponding random Jacobi matrix has a spectrum distributed as $\PP_{n,\beta}$. One can easily check that such  coefficients $b_k$ and $a_k$ satisfy the moments and integrability Assumptions \ref{ass} with $v = 2/\beta$. 
Hence, we get as an immediate consequence of Theorem \ref{maintheo} the following corollary. Denote by $p_{\lambda}(z)$ the logarithmic potential of the $n$-particles configuration $\lambda \in \RR^n$, that is, 
\[ p_\lambda(z) := \frac1n\sum_{i=1}^n \log(z-\lambda_i),  \ z\in \RR,\]
with the convention that $\log (0) = -\infty$.
\begin{corollary}\label{cor-FS} Let $\beta,\eta>0$. Denote by $I_\eta := \{z \in \RR : \eta \leq |z|\leq 2-\eta \}$. 
%In probability 
Under $\PP_{n,\beta}$, 
\[ \frac{\max_{z \in I_\eta} \big( \log |p_{\lambda}(z)| -n \big(\frac{z^2}{4}-\frac12\big) \big) -\sqrt{\frac{2}{\beta}} \log n }{\log \log n} \underset{n\to+\infty}{\longrightarrow}  - \frac{3}{2\sqrt{2\beta}}, \quad \mbox{\rm 
in probability.}\]
\end{corollary}
 Corollary \ref{cor-FS} confirms the first two leading terms of the Fyodorov and Simm conjecture, stated for the GUE in \cite{FS}. It is worthwhile to note that given Corollary \ref{cor-FS}, one obtains the same result (corresponding to $\beta=1,2$) 
for Wigner type matrices with Gaussian-divisible entries, see \cite[Theorem 1.4]{BLZ}.\footnote{As noted by Paul Bourgade, the recent higher moment matching results in \cite[Lemmas 2.3 and 2.4]{Zhang} seem to allow an extension beyond the Gaussian divisible case.}

\subsection{Related work}
\label{sec-related}
 This paper contributes to the body of research motivated by the study of logarithmically correlated fields in the context of random matrices, 
and the seminal conjectures of Fyodorov--Hiary--Keating \cite{FHK} concerning the characteristic polynomial 
of random unitary matrices and their 
links with the maxima of the Riemann zeta function over short intervals (see \cite{FS} for the Hermitian counterpart). We refer to the literature review in \cite{PZ} for a historical account and pointers to the literature for both the circular ensemble results and the logarithmically correlated fields angle.

 We next discuss ensembles with real eigenvalues. The (logarithm of the) characteristic polynomial of random matrices is an example of a linear statistics of the spectrum, albeit with a family of (singular) test functions.
Central limit theorems 
%The interest in the characteristic polynomial of random GUE matrices was one of the motivations of Wigner's work
%\cite{Wigner}. CLT's 
for smooth test functions were discussed (for general $\beta$) in \cite{JohanssonCLT}, while the CLT for
the characteristic polynomials of such matrices in the case $\beta=2$ follow from the analysis in 
\cite{Krasovsky}, see also 
\cite{Charlier}; we refer to \cite{ABZ} 
and \cite{JKOP} for a historical review. The CLT for the determinant of such matrices (and of general Wigner matrices) was derived by \cite{TV}, who also used the basic recursions we use, with $z=0$ (where only the elliptic regime survives). 
More recently, 
%Concerning ensembles with real eigenvalues, previous work includes 
we mention the derivation of multi-points and central limit theorems for the characteristic polynomial of $\beta$ ensembles
in \cite{BMP}, and the refinements developed in \cite{BLZ}, where the leading order in Theorem \ref{maintheo} for $\beta$-ensembles  is derived. For the Gaussian 
$\beta$ ensembles, a detailed study of the recursions described in Section \ref{sec-recursions} below was undertaken in \cite{LP1,LP2,LP3}
(see also \cite{JKOP} for the hyperbolic and parabolic regimes, motivated by the study of spherical spin-glasses).
In the determinantal ($\beta=2$) case, where analytical techniques are available,  we also refer to \cite{LP} and \cite{CFLW}.
 Our work here, while self contained,  builds on ideas in \cite{ABZ}, where a (pointwise) central limit theorem for the logarithm of the characteristic polynomial of
random Jacobi matrices was derived. 
%\\[0.3em]

%\noindent
%{\bf Acknowledgement} This work was partially supported by the Israel Science Foundation (grant $\#$615$\backslash$24).
%$\#$615$\backslash$24).

\section{The three term recursion}
\label{sec-recursions}
Owing to the tridiagonal structure of $J_n$, its characteristic polynomial is naturally linked to a certain three terms recursion. More precisely, for any $z \in (-2,2)$ and  $n \in \NN$, $n\geq 1$, 
let $(q_k(z))_{k \in \{-1,\ldots, n\}}$ be defined by the recursion:
\begin{equation}
  \label{basic-recursion}
 q_{-1}(z) = 0, \ q_0(z) = 1,\;  q_{k}(z) = \big( z\sqrt{n}-b_{k}\big) q_{k-1} - a_{k-1}^2 q_{k-2}(z),
\; k\geq 1,
 \end{equation}
where $a_0=0$ by convention. It is a standard fact that for any  $k\in\{1,\ldots,n\}$ and $z\in \RR$, 
 $q_k(z)= \mathrm{det}(z\sqrt{n} \mathrm{I}_k - J_k)$, and in particular $q_n(z) =n^{n/2}p_n(z)$.
%Fix $\kappa>0$ large and set
%\begin{equation}
%  \label{eq-zkk0}
%  \forall k \in \{1,\ldots,n\}, \ z_k=z\sqrt{n/k}, \quad 
%  k_{0,z} = \lfloor {z^2n}/{4} \rfloor, \quad \ell_{0,z}=\lfloor \kappa k_0^{1/3}\rfloor
%\end{equation}
%One can check that $z_{k_0} = 2+O(1/n)$.
Let $(\phi_k)_{k\in\{ 0,\ldots,n\}}$ denote the scaled variables defined by 
\begin{equation} \label{defphi} {\phi}_k(z) = \frac{q_k(z)}{\sqrt{k!}}, \   k\in \{0,\ldots,n\}, z \in (-2,2).\end{equation}
The choice of this scaling is motivated by the fact that these new polynomials now satisfy the recursion
  \begin{equation}
    \label{eq-recpsik}
    {\phi}_{k}(z)=\Big(z_k-\frac{b_k}{\sqrt{k}}\Big){\phi}_{k-1}(z)
    -\frac{a_{k-1}^2}{\sqrt{k(k-1)}} {\phi}_{k-2}(z) , \quad k\geq 2, z \in (-2,2),
  \end{equation}
  where 
\begin{equation}
\label{eq-zk}
 z_k= z\sqrt{n/k}
\end{equation} and where the second order coefficients are of order $1$, given that by assumption $a_{k-1}^2 \asymp  k$ typically. In turn, we can rewrite this second order recursion as a first order bidimensional recursion by introducing the vectors $X_k^z$ and transition matrix $T_k^z$ defined by
  \begin{equation} \label{deftransition} X_{k}^z = \begin{pmatrix} {\phi}_{k}(z) \\ {\phi}_{k-1}(z) \end{pmatrix}, \quad T_k^z = \begin{pmatrix} z_k - \frac{b_k}{\sqrt{k}} & -\frac{a_{k-1}^2}{\sqrt{k(k-1)}} \\ 1 & 0 \end{pmatrix}, \quad k\geq  2.\end{equation}
With this notation, the recursion \eqref{eq-recpsik} is equivalent to 
\begin{equation} \label{recbi} X_{k}^z =T_k^z X_{k-1}^z, \quad  k\geq 2,z \in (-2,2),\end{equation}  
where $X_1^z = (z\sqrt{n} - b_1, 1)^{\sf T}$.
%For any  $z \in [-2,2]$,  let $X_k^z= ({\phi}_k(z), {\phi}_{k-1}(z))^{\sf T}$ for any $k\in\{1,\ldots,n\}$. $(X_k^z)_{k\in[n]}$ satisfy the recursion
%\[ \forall k \in [n], \ X_{k+1}^z = T_k^z X_k^z,\]
%where $T_k^z = A_k^z + W_k$, with
%\[ A_k^z = \begin{pmatrix} 
%z_k & -1 \\ 1 & 0 \end{pmatrix}, \ W_k = - \begin{pmatrix} \frac{b_k}{\sqrt{k}} & \frac{g_k}{\sqrt{k}} \\ 0& 0\end{pmatrix}.\]
  A central object of this recursion is the expectation of $T_k^z$, which, under our Assumptions \ref{ass}, is very close to the matrix $A_k^z$ defined by
  \begin{equation} \label{defAk} A_{k}^z :=\begin{pmatrix} 
z_k & -1 \\ 1 & 0 \end{pmatrix}.\end{equation}
As  $A_k^z$ belongs to the special linear group of order $2$, $\mathbb{SL}_2(\RR)$, the dynamics of the system will highly depend in which of the three classes of  $\mathbb{SL}_2(\RR)$, hyperbolic, parabolic or elliptic, it belongs to. Since this classification is determined by respectively the value of $|\tr(A_k^z)|$ being strictly greater than $2$, equal to $2$ or strictly smaller than $2$, this leads us to define the critical time $k_{0,z}$ and - although less obvious for now - the critical window $\ell_{0}$ as
\begin{equation} \label{defk0} k_{0,z} := \lfloor \frac{z^2n}{4}\rfloor, \quad \ell_{0}:= \lfloor \kappa n^{1/3}\rfloor, \quad \kappa \geq 1,\end{equation}
and to decompose the recursion into three regimes: a {\em hyperbolic regime} (until time $k_{0,z}-\ell_0$) where the eigenvalues of $A_k^z$ are real, a {\em parabolic regime} (between time $k_{0,z} - \ell_0$ and $k_{0,z}+\ell_0$), and an {\em elliptic regime} (after time $k_{0,z}+\ell_0$) where the eigenvalues of $A_k^z$ are complex conjugated and of modulus $1$. 

In the present work, we study in details the dynamical system \eqref{deftransition} and compute the second order asymptotic expansion of the maximum of $\log \|X_n^z\|$ on a closed subset of $(-2,2)$ bounded away from $0$. To state this result, we introduce further  $\alpha_{k,z}$ when $k\leq  k_{0,z}$,  the spectral radius of $A_k^z$, given by
\begin{equation} \label{defalpha} \alpha_{k,z} := \frac{|z_k| + \sqrt{z_k^2-4}}{2}, \quad 1\leq k \leq k_{0,z}.\end{equation}
With this notation, we have the following result.
\begin{The}\label{maxnorm}
Let  $\eta>0$ and denote by $I_\eta := \{z \in \RR : \eta \leq |z| \leq 2-\eta\}$. Under Assumptions \ref{ass},
\begin{equation} \label{convproba} \frac{\max_{z \in I_\eta} \big( \log \|X_{n}^z\| - \sum_{k=1}^{k_{0,z}} \log \alpha_{k,z} \big) - \sqrt{v}\log n }{\log \log n} \underset{n\to+\infty}{\longrightarrow} - \frac{3\sqrt{v}}{4},\end{equation}
in probability.
\end{The}
As we will show in section \ref{anti} that Theorems \ref{maxnorm} and \ref{maintheo} are actually equivalent by using some anti-concentration argument. The interest of dealing with $\|X_n^z\|$ instead of $\phi_n(z)$ is that we can analyze the recursion through the geometric properties of the transition matrix $T_k^z$.

\subsection{High-level strategy and structure of the proof} \label{highlevel}
As with the earlier \cite{ABZ} (and other works, see 
Section \ref{sec-related}), the proof of Theorem \ref{maxnorm} is based on the 
recursions \eqref{recbi}, and specifically on the norm of the global transfer matrix (that this suffices is a relatively simple anti-concentration argument, presented in Section \ref{anti}). However, it turns out that it is convenient to rewrite these linear recursions after a change of basis (similar to the one employed in \cite{ABZ} for the elliptic regime, but also extended here to the hyperbolic regime). This is done in Section 
 \ref{sec-newbasis}, see \eqref{defXi}. This new form of the recursions translates Theorem \ref{maxnorm} to Theorem \ref{reduc}. After proving that the parabolic regime does not contribute to the maximum at the scale we consider, we develop in Section 
\ref{sec-newbasis} precise exponential moment estimates in the hyperbolic and elliptic regimes. In both regimes, we consider differences of the recursions over certain blocks, In the hyperbolic regime, a block is considered \textit{good} if the recursion remains roughly aligned with the first basis
vector (this is measured in terms of the variable $W_k^z$, see \eqref{def-Wk}). In the elliptic regime, where the recursion is a noisy perturbation of a product of deterministic rotation matrices, the length of blocks correspond to a (large) number of cycles  for the product 
of these deterministic rotation matrices, and a block is considered \textit{good} if the accumulated (random) phase in the block is
close to $0$ (modulus $2\pi$). A good part of the analysis of the recursions, performed in Section \ref{sectionhyperbolic} for the hyperbolic regime and in Section \ref{sectionelliptic} for the elliptic regime, is to show that good blocks dominate, and that the differences 
of the logarithms of the norm of the recursion behaves like a random walk with independent increments, even at the level of exponential moments. In fact, in Section \ref{sec-Precise} we provide the estimate in terms of a $z$-dependent time change. This is then used
in Section \ref{sec-OneRay}, together with a strong approximation result of Sakhanenko, to couple the (logarithm of the) recursion
to a \textit{Gaussian} random walk.

 The rest of the proof employs a path well developed in the context of logarithmically correlated fields, see
e.g. \cite{Ze} for an introduction. The first key step is to show that one can introduce a certain barrier, i.e. that the only points $z$ that
are candidates for achieving a near maximum of the recursion norm at time $n$ are such that  the recursion (in the time changed coordinates) does not cross a certain increasing linear barrier (which interpolates between $C\log\log n$ and $\sqrt{v}\log n+C\log\log n$, for a large enough constant $C$). It is not hard to show using a union bound that with high probability, at every time $t\in [1,\log n]$ and  for any deterministic collection of $e^t$ points $z$, which will be chosen roughly equi-spaced 
in $I_\eta$, indeed the recursion will not cross this rough barrier at time $t$. However, we need to show this does not  happen for \textit{all}
$z\in I_\eta$, or at least (due to Lemma \ref{lem-LP}) for an equispaced  subset of cardinality $cn$. For that, one needs to prove that the recursion, as function of $z$, is continuous at appropriate temporal scales, even at the level of exponential moment. This is arguably the hardest part of the argument, and is carried out in Section \ref{sec-Barrier}. Here we use in a crucial way the fact that we are shooting for a rough barrier, i.e. that we allowed ourselves a leeway of size $C\log\log n$; better precision would be possible but requires a much more refined argument.  Once the barrier is in place, the upper bound (together with the correct $\log\log n$ correction)
is a consequence of a union bound, using that the ballot theorem for random walks supplies the correct constant in front of the $\log \log n$ term.

The proof of the lower bound employs similarly a barrier, except that for the lower bound one does not need to justify a-priori its insertion. One defines a point $z$ in an $n$-net good if the bulk of the recursion (for the logarithm of the norm of the transfer matrix) stays below a barrier and reaches a high value at its (almost) terminal time, and then one 
counts the number of good points within $I_\eta$. The lower bound is then obtained by an application of the
second moment method.
The heart of the argument is a decorrelation estimate, that exhibits decorrelation of the increments of the recursion corresponding to $z,z'$,
 for times beyond an appropriate ``branching time'' that depends on $|z-z'|$. The precise statement is contained in Section
\ref{decorrsection}. Naturally, the proof depends on whether both recursions are in their 
elliptic regime or not (in the elliptic regime, the decorrelation appears since the same Gaussian-like increments are multiplied by (essentially
deterministic) rotating cosines of different frequencies).
Equipped  with these decorrelation estimates, the lower bound proof is a standard application of the truncated second moment method, and is carried out in Section \ref{sec-LBproof}. The proof employs a lower bound on the first moment (which uses the Gaussian coupling and a ballot theorem, recalled in Appendix \ref{sec-Ballot}), and an upper bound on two rays estimates, which uses the decorrelation 
and exponential moments estimates from Section \ref{decorrsection}.

It is worthwhile commenting on the reason that we did not obtain tightness of the centered maximum, as one may have expected. The main reason is that
%There are two main
 %reasons, which are somewhat related.  First,
to obtain enough independence for the second moment method to work we simply truncate the recursions  both in the beginning of the hyperbolic regime and in the end of the elliptic one,
 on an interval where the accumulated variance increases with $n$. This introduces enough independence that allows us to avoid using a
"two rays" Gaussian coupling, which seems technically doable  using \cite{Lifshits} but is technically challenging because unlike the classical case of
branching random walks and their variations, the same variables in the recursions influence different rays at different "variance" times. 
This is particularly annoying for relatively well separated $z$'s, where almost optimal independence is needed for the second moment method to succeed, and where the common driving noise may appear for one of the rays in  the hyperbolic regime while for the second it appears in the extended parabolic regime\footnote{It turns out that the fact that our Gaussian coupling contained in Lemma \ref{coupling}
does not apply with good enough error bounds  in a  relatively large extension of the parabolic region is not a significant  issue, because
it occurs in a region where the barrier is not needed even in a precise computation.}.
\iffalse
To compound this issue,
our Gaussian coupling contained in Lemma \ref{coupling}
does not apply with good enough error bounds  in a  relatively large extension of the parabolic region;   it is possible that an adaptation of the analysis in 
\cite{LP2} would yield the necessary estimates, although it seems it would require a non-negligible work (and a restriction of the class of random variables allowed in the original Jacobi matrix).  Without such precise coupling, there is a gap between our one ray
upper bound and lower bounds, that is proportional to a power of the length of the interval where we do not have Gaussian local estimates. 
This is overcomed in our analysis by the truncation discussed above.

To compound this issue,
our Gaussian coupling contained in Lemma \ref{coupling}
does not apply with good enough error bounds  in a  relatively large extension of the parabolic region;   it is possible that an adaptation of the analysis in 
\cite{LP2} would yield the necessary estimates, although it seems it would require a non-negligible work (and a restriction of the class of random variables allowed in the original Jacobi matrix).  Without such precise coupling, there is a gap between our one ray
upper bound and lower bounds, that is proportional to a power of the length of the interval where we do not have Gaussian local estimates. 

This is overcomed in our analysis by the truncation discussed above. 
\fi
Having introduced the truncation of the region where we carry out a precise analysis, 
 we use the extra slack in many places, notably in the definition of the "good ray  event"
$\mathscr{A}_{z,\veps}$ in Section \ref{sec-LBproof}, but also in our continuity estimates (Section \ref{sec-Barrier}) and the rough barrier we introduce there, which does not start at $0$ and end at the expected typical value of the maximum but rather at height $\log n$ above that.

We emphasize that we expect that all these issues could be handled, that is we do not expect that there are fundamental limitations to the methods we employ. However we decided to not  pursue it further in the current already long work.

\subsection{Notation and conventions} We assume that the underlying probability space to be rich enough to support all the random variables that will appear in this work. We call {\em model parameters} all the constants involved in the Assumptions  \ref{ass}, meaning $v$, $\mathfrak{h}_0$ and all the constants hidden in the $O(.)$'s in \eqref{moments} as well as $\eta$ which is considered as fixed throughout this article. We write for sequences $(x_k)_k$ and $(y_k)_k$ that $x_k =O(y_k)$, $x_k\lesssim y_k$ or $x_k\gtrsim y_k$ if there exists a positive constant $\mathfrak{a}$ depending on the model parameters such that for any $k\geq 1$, $|x_k|\leq \mathfrak{a} |y_k|$. Further we use the notation $x_k \asymp y_k$ if $x_k \lesssim y_k$ and $y_k \lesssim x_k$. Moreover, if $(x_k)_k$ and $(y_k)_k$ are sequence of non-negative numbers, we write $x_k = \Omega(y_k)$ if there exists $\mathfrak{b}>0$ such that $x_k \geq \mathfrak{b} y_k$. We denote by $\| v\|$ the $\ell^2$-norm of a vector $v\in \RR^d$, by $\|M\|_2$ the Hilbert-Schmidt norm of a square matrix $M$, defined by $\|M\|_2 = \sqrt{\tr (M^{\sf T}M)}$ and by $\|M\|$ its operator norm.  We denote by $\llbracket \ell,m\rrbracket$ the the set  $[\ell,m]\cap \NN$ for any $\ell,m\in\NN$, and by   $(e_i)_{i \in \llbracket1, d\rrbracket}$ the canonical basis of $\RR^d$. For any sequence of square matrices of size $2\times 2$, $(M_k)_{k\geq 1}$,  the product $M_\ell M_{\ell-1} \ldots M_k$ is denoted by $M_{k,\ell}$  for any $1\leq k \leq \ell$. We denote by $\mathrm{I}_2$ the identity matrix of size $2\times 2$.  The filtration generated by the variables $(a_k,b_k)_{1\leq k \leq n}$ is denoted by  $(\mathcal{F}_k)_{1\leq  k \leq n}$. For any $1\leq k \leq n$, $\EE_k$, respectively $\PP_k$, denotes the conditional expectation and probability with respect to $\mathcal{F}_k$. 

 Throughout, all our statements assume that $n$ is large enough, where the threshold is universal and depends on the model parameters only but not on $z$ or the randomness.

\section{Truncation of the noise}
Technically, it will be more convenient to work with bounded coefficients $a_k$ and $b_k$. Due to our exponential integrability assumption \eqref{boundlaplace1}, it is possible to truncate our coefficients by a power of $\log n$ without changing much our moment assumptions \eqref{moments1}, although at the price of having to deal with a $n$-dependent distribution of the coefficients of $J_n$. More precisely, we define a new set of assumptions as follows. 

\begin{assumption}\label{ass1}
 $(a_{kn})_{0\leq k \leq n-1}$, $(b_{kn})_{1\leq k\leq n}$ are two independent sequences of independent random variables.
%whose law is absolutely continuous with respect to the Lebesgue measure. 
so that there exist $v$, $\mathfrak{C}$,  $\mathfrak{h}_0$ positive constants independent of $n$ satisfying for any $1\leq k \leq n$,
\begin{equation} \label{moments}  \big|\EE(a_{(k-1)n}^2) - k\big|\leq \mathfrak{C}, \big|\mathrm{Var}(a_{(k-1)n}^2) - kv\big|\leq \mathfrak{C},   \big| \EE(b_{kn})\big|\leq \frac{\mathfrak{C}}{\sqrt{nk}} ,  \big|\mathrm{Var}(b_{kn}) - v\big| \leq \frac{\mathfrak{C}}{k},\end{equation}
\begin{equation} \label{boundlaplace} \  \sup_{1\leq k\leq n} \EE \big(e^{\mathfrak{h}_0  |b_{kn}| }\big)\leq 2,  \ \sup_{1\leq k\leq 1} \EE \big(e^{\frac{\mathfrak{h}_0}{ \sqrt{k}}|a_{kn}^2-\EE(a_{kn}^2)|}\big) \leq 2.\end{equation}
Further, almost surely, 
\begin{equation}\label{boundnoiseass} \max\big(|b_{kn}|, k^{-\frac{1}{2}} |a_{(k-1)n}^2-\EE(a_{(k-1)n}^2)|\big)\leq (\log n)^2, \ 1\leq k\leq n.\end{equation}
\end{assumption}

Under Assumptions \ref{ass1}, we will prove the same asymptotic expansion of the log-characteristic polynomial as in Theorem \ref{maxnorm}.

\begin{The}\label{maxnorm1}
Assume $(a_{kn}, b_{kn})_{1\leq k \leq n}$ satisfies Assumptions \ref{ass1}  and fix $\eta>0$.  Let $J_n$ be the Jacobi matrix with coefficients $(a_{kn},b_{kn})_{1\leq k \leq n}$ and $(X_k^z)_{1\leq k \leq n}$ the solution of \eqref{recbi} for any $z\in (-2,2)$. Then, \eqref{convproba} holds for $(\log \|X^z_n\|)_{z\in I_\eta}$.
% for any $\eta>0$. 
\end{The}

This reduction to bounded coefficients is justified by the fact that if $\mathscr{K}_n := \big\{\max\big(|b_{k}|, k^{-\frac{1}{2}} |a_{k-1}^2-\EE(a_{k-1}^2)|\big)\leq (\log n)^2, 1\leq k \leq n \big\}$ then the exponential integrability assumption \eqref{boundlaplace} entails that $\PP(\mathscr{K}_n^\complement) \leq e^{-\mathfrak{c} (\log n)^2}$ for some positive constant $\mathfrak{c}>0$ depending on the model parameters, and that the conditional distributions of $(b_k)_{1\leq k \leq n}$ and $(a_{k}^2)_{1\leq k \leq n-1}$ given $\mathscr{K}_n$ satisfy the moments and integrability conditions \eqref{moments}-\eqref{boundlaplace}. Hence, conditioning  on $\mathscr{K}_n$,
 Theorem \ref{maxnorm1} entails immediately Theorem \ref{maxnorm}. One can notice that conditioning on $\mathscr{K}_n$ yields a better estimate on $\EE(b_k)$ than the one in \eqref{moments}. As we will see, this bound is sufficient is derive Theorem \ref{maxnorm1}.
\\[0.1em]

{\em In the rest of this paper, we work under Assumptions \ref{ass1} and drop the $n$-dependence in the notation of the distribution of the coefficients of $J_n$ for the  sake of readability.}

\section{Change of basis and description of the new recursions}
\label{sec-newbasis}
To analyze the recursion \eqref{recbi}, we perform a certain change of basis to leverage the geometric properties of the expected transition matrix, which is roughly $A_k^z$. To this end,  recall \eqref{eq-zk}, \eqref{defk0}, \eqref{defalpha}, and define the following time-dependent change of basis $(P_k^z)_{1\leq k \leq n}$ as 
 \begin{equation} \label{changebaseP} P_k^z :=\begin{cases}
 \begin{pmatrix}
1 &  \mbox{\rm sg}(z)\alpha_{k,z}^{-1} \\  \mbox{\rm sg}(z) \alpha_{k,z}^{-1} & 1 \end{pmatrix}, & \ 1\leq k \leq k_{0,z}-\ell_0 , \\
 \begin{pmatrix}
 \frac{\sqrt{4-z_k^2}}{2} & \frac{z_k}{2}  \\
0 & 1
\end{pmatrix},   &  k_{0,z} + \ell_0\leq  k \leq n,\\
P^z_{k_{0,z}-\ell_0}, & |k-k_{0,z}|< \ell_0,
\end{cases} \end{equation}
where $\mathrm{sg}(z)$ denotes the sign of $z$. 
 The choice of this change of basis is motivated by the fact that  $\|P_k^z\|\lesssim 1$ and that $D_k^z:=(P_k^z)^{-1} A_k^z P^z_k$, $1\leq k \leq k_{0,z}-\ell_0$ is a diagonal matrix whereas $R_k^z := (P_k^z)^{-1} A_k^z P^z_k$, $k \geq  k_{0,z}+\ell_0$ is a rotation. More precisely,
\begin{equation} \label{defRk2} D_k^z := \mbox{\rm sg}(z)\begin{pmatrix} 
\alpha_{k,z} & 0 \\ 0 & \alpha_{k,z} ^{-1}
\end{pmatrix}, \ k\leq k_{0,z} - \ell_0, \  R_k^z := \begin{pmatrix}
\cos(\theta_{k}^z)& -\sin(\theta_{k}^z) \\ 
\sin(\theta_{k}^z) & \cos(\theta_{k}^z) \end{pmatrix}, \  k\geq k_{0,z}+\ell_0,  \end{equation} 
where $\theta^z_k\in [0,2\pi)$ is such that $e^{\mathrm{i}\theta^z_k}= (z_k +  \mathrm{i}\sqrt{4-z_k^2})/2$ for $k\geq k_{0,z}+\ell_0$. 
 In the parabolic regime, meaning when $| k-k_{0,z}|<\ell_0$, we chose a constant change of basis  since as we will see this regime has a contribution at most of order $1$ depending on $\kappa$. \\[0.1em]
 
{\em To avoid un-necessary cumbersome notation, in the rest of the paper we only consider the case where $z>0$. The arguments given work also for negative $z$, with  obvious minor changes in notation.}\\[0.1em]

 In the sequel, we denote by $\Xi_k^z$ the new transition matrix and by $Y_k^z$ the coordinate vector of $X_k^z$ in the basis $P_k^z$, that is
 \begin{equation} \label{defXi} \Xi_{k+1}^z := (P_{k+1}^z)^{-1} T_{k+1}^z P_{k}^z, \ Y_k^z =  (P_k^z)^{-1} X_k^z, \ 1\leq k \leq n.\end{equation}
 With this notation, the sequence $Y^z$ satisfies the recursion $Y_{k+1}^z = \Xi_{k+1}^z Y_k^z$ for any $1\leq k \leq n$. 
Our central observable will be the field $\psi_k(z)$ defined by $\psi_0(z)=0$ and
\begin{equation} \label{defpsi} \psi_k(z) := \log \|Y_k^z\| - M_k(z), \ 1\leq k \leq n, \ z \in I_\eta,\end{equation}
where $M_k(z)$ denotes the accumulated mean defined as the sum of the ``instantaneous'' mean $\mu_k(z)$ by
\begin{equation} \label{defmean}
\mu_k(z) :=  \begin{cases}
\log \alpha_{k,z} - \frac{v-1}{4(k_{0,z}-k)} & \text{ if } k\leq k_{0,z}-\ell_0\\
0 & \text{ if } |k-k_{0,z}|<\ell_0 \\
\frac{v-1}{4(k-k_{0,z})} & \text{ if } k\geq k_{0,z}+\ell_0.
\end{cases}, \quad M_k(z) := \sum_{\ell =1}^k \mu_\ell(z).
 \end{equation}  
As we will see in \eqref{estimalphak}, 
\[ \alpha_{k,z} = 1 + O(\sqrt{(k_{0,z}-k)/k_{0,z}}), \quad k\leq k_{0,z}.\]
 One can easily check that this entails that $\sum_{k=k_{0,z} - \ell_0}^{k_{0,z}} \log \alpha_{k,z} = O_\kappa(1)$. As a result,  $M_n(z) = \sum_{k\leq k_{0,z}} \log \alpha_{k,z} +O_\kappa(1)$. Moreover, as $\det(P_n^z) \gtrsim 1$ and $\|P_n^z\|\lesssim 1$, we have  that $\|(P_n^z)^{-1}\|\lesssim 1$ by the comatrix formula. Thus, to prove Theorem \ref{maxnorm1}, it suffices to show the following asymptotic expansion of the maximum of the field $\psi_n(z), z\in I_\eta$.

\begin{The}
\label{reduc}
For any $\eta>0$, 
\[ \frac{\max_{z\in I_\eta} \psi_n(z) - \sqrt{v} \log n}{\log \log n} \underset{n\to+\infty}{\longrightarrow} -\frac{3\sqrt{v}}{4},\]
in probability. 
\end{The}
To prove Theorem \ref{reduc}, we analyse the process $\psi_k(z)$ separately in the hyperbolic, parabolic and elliptic regime. In the hyperbolic regime, we show that $\psi_k(z)$ can be represented up to a small error by a sum of independent random variables with high probability. In the parabolic regime, we only prove an a priori exponential moment estimate, since this regime only has a  contribution of order $1$  to the variance of the field.  In the elliptic regime,  we prove that if one observes the process $\psi_k(z)$ on long enough blocks to be defined later, it can be   approximated as well by a sum of independent random variables with high probability.  Note that because $z\in I_\eta$, there exists $c_\eta \in(0,1)$ such that $c_\eta n \leq k_{0,z}\leq  (1- c_\eta) n$. Hence, for the $z$'s we are considering, the hyperbolic and elliptic regime both contribute in a meaningful way to the recursion.

Throughout the remaining of this section we fix some $z\in I_\eta$ and state the key results on the dynamics, postponing the proofs to sections \ref{sectionhyperbolic} and \ref{sectionelliptic}. We emphasize that all the implicit constants in our estimates, in particular in the $O()$ notation, depend only on the model parameters.

\subsection{Recursion in the hyperbolic regime}
A crucial observable in this regime is the weight of $Y_k^z$ on the smallest eigenvector of $D_k^z$, which we measure through a variable denoted by $W_k^z$:
\begin{equation}
\label{def-Wk}
W_k^z := \frac{\langle Y_k^z, e_2\rangle^2}{\|Y_k^z\|^2}.\end{equation}
Note that the recursion \eqref{defXi} is initialized with $Y_1^z = (P_1^z)^{-1}X_1^z$. As $\alpha_{1,z} = O(n^{-1/2})$, we have that $(P_1^z)^{-1} = I_2+O(n^{-1/2})$, and  since in addition  $X_1^z = (z\sqrt{n} - b_1,1)^{\sf T}$, it follows that 
\begin{equation}\label{initalW} W_1^z = O\Big(\frac{1}{n}\Big), \ \text{a.s.}\end{equation}
The hyperbolic regime is again decomposed into two distinct regions, a {\em negligible regime} and a {\em contributing regime}. To this end, fix some $\delta>0$ small enough so that
\begin{equation}
\label{eq-111025a}
k_{\delta,z}: = k_{0,z}-\delta n\geq \delta n, \quad \mbox{\rm for all $z\in I_\eta$.}
\end{equation}
 The regime until time $k_{\delta,z}$ is what we call the {\em negligible hyperbolic regime}, characterized by the fact that $\alpha_{k,z}$ is bounded away from $1$ and that the angle in $(-\pi,\pi]$ between the eigenvectors of $A_k^z$ is bounded away from $0$. In this regime we show that $W_k^z$ almost surely stays small and that the process $\psi(z)$ is close to a sum of independent random variables. The noise driving the recursion in the hyperbolic regime is a variable denoted by $g_{k,z}$ of variance  $\sigma_{k,z}^2$ defined by
\begin{equation}\label{defgk} g_{k,z} = \frac{1}{1-\alpha_{k,z}^{-2}}\Big(\alpha_{k,z}^{-1} \frac{b_k-\EE(b_k)}{\sqrt{k}} + \alpha_{k,z}^{-2}\frac{a_{k-1}^2-\EE(a_{k-1}^2)}{\sqrt{k(k-1)}} \Big), \ \sigma_{k,z}^2 := \EE[g_{k,z}^2].\end{equation}

\begin{Pro}[Negligible hyperbolic regime]\label{incremhn}
Let $2\leq k \leq  k_{\delta,z}$. There exists $\mathfrak{h}>0$ depending on the model parameters such that
\[ \psi_k(z) - \psi_{k-1}(z) = g_{k,z} + \frac12(W_k^z-W_{k-1}^z) + \mathcal{P}_k^z,\] 
where $\mathcal{P}_k^z$ is $\cF_k$-measurable  and 
\[ \big|\EE_{k-1} \mathcal{P}_k^z\big| \lesssim_\delta  \frac{1}{n}, \ 
 | \mathcal{P}_k^z-\EE_{k-1}(\mathcal{P}_k^z)|\leq \frac{(\log n)^{\mathfrak{h}}}{nk}, \quad \mbox{\rm almost surely.}\]
Moreover, $\sigma_{k,z}^2 \asymp_\delta 1/n$.
\end{Pro}
Further, we show that throughout the negligible hyperbolic regime, $Y_k^z$ stays in the $e_1$-direction and that \eqref{initalW} remains true up a power of $\log n$ almost surely.

\begin{Pro}
\label{boundWhn}
 There exists a deterministic $\mathfrak{h}>0$ depending on the model parameters only, so that
 all $1\leq k \leq k_{\delta,z}$,
\[ W_k^z \leq \frac{(\log n)^\mathfrak{h}}{n},\ \text{almost surely.}\]
\end{Pro}

In the contributing regime, meaning that  $k_{\delta,z} \leq k \leq k_{0,z}-\ell_0$,  the hyperbolic nature of the recursion becomes weaker and weaker while the noise is getting stronger and stronger. As a result, the vector $Y_k^z$ starts to move away from the $e_1$ direction. To quantify this, we introduce blocks $(m_{i,z})_{i_o\leq i \leq i_1}$ defined by $m_{i_1,z} := k_{\delta,z}$, $m_{i_o,z}:= k_{0,z}-\ell_0$, and 
\begin{equation} \label{defblockh}   m_{i,z}: =  k_{0,z} - \lfloor i^{2/3} k_{0,z}^{1/3}\rfloor, \quad i_o  < i <i_1,\end{equation}
where $i_1 := \min\{i : k_{0,z} - \lfloor i^{2/3} k_{0,z}^{1/3}\rfloor\leq k_{\delta,z}\}$ and $i_o:=\min\{i:   \lfloor i^{2/3} k_{0,z}^{1/3}\rfloor \geq \ell_0\}$.
With this notation, define the curve $\eta_{k,z}$,  which is constant on each block $( m_{i+1,z},m_{i,z}]$, by  
\begin{equation} \label{defeta} \eta_{k,z} := i^{-2/3} , \quad k \in (m_{i+1,z},m_{i,z}], \ i_o \leq i <i_1.\end{equation}
 This curve $\eta_{k,z}$ represents a boundary below which, as we will show, the process $W^z$ remains with high probability.  It is by no means sharp, but sufficiently good for our purpose here. We can now state our results in the contributing hyperbolic regime. We start by estimating the probability that $W^z$ gets above the curve $\eta_z$ during a block given that it has started the block a factor below the curve.
% To this end, say the that $i^{\text{th}}$ block is ``good'' if 
% \[ W_{\ell-1} \leq \eta_\ell, \ \forall \ell \in (m_{i+1},m_i],\]
%and denote by $\mathscr{H}_{i,z}$ the event that the $i^{\text{th}}$ block is ``good''.   

 \begin{Pro}[Probability of ``bad'' blocks] \label{controlbadblockh}  Let $i_o\leq i \leq i_1$,  $m_{4i,z} \leq k \leq m_{i,z}$ and $r>0$. On the event where $W_{k}^z \leq r\eta_{m_i,z}$, 
\[ \PP_{k}\big(\exists \ell \in (k, m_{i,z}], W_\ell^z > 2r\eta_{m_i,z}\big) \leq e^{-\mathfrak{c}_r i^{1/6}},\]
where $\mathfrak{c}_r>0$ depends on $r$ and the model parameters. 
\end{Pro}

As a immediate corollary, we deduce that if $W_k^z$ is smaller than $\eta_{k,z}/2$ then the probability that it does not stay under the curve $\eta_z$ until time $k_{0,z}-n^{1/3} (\log n)^{\mathfrak{q}}$ is quasi-exponentially small, provided $\mathfrak{q}$ is large enough.

 \begin{Cor}[Control on the angle process]\label{controlbadblockhall}  Let $r>0$. There exists $\mathfrak{q}$  depending on the model parameters such that  for any $k_{\delta,z} \leq k \leq k_{0,z}-n^{1/3}(\log n)^{\mathfrak{q}}$,  on the event where $W^z_k \leq r\eta_{k,z}$, 
\[ \PP_k\big( \exists k <  \ell \leq k_{0,z} - n^{1/3}(\log n)^{\mathfrak{q}},   W_\ell^z >  2r\eta_{\ell,z}\big) \leq e^{- (\log n)^2}.\]
 \end{Cor}

Next, we state a representation of  $\psi(z)$ as a sum of independent random variables up to an error that is negligible when $W^z$ is small.  Before doing so, we  introduce the following convention which we will use throughout this article.
\begin{fdefinition}
\label{def-s_k}
 For a sequence of real random variables $(\xi_k)_{k}$, a deterministic $\mathfrak{h}>0$ and any deterministic sequence $(\veps_k)_{k}$,  write 
\[ s_{k-1,\mathfrak{h}}(\xi_k) \leq \veps_k,\]
if
%, \corO{for all $k\geq 1$,} 
\begin{align}
\label{eq-021025b}
& \Var_{k-1}(\xi_k):=\EE \big((\xi_k-\EE(\xi_k|\mathcal{F}_{k-1}))^2|\mathcal{F}_{k-1}\big)\leq \veps_k, \quad
\ |\xi_k- \EE_{k-1}(\xi_k)|\leq\sqrt{ \veps_k} (\log n)^{\mathfrak{h}}.
\end{align}
If $(\Xi_k)_{k}$ is a sequence of $2\times 2$ matrices, we write similarly $ s_{k-1,\mathfrak{h}}(\Xi_k)\leq \veps_k$ if
 $s_{k-1,\mathfrak{h}}(\Xi_k(i,j)) \leq \veps_k$ for all  $i,j\in \{1,2\}$.  We also write $s_{0,\mathfrak{h}}(\xi)\leq \veps$ (respectively,  $s_{0,\mathfrak{h}}(\Xi)\leq \veps$) whenever \eqref{eq-021025b} holds for $\xi$ (respectively $\Xi(i,j)$) with the subscript $k$ eliminated and with $\mathcal{F}_{k-1}$ replaced by the trivial $\sigma$-algebra.
\end{fdefinition}

With this notation, we have the following result. 
 \begin{Pro}[Contributing hyperbolic regime]\label{increm}
 There exist $\mathfrak{c},\mathfrak{h}>0$ depending on the model parameters only so that, for 
  $k_{\delta,z}\leq k \leq k_{0,z}-\ell_0$,
\[ \psi_k(z) - \psi_{k-1}(z)= g_{k,z}  +  \frac12 (W^z_k-W^z_{k-1}) + \mathcal{P}^z_k,\]
where $g_{k,z}$ is defined in \eqref{defgk} and  $\mathcal{P}_k^z$ is a $\cF_k$-measurable variable  such that 
\begin{equation} \label{momentP1prop} |\EE_{k-1} {\mathcal{P}}^z_k | \lesssim \frac{1}{\sqrt{k_{0,z}(k_{0,z}-k)}} + \frac{1}{(k_{0,z}-k)^{3/2}} + \frac{\sqrt{W_{k-1}}}{k_{0,z}-k} + (W^z_{k-1})^2 \sqrt{\frac{k_{0,z}-k}{k_{0,z}}},\end{equation}
\begin{equation} \label{momentP2prop}   s_{k-1,\mathfrak{h}}(\mathcal{P}_k^z) \lesssim   \frac{1}{(k_{0,z}-k)^2} + \frac{\sqrt{W^z_{k-1}}}{k_{0,z}-k}.\end{equation}
Moreover, $\sigma_{k,z}^2 = \frac{v}{2(k_{0,z}-k)} + O\big(\frac{1}{\sqrt{k_{0,z}(k_{0,z}-k)}}\big) + O\big(\frac{1}{n}\big)$.
\end{Pro}
Finally, we show that in the contributing hyperbolic regime, the increments of $\psi(z)$ are dominated by the sum of $g_{k,z}$ and of a variable that is negligible when $W_{k-1}^z$ is small and that has a strong negative drift when $W_{k-1}^z$ is bounded away from $0$. 
  \begin{Pro}[Domination of the increments]\label{coreqhyper}
There exist $\mathfrak{c},\mathfrak{h}>0$ depending on the model parameters only so that, for 
$z\in I_\eta$ and $k_{\delta,z} \leq k \leq k_{0,z}-\ell_0$,
%\corO{for all $z\in I_\eta$,} 
 \[ \psi_k(z) - \psi_{k-1}(z) \leq  G_{k,z}+ \frac12 (W^z_k-W^z_{k-1})- \mathfrak{c} \sqrt{\frac{k_{0,z}-k}{k_{0,z}}} (W^z_{k-1})^2  ,\]
where $G_{k,z}$ is a $\cF_k$-random variable such that 
\begin{equation} \label{momentG1}  |\EE_{k-1}(G_{k,z})|\lesssim \frac{1}{\sqrt{k_{0,z}(k_{0,z}-k)}} + \frac{1}{(k_{0,z}-k)^{3/2}} + \frac{\sqrt{W^z_{k-1}}}{k_{0,z}-k}, \ |G_{k,z}-\EE(G_{k,z})|\leq \frac{(\log n)^{\mathfrak{h}}}{\sqrt{k_{0,z}-k}},\end{equation}
\begin{equation} \label{momentG2}   \Var_{k-1}(G_{k,z})  = \frac{v}{2(k_{0,z}-k)} +O\Big( \frac{1}{(k_{0,z}-k)^{3/2}}\Big) +O\Big( \frac{(W^z_{k-1})^{1/4}}{k_{0,z}-k}\Big). \end{equation}
\end{Pro}
This result will enable us to derive  a precise upper bound on the exponential moments of the increments with any initial condition.

 \subsection{Recursion in the elliptic regime}
 In this regime where $k_{0,z} +\ell_0\leq k \leq n$, we show that with high probability the vector $Y_k^z$ rotates with essentially the same angle as $R_k^z$ over short enough blocks and that the increments of the process $\psi(z)$ over these blocks are well-approximated by a sum of independent random variables. More precisely, define the blocks 
 $(k_{i,z})_{j_o\leq i \leq j_1}$ by $k_{j_o, z} := k_{0,z}+\ell_0$, $k_{j_1, z} := n$, and
\begin{equation} \label{defblockosci} k_{i,z}:= k_{0,z}+ \lfloor i^{4} k_{0,z}^{1/3}\rfloor, \ j_o<i<j_1,\end{equation} 
where $j_o := \max\{i :   \lfloor i^{4} k_{0,z}^{1/3}\rfloor \leq  \ell_0\}$ and $j_1:=\min\{i : k_{0,z}+ \lfloor i^{4} k_{0,z}^{1/3}\rfloor \geq  n \}.$
For any $k\geq k_{0,z}+\ell_0$, denote by $\zeta_k^z$ a measure of the argument of $Y_k^z$ in $[0,2\pi)$. Next,  say that the $i^{\text{th}}$ block is {\em good} if for all $k_{i,z}\leq k< k_{i+1,z}$, 
\begin{equation}\label{goodblock}   \zeta^z_k - \zeta^z_{k_{i,z}} + \sum_{\ell=k_{i,z}+1}^k \theta^z_\ell \in [-\delta_i,\delta_i] + 2\pi \ZZ,   \end{equation}
where $\delta_i:= i^{-1/4}$. Denote also by $\mathscr{G}_{i,z}$ the event that the $i^{\text{th}}$ block is good. With this notation, we have the following proposition.
\begin{Pro}[Probability of a ``bad'' block] 
\label{probagoodblock} 
For any $i\geq \kappa^{1/4}$, 
\[  \PP_{k_{i,z}}\big(\mathscr{G}_{i,z}^\complement \big) \leq e^{-\mathfrak{c} i^{1/2}},\]
where $\mathfrak{c}$ is a positive constant depending on the model parameters.
\end{Pro}
As an immediate consequence, we can conclude by using a union bound that all the blocks starting from the $(\log n)^\mathfrak{q}$-th block on are good with overwhelming probability provided $\mathfrak{q}$ is large enough.  
\begin{Cor}\label{goodblocse}Let  $\mathfrak{q}>4$. Denote by  $\mathscr{G}_z := \bigcap_{i \geq (\log n)^{\mathfrak{q}}} \mathscr{G}_{i,z}$. Then,
\[ \PP(\mathscr{G}_z^\complement) \leq e^{-(\log n)^2},\]
where $\mathfrak{c}>0$ depends on the model parameters.
\end{Cor} 

  Next, we show a representation of the increments $\Delta \psi_{k_{i+1,z}}(z) := \psi_{k_{i+1,z}}(z) - \psi_{k_{i,z}}(z)$ as a sum of independent random variables up to some small error on the event where the $i^{\text{th}}$ block is good. This representation will be at the base of all our subsequent results in the elliptic regime.   To describe the noise appearing  in the new recursion in the elliptic regime we define the variables $c_{k,z}$ and $d_k$  as
    \begin{equation} \label{defc} c _{k,z}:=- \frac{2}{\sqrt{4-z_k^2}} \Big(  \frac{z_k(b_k-\EE(b_k))}{2\sqrt{k}} + \frac{a_{k-1}^2-\EE a_{k-1}^2}{\sqrt{k(k-1)}}\Big), \quad d_k := -\frac{b_k-\EE (b_k)}{\sqrt{k}}.\end{equation}
    With this notation, the ``instantaneous''  noise at time $k$ is given by the function $w_k^z$ defined as
    \begin{equation} \label{defw4} w^z_k(\zeta) :=  c_k^z \sin(\theta_k^z+\zeta) \cos(\zeta) + d_k\sin(\theta_k^z + \zeta) \sin(\zeta ), \quad \zeta\in \RR. \end{equation}
We are now ready to state our result in the elliptic regime.

\begin{Pro}[Representation of the increments on ``good'' blocks]
\label{represelliptic}
Let  $i\geq j_o$. On the event $\mathscr{G}_{i,z}$, 
\[ \Delta \psi_{k_{i+1},z}(z) = \sum_{k=k_{i,z}+1}^{k_{i+1,z}} \big(w_k^z(\zeta^z_{i,k-1}) + \mathcal{P}^z_k\big) + O(i^{-5/4}),\]
where $\zeta_{i,k}^z := \zeta_{k_{i,z}}^z + \sum_{\ell= k_{i,z}+1}^k \theta_\ell^z$ for any $k \in [k_{i,z},k_{i+1,z})$, $w_k^z$ is defined in \eqref{defw} and $\mathcal{P}_k^z$ is a $\cF_k$-measurable variable satisfying that 
\begin{equation} \label{momentestimPe} s_{k-1,\mathfrak{h}}(\mathcal{P}_k^z)\lesssim \frac{1}{i^{1/4}(k-k_{0,z})} + \frac{(\log n)^\mathfrak{C}}{(k-k_{0,z})^2},\end{equation}
where $\mathfrak{h},\mathfrak{C}>0$ depend on the model parameters only.
There exists a deterministic sequence $\sigma_{k,z}^2>0$, $k\geq k_{0,z}+\ell_0$ such that 
\begin{equation} \label{varcomput} \Var_{k_{i,z}}\Big(\sum_{k=k_{i,z}+1}^{k_{i+1,z}} w_k(\zeta_{i,k-1}^z) \Big) = \sum_{k=k_{i,z}+1}^{k_{i+1,z}} \sigma_{k,z}^2+ O(i^{-7}), \quad \text{a.s.,}\end{equation}
and
\begin{equation} \label{calcsigmae} \sigma_{k,z}^2 = \frac{v}{k-k_{0,z}} + O\Big(\frac{1}{\sqrt{k_{0,z}(k-k_{0,z})}}\Big).\end{equation}
\end{Pro}

\subsection{A priori exponential moment estimate}
Finally, we state a priori exponential estimates for increments of the process $\psi(z)$. \begin{Pro}\label{apriori}There exists $\mathfrak{h}>0$ depending on the model parameters,   $\mathfrak{C}_\delta >0$ depending on $\delta>0$ and $\mathfrak{C}_\kappa$ depending on $\kappa$ such that for any $1\leq \lambda \leq (\log n)^{-\mathfrak{h}}n^{1/6}$ and $0\leq k\leq k' \leq  k_{\delta,z}$, 
\[ \log \EE_k\big[e^{\lambda(\psi_{k'}(z)-\psi_k(z))} \big] \leq \mathfrak{C}_\delta  \lambda^2,\]
and for any $k_{\delta,z} \leq k \leq k' \leq n$,
\begin{equation} \label{estimaprioriexpopro} \log \EE_k\big[e^{\lambda(\psi_{k'}(z)-\psi_k(z))} \big] \leq \mathfrak{C}_\kappa  \lambda^2 \sum_{\ell = k+1}^{k'} \frac{1}{|k_{0,z}-k|\vee n^{1/3}}.\end{equation}
 Moreover, if $k' \leq k_{0,z}-\ell_0$ or $k\geq k_{0,z}+\ell_0$, $\mathfrak{C}_\kappa$ can be taken independent of $\kappa$.
Further if $k\geq  k_{0,z}-\ell_0$, then \eqref{estimaprioriexpopro} holds for $1\leq - \lambda\leq (\log n)^{-\mathfrak{h}} n^{-1/6}$.
\end{Pro} 
In particular, this result justifies the claim that the parabolic regime has only a contribution to the field $\psi(z)$  of order $1$ depending on $\kappa$.

\section{The hyperbolic regime}\label{sectionhyperbolic}

We study in this section the dynamics in the hyperbolic regime in details and give a proof of  Propositions \ref{incremhn}, \ref{controlbadblockh} and \ref{increm}. We fix in this entire section some $z\in I_\eta$ and we will  often drop the $z$-dependence in the proofs to ease the notation. As already stated, we always assume that $z>0$.

\subsection{Description of the new transition matrix} First, we compute the new transition matrix $\Xi_k^z$  of \eqref{defXi},
obtained after the change of basis $P_k^z$ defined in \eqref{changebaseP}. As a preliminary step, we estimate the norm of $P_k^z$ and its inverse and of the matrix $(P_k^z)^{-1} P_{k-1}^z -\mathrm{I}_2$, both in the negligible hyperbolic regime (until $k_{\delta,z} := k_{0,z} - \delta n$) and in the contributing hyperbolic regime (from $k_{\delta,z}$ to $k_{0,z}$). 

\begin{lemma}\label{estimQlemma} Fix $\delta>0$. For any $1\leq k < k_{0,z}$, 
\begin{equation}
\label{eq-021025a}
\|P_k^z\|\lesssim 1, \ \|(P_k^z)^{-1}\|\lesssim \sqrt{\frac{k_{0,z}}{k_{0,z}-k}}.\end{equation}
Moreover, for any $2\leq k\leq  k_{\delta,z}$, 
\begin{equation} \label{eqQhyperh} \|(P_k^z)^{-1}P_{k-1}^z-\mathrm{I}_2\|\lesssim_\delta \frac{1}{\sqrt{nk}},\end{equation}
and for $k_{\delta,z} \leq k <k_{0,z}-\ell_0$, 
\begin{equation} \label{basischangecontrib} (P_k^z)^{-1}P_{k-1}^z-\mathrm{I}_2= \frac{1}{4(k_{0,z}-k)}\begin{pmatrix} 
1 & -1 \\ -1 & 1 
\end{pmatrix} + O\Big(\frac{1}{\sqrt{k_{0,z}(k_{0,z}-k)}}\Big) + O\Big(\frac{1}{(k_{0,z}-k)^2}\Big).\end{equation}
\end{lemma}
\begin{proof} 
%Without loss of generality we can assume $z\geq 0$.
Clearly, $\|P_k\|\lesssim 1$ since $\alpha_k \geq 1$. From the comatrix formula we find that 
  \begin{equation} \label{norminvgene}\|P_k^{-1}\|\lesssim  1/(1-\alpha_k^{-2}).\end{equation} For $k\leq k_\delta$, we have that $\alpha_k\geq 1 + \mathfrak{c}_\delta$, where $\mathfrak{c}_\delta$ is a positive constant depending on $\delta$ which entails that $\|P_k^{-1}\|\lesssim_\delta 1$, completing the proof of \eqref{eq-021025a} in the case $k\leq k_{\delta}$.  Consider next the case where $k\geq k_{\delta}$. Recalling \eqref{eq-zk}, one checks that 
\begin{equation} \label{estimatezk} z_k = 2 + \frac{k_0-k}{k_0} + O\Big( \frac{1}{k_0}\Big) + O\Big( \Big( \frac{k_0-k}{k_0}\Big)^2\Big).\end{equation}
As a consequence, for $k\neq k_0$,
 \begin{equation} \label{estim1} \sqrt{|z_k^2-4|} = 2\sqrt{\frac{|k_0-k|}{k_0}} + O\Big(\frac{1}{\sqrt{k_0(k_0-k)}}\Big) + O\Big(\Big(\frac{k_0-k}{k_0}\Big)^{3/2}\Big),\end{equation}
which implies that
 \begin{equation} \label{estim} 2(|z_k^2-4|)^{-1/2} = \sqrt{\frac{k_0}{|k_0-k|}} + O\Big( \frac{k_0^{1/2}}{|k_0-k|^{3/2}}\Big) + O\Big(\Big(\frac{|k_0-k|}{k_0}\Big)^{\frac12}\Big).\end{equation}
Together with \eqref{estimatezk} and \eqref{defalpha}, this yields that for any $k_{\delta}\leq k< k_0$,
\begin{equation} \label{estimalphak} \alpha_k = 1+ \sqrt{\frac{k_0-k}{k_0}} + O\Big( \frac{k_0-k}{k_0}\Big).\end{equation}
By \eqref{norminvgene} this implies that  \eqref{eq-021025a}  holds also in case $k_0>k\geq k_{\delta}$, thus completing the proof
of \eqref{eq-021025a}.

%$\|P_k^{-1}\|\lesssim \sqrt{k_0/(k_0-k)}$. 
Next, since $\|P_k^{-1}\| \lesssim_\delta 1$ when $k\leq k_\delta$, it follows that to prove \eqref{eqQhyperh} it suffices to show that $\|P_k-P_{k-1}\|\lesssim 1/\sqrt{n k}$. 
Note that we can write $\alpha_k^{-1} = 2z_k^{-1}/(1+ (1-4z_k^{-2})^{1/2})$. Thus, when $k\leq k_\delta$, as $4z_k^{-2}$ is bounded away from $1$, $\alpha_k$ is a Lipschitz function (independent of $k$ but with Lipschitz constant depending on $\delta$) of $z_k^{-1}$. Since $|z_k^{-1}-z_{k-1}^{-1}|\lesssim 1/(z\sqrt{nk})$, this shows that $\|P_k-P_{k-1}\|\lesssim 1/\sqrt{nk}$ and ends the proof of \eqref{eqQhyperh}.

Finally we prove \eqref{basischangecontrib}.  We have for any $k_\delta \leq k\leq k_0-\ell_0$, 
\[ P_k^{-1} P_{k-1} -\mathrm{I}_2 =\frac{1}{1-\alpha_k^{-2}} \begin{pmatrix}
\alpha_k^{-2}- \alpha_k^{-1} \alpha_{k-1}^{-1} & \alpha_{k-1}^{-1}-\alpha_{k}^{-1} \\ \alpha_{k-1}^{-1} -\alpha_{k}^{-1}  &\alpha_k^{-2} -\alpha_k^{-1} \alpha_{k-1}^{-1} \end{pmatrix}\]
Let $k\geq k_\delta$, $k\neq k_0$. Then,
\begin{equation} \label{5diffz} z_{k-1}-z_k = z_k\Big(\frac{1}{\sqrt{k-1}}-\frac{1}{\sqrt{k}}\Big) = z_k \Big( \frac{1}{2k} +O\Big( \frac{1}{k^2}\Big)\Big) = z_k \Big( \frac{1}{2k_0} + O\Big( \frac{|k_0-k|}{k_0^2}\Big)\Big).\end{equation}
As $k\geq k_\delta$,  $z_k\lesssim 1$  which together with \eqref{5diffz} implies that $z_{k-1} -z_k \lesssim  \frac{1}{k_0}$.  
Further, using \eqref{estim1} and \eqref{5diffz}, it follows that 
\[ \sqrt{\frac{|z_{k-1}^2-4|}{|z_k^2-4|}} = 1 + \frac{1}{2(k_0-k)} + O\Big( \frac{1}{k_0}\Big) + O\Big( \frac{1}{(k_0-k)^2}\Big).\]
Together with \eqref{5diffz} and \eqref{estimalphak}, we find that for $k_\delta\leq k \leq k_0-\ell_0$, 
\begin{equation} \label{calculcoeffmat} \frac{\alpha_k^{-1}-\alpha_{k-1}^{-1}}{1-\alpha_k^{-2}} = \frac{1}{4(k_0-k)} + O\Big(\frac{1}{\sqrt{k_0(k_0-k)}}\Big) +O\Big(\frac{1}{(k_0-k)^2}\Big).\end{equation}
Using \eqref{estimalphak} again and \eqref{calculcoeffmat} we get that
\[ P_k^{-1} P_{k-1} -\mathrm{I}_2 = \frac{1}{4(k_0-k)}\begin{pmatrix} 
1 & -1 \\ -1 & 1 
\end{pmatrix} + O\Big(\frac{1}{\sqrt{k_0(k_0-k)}}\Big) + O\Big(\frac{1}{(k_0-k)^2}\Big),\]
where we used the fact that $k\leq k_0-\ell_0$.
\end{proof}
Equipped with Lemma \ref{estimQlemma}, we now give some properties  of  $\Xi_k^z$ from \eqref{defXi}.
\begin{lemma} \label{decommatrixtrans}
% There exists $\mathfrak{h}>0$ \corF{depending on the model parameters} so that the following holds. 
 Let $2\leq k \leq k_{0,z}-\ell_0$ and set $M_k^z : =\EE(\Xi_k^z)$, $V_k^z := \Xi_k^z-\EE(\Xi_k^z)$. Then, 
\begin{equation} \label{controlnoiseV}
\|V_k^z\| \leq  \frac{\alpha_{k,z}}{\sqrt{k_{0,z}-k}} (\log n)^{\mathfrak{h}}, \ \EE_{k-1}[\|V_k^z\|^p]\lesssim_p \frac{\alpha_{k,z}^p}{(k_{0,z}-k)^{p/2}}, \ p\geq 2, \ \text{a.s.}\end{equation}
Moreover,  $V_k^z e_1= (\alpha_{k,z} g_{k,z} , g_{k,z})^{\sf T}+ \mathcal{V}_k^z$,
where $g_{k,z}$ is defined in \eqref{defgk} and $\mathcal{V}_k^z$ is a centered random vector such that 
\begin{equation} \label{eq:boundV} s_{k-1,2}(\mathcal{V}_k^z)\lesssim_\delta \begin{cases}
\frac{\alpha_{k,z}^2}{n^2k}, &  k\leq k_{\delta,z},\\
\frac{\alpha_{k,z}^2}{(k_{0,z}-k)^3}, & k>k_{\delta,z}.
\end{cases}\end{equation}
%\[ \corO{s_{k-1,\mathfrak{h}}}(\mathcal{V}_k^z)\lesssim_\delta \left\{\begin{array}{ll}
%\frac{\alpha_{k,z}^2}{n^2k}, &\ k\leq \corO{k_{\delta,z}},\\
%% \quad  
%%\corO{s_{k-1,\mathfrak{h}}}(\mathcal{V}_k^z)\lesssim 
%\frac{1}{(k_{0,z}-k)^3}, & k>\corO{k_{\delta,z}}.
%\end{array}
%\right.\]
Further, with $D_k^z$ as in \eqref{defRk2},  if $1\leq k\leq  k_{\delta,z}$ then
\begin{equation} \label{eq1expect} M_k^z = D_k^z + O_\delta\Big(\frac{\alpha_{k,z}}{\sqrt{nk}}\Big),\end{equation}
whereas if  $k_{\delta,z} \leq k \leq k_{0,z}-\ell_0$ then
\begin{equation} \label{eq2expect} M_k^z = D_k^z + \frac{1}{4(k_{0,z}-k)}\begin{pmatrix}  1 & -1 \\ -1 & 1 \end{pmatrix} + O\Big(\frac{1}{\sqrt{k_{0,z}(k_{0,z}-k)}}\Big) + O\Big(\frac{1}{(k_{0,z}-k)^2}\Big).\end{equation}
\end{lemma}
\begin{proof} Recall that $\Xi_k=P_k^{-1} T_k P_{k-1}$, where $T_k$ is from \eqref{deftransition}, and $P_k$ is defined in \eqref{changebaseP}. Using that $\|P_k^{-1}\|\lesssim \sqrt{k_0/(k_0-k)}$,  $\|P_{k-1}\|\lesssim 1$ by Lemma \ref{estimQlemma}, and that $s_{k-1,2}(T_k)\lesssim 1/k$ by Assumption \ref{ass1}, it follows that $s_{k-1,2}(V_k)\lesssim k_0/(k(k_0-k))$. Since $\alpha_{k}^2 \gtrsim |z_k|^2\gtrsim n/k$ and $k_0\asymp n$, we get that $s_{k-1,2}(V_k)\lesssim \alpha_k^2/(k_0-k)$. Since the noise $b_k$ and $(a_{k-1}^2-\EE(a_{k-1}^2))/\sqrt{k}$ satisfy the integrability assumption  \eqref{boundlaplace}, for any $p\geq 2$ their $p^{\text{th}}$-moments are uniformly bounded in $k$. Hence, we get that $\EE[\|T_k-\EE T_k\|^p] \lesssim 1/k^{p/2}$ for any $p\geq 2$. Using again the norm estimates of $P_k^{-1}$ and $P_{k-1}$ from  Lemma \ref{estimQlemma}, the second inequality in \eqref{controlnoiseV} follows.

We now move on to the proof of  \eqref{eq:boundV}.  Write  $\widetilde{V}_k := P_k^{-1} (T_k- \EE(T_k)) P_k$. One can check that $\widetilde{V}_ke_1 =( \alpha_k^{-1} g_k, g_k)^{\sf T}$. Hence, $\mathcal{V}_k =V_ke_1 - \widetilde{V}_k e_1$.   Note that $V_k-\widetilde{V}_k = \widetilde{V}_k(P_k^{-1}P_{k-1}-\mathrm{I}_2)$.
With the same argument as for $V_k$, we have that $s_{k-1,2}(\widetilde{V}_k) \lesssim \alpha_k^2/(k_0-k)$. Now, 
if $k\leq k_{\delta}$, then since $\|P_k^{-1}P_{k-1} - \mathrm{I}_2\|\lesssim_\delta 1/\sqrt{nk}$ and $k_0 -k\asymp  n$ we get that $s_{k-1,2}(\mathcal{V}_k) \lesssim_\delta \alpha_k^2/(n^2 k)$. If on the other hand $k>k_\delta$, then as  $\|P_k^{-1}P_{k-1} - \mathrm{I}_2\|\lesssim 1/(k_0-k)$ by Lemma \ref{estimQlemma}, we get that $s_{k-1,2}(\mathcal{V}_k) \lesssim \alpha_k^2/(k_0-k)^3$.

We now prove \eqref{eq1expect} and \eqref{eq2expect}. Since $P_k^{-1} A_k P_k = D_k$, see 
%where $D_k$ is defined in \
\eqref{defRk2}, we can write that 
\[ M_k = D_k + P_k^{-1}(\EE(T_k)-A_k)P_{k-1} + D_k (P_k^{-1}P_{k-1}-\mathrm{I}_2).\]
From our Assumptions \ref{ass1}, we have that $\|\EE(T_k) -A_k\|\lesssim 1/ k$. If $k\leq k_\delta$, then as $\|P_k^{-1}\|\lesssim_\delta 1$ and $\|P_{k-1}\|\lesssim 1$ by Lemma \ref{estimQlemma}, we deduce that $\|P_k^{-1}(\EE(T_k)-A_k)P_{k-1}\|\lesssim_\delta 1/k \lesssim \alpha_k/\sqrt{nk}$. Moreover, again by Lemma \ref{estimQlemma}, we get that $\| D_k (P_k^{-1}P_{k-1} -\mathrm{I}_2)\|\lesssim_\delta \alpha_k/\sqrt{nk}$ if $k\leq k_\delta$. 
%Further, one can check that using the fact that $-\alpha_k^{-1}z +\alpha_k^{-2}+1=0$, we have that 
%\[ M_k(2,1) = \frac{\alpha_k^{-1}}{1-\alpha_k^{-2}}\big(\alpha_{k-1}^{-1}-\alpha_k^{-1} + \EE(\widetilde{b}_k) +  \alpha_{k-1}^{-1}\EE(\widetilde{c}_{k})\big),\]
%where $\widetilde{b}_k := b_k/\sqrt{k}$ and $\widetilde{c}_{k} := a_{k-1}^2/\sqrt{k(k-1)}-1$. Since $\alpha_k^{-1} \lesssim \sqrt{k/n}$, $1-\alpha_k^{-2}\gtrsim_\delta 1$, $|\EE(\widetilde{b}_k)|\lesssim 1/\sqrt{nk}$, $|\EE(\widetilde{c}_k)|\lesssim 1/k$ and $|\alpha_k^{-1}-\alpha_{k-1}^{-1}|\lesssim_\delta |z_k^{-1}-z_{k-1}^{-1}|\lesssim 1/\sqrt{nk}$ by the same argument as in the proof of \eqref{eqQhyperh}, we get the second claim of \eqref{eq1expect}.

Finally, assume $k\geq k_{\delta}$. Using the estimates on the norms of $P_{k}^{-1}$ and $P_{k-1}$ from Lemma \ref{estimQlemma} and the fact that $\|\EE(T_k) -A_k\|\lesssim 1/ k$, we obtain that 
\begin{equation} \label{estimesph} \|P_k^{-1}(\EE(T_k)-A_k)P_{k-1}\|\lesssim \frac{1}{k} \sqrt{\frac{k_0}{k_0-k}} \lesssim \frac{1}{\sqrt{k_0(k_0-k)}}.\end{equation}
Using that $\|D_k-\mathrm{I}_2\|\lesssim \sqrt{(k_0-k)/k_0}$ by \eqref{estimalphak} and Lemma \ref{estimQlemma}, we get that 
\[ D_k (P_k^{-1}P_{k-1} -\mathrm{I}_2) = P_k^{-1}P_{k-1} - \mathrm{I}_2 + O\Big(\frac{1}{\sqrt{k_0(k_0-k)}}\Big).\]
 Together with \eqref{basischangecontrib}, this ends the proof of \eqref{eq2expect}.
\end{proof}

\subsection{A linearization lemma}
Given a random matrix $\Xi$ satisfying certain moment bounds, we show that $\log \|\Xi u\|^2$, where $u$ is a deterministic unit norm vector, can be linearized up to a small error. This general result will be instrumental in proving Propositions  \ref{incremhn} and \ref{increm}. Recall the notation $s_0$ from Definition \ref{def-s_k}.

\begin{lemma}
\label{lemgene} Let 
%$(M_k)_{1\leq k< k_0}$ be a sequence of 
$M$ be a deterministic $2\times 2$ matrix and assume 
%$(V_k)_{1\leq k< k_0}$ a sequence of 
that $V$ is a 
%independent 
centered random $2\times 2$ matrix.   Assume that for some constants $\mu\geq 1, \mathfrak{h}>0$,
\begin{enumerate}
\item $\|M\|\lesssim 1$.
\item  $\mu \cdot s_{0,\mathfrak{h}}(V)\lesssim 1$,  and  $\mu^{p/2}\EE[\|V\|^p]\lesssim_p 1$ for any $p\geq 2$.
\end{enumerate}
(Here and in the rest of the statements, the implicit constants do not depend on $\mu$ or $\mathfrak{h}$).
Let $\Xi := M+V$,   $u\in \mathbb{S}^1$ be deterministic,  and write $\Lambda := \|Mu\|$. If $\Lambda \gtrsim 1$, then 
\begin{align*}
 \log\big[\Lambda^{-2}\|\Xi u\|^2\big] &= W- \Lambda^{-2}\langle Mu,e_2\rangle^2 + 2\Lambda^{-2}\langle Mu,e_1\rangle\langle Vu,e_1 \rangle \\
&+ \Lambda^{-2}\EE[\langle V u, e_1\rangle^2]- 2\Lambda^{-4}\EE[\langle Mu, Vu\rangle^2]  +\mathcal{X},
\end{align*}
where $W=W(\Xi,u): = \langle \Xi u, e_2\rangle^2/\|\Xi u\|^2$ and $\mathcal{X}$ satisfies that
\[ |\EE(\mathcal{X})|\lesssim \frac{1}{\mu^{3/2}} + \frac{|\langle M  u, e_2\rangle|}{\mu},  \ s_{0,\mathfrak{h}}(\mathcal{X}) \lesssim \frac{1}{\mu^2}+ \frac{|\langle M u, e_2\rangle|}{\mu}.\]
%again with implicit constants independent of $\mu$.
Moreover,
\[W = \Lambda^{-2} \langle M u, e_2\rangle^2 + \Lambda^{-2}\EE[\langle V u,e_2\rangle^2 ]+ \mathcal{S},\]
where $\mathcal{S}$ satisfies that, for some $\mathfrak{h}'=\mathfrak{h}'(\mathfrak{h})>0$, 
\[ |\EE(\mathcal{S})|\lesssim \frac{|\langle M u, e_2\rangle|}{\mu}+ \frac{1}{\mu^{3/2}}, \  s_{0,\mathfrak{h}'}(\mathcal{S})\lesssim  \frac{\langle M u, e_2\rangle^2}{\mu}+ \frac{1}{\mu^2}.\]
%again with implicit constants independent of $\mu$.}
\end{lemma}

\begin{proof}
Expanding the norm, we get that 
\begin{equation} \label{eqnorm1} \Lambda^{-2}\|\Xi u\|^2 = 1 + 2\Lambda^{-2}\langle M u, V u\rangle + \Lambda^{-2}\|V u\|^2.\end{equation}
Taking the logarithm and Taylor expanding up to the second order, we find that 
\begin{equation} \label{eqnormg} \log\big[\Lambda^{-2}\|\Xi u\|^2\big]  = 2\Lambda^{-2}\langle M u, Vu\rangle +\Lambda^{-2}\|V u\|^2 - 2\Lambda^{-4} \langle Mu, Vu\rangle^2 + \mathcal{T},\end{equation}
where $\mathcal{T}$ is given by
\[ \mathcal{T} : = -2\Lambda^{-4}\langle M u, V u\rangle\|Vu\|^2 -\frac{\Lambda^{-4}}{2}\|V u\|^4  +O\Big(\big|2\Lambda^{-2}\langle M u, V u\rangle + \Lambda^{-2}\|Vu\|^2\big|^3\Big).\]
Since $\Lambda \gtrsim 1$, $\|M\|\lesssim 1$ and $\EE[\|V\|^p]\lesssim_p \mu^{-p/2}$ for any $p\geq 2$, we get that 
\[ |\EE(\mathcal{T})|\lesssim \frac{1}{\mu^{3/2}}, \ s_{0,6\mathfrak{h}}(\mathcal{T}) \lesssim \frac{1}{\mu^{3}}. \]
Besides, expanding $ W= (\Lambda^{-2} + \|\Xi u\|^{-2} -\Lambda^{-2})\langle (M+V) u,e_2\rangle^2$, we have that
\begin{equation} \label{eqWg1} W = \Lambda^{-2} \langle M u, e_2\rangle^2 +2 \Lambda^{-2}\langle M u, e_2\rangle \langle V u, e_2\rangle + \Lambda^{-2}\langle V u,e_2\rangle^2 + \mathcal{S}\end{equation}
where $\mathcal{S}$ is given by
\[ \mathcal{S} := \Big( \|\Xi u\|^{-2} -\Lambda^{-2}\Big) \Big( \langle M u, e_2\rangle^2 + 2\langle M u, e_2\rangle \langle V u, e_2\rangle + \langle Vu,e_2\rangle^2\Big). \]
By Taylor's theorem and \eqref{eqnorm1}, we obtain, denoting by $\mathcal{L} := \|\Xi u \|^{-2} -\Lambda^{-2}$, that
\[ |\EE(\mathcal{L})|\lesssim \frac{1}{\mu}, \ s_{0,4\mathfrak{h}}(\mathcal{L})\lesssim \frac{1}{\mu}.\]
Thus, using the facts that $s_{0,\mathfrak{h}}(V)\lesssim 1/\mu$, that $\EE[\|V\|^p]\lesssim_p \mu^{-p/2}$ for any $p\geq 2$
and that
$\|M\|\lesssim 1$, we find that 
\[ |\EE(\mathcal{S})|\lesssim \frac{|\langle M u, e_2\rangle|}{\mu}+ \frac{1}{\mu^{3/2}}, \ s_{0,6\mathfrak{h}}(\mathcal{S})\lesssim  \frac{\langle M u, e_2\rangle^2}{\mu}+ \frac{1}{\mu^2}\]
Subtracting \eqref{eqnormg}  and \eqref{eqWg1}  we  conclude that 
\begin{align*}
 \log\big[\Lambda^{-2}\|\Xi u\|^2\big]  - W &=  - \Lambda^{-2}\langle Mu,e_2\rangle^2 + 2\Lambda^{-2}\langle Mu,e_1\rangle\langle Vu,e_1 \rangle \\
&+ \Lambda^{-2}\langle V u, e_1\rangle^2 - 2\Lambda^{-4}\langle Mu, Vu\rangle^2  +\mathcal{T} - \mathcal{S}.\end{align*}
Finally, using the fact that $\mu^2 \EE[\|V\|^4] \lesssim 1$ and that $\mu s_{0,\mathfrak{h}}(V) \lesssim 1$, we deduce that 
\[ \mu^2 s_{0,2\mathfrak{h}}(\langle V u, e_1\rangle^2) \lesssim 1, \ \mu^2 s_{0,2\mathfrak{h}}(\langle Mu, Vu\rangle^2) \lesssim 1.\]
Thus, replacing $\langle V u, e_1\rangle^2$ and $\langle Mu, Vu\rangle^2$ by their expectations and gathering the noise terms as well as the variables $\mathcal{T}$, $\mathcal{S}$ into the error term $\mathcal{X}$ ends the proof.
\end{proof}
\subsection{Dynamics of the logarithm of the norm}  
 Equipped with the properties of $\Xi_k$
shown in Lemma \ref{decommatrixtrans} and the general linearization result of Lemma \ref{lemgene}, we can describe the dynamics of $\log \|Y_k^z\|$,  in the negligible and contributing hyperbolic regimes. We first start by the negligible regime. 
 Recall the variable $W_k^z$, see \eqref{def-Wk}.

\begin{Pro}\label{rechn}Let $2\leq k \leq k_{\delta, z}$.  There exists $\mathfrak{h}>0$ depending on the model parameters only such that on the event where $W_{k-1}^z \leq 1/2$,
\[  \log\Big[ \frac{\|Y_k^z\|}{\alpha_{k,z}\|Y_{k-1}^z\|} \Big]  = g_{k,z} +\frac12(W_k^z-W_{k-1}^z) + \mathcal{X}_k^z,\]
where $g_{k,z}$ is defined in \eqref{defgk} and $\mathcal{X}_k^z$ is a $\cF_k$-measurable variable such that 
\[ \big|\EE_{k-1} \mathcal{X}_k\big| \lesssim W_{k-1}^2 +  \frac{1}{\sqrt{nk}}, \ s_{k-1,\mathfrak{h}}(\mathcal{X}_k) \lesssim \frac{1}{n^2}+ \frac{W_{k-1}}
{n}.\]
Moreover, $\EE[g_k^2] \lesssim_\delta 1/n$.

\end{Pro}
\begin{proof} 
We start by computing the variance of $g_k$. Using that $1-\alpha_k^{-2} \gtrsim_\delta 1$, $\alpha_k^{-1} \lesssim 1/|z_k|\lesssim \sqrt{k/n}$ and the fact that $b_k$ and $a_{k-1}$ are independent with $\Var(b_k) \lesssim 1$, $\Var(a_{k-1}^2) \lesssim k$, we find that  $\EE[g_k^2] \lesssim_\delta  1/n$.

Denote by $M_k = \EE (\Xi_k)$, $V_k = \Xi_k-M_k$, $U_{k-1} = Y_{k-1}/\|Y_{k-1}\|$ and $\Lambda_k = \| M_k U_{k-1}\|$. Assume that $W_{k-1}\leq 1/2$. By  Lemma \ref{decommatrixtrans}, $M_k/\alpha_k$ and $V_k/\alpha_k$ satisfy the assumptions of the linearization Lemma \ref{lemgene} with $\mu =n$. For the application of the latter lemma, we also need to 
%It remains to 
show that $\Lambda_k \gtrsim \alpha_k$. Using \eqref{eq1expect} we find that 
\begin{equation} \label{devM} \Lambda_k^2 = \|D_kU_{k-1}\|^2 + O\Big(\alpha_{k}^2/\sqrt{nk}\Big).\end{equation}
Since 
\begin{equation}
\label{eq-021025c}
 \|D_kU_{k-1}\|^2 = \alpha_k^2 -(\alpha_k^2-\alpha_k^{-2}) W_{k-1}\geq \alpha_k^2-\alpha_k^2 W_{k-1}\geq \alpha_k^2/2 
\quad \mbox{\rm  when $W_{k-1}\leq 1/2$},
\end{equation}
 we have from \eqref{devM}
that on the latter event, $\Lambda_k\gtrsim \alpha_k$ for all $n$ large enough. Before applying Lemma \ref{lemgene}, note that since $\alpha_k\geq 1$ and $\alpha_k \lesssim  |z_k|\lesssim \sqrt{n/k}$,  we have using \eqref{eq1expect} that $|\langle M_kU_{k-1},e_2\rangle|\lesssim \sqrt{W_{k-1}}+1/n$. 
Hence, using Lemma \ref{lemgene} with $u= U_{k-1}$, $M=M_k/\alpha_k$, $V=V_k/\alpha_k$, $s_{k-1}$ instead of $s_0$
 and $\EE$ replaced by $\EE_{k-1}$, and Lemma  \ref{decommatrixtrans} to estimate the expectations of the quadratic terms in the noise appearing in the linearization, we find that  
\[ \log \Big[\Lambda_k^{-2} \frac{\|Y_k\|^2}{\|Y_{k-1}\|^2}\Big] = W_k  - \Lambda_k^{-2} \langle M_kU_{k-1},e_2\rangle^2 + 2\Lambda_k^{-2} \langle M_k u, e_1\rangle \langle V_k U_{k-1},e_1\rangle + \mathcal{X}_k,\]
where 
\[ |\EE_{k-1} (\mathcal{X}_k)|\lesssim_\delta \frac{1}{n}, \ s_{k-1,\mathfrak{h}}(\mathcal{X}_k) \lesssim_\delta  \frac{1}{n^2} + \frac{W_{k-1}}{n},\]
and $\mathfrak{h}>0$ depends on the model parameters.
Using \eqref{devM} and the fact that $\|D_kU_{k-1}\|\gtrsim \alpha_k$ we get that  $\Lambda_k^{-2} = \|D_kU_{k-1}\|^{-2} + O\Big(1/(\alpha_k^{2}\sqrt{nk})\Big)$, which implies together with \eqref{eq-021025c} that 
\begin{equation}\label{Lambdakhn} \Lambda_k^{-2} - \alpha_k^{-2} = O\big(W_{k-1}/\alpha_k^2\big) + O\Big(1/(\alpha_k^{2}\sqrt{nk})\Big).\end{equation} 
Using \eqref{Lambdakhn} and in addition the facts that $\|M_k-D_k\|\lesssim \alpha_k/\sqrt{nk}$ and $s_{k-1,2}(V_k) \lesssim \alpha_k^2/n$ by Lemma  \ref{decommatrixtrans}, we deduce that 
\begin{equation} \label{noisehypern}   2\Lambda_k^{-2} \langle M_k U_{k-1}, e_1\rangle \langle V_k U_{k-1},e_1\rangle  = 2\alpha_k^{-1}\langle U_{k-1},e_1\rangle \langle V_k U_{k-1},e_1\rangle  + \mathcal{Y}_k, \end{equation}
where $\EE_{k-1}(\mathcal{Y}_k) = 0$ and $s_{k-1,2}(\mathcal{Y}_k)\lesssim W_{k-1}/n + 1/n^2$. Using that $V_k(1,1)= \alpha_kg_k + \mathcal{V}_k$ where $s_{k-1,2}(\mathcal{V}_k) \lesssim \alpha_k^2/(n^2k)$, $\EE[\|V_k\|^2]\lesssim \alpha_k^2/n$ by  Lemma \ref{decommatrixtrans}   we find that 
\begin{equation} \label{noisehypern1} 2\alpha_k^{-1}\langle U_{k-1},e_1\rangle \langle V_k U_{k-1},e_1\rangle = 2g_k +\mathcal{Z}_k,\end{equation}
where $\mathcal{Z}_k$ is a centered and  $s_{k-1,2}(\mathcal{Z}_k) \lesssim 1/(n^2k)+W_{k-1}/{n}$.  Finally, we compute the contribution of $\log(\Lambda_k^2/\alpha_k^2) - \Lambda_k^{-2}  \langle M_k U_{k-1},e_2\rangle^2$.
Note that  \eqref{devM}-\eqref{eq-021025c} implies that $\alpha_k^{-2} \Lambda_k^2= 1-(1-\alpha_k^{-4})W_{k-1} +O(1/\sqrt{nk})$. Hence, Taylor expanding the logarithm to the second order we find that 
\begin{equation}\label{devloghn}  \log(\Lambda_k^2/\alpha_k^2) = -(1-\alpha_k^{-4}W_{k-1}) + O(W_{k-1}^2) + O(1/\sqrt{nk}). \end{equation}
Moreover, using \eqref{eq1expect}  we deduce that $ \alpha_k^{-2} \langle M_k U_{k-1},e_2\rangle^2 = \alpha_k^{-4}W_{k-1} + O(1/n)$. Together with \eqref{Lambdakhn}, this implies that
\begin{equation} \label{devM2hn}
\Lambda_k^{-2}  \langle M_k U_{k-1},e_2\rangle^2  = \alpha_k^{-4}W_{k-1}+ O(W_{k-1}^2) +O\Big(1/\sqrt{nk}\Big),
\end{equation}
which finally ends the proof.
\end{proof}
In the contributing regime, the dynamics of $\log \|Y_k^z\|^2$ is described by the following equation.
\begin{Pro}\label{rech}  There exists $\mathfrak{h}>0$ so that the following holds for all $z\in I_\eta$.
Let $k_{\delta,z} \leq k \leq k_{0,z}-\ell_0$. Denote by $U_k^z := Y_k^z/\|Y_k^z\|$. Then, 
\begin{align*}  \log\Big[ \frac{\|Y_k^z\|}{\alpha_{k,z}\|Y_{k-1}^z\|} \Big] & =\frac12(W^z_k-W^z_{k-1}) + g_{k,z}  - \frac{v-1}{4(k_{0,z}-k)}\\
& + \frac12\Big(\log \Big(\frac{\|D_k^zU_k^z\|^2}{\alpha_{k,z}^2}\Big) +\Big(1- \frac{1}{\alpha_{k,z}^2\|D_k^z U_k^z\|^2} \Big)W^z_{k-1}\Big) + \mathcal{X}^z_k,
\end{align*}
where $g_{k,z}$ is defined in \eqref{defgk}
and $\mathcal{X}_k^z$ is a $\cF_k$-measurable variable such that 
\begin{equation} \label{rechmoment} |\EE_{k-1}(\mathcal{X}_k^z)|\lesssim \frac{1}{\sqrt{k_{0,z}(k_{0,z}-k)}} + \frac{1}{(k_{0,z}-k)^{3/2}} + \frac{\sqrt{W^z_{k-1}}}{k_{0,z}-k},\ s_{k-1,\mathfrak{h}}(\mathcal{X}_k^z)\lesssim \frac{1}{(k_{0,z}-k)^2} + \frac{\sqrt{W^z_{k-1}}}{k_0-k}.\end{equation}
Moreover, $\EE[g_{k,z}^2] = v/(2(k_0-k)) + O\big(1/\sqrt{k_0(k_0-k)}\big)$, $|g_{k,z}|\lesssim  (\log n)^{2}/\sqrt{k_{0,z}-k}$ \corF{a.s.}
%, where $\mathfrak{C}>0$ depends on the model parameters. 
\end{Pro}
Contrary to the negligible regime, we see two drift terms appear. The first one, $-(v-1)/(2(k_{0,z}-k))$, is  due to the $k$-dependent change of basis $P_k^z$ and to the quadratic terms involving the noise. In addition, one gets another drift term,
namely $\log (\|D_k^z U_k^z\|^2/\alpha_{k,z}^2) + (1-1/[\alpha_{k,z}^2\|D_k^z U_k^z\|^2) W_{k-1}^z$, that
 depends on the direction of $Y_{k-1}^z$. We will later show that this second drift term
is negligible with high probability.
%  which we will show later to be  negligible with high probability.

\begin{proof} We start by computing the variance of $g_k$. By \eqref{estimalphak} we know that $\alpha_k = 1 + O(\sqrt{(k_0-k)/k_0})$ and moreover, $1-\alpha_k^{-2}= 2\sqrt{(k_0-k)/k_0} + O((k_0-k)/k_0)$. Together with the fact that $b_k$ and $a_{k-1}$ are independent, that $k\gtrsim n$, and that $\Var(b_k) = v + O(1/k)$, $\Var(a_{k-1}^2) = kv + O(1)$, we find that 
\begin{equation} \label{vargproof} \EE[g_k^2]= \frac{v}{2(k_0-k)} \frac{k_0}{k} + O\Big(\frac{1}{\sqrt{k_0(k_0-k)}}\Big) = \frac{v}{2(k_0-k)} + O\Big(\frac{1}{\sqrt{k_0(k_0-k)}}\Big).\end{equation}
%\corF{where we used in addition that $k\gtrsim n$.}
Moreover, since $|b_k|\lesssim (\log n)^2$ and $|a_{k-1}^2-\EE(a_{k-1}^2)|\lesssim (\log n)^2$ a.s., we deduce similarly that $|g_k|\lesssim (\log n)^2/\sqrt{k_{0,z}-k}$ a.s.

Next,  denote by $U_{k-1} = Y_{k-1}/\|Y_{k-1}\|$, $M_k = \EE(\Xi_k)$, $V_k = \Xi_k-M_k$ and $\Lambda_k = \|M_kU_{k-1}\|$. By Lemma \ref{decommatrixtrans}, $M_k$ and $V_k$ satisfy the assumptions of Lemma \ref{lemgene}, with $\mu=k_0-k$. We now check that $\Lambda_k \gtrsim 1$. Indeed, using \eqref{eq2expect} we have that
\begin{equation} \label{eqLambda} \Lambda_k^2 =  \|D_k U_{k-1}\|^2 + O\Big(\frac{1}{k_0-k}\Big) = \alpha_k^2 - (\alpha_k^2-\alpha_k^{-2})W_{k-1}  + O\Big(\frac{1}{k_0-k}\Big).\end{equation}
By  \eqref{estimalphak} we know that $\alpha_k^2-\alpha_k^{-2} \lesssim \sqrt{(k_0-k)/k_0}\lesssim \sqrt{\delta}$. As $W_k \in [0,1]$ and $\alpha_k\geq 1$, we deduce that for $\delta$ small enough $\Lambda_k^2\gtrsim 1$. 
Further,  using again \eqref{eq2expect} we have that $|\langle M_k U_{k-1},e_2\rangle|\lesssim \sqrt{W_{k-1}} + 1/(k_0-k)$. 
Applying Lemma \ref{lemgene} with  $u= U_{k-1}$ and using $\EE_{k-1}$ as the expectation, we get that 
\begin{align*} \log \Big[\frac{\|Y_k\|^2}{\Lambda_k^{2} \|Y_{k-1}\|^2}\Big] &= W_k- \Lambda_k^{-2}\langle M_kU_{k-1},e_2\rangle^2 + 2\Lambda_k^{-2}\langle M_kU_{k-1},e_1\rangle\langle V_kU_{k-1},e_1 \rangle \\
&+ \Lambda_k^{-2}\EE[\langle V_k U_{k-1}, e_1\rangle^2]- 2\Lambda_k^{-4}\EE[\langle M_kU_{k-1}, V_kU_{k-1}\rangle^2]  +\mathcal{X}_k,
\end{align*}
where $\mathcal{X}_k$ satisfies that
\[ |\EE_{k-1}(\mathcal{X}_k)|\lesssim \frac{1}{(k_0-k)^{3/2}} + \frac{\sqrt{W_{k-1}}}{k_0-k},  \ s_{k-1,\mathfrak{h}}(\mathcal{X}_k) \lesssim \frac{1}{(k_0-k)^2}+ \frac{\sqrt{W_{k-1}}}{k_0-k},\]
where $\mathfrak{h}>0$ depends on the model parameters.
We first compute the noise term in this linearization. Note that since $\alpha_k^2-\alpha_k^{-2}\lesssim \sqrt{(k_0-k)/k_0}$ by \eqref{estimalphak}, we obtain from \eqref{eqLambda}  that
\begin{equation} \label{estimlambdak}\Lambda_k^{-2} = \alpha_k^{-2} + O\Big(W_{k-1}\sqrt{\frac{k_0-k}{k_0}}\Big)+O\Big(\frac{1}{k_0-k}\Big).\end{equation} 
%\corF{O\Big(\frac{1}{k_0-k}\Big)}.\end{equation} 
Moreover, again by \eqref{eq2expect}, $ \langle M_kU_{k-1}, e_1\rangle = \alpha_k \langle U_{k-1},e_1\rangle + O\Big(\frac{1}{k_0-k}\Big)$.
As a result, using in addition that $\EE_{k-1}[\|V_k\|^2] \lesssim 1/(k_0-k)$ and that $V_k(1,1) = \alpha_k g_k +\mathcal{V}_k$ where
%$\corF{s_{k-1,2}}(\mathcal{V}_k) \lesssim 1/(k_0-k)^3$ \corF{by Lemma  \ref{decommatrixtrans}}, we conclude  that 
$s_{k-1,2}(\mathcal{V}_k) \lesssim 1/(k_0-k)^3$ by Lemma  \ref{decommatrixtrans}, we conclude  that
\begin{equation} \label{noisetermrec}  2\Lambda_k^{-2}\langle M_kU_{k-1},e_1\rangle\langle V_kU_{k-1},e_1 \rangle =  2g_k + \mathcal{Y}_k,\end{equation}
where $\mathcal{Y}_k$ is $\cF_k$-measurable and 
\[ \EE_{k-1}(\mathcal{Y}_k) =0, \ s_{k-1,2}(\mathcal{Y}_k) \lesssim \frac{W_{k-1}}{k_0-k} + \frac{1}{(k_0-k)^3}.\]
Next, we turn our attention to the drift terms.  We compute first the quadratic terms involving the noise. Using that $\EE_{k-1}[\|V_k\|^2] \lesssim 1/(k_0-k)$ by Lemma \ref{decommatrixtrans}, we find that
\begin{equation} \label{driftterm1}
 \alpha_k^{-2}\EE_{k-1}[\langle V_kU_{k-1},e_1\rangle^2] =   \alpha_k^{-2}\EE[V_k(1,1)^2] + O\Big(\frac{\sqrt{W_{k-1}}}{k_0-k}\Big).
\end{equation}
But $V_k(1,1) = \alpha_k g_k +\mathcal{V}_k$ where $s_{k-1,2}(\mathcal{V}_k) \lesssim 1/(k_0-k)^3$ by Lemma \ref{decommatrixtrans}. Together with the computation of the variance of $g_k$ in \eqref{vargproof}, this entails that 
\begin{equation} \label{driftterm2} \alpha_k^{-2}\EE[V_k(1,1)^2]  = \frac{v}{2(k_0-k)} + O\Big(\frac{1}{\sqrt{k_0(k_0-k)}}\Big)  + O\Big(\frac{1}{(k_0-k)^2}\Big).\end{equation}
Using the development for $\Lambda_k^{-2}$ in \eqref{estimlambdak} and putting together \eqref{driftterm1} and \eqref{driftterm2}, we conclude that
\begin{equation} \label{driftnoise1}   \Lambda_k^{-2}\EE_{k-1}[\langle V_kU_{k-1},e_1\rangle^2]  = \frac{v}{2(k_0-k)} +O\Big(\frac{\sqrt{W_{k-1}}}{k_0-k}\Big) + O\Big(\frac{1}{\sqrt{k_0(k_0-k)}}\Big) + O\Big(\frac{1}{(k_0-k)^2}\Big).\end{equation}
Similarly, we compute the contribution of the second drift term:
\begin{align*}
 \Lambda_k^{-4}\EE_{k-1}[\langle M_k U_{k-1},V_k U_{k-1}\rangle^2] & = \frac{v}{2(k_0-k)} + O\Big(\frac{1}{\sqrt{k_0(k_0-k)}}\Big) + O\Big(\frac{1}{(k_0-k)^2}\Big) + O\Big(\frac{\sqrt{W_{k-1}}}{k_0-k}\Big).
 \end{align*}
We now turn our attention to the drift term coming from the change of basis. We first compute a more precise development for $\Lambda_k^2$. 
Writing $E_k = M_k-D_k$, we have by Lemma \ref{decommatrixtrans} that $\|E_k\|\lesssim 1/(k_0-k)$. Since $\|D_k-\mathrm{I}_2\|\lesssim \sqrt{(k_0-k)/k_0}$ by \eqref{eq2expect}, we deduce that 
\begin{align*}
 \Lambda_k^2
& = \|D_kU_{k-1}\|^2 + 2\langle U_{k-1},E_k U_{k-1}\rangle + O\Big(\frac{1}{(k_0-k)^2}\Big) +O\Big(\frac{1}{\sqrt{k_0(k_0-k)}}\Big).
\end{align*}
Using the description of $E_k$ from \eqref{eq2expect}, it follows that 
\begin{equation} \label{drift2}  \Lambda_k^2  = \|D_k U_{k-1}\|^2+ \frac{1}{2(k_0-k)} + O\Big(\frac{\sqrt{W_{k-1}}}{k_0-k}\Big)+O\Big(\frac{1}{\sqrt{k_0(k_0-k)}}\Big).\end{equation}
Moreover, as $M_k=D_k+O(1/(k_0-k))$, we find that $\langle M_kU_{k-1},e_2\rangle^2 = \alpha_k^{-2}W_{k-1} + \sqrt{W_{k-1}}/(k_0-k) + 1/(k_0-k)^2$.
%Note that $|\|D_kU_{k-1}\|^2 -\alpha_k^{2}| \lesssim \sqrt{(k_0-k)/k_0}$ since $\alpha_k^2-\alpha_k^{-2}\lesssim \sqrt{(k_0-k)/k_0}$ by \eqref{estimalphak}. $
Since $\Lambda_k^2 = \|D_kU_{k-1}\|^2 + O(1/(k_0-k))$ by  \eqref{drift2},
 we obtain by Taylor's theorem that
\iffalse
\corF{Moreover, as $M_k=D_k+O(1/(k_0-k))$, we find that $\langle M_kU_{k-1},e_2\rangle^2 = \alpha_k^{-2}W_{k-1} + \sqrt{W_{k-1}}/(k_0-k) + 1/(k_0-k)^2$}.
%Note that $|\|D_kU_{k-1}\|^2 -\alpha_k^{2}| \lesssim \sqrt{(k_0-k)/k_0}$ since $\alpha_k^2-\alpha_k^{-2}\lesssim \sqrt{(k_0-k)/k_0}$ by \eqref{estimalphak}. $
\corF{Since $\Lambda_k^2 = \|D_kU_{k-1}\|^2 + O(1/(k_0-k))$ by  \eqref{drift2},
 we obtain by Taylor's theorem} that
\fi
\begin{equation} \label{meanhrec}\Lambda_k^{-2} \langle M_k U_{k-1},e_2\rangle^2 = \frac{\alpha_k^{-2} W_{k-1}}{\|D_kU_{k-1}\|^2} + O\Big(\frac{\sqrt{W_{k-1}}}{k_0-k}\Big) + O\Big(\frac{1}{\sqrt{k_0(k_0-k)}}\Big). \end{equation}
Further, using the fact that $\|D_kU_{k-1}\| \gtrsim 1$  since $\alpha_k^2-\alpha_{k-1}^{-2} \lesssim \sqrt{\delta}$, we deduce from \eqref{drift2}
that by  Taylor expanding the logarithm to the first order, 
\[ \log \Big(\frac{\Lambda_k^2}{\alpha_k^2}\Big) = \log \Big(\frac{\|D_k U_{k-1}\|^2}{\alpha_k^2}\Big) + \frac{1}{2(k_0-k)} + O\Big(\frac{\sqrt{W_{k-1}}}{k_0-k}\Big) +  O\Big(\frac{1}{(k_0-k)^2}\Big) +O\Big(\frac{1}{\sqrt{k_0(k_0-k)}}\Big).\]
This ends the proof of the claim.
\end{proof}

\subsection{Dynamics of the angle process}
 We derive an equation for the process $W^z$ measuring the angle between $Y_k^z$ and the eigenvector of $D_k^z$ associated with its smallest eigenvalue. We start first with the negligible hyperbolic regime.

 \begin{Pro}\label{anglenegli}
 There exists $\mathfrak{h}>0$ so that the following holds for all $z\in I_\eta$.
 Let $2\leq k \leq k_{\delta,z}$. Then, on the event where $W^z_{k-1}\leq 1/2$, 
 \[ W^z_k = \alpha_k^{-4} W^z_{k-1} 
 + \mathcal{S}^z_k, \]
 where $\mathcal{S}_k$ is $\cF_k$-measurable and 
 \[ |\EE_{k-1}(\mathcal{S}^z_k)|\lesssim  \frac{1}{n} + (W^z_{k-1})^2, \ s_{k-1,\mathfrak{h}}(\mathcal{S}^z_k) \lesssim \frac{W^z_{k-1}}{n} + \frac{1}{n^2}.\]

 \end{Pro}
 \begin{proof}
%\corO{Throughout the proof, we omit $\mathfrak{h}$ from the notation, with the understanding that it represents a constant that may change from line to line but depends only on the model parameters and not on $z$. We also omit $z$ from the notation.
As before, denote by $M_k = \EE(\Xi_k)$, $V_k=\Xi_k-M_k$, $U_k=Y_k/\|Y_{k-1}\|$ and 
 $\Lambda_k = \|M_k U_{k-1}\|$. As already noted in the proof of Proposition \ref{rechn}, the matrices $M_k/\alpha_k$ and $V_k/\alpha_k$ satisfy the assumption of Lemma \ref{lemgene} (with $\mu=k_0-k$), $\Lambda_k\gtrsim \alpha_k$ on the event $W_{k-1}\leq 1/2$, and moreover  $|\langle M_kU_{k-1},e_2\rangle|\lesssim \sqrt{W_{k-1}} + 1/n$ by \eqref{eq1expect}. Applying Lemma \ref{lemgene} to $u=U_{k-1}$ conditionnally on $\mathcal{F}_{k-1}$, we get that
  \[ W_k = \Lambda_k^{-2} \langle M_kU_{k-1},e_2\rangle^2 + \Lambda_k^{-2}\EE_{k-1}[\langle V_kU_{k-1},e_2\rangle^2] + \mathcal{S}_k,\]
 where $\mathcal{S}_k$ is $\cF_k$-measurable and 
 \[ |\EE_{k-1}(\mathcal{S}_k)|\lesssim \frac{\sqrt{W_{k-1}}}{n} + \frac{1}{n^{3/2}}, \ s_{k-1,\mathfrak{h}}(\mathcal{S}_k)\lesssim \frac{W_{k-1}}{n}  + \frac{1}{n^2}.\]
Since $\Lambda_k\gtrsim \alpha_k $ and  $\EE[\|V_k\|^2]\lesssim  \alpha_k^2/n$ by Lemma  \ref{decommatrixtrans}  we obtain that
 \[ \Lambda_k^{-2}\EE[\langle V_kU_{k-1},e_2\rangle^2]\lesssim  \frac{1}{n} +\frac{ \sqrt{W_{k-1}}}{n}.\]
Putting this term together with the error term $\mathcal{S}_k$ and using \eqref{devM2hn} end the proof of the claim. 
 \end{proof}
 Similarly, we get the following equation for the increments of the angle process in the contributing hyperbolic regime. 
 \begin{Pro}\label{eqWh}
There exists $\mathfrak{h}>0$ so that the following holds for all $z\in I_\eta$.
 Let $k_{\delta,z} \leq k\leq k_{0,z}-\ell_0$. Then,
 \[ W_k ^z- W^z_{k-1} = H^z_k + \frac{v}{2(k_{0,z}-k)} - 4W^z_{k-1} \sqrt{\frac{k_{0,z}-k}{k_{0,z}}},\]
 where $H^z_k$ is a $\cF_k$-measurable variable such that 
 \begin{equation} \label{momentH1} |\EE_{k-1} H^z_k |\lesssim \frac{W^z_{k-1}}{k_{0,z}-k} + (W^z_{k-1})^2\sqrt{\frac{k_{0,z}-k}{k_{0,z}}}+ \frac{1}{\sqrt{k_{0,z}(k_{0,z}-k)}} + \frac{1}{(k_{0,z}-k)^{3/2}},\end{equation}
 \begin{equation} \label{momentH2} s_{k-1,\mathfrak{h}}(H^z_k) \lesssim \frac{W^z_{k-1}}{k_{0,z}-k} + \frac{1}{(k_{0,z}-k)^2}. \end{equation}
 \end{Pro}
 
 \begin{proof}
%\corO{Throughout the proof, we omit $\mathfrak{h}$ from the notation, with the understanding that it represents a constant that may change from line to line but depends only on the model parameters and not on $z$. We also omit $z$ from the notation.
 We continue to denote  by $M_k =\EE(\Xi_k)$, $V_k=\Xi_k-M_k$, $U_{k-1}= Y_{k-1}/\|Y_{k-1}\|$ and $\Lambda_k = \|M_kU_{k-1}\|$.
We know from Lemma \ref{decommatrixtrans} that $M_k,V_k$ satisfy the assumption of Lemma \ref{lemgene} (with $\mu=k_0-k$)
and as we argued in the proof of Proposition \ref{rech}, $\Lambda_k \gtrsim 1$. Moreover, we have by \eqref{eq2expect} that  $|\langle M_k U_{k-1},e_2\rangle|\lesssim \sqrt{W_{k-1}} + 1/(k_0-k)$. Applying Lemma \ref{lemgene} to $u=U_{k-1}$ conditionally on $\mathcal{F}_{k-1}$, we find that $W_k$ satisfies the following equation 
  \begin{align}W_k &=  \Lambda_k^{-2}\langle M_kU_{k-1},e_2\rangle^2 + \Lambda_k^{-2} \EE_{k-1}\langle V_kU_{k-1}, e_2\rangle^2  + H_k, \label{eqWU}
 \end{align}
 where $H_k$ is a $\cF_k$-measurable variable such that
\[ |\EE_{k-1} (H_k)| \lesssim \frac{\sqrt{W_{k-1}}}{k_0-k} + \frac{1}{(k_0-k)^{3/2}},  \ s_{k-1,\mathfrak{h}}(\mathcal{X}_k) \lesssim \frac{W_k}{k_0-k} + \frac{1}{(k_0-k)^2},\]
and $\mathfrak{h}>0$ depends on the model parameters.
Since $V_k(2,1) =  g_k+ \mathcal{T}_k$ where $\mathcal{T}_k$ is a centered random variable such that $s_{k-1,2}(\mathcal{T}_k)\lesssim 1/(k_0-k)^3$ by  Lemma \ref{decommatrixtrans},  we deduce by a similar argument as in the proof of \eqref{driftnoise1} that 
\[ \Lambda_k^{-2}\EE_{k-1}\big[\langle V_kU_{k-1},e_2\rangle^2\big] = \frac{v}{2(k_0-k)} + \widetilde{\mathcal{T}}_k,\]
where $0\leq \widetilde{\mathcal{T}}_k\lesssim  \frac{1}{\sqrt{k_0(k_0-k)}} +\frac{1}{(k_0-k)^2} + \frac{W_{k-1}}{k_0-k}$. Finally,
using \eqref{estimlambdak} and the fact  that  $M_k =D_k + O(1/(k_0-k))$ by \eqref{eq2expect}, we deduce that 
\[ \Lambda_k^{-2}\langle M_kU_{k-1},e_2\rangle^2 =\alpha_k^{-4} W_{k-1} + O\Big(W_{k-1}^2\sqrt{\frac{k_0-k}{k_0}}\Big) + O\Big(\frac{\sqrt{W_{k-1}}}{k_0-k}\Big) + O\Big(\frac{1}{(k_0-k)^2}\Big). \]
 Besides $1-\alpha_k^{-4} = 4\sqrt{(k_0-k)/k_0} + O((k_0-k)/k_0)$ by \eqref{estimalphak}, which ends the proof.
 \end{proof}
 
 \subsection{Control of the angle process}
 Using the description of the dynamics of the angle process, we can now give a proof of Propositions \ref{boundWhn} and \ref{controlbadblockh}.
 
\begin{proof}[Proof of Proposition \ref{boundWhn}]
By Proposition \ref{anglenegli}, we know that there exists $\mathfrak{h}>0$ depending on the model parameters 
 only, i.e. not on $z\in I_\eta$, such that almost surely on the event where $W_{k-1}\leq 1/2$,
\begin{equation} \label{ineqW} W_k \leq \alpha_k^{-4} W_{k-1} + \mathfrak{h}\Big( \frac{1}{n} + W_{k-1}^2\Big) +(\log n)^{\mathfrak{h}}\Big( \sqrt{\frac{W_{k-1}}{n}} +\frac1n\Big),\end{equation}
%\corF{almost surely on the event where} $\corF{W_{k-1}}\leq 1/2$. 
We know by \eqref{initalW} that $W_1 \lesssim 1/n$, almost surely. Moreover, there exists $\mathfrak{c}_\delta>0$ depending on the model parameters such that $\alpha_k^{-4}\leq 1-\mathfrak{c}_\delta$ for any $k\leq k_{\delta}$. One can check that by induction on $k$ the inequality \eqref{ineqW} implies that   
for any $k\leq k_{\delta}$, 
\[ W_k \leq \frac{4}{\mathfrak{c}_\delta^2} \frac{(\log n)^{\corF{\mathfrak{2h}}}}{n}, \ \text{a.s.,}\]
which gives the claim.
\end{proof}

\begin{proof}[Proof of Proposition \ref{controlbadblockh}] 
The idea of the proof  is that if the process $W_\cdot$ starts below $r\eta_{m_i}$ at some time $k$ and goes above $2r\eta_{m_i}$ before time $m_{i}$, then there exists an excursion of the process $W_\cdot$ above $r\eta_{m_i}$. On such an excursion, the process $W_\cdot$ according to Lemma \ref{eqWh} feels a strong negative drift which translates itself into a large deviation event for the process $H$; the probability of the latter event is estimated 
%. Finally, we estimate this probability 
using our moment estimates \eqref{momentH1}-\eqref{momentH2}. 

More precisely, 
let $\tau$ be the first time $\ell$ between times $k$ and  $m_i$ that $W_\ell\geq 2r\eta_{m_i}$  and $\tau'$  the last time $\ell'$ between $k$ and $m_i$ such that $W_{\ell'} \leq r \eta_{m_i}$.  Note that on the event $\mathscr{H}_i^\complement\cap\{W_{k}\leq r\eta_{m_i}\}$, the times $\tau$ and $\tau'$ are well-defined and we have that  $W_\ell > r\eta_{m_i}$, for any  $\ell \in (\tau',\tau]$.
Now, by Lemma \ref{eqWh} we have for any $\ell \geq k_\delta$,
\begin{equation} \label{recurWeq1}  W_{\ell+1}-W_{\ell} = H_{\ell+1} + \frac{v}{k_0-\ell}-4W_{\ell}\sqrt{\frac{k_0-\ell}{k_0}}\end{equation}
where $H$ is an adapted process satisfying the moment bounds \eqref{momentH1}-\eqref{momentH2}. 
 Since for any $\ell \in (\tau',\tau]$, $W_\ell > r\eta_{m_i}$ and $\eta_{m_i} \geq  vk_0^{1/2}/[4r (k_0-m_{i})^{3/2}] \geq vk_0^{1/2}/[4r(k_0-\ell)^{3/2}]$, we have that 
\begin{equation} \label{driftbound} 4W_{\ell}\sqrt{\frac{k_0-\ell}{k_0}} - \frac{v}{k_0-\ell} \geq 0, \quad \ell \in (\tau',\tau].\end{equation}
Therefore, \eqref{recurWeq1} implies that $ r\eta_{m_i} \leq W_\tau - W_{\tau'}  \leq    \sum_{\ell=\tau'+1}^{\tau} \widetilde{H}_\ell$, 
 where $\widetilde{H}_\ell:= H_\ell \Car_{W_{\ell-1}\leq r \eta_{m_i}}$ for any $\ell$. We have proven that on the event where $W_{k}\leq r\eta_{m_i}$,
 \begin{equation} \label{probaggodproof} \PP_{k}\big(\mathscr{H}_i^\complement\big) \leq  \PP_{k}\Big(\sup_{k\leq  t' \leq t \leq m_i} \sum_{\ell = t'+1}^t \widetilde{H}_\ell \geq r\eta_{m_i}\Big).\end{equation}
  From the moment estimates \eqref{momentH1}-\eqref{momentH2}, we find that
  \[ \sum_{\ell= m_{4i}+1}^{m_i} \|\EE_{\ell-1} \widetilde{H}_\ell \|_{L^\infty} \lesssim \eta_{m_i}/4, \ \sum_{\ell=m_{4i}+1}^{m_i} \|s_{\ell-1}(\widetilde{H}_\ell)\|_{L^\infty} \lesssim \frac{1}{i^{5/3}},\]
 where we used the fact that $m_i-m_{4i}\lesssim (n/i)^{1/3}$ and $\|s_{\ell-1}(\widetilde{H}_\ell)\|_{L^\infty} \lesssim (i^{4/3} n^{1/3})^{-1}$ for any $m_{4i} \leq \ell \leq m_i$.  Since $\max_{m_{4i} \leq \ell \leq m_i} \|s_{\ell-1}(\widetilde{H}_\ell)\|_{L^\infty} \lesssim n^{-1/3}$ and $k\geq m_{4i}$, we deduce by Lemma \ref{tailpropH} that 
 \begin{equation} \label{bounprobaW} \PP_{k}\Big(\sup_{k\leq t' \leq t \leq m_i} \sum_{\ell = t'+1}^t \widetilde{H}_\ell \geq \eta_{m_i}/2\Big) \leq \exp\Big( - \mathfrak{c} \min \big\{i^{1/3}, n^{1/6}\big\}\Big),\end{equation}
 where $\mathfrak{c}$ is a positive constant depending on the model parameters only. Using that $i$ is at most of order $n$ ends the proof of the claim. 
\end{proof}

Using the result of Proposition \ref{controlbadblockh} we can now give a proof of Corollary \ref{controlbadblockhall}.

\begin{proof}[Proof of Corollary \ref{controlbadblockhall}] Let $\mathfrak{q}>0$ to be chosen later and write $\ell_{\mathfrak{q}} = n^{1/3} (\log n)^{\mathfrak{q}}$. Denote by $\mathscr{H}_k$ the event where $W_\ell > \eta_\ell$ for any $k<\ell\leq k_0-\ell_\mathfrak{q}$. Let $\tau$ be the first time $\ell$ between times $k$ and $k_0-\ell_{\mathfrak{q}}$ that $W_\ell\geq \eta_{\ell}$ and denote by $i_\tau$ the integer such that $\tau \in (m_{i_\tau+1},m_{i_\tau}]$. Note that $i_\tau\gtrsim \mathfrak{a}(\log n)^{3\mathfrak{q}/2}$, where $\mathfrak{a}>0$ is some numerical constant. If $i\geq \mathfrak{a}(\log n)^{3\mathfrak{q}/2}$ is such that $m_{4i}\geq k$, then we obtain
 from  Proposition \ref{controlbadblockh} that
\begin{equation} \label{badblock5} \PP_k(\mathscr{H}_k^\complement\cap \{i_\tau = i\} ) \leq  e^{-\mathfrak{c} i^{1/6}},\end{equation}
where $\mathfrak{c}>0$ depends on the model parameters. If on the the other hand $i\geq \mathfrak{q}(\log n)^{2\mathfrak{q}/2}$ and $m_{4i} <k$ then we can deduce again with the same argument as in the proof of Proposition \ref{controlbadblockh} that on the event where $W_k\leq \eta_k/2$, we have that
\begin{equation} \label{badblock52}  \PP_k(\mathscr{H}_k^\complement\cap \{i_\tau =i \} ) \leq e^{-\mathfrak{c} i^{1/6}}.\end{equation}
Hence, if $\mathfrak{q}>8$, then using a union bound and \eqref{badblock5}-\eqref{badblock52} we deduce that $\PP_k(\mathscr{H}_k^\complement) \leq e^{-\mathfrak{c}' (\log n)^2}$, where $\mathfrak{c}'>0$ depends on the model parameters. 
\end{proof}

 \subsection{Representation of the increments of the process $\psi(z)$}
 We now give a proof of Propositions \ref{incremhn}, \ref{increm} and \ref{coreqhyper}.  Using a convexity argument we first derive the domination result of Proposition \ref{coreqhyper} from the equation for the increments of $\log \|Y_k^z\|$ proved in Proposition \ref{rech}

 \begin{proof}[Proof of Proposition \ref{coreqhyper}]
 By definition of the process  $\psi(z)$, we have that 
 \[   2(\psi_k(z) - \psi_{k-1}(z)) = \log \frac{\|Y_k\|^2}{\alpha_k^2 \|Y_{k-1}\|^2} + \frac{v-1}{2(k_0-k)}.\]
 Using Proposition \ref{rech} we get that 
  \begin{align}    2(\psi_k(z) - \psi_{k-1}(z)) &= (W_k-W_{k-1}) + g_k \nonumber \\
  & +\log\Big(\frac{\|D_kU_{k-1}\|^2}{\alpha_k^2}\Big) - \Big(1- \frac{1}{\alpha_k^2\|D_k U_{k-1}\|^2}\Big) W_{k-1} +\mathcal{X}_k, \label{eqpsih}\end{align}
 where $\mathcal{X}_k$ satisfies the moment bounds of \eqref{rechmoment}. Writing $G_k = g_k+\mathcal{X}_k$, one can check using  \eqref{rechmoment} and the facts that $\EE[g_k^2] = v/(2(k_0-k) + O(1/\sqrt{k_0(k_0-k)})+O(1/n)$ and $|g_k|\leq (\log n)^{2}/\sqrt{k_0-k}$,  that  $G$ satisfies the bounds \eqref{momentG1}-\eqref{momentG2}. 
 
 Next, using the concavity of the logarithm, we find that $\log( \|D_kU_{k-1}\|^2/\alpha_k^2) = \log(1- (1-\alpha_k^{-4})W_{k-1})\leq - (1-\alpha_k^{-4}) W_{k-1}$. Moreover, again by convexity, 
 \[\frac{\alpha_k^{2}}{\|D_k U_{k-1}\|^2} = \frac{1}{1-(1-\alpha_k^{-4})W_k} \geq 1 + (1-\alpha_k^{-4})W_{k-1}.\]
Therefore, 
\begin{align*}
&\log\Big(\frac{\|D_kU_{k-1}\|^2}{\alpha_k^2}\Big) - \Big(1- \frac{1}{\alpha_k^2\|D_k U_{k-1}\|^2}\Big) W_{k-1} \\
&\leq  - (1-\alpha_k^{-4}) W_{k-1}  + \Big( 1- \alpha_k^{-4}(1+\big(1-\alpha_k^{-4})W_{k-1}\big)\Big)W_{k-1} 
  \leq -\alpha_k^{-4}(1-\alpha_k^{-4}) W^2_{k-1}.
 \end{align*}
Since $1-\alpha_k^{-4}\lesssim \sqrt{(k_0-k)/k_0}$ by \eqref{estimalphak}, this ends the proof of the claim.
 \end{proof}

\begin{proof}[Proof of Proposition \ref{increm}]
In view of \eqref{eqpsih} that we proved in the proof of Proposition \ref{coreqhyper}, it suffices to show that 
\[  \log\Big( \frac{\|D_k U_{k-1}\|^2}{\alpha_k^2}\Big)  + \Big(1- \frac{1}{\alpha_k^2\|D_kU_{k-1}\|^2}\Big) W_{k-1} = O\Big( \sqrt{\frac{k_0-k}{k_0}} W_{k-1}^2\Big).\]
This follows from the fact that $\|D_k U_{k-1}\|^2/\alpha_k^2 = 1 - (1-\alpha_k^{-4})W_{k-1}$ and  $1-\alpha_k^{-4} = O(\sqrt{(k_0-k)/k})$ by \eqref{estimalphak}. Indeed, 
\begin{align*}
&\log\Big(\frac{\|D_kU_{k-1}\|^2}{\alpha_k^2}\Big)  + \Big(1- \frac{1}{\alpha_k^2\|D_kU_{k-1}\|^2}\Big) W_{k-1}  \\
&= -(1-\alpha_k^{-4})W_{k-1} +\Big(1 - \alpha_k^{-4}\Big(1+ O\Big(\sqrt{\frac{k_0-k}{k_0}}W_{k-1}\Big)\Big)\Big) W_{k-1} + O\Big(\sqrt{\frac{k_0-k}{k_0}}W_{k-1}^2\Big)\\
& = O\Big(\sqrt{\frac{k_0-k}{k_0}}W_{k-1}^2\Big),
\end{align*}
which ends the proof.
\end{proof}

Finally, using Proposition \ref{rechn}  and the control on $W^z$ in the negligible regime given by Proposition \ref{boundWhn} we can now give a proof of Proposition \ref{increm}. 

\begin{proof}[Proof of Proposition \ref{increm}]In this regime, $\log \|Y_k\|/[\alpha_{k} \|Y_{k-1}\|]$ and $\psi_k$ differ by a constant of order $1/n$. 
Proposition \ref{rechn} gives us a representation of the increment $\log \|Y_k\|/[\alpha_{k} \|Y_{k-1}\|]$ as a sum of $g_k$, $\frac12 (W_k-W_{k-1})$ and an error $\mathcal{X}_k$ satisfying certain moment bounds depending on $W_{k-1}$. Since $W_{k-1}\leq (\log n)^\mathfrak{C}/n$ almost surely by Proposition \ref{boundWhn}, where $\mathfrak{C}>0$ depends on the model parameters, we deduce that this error $\mathcal{X}_k$ indeed satisfies the claimed moment bounds.
\end{proof}

\section{The parabolic and elliptic regime}\label{sectionelliptic}
We analyse in this section the parabolic and elliptic regime and give a proof of Propositions \ref{apriori}, \ref{probagoodblock} and \ref{represelliptic}. In the following, we fix some $z\in I_\eta$ and drop the $z$-dependence in the notation for the sake of clarity.

 \subsection{Description of the new transition matrix}
As a first step, we show that the new transition matrix $\Xi_k^z$  can be decomposed up to a small error as the sum of  the rotation $R_k^z$ (see  \eqref{defRk2}), a noise term and a drift term. 
\begin{lemma}\label{changebasisgene}
The following holds for all $z\in I_\eta$.
Let  $k\geq  k_{0,z}+\ell_0$. Then,
\begin{equation} \label{decompXi} \Xi_k^z =  R_k^z + G_k^z + \Delta_k^z + \mathcal{B}^z_k,\end{equation}
where $R_k^z$ is defined in \eqref{defRk2}, 
\begin{equation} \label{descrG} G_k^z =  \begin{pmatrix} 0 & 0 \\ c_{k,z} & d_k \end{pmatrix}, \ \Delta_k^z = \begin{pmatrix} 0 & 0 \\ 0 & -\frac{1}{2(k-k_{0,z})} \end{pmatrix},\end{equation}
and $c_{k,z}$, $d_k$ are defined in \eqref{defc} and  where  
\begin{equation} \label{descrB} \|\EE (\mathcal{B}_k^z)\| \lesssim \frac{1}{\sqrt{k_{0,z}(k-k_{0,z})}} , \, s_{k-1,2}(\mathcal{B}_k^z)\lesssim \frac{1}{(k-k_{0,z})^3}.\end{equation}
Moreover, 
\begin{equation}\label{estimnoiseXie}  
\|\Xi_k^z-\EE\Xi_k^z\|\lesssim \frac{(\log n)^{2}}{\sqrt{k-k_{0,z}}}, \quad
%\corO{s_{k-1,\mathfrak{h}}}(\Xi_k^z) \lesssim \frac{1}{k-k_{0,z}}, \ 
\EE[\|\Xi_k^z-\EE(\Xi_k^z)\|^p] \lesssim_p \frac{1}{(k-k_{0,z})^{p/2}}, p\geq 2.\end{equation}
\end{lemma}

\begin{proof} We recall that without loss of generality, we may and will  assume that $z>0$.  
We start by deriving estimates of $\|P_k\|$ and $\|P_k^{-1}\|$. We claim that 
\begin{equation} \label{estimatePk} \|P_k\| \lesssim 1, \ \|P_k^{-1}\|\lesssim \sqrt{\frac{k_0}{k-k_0}},  k\geq k_{\delta}^-.\end{equation}
Note that as $k \gtrsim n$, we have that $z_k \asymp 1$. Thus, the entries of $P_k$ are of order $1$ and as a result $\|P_k\|\lesssim 1$. Moreover, 
\begin{equation} \label{Pinv} P_k^{-1} = \begin{pmatrix} \frac{2}{\sqrt{4-z_k^2}} & - \frac{z_k}{\sqrt{4-z_k^2}} \\ 0 & 1\end{pmatrix}.\end{equation}
Using \eqref{estim}, we deduce that $\|P_k^{-1}\|\lesssim \sqrt{k_0/(k-k_0)}$.

Let $   k \geq k_0+\ell_0$. Note that $s_{k-1,2}(T_k) \lesssim 1/k$ by Assumption \ref{ass1}. Hence, the estimates on the norm of $P_k^{-1}$ and $P_{k-1}$ from \eqref{estimatePk} imply that $s_{k-1,2}(\Xi_k^z) \lesssim 1/(k-k_0)$. Moreover, using a similar argument as in the proof of \eqref{controlnoiseV}, we obtain that for any $p\geq 2$, $\EE[\|\Xi_k^z-\EE(\Xi_k^z)\|^p] \lesssim_p (k-k_{0,z})^{-p/2}$.
% for any $p\geq 2$.
Next, we decompose $\Xi_k$ into
\begin{equation} \label{devxi}  \Xi_k = P_k^{-1}T_k P_k +(P_k^{-1}P_{k-1} -\mathrm{I}_2)+ (P_k^{-1}T_kP_k-\mathrm{I}_2)(P_k^{-1} P_{k-1}-\mathrm{I}_2), \end{equation}
and compute each terms. 
 Recall the definition of $A_k$ in \eqref{defAk}. We find that 
 \[ P_k^{-1}A_k P_k = \begin{pmatrix}
 \frac{z_k}{2} & -  \frac{\sqrt{4-z_k^2}}{2} \\ \frac{\sqrt{4-z_k^2}}{2} & \frac{z_k}{2} \end{pmatrix}.\]
 Since $e^{i\theta_k} = (z_k+i\sqrt{4-z_k^2})/2$ for $k>k_0$, this shows that $P_k^{-1} A_k P_k$ coincides with $R_k$, defined in \eqref{defRk2}.
Since $\EE (a_{k-1}^2 )= k+O(1)$ and $\EE(b_k) =O(1/k)$, we deduce that,
 \[ \| \EE T_k-A_k\| \leq \Big| \frac{\EE (a_{k-1}^2)}{\sqrt{k(k-1)}}-1\Big| + \Big| \frac{\EE(b_k)}{\sqrt{k}}\Big|\lesssim \frac{1}{k}\lesssim \frac{1}{k_0}.\]
Together with \eqref{estimatePk}, it yields that 
\begin{equation} \label{estimexpec} \| P_k^{-1} (\EE (T_k )-A_k)P_k \| \lesssim \frac{1}{k_0} \sqrt{\frac{k_0}{k-k_0}} = \frac{1}{\sqrt{k_0(k-k_0)}},\end{equation}
where we used the fact that the operator norm is submultiplicative.
We have shown so far that 
\begin{equation} \label{e1} P_k^{-1} \EE (T_k) P_k = R_k + O\Big(\frac{1}{\sqrt{k_0(k-k_0)}}\Big).\end{equation}
Let $G_k := P_k^{-1} (T_k- \EE (T_k)) P_k$. One can check that $V_k$ is indeed described as in \eqref{descrG}.  It now remains to compute the two last terms in \eqref{devxi}. First, we find that  
\begin{equation} \label{diffchangebasis} P_k^{-1}P_{k-1} - \mathrm{I}_2 = \begin{pmatrix} 1 & 0 \\ \frac{z_{k-1} - z_k}{\sqrt{4-z_k^2}} & \sqrt{\frac{4-z_{k-1}^2}{4-z_k^2}}\end{pmatrix}.\end{equation}
Since $z_{k-1}-z_k\lesssim 1/k_0$ by \eqref{5diffz}, we deduce together with   \eqref{estim} that 
\[ 0\leq \frac{z_{k-1}-z_k}{\sqrt{4-z_k^2}} \lesssim \frac{1}{\sqrt{k_0(k-k_0)}}.\]
Further, using \eqref{estim} and \eqref{5diffz}, one can check that 
\begin{equation} \label{ratioz}  \sqrt{\frac{4-z_{k-1}^2}{4-z_k^2}} = 1 + \frac{1}{2(k_0-k)} + O\Big( \frac{1}{k_0}\Big) + O\Big( \frac{1}{(k_0-k)^2}\Big).\end{equation}
Using that $k-k_0\geq \ell_0$, this entails  that 
\begin{equation} \label{e3} P_k^{-1} P_{k-1} - \mathrm{I}_2 = \Delta_k + \mathcal{D}_k,\end{equation}
where $\|\mathcal{D}_k\|\lesssim 1/\sqrt{k_0(k-k_0)}$. Finally, we compute the last term in \eqref{devxi}, $\mathcal{S}_k:= (P_k^{-1}T_k P_k -\mathrm{I}_2)(P_k^{-1}P_{k-1}-\mathrm{I}_2)$. Using \eqref{estimatezk} and \eqref{estim1}, we get that
\begin{equation} \label{diffRI} \|R_k-\mathrm{I}_2\|\lesssim |\theta_k|\lesssim \sqrt{(k-k_0)/k_0}.\end{equation}
Combined with \eqref{e1} and \eqref{e3}, this entails that 
\begin{equation} \label{e4} \|\EE (\mathcal{S}_k)\| \leq  \|P_k^{-1}\EE (T_k) P_k -\mathrm{I}_2\| \cdot\|P_k^{-1} P_{k-1}-\mathrm{I}_2\|\lesssim  \sqrt{\frac{k-k_0}{k_0}}. \frac{1}{k-k_0}\lesssim \frac{1}{\sqrt{k_0(k-k_0)}}.\end{equation}
Note that by \eqref{estimatePk} and the fact that $\EE[\|T_k-\EE(T_k)\|^2]\lesssim 1/k \lesssim 1/k_0$ by Assumption \ref{ass1}, we have that $\EE[\|G_k\|_2^2] \lesssim 1/(k-k_0)$. Together with \eqref{e3} and Hölder's noncommutative inequality, we deduce that  
\begin{align}
 \EE [\|\mathcal{S}_k - \EE( \mathcal{S}_k)\|_2^2] &= \EE[ \|G_k(P_k^{-1}P_{k-1} - \mathrm{I}_2)\|_2^2] \nonumber \\
& \leq  \|P_k^{-1}P_{k-1}-\mathrm{I}_2\|^2 \EE [\|G_k \|_2^2] \lesssim \frac{1}{(k-k_0)^3}. \label{var4}
\end{align}
Moreover, $\|\mathcal{S}_k - \EE(\mathcal{S}_k)\|_2\lesssim (\log n)^2/(k-k_0)^{3/2}$ as $\|G_k\|_2\lesssim (\log n)^2/\sqrt{(k-k_0)}$ using the fact that the noise is bounded according to Assumption \ref{ass1}.
Putting together \eqref{e1}, \eqref{e3}, \eqref{e4} and \eqref{var4}, this ends the proof of  \eqref{decompXi}. 
\end{proof}

Next, we  show that the random variable $c_{k,z}$ appearing in the noise part of the transition matrix $\Xi_k^z$ does not depend strongly on $z$. More precisely, we have the following result. 
\begin{lemma}\label{noisec}
There exists a sequence of independent centered random variables $(g_k)_{k\geq 1}$ 
which does not depend on $z$, such that  $\EE(g_k^2) = 2v$ for any $k\geq 1$ and  
\begin{equation} \label{diffcgvar} \EE\Big[\Big(c_{k,z}-\frac{g_k}{\sqrt{k-k_{0,z}}}\Big)^2\Big]\lesssim  \frac{1}{n} +\frac{1}{(k-k_{0,z})^2} \quad k\geq k_{0,z} +\ell_0.\end{equation}
\end{lemma}

\begin{proof}
Define $(g_k)_{1\leq k\leq n}$ by 
\[  g_k:=- \sqrt{v}  \Big(\frac{b_k- \EE(b_k)}{\sqrt{\Var(b_k)}} + \frac{a_{k-1}^2-\EE (a_{k-1}^2)}{\sqrt{\Var(a_{k-1}^2)}}\Big), \quad 1\leq k \leq n.\]
By construction $(g_k)_k$ defines a sequence of independent centered variables of variance $2v$ which does not depend on $z$. One can check that the  estimates on $z_k$, \eqref{estimatezk}, \eqref{estim} and the moment assumptions \eqref{ass1} imply  \eqref{diffcgvar}.
\end{proof}

%In the hyperbolic regime, it is more natural to work in the basis $Q_k^z$ where $A_k^z$ is the diagonal matrix $D_k^z$, which is related to $P_k^z$ by the relation $Q_k^z = P_k^z U$, where 
%\[ U =\frac{1}{\sqrt{2}}\begin{pmatrix} 1 & 1 \\ 1 & -1\end{pmatrix}, \quad D_k^z = \begin{pmatrix} \alpha_k^z & 0\\ 0&1/\alpha_k^z\end{pmatrix}\]
%
%\begin{Lem}
%Let $k_\delta^- \leq k \leq k_0$. Then, 
%\[ \Xi_k^z = D_k^z +V_k^z + \Delta\]
%
%
%\end{Lem}

Finally, we only give a rough description of the transition matrix in the parabolic regime, $|k-k_{0,z}|\leq \ell_0$,  as this regime will be negligible in terms of variance and expectation.

\begin{lemma} 
\label{parabolic}
%There exists $\mathfrak{h}>0$ so that 
The following holds for all $z\in I_\eta$.
For any $|k-k_{0,z}|\leq \ell_0$, 
\[  \| \EE (\Xi_k^z) - \mathrm{I}_2\| \lesssim   \sqrt{\frac{\ell_0}{k_0}}, \quad 
%\corO{s_{k-1,\mathfrak{h}}}( \Xi_k^z)
\|\Xi_k^z-\EE \Xi_k^z\|\lesssim  \frac{(\log n)^2}{\sqrt{\ell_0}}, \ \EE_{k-1}[\|\Xi_k^z-\EE(\Xi_k^z)\|_2^p]\lesssim_p \frac{1}{\ell_0^{p/2}}, \ p\geq 2.\]
\end{lemma}

\begin{proof} For $k= k_{0}-\ell_0$, the claim follows from Lemma \ref{decommatrixtrans}, so we can assume $k_0-\ell_0< k\leq k_0+\ell_0$. 
From our choice of change of basis (see \eqref{changebaseP}), we have that  $\Xi_k = Q_kP_{k_0-\ell_0}^{-1} T_k P_{k_0-\ell_0}$, where $Q_k := \mathrm{I}_2$ for $| k-k_0|< \ell_0$, $Q_{k_0+\ell_0} : = P_{k_0+\ell_0}^{-1} P_{k_0-\ell_0}$. We claim that $\|Q_k\|\lesssim  1$. This follows from \eqref{diffchangebasis} for $k= k_0-\ell$, whereas for $k=k_0+\ell_0$, we have explicitly that 
\[ Q_{k_0+\ell_0}=  \begin{pmatrix} 
\frac{2 -z_{k_0+\ell_0}\alpha_{k_0-\ell_0}^{-1}}{\sqrt{4-z_{k_0+\ell_0}^2}} & \frac{2\alpha_{k_0-\ell_0}^{-1}-z_{k_0+\ell_0}}{\sqrt{4-z_{k_0+\ell_0}^2}} \\
\alpha_{k_0-\ell_0}^{-1} & 1 \end{pmatrix}.\]
Since $\alpha_{k_0-\ell_0} \asymp 1$, it boils down to prove that $(2 -z_{k_0+\ell_0}\alpha_{k_0-\ell_0}^{-1})/\sqrt{4-z_{k_0+\ell_0}^2} =O(1)$ and $(2 \alpha_{k_0-\ell_0}^{-1}+z_{k_0+\ell_0})/\sqrt{4-z_{k_0+\ell_0}^2} =O(1)$. We only prove the first bound, the proof of the second one being very similar. Since  $\alpha_{k_0-\ell_0}^{-2} - z_{k_0-\ell_0} \alpha_{k_0-\ell_0}^{-1} +1=0$, we have that 
\begin{equation} \label{eqcoeffQ} 2 \alpha_{k_0-\ell_0}^{-1}+z_{k_0+\ell_0}  = 1- \alpha_{k_0-\ell_0}^{-2} + \alpha_{k_0-\ell_0}^{-1}(z_{k_0-\ell_0}-z_{k_0+\ell_0}).\end{equation}
Note that \eqref{5diffz} entails that $z_{k_0-\ell_0}-z_{k_0+\ell_0} \asymp \ell_0/k_0$. Moreover, we know by \eqref{estim} that $(4-z_{k_0+\ell_0}^2)^{-1/2} \asymp  \sqrt{k_0/\ell_0}$ and by \eqref{estimalphak}  that $1-\alpha_{k_0-\ell_0}^{-2} \asymp \sqrt{\ell_0/k_0}$. Using \eqref{eqcoeffQ}, it follows that  $(2 -z_{k_0+\ell_0}\alpha_{k_0-\ell_0}^{-1})/\sqrt{4-z_{k_0+\ell_0}^2} =O(1)$.

Next, since $\|P_{k_0-\ell_0}^{-1}\|\lesssim \sqrt{k_0/\ell_0}$, $\|P_{k_0-\ell_0-1}\|\lesssim1$ by \eqref{estimatePk}, $\|Q_{k}\|\lesssim  1$ and the fact that $\EE [\| T_k - \EE(T_k)\|_2^2] \lesssim 1/k_0$ by Assumption \ref{ass1}, it follows that 
\[ \EE [\|\Xi_k - \EE(\Xi_k)\|_2^2] \lesssim 1/\ell_0.\]
Moreover, $\|\Xi_k - \EE(\Xi_k)\|\lesssim (\log n)^2/\sqrt{\ell_0}$ since $\|T_k-\EE(T_k)\|\lesssim (\log n)^2/\sqrt{k_0}$ by \eqref{boundnoiseass}. Finally, we prove the estimate on the norm of $\EE(\Xi_k)-\mathrm{I}_2$. We write:
 \begin{align*} \EE(\Xi_k)-\mathrm{I}_2& = Q_kP_{k_0-\ell_0}^{-1} (\EE(T_k) - \EE (T_{k_0-\ell_0})) P_{k_0-\ell_0}  + Q_kP_{k_0-\ell_0}^{-1} ( \EE (T_{k_0-\ell_0})-\mathrm{I}_2) P_{k_0-\ell_0}   + (Q_k -\mathrm{I}_2)\\
& = : K_1+K_2+K_3. \end{align*}
By \eqref{5diffz}, we deduce that  $ z_{k_0-\ell_0}-z_k \lesssim (k-k_0+\ell_0)/k_0\lesssim \ell_0/k_0$. Thus,  $\|\EE (T_{k_0-\ell_0})  - \EE (T_k )\| \lesssim \ell_0/k_0$ and it follows that $\|K_1\|\lesssim  \sqrt{\ell_0/k_0}$ since $\|P_{k_0-\ell_0}^{-1}\|\lesssim \sqrt{k_0/\ell_0}$ by \eqref{estimatePk}. Next, using \eqref{estimesph} and the fact that $\|D_{k_0-\ell_0}-I_2\|  \lesssim \sqrt{\ell_0/k_0}$ by \eqref{estimalphak}, we have that 
\[ K_2 = Q_k\Big(D_{k_0-\ell_0} -\mathrm{I}_2+ O\Big(\frac{1}{\sqrt{k_0\ell_0}}\Big) \Big) = Q_kO\Big(\sqrt{\frac{\ell_0}{k_0}}\Big).\]
Since $\|Q_k\|\lesssim 1$, we obtain  that $\|K_2\|\lesssim  \sqrt{\ell_0/k_0}$. Finally, for $|k-k_0|<\ell_0$, $Q_k=\mathrm{I}_2$ while $\|Q_{k_0+\ell_0}-\mathrm{I}_2\|\lesssim  \sqrt{\ell_0/k_0}$ by \eqref{e3}, which ends the proof of the claim.
\end{proof}

\subsection{A priori domination of the increments and a priori exponential estimate}
A simple consequence of the structure of the new transition matrix $\Xi_\ell^z$ obtained in Lemmas  \ref{decommatrixtrans}, \ref{changebasisgene}, and \ref{parabolic} is an a apriori moment estimate  the norm of products of the transition matrices $\Xi_\ell^z$. We first start with an a priori estimate showing that the contribution of the negligible hyperbolic regime is of order $1$ at the level of moments. 

\begin{lemma}\label{apriorinegli}
Let $\delta>0$. There exists $\mathfrak{h}>0$ depending on the model parameters  only
and  $\mathfrak{C}_\delta>0$ depending on $\delta>0$ such that for any $z\in I_\eta$ and any  $0\leq k \leq k' \leq k_{\delta,z}$ and  $1\leq \lambda \leq (\log n)^{-\mathfrak{h}}n^{-1/6}$, 
\begin{equation} \label{estimexpoXin} \log \EE\big[ \|\widehat{\Xi}_{k,k'}^z \|^\lambda \big] \leq\mathfrak{C}_\delta  \lambda^2.\end{equation}
\end{lemma}
\noindent
The proof of Lemma \ref{apriorinegli} is immediate from a union bound together with Lemma \ref{decommatrixtrans}.

In the other regimes, we prove the following a priori domination of $\log \|\Xi_\ell^z u\|$, where $u$ is a unit vector, and the induced moment estimate of the norm of products of the transition matrices.
\begin{lemma}\label{aprioridom} There exists $\mathfrak{h}>0$ so that the following holds for all $z\in I_\eta$.
 Let $\ell \geq k_{\delta,z}$ and set $v_\ell :=  \frac{1}{|\ell-k_{0,z}|}$ if $|\ell-k_{0,z}|>\ell_0$ and $v_\ell:= \frac{\kappa}{n^{1/3}}$. For any $u\in \mathbb{S}^1$, 
\begin{equation} \label{domiincrem} \log \|\widehat{\Xi}^z_\ell u\| \leq \xi_\ell^z,\end{equation}
where $\widehat{\Xi}_\ell^z := \Xi_\ell^z / \alpha_{\ell, z}$ when $\ell \leq k_{0,z}-\ell_0$ and $\widehat{\Xi}_\ell^z := \Xi^z_\ell$ otherwise, and $\xi_\ell^z$ is a random variable such that 
\begin{equation} \label{momentxigene} |\EE[\xi^z_\ell] | \lesssim v_\ell, \ s_{\ell-1,\mathfrak{h}}(\xi^z_\ell) \lesssim v_\ell.\end{equation}
Moreover, there exist $\mathfrak{C},\mathfrak{h}>0$ depending on the model parameters such that for any $k_{\delta,z}\leq k \leq k' \leq n$ and  $1\leq \lambda \leq (\log n)^{-\mathfrak{h}}n^{1/6}$, 
\begin{equation} \label{estimexpoXi} \log \EE\big[ \|\widehat{\Xi}_{k,k'}^z \|^\lambda \big] \leq\mathfrak{C}  \lambda^2\big( \sum_{\ell=k+1}^{k'} v_\ell\big)+2 \lambda\leq \mathfrak{C}  \lambda^2\big( \sum_{\ell=k+1}^{k'} v_\ell+1\big).\end{equation}
Further if $k\geq  k_{0,z}-\ell_0$, then the same estimate as in \eqref{estimexpoXi}, without the linear in $\lambda$ term, 
holds for $1\leq - \lambda\leq (\log n)^{-\mathfrak{h}} n^{1/6}$. \end{lemma} 

Clearly, Lemmas \ref{apriorinegli} and \ref{aprioridom} entail the a priori estimate of Proposition \ref{apriori}. 
\begin{proof}
Throughout, we omit $z$  and $\mathfrak{h}$ from the notation. Let $\ell \geq k_\delta^-$.
By Lemmas \ref{decommatrixtrans}, \ref{changebasisgene} and \ref{parabolic}, we can write the transition matrix as $\widehat{\Xi}_\ell = M_\ell + N_\ell$, where $M_\ell$ is a deterministic matrix with operator norm bounded by $1$ and $N_\ell$ is a sequence of independent matrices  such that
\begin{equation} \label{momentNapriori} \| \EE N_\ell \| \lesssim v_\ell, \quad \|N_\ell-\EE N_\ell\|\leq \sqrt{v_\ell}(\log n)^{\mathfrak{h}} , \ \EE[\|N_\ell-\EE N_\ell\|^p] \lesssim_p v_\ell^{p/2}, p\geq 2.\end{equation}
where $\mathfrak{h}>0$ depends on the model parameters. Let $u\in \mathbb{S}^1$.
 Using the concavity of the $\log$, we find that
\begin{align}
 \log \|\widehat{\Xi}_\ell u\|^2 &= \log \big( \|M_\ell u\|^2 + 2\langle M_\ell u, N_\ell u\rangle + \|N_\ell u\|^2\big)\label{deblogincrem}\\
& \leq \log \big( 1 + 2\langle M_\ell u, N_\ell u\rangle + \|N_\ell u\|^2\big)\nonumber\\
& \leq  2\langle M_\ell u, N_\ell u\rangle + \|N_\ell u\|^2.\nonumber
\end{align}
Set $\xi_\ell := 2\langle M_\ell u, N_\ell u\rangle + \|N_\ell u\|^2$. Using \eqref{momentNapriori}, we find that $\xi_\ell$ satisfies that
\[|\EE \xi_\ell |\lesssim v_\ell, \ \Var(\xi_\ell) \lesssim v_\ell, \ |\xi_\ell-\EE(\xi_\ell)| \leq \sqrt{v_\ell}(\log n)^{\mathfrak{h}'},\]
where $\mathfrak{h}'>0$ depends on the model parameters,
which proves \eqref{domiincrem} and \eqref{momentxigene}. Next, let $k_\delta^- \leq k \leq k'\leq n$.  For any $u\in \mathbb{S}^1$, we can write 
\begin{equation}
\label{eq-250925}
 \log \|\widehat{\Xi}_{k,k'}u\| = \sum_{\ell=k+1}^{k'} \log \|\widehat{\Xi}_\ell u_{\ell-1}\|, \end{equation}
where $u_{\ell-1}:= \widehat{\Xi}_{k,\ell-1}u/\|\widehat{\Xi}_{k,\ell-1}u\|$ for any $k \leq \ell \leq k'-1$.  Using \eqref{domiincrem} and Lemma \ref{tailpropH}, we find that for any $1\leq \lambda \leq (\log n)^{-\mathfrak{h}'}/\max_{k+1\leq \ell \leq k'} \sqrt{v_\ell}$, 
\[ \log \EE  \big[ \|\widehat{\Xi}_{k,k'} u\|^{2\lambda}\big] \leq \mathfrak{C}  \lambda^2 \Big[\sum_{\ell=k+1}^{k'} v_\ell \Big],\]
where $\mathfrak{C}>0$ depends on the model parameters, and we can take $\mathfrak{C}>2$. Finally, note that $\|\widehat{\Xi}_{k,k'}\| \leq \|\widehat{\Xi}_{k,k'}e_1\| +  \|\widehat{\Xi}_{k,k'}e_2\|$. Hence, by convexity we deduce that for any  $1\leq \lambda \leq (\log n)^{-\mathfrak{h}'}/\max_{k+1\leq \ell \leq k'} \sqrt{v_\ell}$, 
 \[ \log \EE  \big[ \|\widehat{\Xi}_{k,k'} \|^{2\lambda}\big] \leq 2\lambda+ \mathfrak{C}  \lambda^2  \big(\sum_{\ell=k+1}^{k'} v_\ell \big).\] %\leq \mathfrak{C}'  \lambda^2  \big(\sum_{\ell=k+1}^{k'} v_\ell \big),\]
% where \corO{$\mathfrak{C}>0$ depends again on the model parameters.
Since $\max v_k \leq n^{-1/3}$, this ends the proof of the claim \eqref{estimexpoXi}. It now remains to show that  \eqref{estimexpoXi} still holds when $k\geq k_0-\ell_0$ and $1\leq -\lambda \leq (\log n)^{\mathfrak{h}''} n^{1/6}$ for some $\mathfrak{h}''>0$ depending on the model parameters. Let $\ell \geq k_0-\ell_0$. From Lemmas  \ref{changebasisgene} and \ref{parabolic}, we have that $\Xi_\ell = M_\ell+N_\ell$ where $M_\ell$ is a deterministic matrix such that $\|M_\ell\|=1$ and $N_\ell$ satisfies \eqref{momentNapriori}. Let $u \in \mathbb{S}^1$. Using \eqref{deblogincrem} and the fact that $\log(1+x) \geq x-x^2/2$ for any $x>-1$, we find that 
\[ \log [\|{\Xi}_\ell u\|^2]\geq \xi_\ell - \xi_\ell^2/2.\]
Now, if $Q_\ell=\xi_\ell-\xi_\ell^2/2$, one can check that the moments estimates in \eqref{momentNapriori} entail that 
\[ |\EE Q_\ell|\lesssim v_\ell, \ \Var(Q_\ell) \lesssim v_\ell, \  |Q_\ell-\EE(Q_\ell)|\leq\sqrt{v_\ell} (\log n)^{\mathfrak{h}''},\]
where $\mathfrak{h}''>0$ depends on the model parameters. 
Using now that $\|\widehat{\Xi}_{k,k'}\|^{\lambda}
\leq \|\widehat{\Xi}_{k,k'}u\|^{\lambda}$ for $\lambda<0$, \eqref{eq-250925}
%for the proof of \eqref{estimexpoXi} \corO{in the case of positive $\lambda$}
 yields that \eqref{estimexpoXi} holds as well for $1\leq - \lambda \leq (\log n)^{-\mathfrak{h}''} n^{-1/6}$ for some $\mathfrak{h}''>0$ depending on the model parameters.
\end{proof}

\subsection{Norm and angle recursion in the elliptic regime}
We now turn our attention to the elliptic regime.  
 First we derive an equation for the evolution of $\log \|Y_k^z\|$, which will involve the argument of $Y_k^z$,  denoted by $\zeta^z_k \in [0,2\pi)$.

\begin{lemma}\label{recurnormell} There exists $\mathfrak{h}>0$ so that the following holds for all $z\in I_\eta$.
For any $k\geq k_{0,z}+\ell_0$, 
\[ \log \frac{\|Y^z_k\|}{\|Y^z_{k-1}\|}  = w^z_k(\zeta^z_{k-1})+  \frac{1}{2(k-k_{0,z})}\big(2v \cos^2(\zeta^z_{k-1}) -\sin^2(\zeta^z_{k-1})-v \sin^2(2\zeta^z_{k-1})\big) +\mathcal{Z}^z_k,\]
where, as in \eqref{defw4},
\begin{equation} \label{defw} w^z_k(\zeta) :=  c_k^z \sin(\theta_k^z+\zeta) \cos(\zeta) + d_k\sin(\theta_k^z + \zeta) \sin(\zeta ), \quad \zeta\in \RR, \end{equation}
 and  $\mathcal{Z}_k^z$ is a $\mathcal{F}_k$-measurable variable satisfying that
\begin{equation} \label{momentZlemma}  |\EE_{k-1} \mathcal{Z}_k^z |\lesssim \frac{1}{(k-k_0)^{3/2}} + \frac{1}{\sqrt{k_0(k-k_0)}}, \ 
s_{k-1,\mathfrak{h}}(\mathcal{Z}_k^z)\lesssim  \frac{1}{(k-k_0)^{2}}.\end{equation}
%where $\mathfrak{C}>0$ depends on the model parameters.
\end{lemma}

\begin{proof}
We again omit $z$ and $\mathfrak{h}$ from the notation.
Let $U_k:= Y_k/\|Y_k\|$ for any $k\geq k_0+\ell_0$. By Lemma \ref{changebasisgene}, we can write for any $k\geq k_0+\ell_0$,  
\[ \frac{\|Y_k\|^2}{\|Y_{k-1}\|^2} = \|\Xi_k U_{k-1}\|^2= \|(R_k+G_k+\Delta_k+\mathcal{B}_k)U_{k-1}\|^2.\]
Expanding the norm and using that $R_k$ is a rotation matrix, we get
\[ \frac{\|Y_k\|^2}{\|Y_{k-1}\|^2} = 1 + 2\langle R_k U_{k-1},G_kU_{k-1}\rangle + 2\langle U_{k-1},\Delta_k U_{k-1}\rangle + \EE_{k-1} [\|G_k U_{k-1}\|^2] +\widetilde{\mathcal{Z}}_k,\]
where $\widetilde{\mathcal{Z}}_k$ represents all the lower order terms and is defined by
\begin{align*} \widetilde{\mathcal{Z}}_k &:= 2\langle B_kU_{k-1},(R_k+G_k+\Delta_k)U_{k-1}\rangle + \|B_k U_{k-1}\|^2 \\
& + 2\langle (R_k-\mathrm{Id})U_{k-1}, \Delta_k  U_{k-1}\rangle + \|G_k U_{k-1}\|^2-\EE_{k-1} [\|G_kU_{k-1}\|^2] \\
& + 2\langle G_kU_{k-1}, \Delta_k U_{k-1}\rangle + \|\Delta_k U_{k-1}\|^2.
\end{align*}
Using Lemma \ref{changebasisgene}  and \eqref{diffRI},  we find that 
\begin{equation} \label{vareZ}  |\EE_{k-1} \widetilde{\mathcal{Z}}_k |\lesssim \frac{1}{(k-k_0)^{2}} + \frac{1}{\sqrt{k_0(k-k_0)}}, \ s_{k-1}(\widetilde{\mathcal{Z}}_k )  \lesssim \frac{1}{(k-k_0)^2},\end{equation}
where $\mathfrak{C}>0$ depends on the model parameters.   Using Lemma  \ref{changebasisgene}, we find that  $\langle R_k U_{k-1}, G_k U_{k-1}\rangle = w_k(\zeta_{k-1})$, $2\langle U_{k-1}, \Delta_k U_{k-1}\rangle  = -\sin^2(\zeta_{k-1})/(k-k_0)$ and that $\EE_{k-1}[\|G_k U_{k-1}\|^2]  = \EE[(c_k\cos(\zeta_{k-1})+d_k\sin(\zeta_{k-1})^2]$. Using Lemma \ref{noisec} and the fact that $\EE(d_k^2) \lesssim 1/k_0$, 
we deduce that  $\EE_{k-1}[\|G_k U_{k-1}\|^2] = 2v \cos^2(\zeta_{k-1})/(k-k_0) +O(1/\sqrt{k_0(k-k_0)}) + O(1/(k-k_0)^2)$.  Hence, 
\begin{equation} \label{increnorm2} \frac{\|Y_k\|^2}{\|Y_{k-1}\|^2} = 1 +2w_k(\zeta_{k-1}) + \frac{1}{k-k_0}\big(2v \cos^2(\zeta_{k-1}) -\sin^2(\zeta_{k-1})\big) + {\mathcal{Z}}_k,\end{equation}
where $\mathcal{Z}_k$ is a $\cF_k$-measurable variable satisfying \eqref{vareZ}. 
Now, since $c_k, d_k, \mathcal{Z}_k$ are bounded by $(\log n)^\mathfrak{C}/\sqrt{k-k_0}$ for some constant $\mathfrak{C}>0$ almost surely,  we obtain by Taylor expanding the logarithm up to second order that 
\[ \log \frac{\|Y_k\|^2}{\|Y_{k-1}\|^2}  = 2w_k(\zeta_{k-1})+  \frac{1}{k-k_0}\big(2v \cos^2(\zeta_{k-1}) -\sin^2(\zeta_{k-1})\big) - 2 \EE_{k-1}[w_k(\zeta_{k-1})^2] +\mathcal{W}_k,\]
where $\mathcal{W}_k$ satisfies
\[ |\EE_{k-1} \mathcal{W}_k|\lesssim  \frac{1}{(k-k_0)^{2}} + \frac{1}{\sqrt{k_0(k-k_0)}}, \ s_{k-1}(\mathcal{W}_k) \lesssim \frac{1}{(k-k_0)^{3/2}},\]
where we used the fact that $\EE(|w_k|^p)\lesssim_p 1/(k-k_0)^{p/2}$  for any $p\geq 2$ by \eqref{descrB}-\eqref{estimnoiseXie}.
 Finally, using that $\EE(c_k^2) = 2v/(k-k_0) + O(1/\sqrt{k_0(k-k_0)}) + O(1/(k-k_0)^2)$, $\EE(d_k^2) =O(1/k_0)$ and that $\theta_k\lesssim \sqrt{(k-k_0)/k_0}$ by \eqref{diffRI}, we find that 
\[ \EE_{k-1}[w_k(\zeta_{k-1})^2]  = \frac{1}{2(k-k_0)} \sin^2(2\zeta_{k-1}) +O\Big(\frac{1}{(k-k_0)^2}\Big)+O\Big(\frac{1}{\sqrt{k_0(k-k_0)}}\Big),\]
which ends the proof of the claim. 
\end{proof}

Next, we prove that on a short enough interval $\llbracket k,k+m\rrbracket$, the vector 
%$Y_\ell^z$, 
$\zeta_\ell^z$ rotates with the almost the same angle as the rotation matrix $R_\ell^z$. More precisely, we have the following result. 
\begin{lemma}
\label{probagoodblockgene} There exists $\mathfrak{C}>0$ depending on the model parameters only such that for all $z\in I_\eta$ the following holds.
Let $k\geq k_{0,z} +\ell_0$ and $m\geq 1$. Then
%There exists $\mathfrak{C}>0$ depending on the model parameters such that 
\[ \PP_k\Big(\exists k<k' \leq k+m, |\zeta_{k'}^z-\zeta^z_k-\theta_{k,k'}|>\delta\big) \leq \exp\Big(-\min\Big\{\delta^2 \frac{k-k_{0,z}}{\mathfrak{C}m}, \delta n^{1/6} (\log n)^{\mathfrak{C}}\Big\}\Big),\]
for any $ \mathfrak{C} m/(k-k_{0,z})\leq \delta <\pi$.
\end{lemma}

\begin{proof} 
Using complex notation, we deduce from Lemma \ref{changebasisgene}  that 
\begin{equation} \label{eqangle} \frac{\|Y_\ell\|}{\|Y_{\ell-1}\|} e^{\mathrm{i} \zeta_\ell} =\Xi_\ell U_{\ell-1} = e^{\mathrm{i}(\zeta_{\ell-1} +  \theta_\ell)} + L_\ell,\end{equation} 
where $|\EE_{\ell-1} L_\ell|\lesssim 1/(\ell-k_0)$, $s_{\ell-1}(L_\ell) \lesssim 1/(\ell-k_0)$. Now, using Lemma \ref{recurnormell} we can Taylor expand $\|Y_{\ell-1}\|/\|Y_\ell \|$ up to the second order, which gives that
\begin{equation} \label{taylornorminv} \frac{\|Y_{\ell-1}\|}{\|Y_\ell\|} = 1 + H_\ell,\end{equation}
where $|\EE_{\ell-1} H_\ell|\lesssim 1/(\ell-k_0)$, $s_{\ell-1}(H_\ell) \lesssim 1/(\ell-k_0)$. From \eqref{eqangle} and \eqref{taylornorminv}, it follows that 
\begin{equation} \label{recurangle} e^{\mathrm{i} (\zeta_\ell-\zeta_{\ell-1} - \theta_\ell)} = 1 + {S}_\ell,\end{equation}
where ${S}_\ell = H_\ell(e^{\mathrm{i}(\zeta_{\ell-1}+\theta_\ell)} + L_\ell) + L_\ell$. Using the moments estimates of $H_\ell$ and $L_\ell$, we find that $|\EE_{\ell-1} S_\ell|\lesssim 1/(\ell-k_0)$ and $s_{\ell-1}(S_\ell) \lesssim 1/(\ell-k_0)$.
Since $S_\ell$ is almost surely bounded by $(\log n)^\mathfrak{C}/\sqrt{\ell-k_0}$, where $\mathfrak{C}>0$ depends on the model parameters, we can take the principal branch of the logarithm on both sides of \eqref{recurangle} and use Taylor's theorem  to the second order to get that 
\begin{equation} \label{represincremarg} \zeta_\ell-\zeta_{\ell-1} - \theta_\ell  \equiv T_\ell  \  [2\pi],\end{equation}
where $T_\ell$ is a $\cF_\ell$-random variable such that $|\EE_{\ell-1} T_\ell|\lesssim 1/(\ell-k_0)$ and $s_{\ell-1}(T_\ell)\lesssim 1/(\ell-k_0)$. Now, denote by $\mathscr{B}$ the event where there exists $k'\in \llbracket k,k+m\rrbracket$ such that $|\zeta_{k'}-\zeta_k-\theta_{k,k'}|>\delta$. From  \eqref{represincremarg}, we deduce that there exists $k' \in \llbracket k,k+m\rrbracket$ such that $\sum_{\ell=k+1}^{k'} T_\ell \notin [-\delta,\delta] +2\pi \ZZ$. Using that $\delta<\pi$, we get that 
\[ \PP_k(\mathscr{B}) \leq  \PP_k\Big(\max_{k+1\leq k' \leq k+m}\big|\sum_{\ell= k+1}^{k'} T_\ell \big|>\delta\Big). \]
But, 
\[ \sum_{\ell = k+1}^{k+m} \|\EE_{\ell-1}(T_\ell)\|_{L^\infty}\lesssim \frac{m}{k-k_0}, \ \sum_{\ell=k+1}^{k+m} \| \Var_{\ell-1}(T_\ell)\|_{L^\infty} \lesssim \frac{m}{k-k_0},\]
and $\max_{k+1\leq \ell \leq k+m} \|s_{\ell-1}(T_\ell)\|_{L^\infty}\lesssim n^{1/3}$ since $k\geq k_0+\ell_0$. Lemma \ref{tailpropH} finally ends the proof of the claim. 
\end{proof}

Specializing the result of Lemma \ref{probagoodblockgene} to the blocks $(k_{i,z})_i$ defined in \eqref{defblockosci}, we can now give a proof of Lemma \ref{probagoodblock}.

\begin{proof}[Proof of Lemma \ref{probagoodblock}]
Let $i_o\leq i \leq i_1$. Note that $(k_{i+1}-k_i)/k_i\asymp 1/i$. Hence, applying Lemma \ref{probagoodblockgene} with $\delta = i^{-1/4}$, $k= k_i$, $m = k_{i+1}-k_i$ and using the fact that $i_1\lesssim n^{1/6}$, we get the claim. 
\end{proof}

Finally, we close this section by showing the representation of the increments $\Delta \psi_{k_{i+1,z}}(z) := \psi_{k_{i+1,z}}(z) - \psi_{k_{i,z}}(z)$ claimed in Proposition \ref{represelliptic}.
Before doing so, we state a technical result on the approximation of Riemann sums by integrals.

\begin{lemma}\label{approxint} Let $f:\RR\to\RR$ be a  $2\pi$-periodic, $1$-Lipschitz function.
% such that $\sup_{x\in \RR}|f(x)| \leq 1$. 
For any  $k'\geq k \geq k_{0,z}+\ell_0$, $\zeta \in \RR$,
\begin{align*}& \sum_{\ell=k+1}^{k'}\frac{1}{\ell-k_{0,z}}f(\zeta+\theta^z_{k, \ell}) \\&= \Big( \sum_{\ell= k+1}^{k'} \frac{1}{\ell-k_{0,z}} \Big)\int_0^{2\pi} f(x) \frac{dx}{2\pi} +  O\Big( k_{0,z}^{\frac12} \big( (k-k_{0,z})^{-\frac32} - (k'-k_{0,z})^{-\frac32}\big)\Big)\\
&\quad  + O\Big(k_{0,z}^{-1/2}\big((k'-k_{0,z})^{1/2}-(k-k_{0,z})^{1/2}\big)\Big),
\end{align*}
where $\theta^z_{k,\ell} = \sum_{j=k+1}^\ell \theta^z_j$ for any $k\leq \ell \leq k'$. 
\end{lemma}

\begin{proof}[Proof of Lemma \ref{approxint}]
We may and will assume that $f(0)=0$. In that case, $\|f\|_\infty\leq \pi$. As always, we omit $z$ from the notation in the proof.
Define $(p_i)_{i\geq 0}$  inductively by  $p_0=k$ and 
\[ p_{i+1} := p_i + \lceil \frac{2\pi}{\theta_{p_i}}\rceil, \  i\geq 0.\]
Let $r= \max\{i : p_i <k'\}+1$ and redefine $p_r = k'$.
One can check that this implies that $p_i \asymp  k_0^{1/3} (i+s)^{2/3} $ for any $1\leq i \leq r-1$, where  $s \asymp (k-k_0)^{3/2} k_0^{-1/2}$, and that moreover  $r +s\asymp  (k'-k_0)^{3/2} k_0^{-1/2}$. Denote by $\zeta_i = \zeta +\theta_{p_0,p_i}$ for any $i\geq 0$.  Using that $f$ is bounded by $1$, we get that 
\[  \sum_{\ell=p_i+1}^{p_{i+1}} \frac{1}{\ell-k_0} f(\zeta_i+\theta_{p_i, \ell})= \frac{1}{p_i-k_0} \sum_{\ell=p_i+1}^{p_{i+1}} f(\zeta_i+\theta_{p_i, \ell})  +O\Big( \Big(\frac{\Delta p_i}{p_i-k_0}\Big)^2\Big).\]
Now, by definition $\theta_\ell = \arcsin(\frac{\sqrt{4-z_\ell^2}}{2})$ for any $\ell\geq k_0+\ell_0$, so that 
\[ \partial_\ell \theta_\ell = \frac{z_\ell}{2\ell \sqrt{4-z_\ell^2}}, \quad \ell \geq k_0+\ell_0.\]
Thus, using \eqref{estim} we find that 
\[ |\theta_{p_i,\ell} - \theta_{p_i}(\ell- p_i)| \lesssim \frac{(\Delta p_i)^2}{\sqrt{k_0(p_i-k_0)}}, \quad \ell \in [p_i,p_{i+1}], i\geq 0.  \]
Since $f$ is $1$-Lipschitz, the above estimate implies that for any $p_i < \ell \leq p_{i+1}$
\begin{equation} \label{eqapprox}
 f(\zeta_i+\theta_{p_i,\ell}) = f(\zeta_i + \theta_{\ell_i}(\ell-p_i)) + O\Big(\frac{(\Delta p_i)^2}{\sqrt{k_0(p_i-k_0)}}\Big).
\end{equation}
Using again that $f$ is $1$-Lipschitz and $\|f\|_\infty\leq \pi$, we can now compare the sum to an integral:
\begin{equation} \label{eqintsum}\sum_{\ell=p_i+1}^{p_{i+1}} f(\zeta_i + \theta_{p_i}(\ell-p_i))= \int_{\ell_i}^{p_{i+1}} f(\zeta_i + \theta_{p_i} (x-p_i)) dx + O(1),\end{equation}
where we used the fact that $\Delta p_i \theta_{p_i} \lesssim 1$. Now, by definition we have
\begin{equation} \label{clock2} \theta_{p_i} \Delta p_i,\end{equation}
which  implies, as $f$ is  $2\pi$-periodic and $\|f\|_\infty\leq \pi$, that
% is bounded by $1$ and,  that 
\begin{equation}  \int_{p_i}^{p_{i+1}} f(\zeta_i + \theta_{p_i} (x-p_i)) dx =\frac{1}{\theta_{p_i}} \int_0^{2\pi} f(x) dx +O(1).\end{equation}
Next, using \eqref{clock2} we get that 
\[\Big|\frac{1}{(p_i-k_0)\theta_{p_i}} - \frac{\Delta p_i}{2\pi(p_i-k_0)}\Big| \lesssim \frac{1}{p_i-k_0}.\]
Therefore, 
\begin{equation}\label{intperiod}  \frac{1}{p_i-k_0} \int_{p_i}^{p_{i+1}} f(\zeta_i + \theta_{p_i} (x-p_i)) dx =\frac{\Delta p_i}{p_i-k_0} \int_0^{2\pi} f(x) \frac{dx}{2\pi} +O\Big(\frac{1}{p_i-k_0} \Big).\end{equation}
Putting together \eqref{eqapprox}, \eqref{eqintsum}, \eqref{intperiod}, and using that  $p_i = \mathfrak{C} k_0^{1/3} (i+s)^{2/3} + O(1)$ for any $1\leq i \leq r-1$, we find that 
\begin{align*}
 & \sum_{\ell=p_i+1}^{p_{i+1}} \frac{1}{\ell-k_0} f(\zeta_i+\theta_{p_i, \ell}) \\
&= \frac{\Delta p_i}{p_i-k_0} \int_0^{2\pi} f(x) \frac{dx}{2\pi} +O\Big(\frac{(\Delta p_i)^3}{k_0^{1/2}(p_i-k_0)^{3/2}}\Big)+O\Big( \frac{1}{p_i-k_0}\Big) + O\Big(\Big(\frac{\Delta p_i}{p_i-k_0}\Big)^2\Big) \\
& =\Big( \sum_{\ell=p_i+1}^{p_{i+1}}\frac{1}{\ell-k_0}\Big) \int_0^{2\pi} f(x) \frac{dx}{2\pi} + O\Big(\frac{1}{n^{1/3}(i+s)^{2/3}}\Big) + O\Big(\frac{1}{(i+s)^{2}}\Big). 
\end{align*}
Finally, as  $s\asymp (k-k_0)^{3/2} k_0^{-1/2}$ and $r +s\asymp  (k'-k_0)^{3/2} k_0^{-1/2}$, summing over $i$ gives the claim. 
\end{proof}

Equipped with the result of Lemma \ref{approxint} we can now state  a generalization of Lemma \ref{represelliptic}, where we seek a representation of a general increment of $\psi(z)$ on a good event as a sum of independent random variables.

\begin{lemma}\label{represellipticg}
There exists $\mathfrak{h}>0$ so that for any $z\in I_\eta$, the following holds.
Let $k\geq k_{0,z} +\ell_0$ and $m\geq 1$. Denote by  $\mathscr{G}^\delta_{k,m} := \{|\zeta_{\ell} - \zeta_{k} - \theta_{k,j} | \leq \delta, \forall k+1\leq \ell \leq k+m\}$. On the event $\mathscr{G}^\delta_{k,m} $,
\begin{align} \label{represpsieq} &\psi_{k+m}(z) - \psi_k(z) \\
&= \sum_{\ell=k+1}^{k+m} \big(w_\ell^z(\widehat{\zeta}^z_\ell) + \mathcal{P}_\ell^z) + O\Big(\frac{m}{\sqrt{k_{0,z}(k-k_{0,z}})}\Big)+O\Big(\frac{\delta m}{k-k_0}\Big) +O\Big(\frac{n^{1/2}m}{(k-k_0)^{5/2}}\Big), \nonumber\end{align}
where $\widehat{\zeta}^z_\ell := \zeta^z_k + \theta^z_{k,\ell}$ for any $\ell\geq k$, $w_\ell^z$ is defined in \eqref{defw} and $\mathcal{P}_\ell^z$ is an $\cF_k$-measurable variable satisfying that 
\begin{equation} \label{momentestimPeg}   \EE_{\ell-1}(\mathcal{P}_\ell^z) = 0, \ s_{\ell-1,\mathfrak{h}}(\mathcal{P}_\ell^z) \lesssim  \frac{\delta }{k-k_{0,z}} + \frac{1}{(k-k_{0,z})^2}\end{equation}
Further, there exists a deterministic sequence $(\sigma_{\ell,z}^2)_{k<j\leq k+m}$ of  positive real numbers such that 
\begin{equation} \label{varclaimeqw} \Var_k \Big(\sum_{\ell=k+1}^{k+m} w_\ell^z(\widehat{\zeta}_{\ell-1})\Big) = \sum_{\ell=k+1}^{k+m} \sigma_{\ell,z}^2 + O\Big(\frac{mn^{1/2}}{(k-k_0)^{5/2}}\Big) , \ \text{a.s,}\end{equation}
and
\begin{equation} \label{varclaimeq}  \sigma_{\ell,z}^2 = \frac{v}{\ell-k_{0,z}} + O\Big(\frac{1}{\sqrt{k_{0,z}(\ell-k_{0,z})}}\Big).\end{equation}
\end{lemma}
\begin{proof}
We omit $z$ from the notation in the proof.
For any $k\leq \ell \leq k+m$, denote by $\mathscr{G}^\delta_\ell := \{ |\zeta_\ell-\widehat{\zeta}_\ell|\leq \delta\}$. Recall the definition of the process $\psi_k$ in \eqref{defpsi}.
By Lemma  \ref{recurnormell} we can write on the event $\mathscr{G}_{k,m}^\delta$,
\[ \psi_{k+m} -\psi_k = \sum_{\ell=k+1}^{k+m} w_\ell(\widehat{\zeta}_{\ell-1}) +\sum_{\ell=k+1}^{k+m}  \Big(m_\ell(\zeta_{\ell-1})  -\frac{v-1}{4(\ell-k_0)}+  \mathcal{Q}_\ell\Big), \]
where $w_\ell$ is defined in \eqref{defw}, $\mathcal{Q}_\ell:= (w_\ell(\zeta_{\ell-1}) - w_\ell(\widehat{\zeta}_{\ell-1}))\Car_{\mathscr{H}_{\ell-1}} + \mathcal{Z}_\ell$, with $\mathcal{Z}_\ell$ satisfying \eqref{momentZlemma} and 
\[ m_\ell(\zeta) :=  \frac{1}{\ell-k_{0}}\Big(2v \cos^2(\zeta) -\sin^2(\zeta)-v \sin^2(2\zeta)\Big), \quad \zeta \in \RR. \]
As $m_\ell$ is $O(1/(\ell-k_0))$-Lipschitz, we have that on the event $\mathscr{G}_{k,m}^\delta$, 
\[ \sum_{\ell=k+1}^{k+m}\big(m_\ell(\zeta_{\ell-1})-m_\ell(\widehat{\zeta}_{\ell-1})\big) = O\Big( \frac{\delta m}{k-k_0}\Big).\]
Using Lemma \ref{approxint}, we get that 
\[ \sum_{\ell=k+1}^{k+m}  \Big(m_\ell(\widehat{\zeta}_{\ell-1}) - \frac{v-1}{4(\ell-k_0)}\Big) =  O\Big(\frac{m n^{1/2}}{(k-k_0)^{5/2}}\Big) + O\Big(\frac{m}{\sqrt{k_0(k-k_0)}}\Big).\]
Using the estimates on $\mathcal{Z}_\ell$ from Lemma \ref{recurnormell} together with the fact that $ \EE_{\ell-1}[(w_\ell(\zeta_{\ell-1})-w_\ell(\widehat{\zeta}_{\ell-1}))^2] \lesssim \delta/(\ell-k_0)$ on the event $\mathscr{G}_\ell^\delta$,  we obtain \eqref{represpsieq} and \eqref{momentestimPeg}. 

We now turn our attention to \eqref{varclaimeqw}.
%eqref{varcomput}. 
By the definition of $w_\ell(\zeta)$, we have that 
\begin{align*} \Var_{k}\Big(\sum_{\ell=k+1}^{k+m} w_\ell(\zeta_{\ell-1}) \Big)  &= 4\sum_{\ell= k+1}^{k+m} \big(\EE (c_\ell^2 )\cos^2(\widehat{\zeta}_{\ell-1}) +\EE (c_\ell d_\ell) \sin(2\widehat{\zeta}_{j-1})\big)\sin^2(\theta_j+\widehat{\zeta}_{j-1})\\
&  +4\sum_{\ell = k+1}^{k+m} \EE (d_\ell^2) \sin^2(\theta_{\ell} +\widehat{\zeta}_{\ell-1})\sin^2(\widehat{\zeta}_{\ell-1}).\end{align*}
Expanding $\sin^2(\theta_\ell+\widehat{\zeta}_{\ell-1})$ in order to write the three sums  as $\sum_j p_{\ell,j} f_j(\widehat{\zeta}_{\ell-1})$ where $f_\ell$ are bounded and $2\pi$-periodic functions, and applying Lemma \ref{approxint} to each of the terms, we find that 
\begin{align*}  \Var_{k}\Big(\sum_{\ell=k+1}^{k+m} w_\ell(\widehat{\zeta}_{\ell-1}) \Big) &= \sum_{\ell=k+1}^{k+m}\Big[ \big(\EE c_\ell^2+\EE d_\ell^2\big)\big(1+\frac12 \sin^2(\theta_\ell)\big) + \EE (c_\ell d_\ell) \big(\sin(2\theta_\ell) -\frac12 \cos^2(\theta_\ell)\big)\Big]\\
& + O\Big(\frac{mn^{1/2}}{(k-k_0)^{5/2}}\Big) \end{align*}
where we used the fact that $\EE(c_\ell^2),\EE(d_\ell^2),|\EE(c_\ell d_\ell)|\lesssim 1/(\ell-k_0)$. Set 
\begin{equation}\label{expresigmaelliptic} \sigma_\ell^2 := \big(\EE c_\ell^2+\EE d_\ell^2\big)\big(1+\frac12 \sin^2(\theta_\ell)\big) + \EE (c_\ell d_\ell) \big(\sin(2\theta_\ell) -\frac12 \cos^2(\theta_\ell)\big).\end{equation} Using that $\EE (c_\ell^2) = 2v/(\ell-k_0) + O(1/n) +O(1/(\ell-k_0)^2)$ by Lemma \ref{noisec}, $\EE (d_\ell^2) =O(1/n)$ and that $\theta_\ell \lesssim \sqrt{(\ell-k_0)/k_0}$ by \eqref{diffRI}, we get the last claim \eqref{varclaimeq}.
\end{proof}

Applying Lemma \ref{represellipticg} to each of the blocks $\llbracket k_{i,z}, k_{i+1,z}\rrbracket$ defined in \eqref{defblockosci}, we can now give a proof of Lemma \ref{represelliptic}. 
\begin{proof}[Proof of Lemma \ref{represelliptic}] 
Let $j_o\leq i \leq j_1 $. Recall the definition of the blocks $\llbracket k_{i,z}, k_{i+1,z}\rrbracket$. Write $\Delta k_{i+1} = k_{i+1}-k_i$. Using the previous Lemma \ref{represellipticg} with $k= k_i$, $m= \Delta k_{i+1} $ and $\delta= \delta_i = i^{-1/4}$,  we find that on the event $\mathscr{G}_i$, the error terms in \eqref{defblockosci} are bounded as follows:
\[ \frac{\Delta k_{i+1}}{\sqrt{k_0(k_i-k_0)}} \lesssim  n^{-1/3} i, \ \frac{\delta_i \Delta k_{i+1}}{k_i-k_0} \lesssim i^{-5/4}, \ \frac{\Delta k_{i+1} n^{1/2}}{(k_i-k_0)^{5/2}}\lesssim i^{-7},\]
where we used the fact that $\Delta k_{i+1} \lesssim n^{1/3} i^3$. Note that $n^{-1/3} i \lesssim i^{-5/4}$ since $j_1\lesssim n^{1/6}$. This ends the proof of the lemma.
\end{proof}

\section{Precise exponential moment estimates}
\label{sec-Precise}
Building on the representations of the process $\psi(z)$ from Propositions \ref{incremhn}, \ref{increm} and \ref{represelliptic}, we compute  precise bounds on the high (but bounded) exponential moments of the increments of $\psi(z)$.

To state our result, we introduce the notion of {\em approximate instantaneous variance} $\widehat{\sigma}^2_{k,z}$ at time $k$ and {\em approximate accumulated variance} $\widehat{\Sigma}^2_{k,z}$  until time $k$ by setting
\begin{equation} \label{defSigma} \widehat{\sigma}^2_{k,z}  := \begin{cases}
\frac{v}{2(k_{0,z}-k)} & \text{ if } k\leq k_{0,z}-\ell_0\\
0 & \text{ if } |k-k_{0,z} |\leq \ell_0 \\
\frac{v}{4(k-k_{0,z})} & \text{ if } k>k_{0,z}+\ell_0.
\end{cases}, \quad \widehat{\Sigma}^2_{k,z} := \sum_{\ell =1}^k \sigma^2_\ell(z).\end{equation}
$\widehat{\sigma}_k^2(z)$ corresponds to the first order term of the variance $\sigma_{k,z}^2$ of the noise driving the  recursions in Propositions \ref{increm} and \ref{represelliptic}.
With this notation, we will prove the  following exponential moment upper bound estimates for the increments of the process $\psi(z)$.

\begin{Pro}\label{expopreciseincre}
There exists $\mathfrak{C}_\kappa>0$ depending on $\kappa$ and $\mathfrak{d}\in (0,1)$ depending on the model parameters only,
such that for any  $z\in I_\eta$, $k_{\delta,z} \leq k \leq k' \leq n$ and $0\leq \lambda \leq \kappa^{\mathfrak{d}}$,
\begin{equation} \label{upperboundstatement}   \log  \EE_k\big[ e^{\lambda(\psi_{k'}(z)-\psi_k(z))} \big] =  \frac{\lambda^2}{2}  (\widehat{\Sigma}_{k',z}^2-\widehat{\Sigma}_{k,z}^2) + \mathfrak{C}_\kappa (|\lambda| \vee \lambda^3),   \end{equation}
on the event where $W_k^z \leq \eta_{k,z}$ if $k\leq k_{0,z}-\ell_0$. 
Further, if $k' \leq k_{0,z}-\ell_0$ or $k\geq k_{0,z}+\ell_0$, then $\mathfrak{C}_\kappa$ can be taken independent of $\kappa$.
\end{Pro}

\begin{Rem}\label{exponorm} We note that an exponential moment upper bound is true as well for the logarithm of the norm of products of the transition matrices. Indeed, 
the proof of Proposition \ref{expopreciseincre} actually shows that for any $u \in \mathbb{S}^1$, $\log \|\Xi_{k,k'}^zu\| - (M_{k'}(z) - M_k(z))$ satisfies the same exponential moment upper bound as $\psi_{k'}(z) - \psi_k(z)$ conditionally on $\mathcal{F}_k$. With a similar argument as in the proof of Lemma \ref{aprioridom}, it follows that \eqref{upperboundstatement} holds as well for $\log \|\Xi_{k,k'}^z\| - (M_{k'}(z) - M_k(z))$. We will use this fact in the proof of the barrier estimate in the elliptic regime in Section \ref{barrierelliptic}.
\end{Rem}

\begin{Rem}
One could compute higher exponential moments provided one considers increments away from the parabolic regime. Indeed, a close attention to the proof shows that the same exponential moments estimate holds for $0\leq \lambda \leq  ((k_{0,z}-k')/n^{1/3})^{\mathfrak{d}}$ if $k'\leq k_{0,z}-\ell_0$ and for  $|\lambda| \leq ((k'-k_{0,z})/n^{1/3})^{\mathfrak{d}}$ if $k\geq k_{0,z}+\ell_0$, where $\mathfrak{d}>0$ is some numerical constant. 
\end{Rem}

Combining the exponential moment estimate of Proposition \ref{expopreciseincre} together with the a priori estimate of Proposition \ref{apriori} and the fact that $W_k^z\leq \eta_{k,z}$ at time $k= k_{\delta,z}$ almost surely by Proposition \ref{boundWhn}, we get the following immediate corollary. 

\begin{corollary}\label{expoprecise}
There exists $\mathfrak{C}_\kappa>0$ depending on $\kappa$ and $\mathfrak{d}\in (0,1)$ depending on the model parameters such that for any   $1 \leq k  \leq n$, $0\leq \lambda \leq \kappa^{\mathfrak{d}}$,
\begin{equation} \label{upperboundstatementcor}   \log  \EE\big[ e^{\lambda \psi_{k}(z) } \big] =  \frac{\lambda^2}{2}  \widehat{\Sigma}_{k,z}^2  + \mathfrak{C}_\kappa (|\lambda| \vee \lambda^3),   \end{equation}
Further, if  $k\geq k_{0,z}+\ell_0$, then $\mathfrak{C}_\kappa$ can be taken independent of $\kappa$.
\end{corollary}

 Further, using the domination of the increments results of Proposition \ref{coreqhyper}, we obtain a precise upper bound estimate on perturbed exponential moments with positive parameters, {\em regardless of the initial condition}. This result will be instrumental in proving the decorrelation estimate of section \ref{decorrsection}.

\begin{Pro}
\label{expomomentpreciseperturb}
Let  $0\leq k \leq k' \leq n$. Assume that $(S_\ell)_{k< \ell \leq k'}$ is a martingale such that $s_{\ell-1,\mathfrak{h}}(S_\ell) \lesssim 1/|\ell-k_0|$ for some $\mathfrak{h}>0$.
Then, there exist $\mathfrak{d},\mathfrak{C}>0$ depending  only on the model parameters and $\mathfrak{h}$
such that for  any  $|\lambda|\vee|\mu| \leq  \kappa^{\mathfrak{d}}$,
\begin{align*} \label{expoeq} \log \EE_k\Big[e^{\lambda(\psi_{k'}(z)- \psi_k(z))+\mu \sum_{\ell= k+1}^k S_\ell} \Big] &\leq  \frac{\lambda^2}{2} \big(\widehat{\Sigma}_{k',z}^2 -  \widehat{\Sigma}_{k,z}^2\big) +\frac{\mu^2}{2}s^2 + \mathfrak{C}\lambda |\mu |c +\mathfrak{C}(|\lambda| \vee |\lambda|^3\vee |\mu|^3),
\end{align*}
where $s^2 := \sum_{\ell = k+1}^{k'} \| \Var_{\ell-1}(S_\ell) \|_{L^\infty}$, $c:=\sum_{\ell =k+1}^{k'}  \|\sqrt{\frac{\Var_{\ell-1}(S_\ell)}{\ell-k_{0,z}}}\|_{L^\infty}$.
\end{Pro}

To prove  Propositions \ref{expopreciseincre} and \ref{expomomentpreciseperturb}, we argue separately in the hyperbolic and elliptic regime, the parabolic regime being covered by the a priori estimate of Proposition \ref{apriori}.

\subsection{The hyperbolic regime}
We start by proving the claimed exponential moment estimate of the increments of the process $\psi(z)$ in the contributing hypergolic regime. 
\begin{lemma}
\label{expomomentpreciseperturbh}
For $\kappa$ large enough, there exist $\mathfrak{d},\mathfrak{C}>0$ depending on the model parameters only 
such that the following holds for any $z\in I_\eta$.
Let $k_{\delta,z}\leq k \leq k' \leq k_{0,z} - \ell_0$.  Then, for any  $|\lambda|\leq \kappa^\mathfrak{d}$, on the event where $W_k^z \leq \eta_{k,z}$, 
\begin{align*} \log \EE_k\Big[e^{\lambda(\psi_{k'}(z)- \psi_k(z))} \Big] &=\frac{\lambda^2}{2} \big(\widehat{\Sigma}_{k',z}^2 -  \widehat{\Sigma}_{k,z}^2\big)  +\mathfrak{C}(|\lambda| \vee |\lambda|^3).\end{align*}  \end{lemma}

\begin{proof} 
Recall the blocks $(m_i)_{i_o\leq i \leq i_1}$ defined in \eqref{defblockh}. We claim that it is enough to prove the claimed exponential moment estimate for $\psi_{m_{i'}}- \psi_{m_i}$ where  $i_o\leq i' <i \leq i_1$ are such that $k\in (m_{i+1},m_i]$
and $k'\in (m_{i'+1},m_{i'}]$.  Indeed, $\sum_{\ell= m_{j+1}+1}^{m_j} 1/(k_0-\ell) \lesssim 1/j \lesssim 1$, so that by using the a priori estimate of Proposition \ref{apriori} it follows that for any $|\lambda|  \leq (\log n)^{-\mathfrak{h}} n^{1/6}$, and $m_{j+1} \leq k\leq k' \leq m_{j}$, 
\[ \EE_{k}\Big[e^{\lambda (\psi_{k'} -\psi_{k})}\Big] \leq e^{\mathfrak{C}(|\lambda|\vee \lambda^2)},\]
where $\mathfrak{h},\mathfrak{C}>0$ depend only on the model parameters. Moreover, by Proposition \ref{increm}, we know that almost surely, $|\EE_{k}( \psi_{k'}-\psi_k)| \lesssim1$ for any  $m_{j+1}\leq k \leq k' \leq m_j$, using the fact that $\sum_{\ell=m_{j+1}+1}^{m_j} 1/(k_0-\ell) \lesssim 1$ and that $\sum_{\ell=m_{j+1}+1}^{m_j} \sqrt{(\ell-k_0)/k_0} \lesssim 1$.
 Hence, by Jensen's inequality, we get that for any $m_{j+1}\leq k \leq k' \leq m_j$, 
\[ \EE_{k}\Big[e^{\lambda (\psi_{k'} -\psi_{k})}\Big] \geq e^{-\mathfrak{C}|\lambda|},\]
where $\mathfrak{C}$ changed value without changing name.

Thus, it suffices to compute the exponential moment of $\psi_{m_{i'}}- \psi_{m_i}$ with $i_o\leq i' <i \leq i_1$. Throughout,
we let  $k\in (m_{i+1},m_i]$.
%   for some $i_o\leq i' <i \leq i_1$ such that $k\in (m_{i+1},m_i]$. 
For any $i_o\leq j \leq i_1$, say that the $j^{\text{th}}$ block  is good if for all $k\in (m_{j+1},m_j]$, $W_{k-1}\leq 2\eta_k$,
 and denote by $\mathscr{H}_j$ this event.  We first start with the upper bound. Let $\tau$ be the largest index $i'\leq j  \leq i-1$ such that the $j^{\text{th}}$ block is bad, setting $\tau=0$ if no bad blocks exist.
%if there are any and set $\tau=0$ otherwise.
 Fix $i'\leq j \leq i-1$. We distinguish between two cases.

 If $4j\geq i$, then on the event where $W_k\leq \eta_k$, we have by Proposition \ref{controlbadblockh} that $\PP_k(\mathscr{H}_j^\complement) \leq e^{-\mathfrak{c} j^{1/6}}$, where $\mathfrak{c}>0$ depends on the model parameters. Using the Cauchy-Schwarz inequality, the a priori estimate of Proposition \ref{apriori} and Proposition \ref{controlbadblockh}, we get that on the event where $W_k\leq \eta_k$, 
 \[ \log \EE_{k}\big[\Car_{\tau=j}e^{\lambda (\psi_{m_{i'}}-\psi_{m_i})}\big] \leq \mathfrak{C}(|\lambda|\vee \lambda^2) \log i - 
\mathfrak{C}^{-1} j^{1/6},\]
 where $\mathfrak{C}>0$ depends on the model parameters and changed value without changing name and where we used the fact that $\sum_{\ell = m_{i+1}+1}^{k_0-\ell_0} 1/(k_0-\ell) \lesssim \log i$. Since $i\geq i_o\gtrsim \kappa^{3/2}$, we deduce that there exists $\mathfrak{d}>0$ such that for $\kappa$ large enough and $|\lambda|,|\mu|\leq \kappa^\mathfrak{d}$, 
\begin{equation} \label{probaexpobadblock1}  \log \EE_{k}\big[\Car_{\tau=j}e^{\lambda (\psi_{m_{i'}}-\psi_{m_{i}})}\big] \leq  - \frac12\mathfrak{C}^{-1} j^{1/6},\quad \mbox{\rm if}\, 4j\geq i.\end{equation}

On the other hand, if
 $4j< i$, then as $\eta_{m_{4j}}\leq \eta_{m_j}/2$,  we get by Proposition \ref{controlbadblockh}, that on the event where $W_{m_{4j}}\leq 2\eta_{m_{4j}}$, it holds that 
 $\PP_{m_{4j}}(\mathscr{H}_j^\complement) \leq e^{-\mathfrak{c} j^{1/6}}$. Using the same argument as in \eqref{probaexpobadblock1} we deduce  that for $\kappa$ large enough and $|\lambda|,|\mu|\leq \kappa^\mathfrak{d}$, 
\begin{equation} \label{probaexpobadblock2}  \log \EE_{m_{4j}}\big[\Car_{\mathscr{H}_j^\complement}e^{\lambda (\psi_{m_{i'}}-\psi_{m_{4j}}) }\big] \leq  - \frac12\mathfrak{C}^{-1} j^{1/6},\quad \mbox{\rm if} \,j\neq 0, 4j< i.\end{equation}
Next, on the event $\bigcap_{p=4j}^{i-1} \mathscr{H}_p$, we have by Proposition \ref{increm} that 
\begin{equation} \label{eqpsiggodvent} \psi_{m_{4j}} - \psi_{m_i} = \sum_{\ell = m_i+1}^{m_{4j}} (g_\ell +\widetilde{\mathcal{P}}_\ell )+ O(1), \end{equation}
where $ \widetilde{\mathcal{P}}_\ell:= \mathcal{P}_\ell \Car_{W_{\ell-1}\leq \eta_{\ell-1}}$, with $g_\ell$ defined in \eqref{defgk} and $\mathcal{P}_\ell$ satisfying the moments estimates of \eqref{momentP1prop}-\eqref{momentP2prop}. In particular,  one can check that 
\begin{equation} \label{estimmomentPtilde} \sum_{\ell \leq k_0-\ell_0} \|\sqrt{\frac{\Var_{\ell-1}(\widetilde{\mathcal{P}}_\ell)}{k_0-\ell}}\|_{L^\infty} \lesssim 1, \ \sum_{\ell \leq k_0-\ell_0} \|\Var_{\ell-1}(\widetilde{\mathcal{P}}_\ell)\|_{L^\infty}, \ \sum_{\ell\leq k_0-\ell_0} \|\EE_{\ell-1}(\widetilde{\mathcal{P}}_\ell)\|_{L^\infty} \lesssim 1.\end{equation}
Together with the fact that $\Var_{\ell-1}(g_\ell) = 2v/(k_0-\ell) + O(1/\sqrt{k_0(k_0-k)})+O(1/n)$ by Proposition \ref{increm}, and that $\Var_{\ell-1}(g_\ell+\widetilde{\mathcal{P}}_\ell) \leq 1/(k_0-\ell)$, this yields that 
\[ \sum_{\ell=m_i+1}^{m_{i'}} \Big\|\Var_{\ell-1}(g_\ell+\widetilde{\mathcal{P}}_\ell) - \frac{v}{2(k_0-\ell)}\big\|_{L^\infty} \lesssim 1.\]
Further, note that by \eqref{momentP2prop} and the fact that the noise is bounded (see \eqref{boundnoiseass}) we have that  $|g_\ell +\widetilde{\mathcal{P}}_\ell-\EE_{\ell-1}(\widetilde{\mathcal{P}}_\ell)|  \leq (\log n)^{\mathfrak{h}}/(k_0-\ell)$ for any $\ell$, where $\mathfrak{h}>0$ depends on the model parameters. Moreover, $\sum_{\ell\leq k_0-\ell_0} (\log n)^{3\mathfrak{h}/2}/(k_0-\ell)^{3/2} \lesssim 1$. Hence, using Lemma \ref{moddev} and \eqref{eqpsiggodvent}, we deduce that for $|\lambda|,|\mu|\leq (\log n)^{-\mathfrak{h}} n^{1/6}$, 
\[ \log \EE_k\big[\Car_{\bigcap_{p=4j}^{i-1} \mathscr{H}_p} e^{\lambda  (\psi_{4j}-\psi_{m_i}) }\big]\leq   \frac{\lambda^2}{4}  \sum_{\ell=m_i+1}^{m_{i'}} \frac{v}{k_0-\ell}  +\mathfrak{C}(|\lambda| \vee  |\lambda|^3),  \]
where $\mathfrak{C}$ changed value without changing name. Combined with \eqref{probaexpobadblock2}, this implies that 
\begin{align} \log \EE_{k}\big[\Car_{\tau = j} e^{\lambda (\psi_{m_{i'}}-\psi_{m_{i}}) }\big] &\leq    \frac{\lambda^2}{4}  \sum_{\ell=m_i+1}^{m_{i'}} \frac{v}{k_0-\ell}  +\mathfrak{C}(|\lambda| \vee |\lambda|^3) - \mathfrak{C}^{-1} j^{1/6}, \quad \mbox{\rm if}\, 4j\leq i
,\label{momentexpo2}\end{align}
for any $|\lambda| \leq \kappa^\mathfrak{d}$, where $\mathfrak{C}$ again changed value without changing name. Note that the same estimate clearly holds as well when $j=0$ since in that case all blocks from $m_i$ to $m_{i'}$ are good. Summing (the exponential of) \eqref{probaexpobadblock1} and \eqref{momentexpo2} over $j$ concludes the proof of the upper bound. 
%\textcolor{blue}{Question: I do not understand this summation. You do not sum the logs, you sum the exponential of the logs, and this %introduces a cardinality term?}

We now turn our attention to the lower bound. Denote by $\mathscr{H} := \{W_{\ell-1}\leq 2\eta_\ell, m_i+1\leq \ell \leq m_{i'}\}$. On the event $\mathscr{H}$, we have by \eqref{eqpsiggodvent}-\eqref{estimmomentPtilde} that
\begin{equation} \label{lowerboundpsi} \psi_{m_{i'}} - \psi_{m_i} = \sum_{\ell= m_i+1}^{m_{i'}} H_\ell  +O(1), \end{equation}
where $H_\ell := g_\ell + \widetilde{\mathcal{P}}_\ell-\EE_{\ell-1}(\widetilde{\mathcal{P}}_\ell)$. Since $|H_\ell|\leq (\log n)^{\mathfrak{h}}/(k_0-\ell)$ for any $\ell$, we deduce by Lemma \ref{moddev}  that for any $|\lambda|\leq (\log n)^{-\mathfrak{h}} n^{-1/6}$, 
\[ \log \EE_{k}\big[e^{\lambda \sum_{\ell=m_i+1}^{m_{i'}} H_\ell } \big] \geq \frac{\lambda^2}{4}\sum_{\ell= k+1}^{m_{i'}} \frac{v}{k_0-\ell} - \mathfrak{C}(|\lambda|^3\vee|\lambda|), \]
where we used again the fact that $\sum_{\ell \leq k_0-\ell_0} (\log n)^{3\mathfrak{h}/2}/(k_0-\ell)^{3/2} \lesssim 1$. In view of \eqref{lowerboundpsi}, it remains to prove that the exponential moment of $\sum_{\ell= m_i+1}^{m_{i'}} H_\ell$ restricted to $\mathscr{H}^\complement$ is of smaller order.    To achieve this, we use a similar argument as in the proof of Lemma \ref{expomomentpreciseperturbe}. Fix $i_o \leq j \leq i_1$. With the same argument as in the proof of \eqref{momentexpo2}, we deduce that there exists $\mathfrak{d}>0$ such that for $|\lambda|\leq \kappa^\mathfrak{d}$, on the event where $W_k\leq \eta_k$, 
\begin{equation}\label{expotauj}  \log \EE_{k}\big[\Car_{\tau = j} e^{\lambda \sum_{\ell=m_{i}+1}^{m_{i'}} H_\ell} \big]\leq  \frac{\lambda^2}{4} \sum_{\ell= m_i+1}^{m_{i'}} \frac{v}{k_0-\ell}   +\mathfrak{C} (|\lambda|^3\vee |\lambda|) -\mathfrak{C}^{-1}j^{1/6}.  \end{equation}
Take $\kappa$ large enough and $\mathfrak{d}$ small enough so that for any $|\lambda|\leq \kappa^{\mathfrak{d}}$,
\[ \sum_{j\geq i_o} e^{-\frac{\mathfrak{C}^{-1}}{2}j^{1/6}} \leq \frac12e^{-2\mathfrak{C} (|\lambda|^3\vee |\lambda|)},\]
using the fact that $i_o\gtrsim \kappa^{3/2}$.
Then, exponentiating \eqref{expotauj} and summing over $j$ yields that for any $|\lambda|\leq \kappa^{\mathfrak{d}}$, 
\[ \EE_{k}\big[\Car_{\mathscr{H}^{\complement}} e^{\lambda \sum_{\ell=m_i+1}^{m_{i'}} H_\ell} \big]\leq   \frac12 \EE_{k}\big[e^{\lambda \sum_{\ell=m_i+1}^{m_{i'}} H_\ell}\big],\]
 which concludes the proof of the lower bound.
 \end{proof}
 
 Next, we show an upper bound estimate of the joint exponential moments of increments of $\psi(z)$ and of a process $S$ weakly correlated to $\psi(z)$ in the hyperbolic regime. 
 \begin{lemma}
\label{expomomentpreciseh}
For any  constant $\mathfrak{h}'>0$ there exist constants  $\mathfrak{h},\mathfrak{C}>0$ depending on the model parameters 
and $\mathfrak{h}'$ only so that the following holds for all $z\in I_\eta$.
Let $0\leq k \leq k' \leq k_{0,z} - \ell_0$ and let $(S_\ell)_{k< \ell \leq k'}$ be a martingale such that $s_{\ell-1,\mathfrak{h}'}(S_\ell) \leq 1/(k_{0,z}-\ell)$ for any $\ell$.   Then,  for any  $\lambda \geq 0$, $\mu \in \RR$ such that $\lambda\vee|\mu| \leq (\log n)^{-\mathfrak{h}} n^{1/6}$, 
\begin{align*}& \log \EE_k\Big[e^{\lambda(\psi_{k'}(z)- \psi_k(z))+\mu \sum_{\ell= k+1}^{k'} S_\ell} \Big]\\
 &\leq  \frac{\lambda^2}{2} \big(\widehat{\Sigma}_{k',z}^2 -  \widehat{\Sigma}_{k,z}^2\big) +\frac{\mu^2}{2}s^2 + \mathfrak{C}\lambda |\mu |c +\mathfrak{C}(\lambda \vee \lambda^3\vee \mu^2\vee |\mu|^3) ,\end{align*}
where $s^2 := \sum_{\ell = k+1}^{k'} \| \Var_{\ell-1}(S_\ell) \|_{L^\infty}$ and $c:= \sum_{\ell =k+1}^{k'}  \|\sqrt{\frac{\Var_{\ell-1}(S_\ell)}{\ell-k_{0,z}}}\|_{L^\infty}$.  \end{lemma}

\begin{proof} In the following, the constants $\mathfrak{h},\mathfrak{C}>0$ will change from line to line but will only depend 
% two constants depending 
on the model parameters 
and $\mathfrak{h}'$.
%which will change value without changing name. 
First, we argue that it suffices to prove the claim estimate in the case where $k\geq k_\delta^-$. Indeed, since $s_{\ell-1,\mathfrak{h}}(S_\ell) \lesssim 1/n$ for any $\ell \leq k_\delta^-$, it follows from Lemma  \ref{tailpropH} - \eqref{subgauss} that  for any $| \mu |\leq (\log n)^{-\mathfrak{h}} \sqrt{n}$ and any $0\leq k \leq k' \leq k_\delta^-$,
\[ \EE_k\big[e^{\mu \sum_{\ell=k+1}^{k'} S_\ell}\big] \leq e^{\mathfrak{C} (|\mu|\vee \mu^2)}.\]
By Cauchy-Schwarz inequality and the a priori estimate of Lemma \ref{aprioridom}, this implies that it suffices to prove Lemma \ref{expomomentpreciseperturbh} when $k\geq k_\delta^-$, meaning that we are looking at an increment in the contributing hyperbolic regime. 

By Corollary \ref{coreqhyper}, there exists $\mathfrak{c}>0$ depending on the model parameters such that  for any $k_\delta^- \leq \ell \leq k_0-\ell_0$,
\begin{equation} \label{dominincremproof} \psi_{\ell}(z) - \psi_{\ell-1}(z) \leq  G_\ell -\mathfrak{c} \sqrt{\frac{k_0-\ell}{k_0}}W_{\ell-1}^2, + \frac12 (W_\ell-W_{\ell-1}),\end{equation}
where $G_\ell$ is $\cF_\ell$-measurable and satisfies the bounds \eqref{momentG1}-\eqref{momentG2}. Define for any $k_\delta^-\leq\ell \leq k_0-\ell_0$, $H_\ell = G_\ell -\mathfrak{c}\sqrt{(k_0-\ell)/k_0}W_{\ell-1}^2$. Since $W_{k'}-W_k\leq 1$,  it suffices to compute joint exponential moment of $\sum_{\ell=k+1}^{k'} H_\ell$ and $\sum_{\ell=k+1}^{k'} S_\ell$.
Let $k_\delta^-\leq\ell \leq k_0-\ell_0$.
 By assumption and \eqref{momentG2},  $|H_\ell-\EE_{\ell-1}(H_\ell)|+|S_\ell|\leq (\log n)^{\mathfrak{h}}/\sqrt{k_0-\ell}$.
Let $\lambda\geq 0, \mu \in \RR$ such that $\lambda \vee |\mu|\leq (\log n)^{-\mathfrak{h}}n^{1/6}$. Using Lemma \ref{moddev}, we find that 
    \begin{align}
   \log \EE_{\ell-1}\big[e^{\lambda H_\ell + \mu S_\ell}\big] &\leq  \lambda \EE_{\ell-1}(H_\ell) + \frac{\lambda^2}{2} \Var_{\ell-1}(H_\ell) + \frac{\mu^2}{2}\Var_{\ell-1}(S_\ell) \nonumber \\
  &  +\mathfrak{C} \lambda| \mu| \sqrt{\frac{\Var_{\ell-1}(S_\ell)}{k_0-\ell}} + \mathfrak{C} (\lambda^3\vee |\mu|^3) \frac{(\log n)^{\mathfrak{h}}}{(k_0-\ell)^{3/2}},\label{ineqlogE7} \end{align}
  where we used in addition Cauchy-Schwarz inequality to bound the covariance between $G_\ell$ and $S_\ell$ and the fact that $\Var_{\ell-1}(G_\ell) \lesssim 1/(k_0-\ell)$.

 Next, note that the terms involving $W_{\ell-1}$ in \eqref{momentG1} and \eqref{momentG2} are not summable without any control on $W_{\ell-1}$, but their contributions  to either the expectation or variance of $G_\ell$ is at most of order $W_{\ell-1}^{1/4}/(k_0-\ell)$ since $W_{\ell-1} \in [0,1]$. We show that these terms  will be compensated by the negative drift of order $W_{\ell-1}^2\sqrt{(k_0-\ell)/k_0}$ in the expectation of $H_\ell$. Indeed, one can check that for any  $\lambda \geq 0$ and $\mathfrak{D}>0$, 
 \[ \lambda \sqrt{\frac{k_0-\ell}{k_0}} W_{\ell-1}^2-\mathfrak{D}(\lambda^2\vee\lambda) \frac{W_{\ell-1}^{1/4}}{k_0-\ell} \geq - \mathfrak{D}^{8/7}(\lambda^3\vee \lambda) \frac{k_0^{1/14}}{(k_0-k)^{17/14}}.\]
Using \eqref{momentG1}-\eqref{momentG2} this implies that 
 \begin{align*}& \lambda \EE_{\ell-1}(H_\ell) + \frac{\lambda^2}{2} \Var_{\ell-1}(H_\ell) \\
&\leq \frac{\lambda^2v}{4(k_0-\ell)}  + \mathfrak{C}(\lambda^3\vee \lambda)\Big[ \frac{1}{(k_0-\ell)^{3/2}} + \frac{1}{\sqrt{k_0(k_0-\ell)}} + \frac{k_0^{1/14}}{(k_0-k)^{17/14}}\Big].\end{align*}
 Finally, gather all the error terms into $\veps_\ell$, defined as 
 \[ \veps_\ell :=  \frac{1}{\sqrt{k_0(k_0-\ell)}} +\frac{(\log n)^{\mathfrak{h}}}{(k_0-\ell)^{3/2}} + \frac{k_0^{1/14}}{(k_0-\ell)^{17/14}}.\]
 With this notation we have that for any $k_\delta^- \leq \ell \leq k_0-\ell_0$, 
 \[    \log \EE_{\ell-1}\big[e^{\lambda H_\ell + \mu S_\ell}\big] \leq \frac{\lambda^2v}{4(k_0-\ell)}    + \frac{\mu^2}{2}\Var_{\ell-1}(S_\ell) + \mathfrak{C}\lambda |\mu|\sqrt{\frac{\Var_{\ell-1}(S_\ell)}{k_0-\ell}} + \mathfrak{C} (\lambda \vee \lambda^3\vee |\mu|^3)\veps_\ell.\]
 One can check that $\sum_{k_\delta^-\leq \ell \leq k_0-\ell_0} \veps_\ell \lesssim 1$. Hence, an immediate induction argument  ends the proof of the claim.\end{proof}

\subsection{The elliptic regime}
We separate the proof of the upper and lower bound. For the upper bound, we prove the analog of 
%a corresponding result of 
Lemma \ref{expomomentpreciseperturbh} in the elliptic regime. 

\begin{lemma}
\label{expomomentpreciseperturbe}
Fix $\mathfrak{h}'>0$. There exist $\mathfrak{d},\mathfrak{C}>0$ depending on the model parameters and $\mathfrak{h}'$ 
only so that the following holds for all
$z\in I_\eta$.
Let  $k_{0,z}+\ell_0\leq k \leq k' \leq n$. Assume that $(S_\ell)_{k< \ell \leq k'}$ is a martingale such that $s_{\ell-1,\mathfrak{h}'}(S_\ell) \lesssim 1/(\ell-k_0)$.
Then, 
%there exists $\mathfrak{d},\mathfrak{C}>0$ depending on the model parameters such that 
for  any  $|\lambda|\vee|\mu| \leq  \kappa^{\mathfrak{d}}$,
\begin{align*} \label{expoeq} &\log \EE_k\Big[e^{\lambda(\psi_{k'}(z)- \psi_k(z))+\mu \sum_{\ell= k+1}^k S_\ell} \Big]\\
 &\leq  \frac{\lambda^2}{2} \big(\widehat{\Sigma}_{k',z}^2 -  \widehat{\Sigma}_{k,z}^2\big) +\frac{\mu^2}{2}s^2 + \mathfrak{C}\lambda |\mu |c +\mathfrak{C}(|\lambda| \vee |\lambda|^3\vee |\mu|^3),
\end{align*}
where $s^2 := \sum_{\ell = k+1}^{k'} \| \Var_{\ell-1}(S_\ell) \|_{L^\infty}$, $c:=\sum_{\ell =k+1}^{k'}  \|\sqrt{\frac{\Var_{\ell-1}(S_\ell)}{\ell-k_{0,z}}}\|_{L^\infty}$.
\end{lemma}

\begin{proof} 
In the following, the constants $\mathfrak{h},\mathfrak{C},\mathfrak{c}>0$ will change from line to line but will only depend 
% two constants depending 
on the model parameters 
and $\mathfrak{h}'$. 
Recall the definition of the blocks $(k_j)_{j_o\leq j\leq j_1}$ in \eqref{defblockosci} and of a good block from \eqref{goodblock}.
With the same argument as in the proof of Lemma \ref{expomomentpreciseperturbh}, we can assume that $k=k_i$ and $k'=k_{i'}$ for some $j_o \leq i\leq i' \leq j_1$. 
Let $j_o\leq i \leq i'\leq j_1$.  Set $\tau$ as the largest index $j\in \llbracket  i, i'-1\rrbracket$ such that the $j^{\text{th}}$ block is bad if there are any, and set $\tau =0$ otherwise. Fix $i\leq j\leq i'-1$.  By Proposition \ref{represelliptic}, we know that on the event $\mathscr{G}_p$, 
\begin{equation} \label{goodblockeincrem} \Delta \psi_{k_{p+1}} = \sum_{\ell=k_p+1}^{k_{p+1}} \big(w_\ell(\zeta_{p,\ell-1}) +\mathcal{P}_\ell\big) +O(p^{-5/4}),\end{equation}
where $\mathcal{P}_\ell$ is a centered $\cF_\ell$-measurable random variable satisfying the moment estimates of \eqref{momentestimPe}. Write $H_\ell = w_\ell(\zeta_{p,\ell-1})+\mathcal{P}_\ell$ for any $\ell \geq k_0+\ell_0$. Using \eqref{estim}, the fact that the noise is bounded by $(\log n)^2$ by Assumption \ref{ass1} and \eqref{momentestimPe}, we deduce that there exists $\mathfrak{h}>0$ depending on the model parameters such that $|H_\ell|\leq (\log n)^{\mathfrak{h}}/\sqrt{\ell-k_0}\leq (\log n)^{\mathfrak{h}} n^{-1/3}$ for any $\ell \in (k_p,k_{p+1}]$.   Moreover, 
using  \eqref{momentestimPe} and the fact that $\Delta k_{p+1} \lesssim n^{1/3}p^3$ and $k_p-k_0\gtrsim n^{1/3} p^4$, we find that 
\begin{equation} \label{varwP} \Var_{\ell-1}(H_\ell) = \Var_{k_p}(w_\ell(\zeta_{p,\ell-1})) + \veps_{p,\ell}, \ \ell \in (k_p, k_{p+1}],\end{equation}
where   $\sum_{\ell=k_p+1}^{k_{p+1}} \|\veps_{p,\ell} \|_{L^\infty} \leq pn^{-1/3} + p^{-9/8}$.
Using Proposition, we deduce that 
\[  \sum_{\ell=k_p+1}^{k_{p+1}}\|  \Var_{\ell-1}(H_\ell)\|_{L^\infty} \leq \frac{v}{4} \sum_{\ell=k_p+1}^{k_{p+1}} \frac{1}{\ell-k_0} + O(pn^{-1/3}) + O(p^{-9/8}).\]
Hence, using  Lemma \ref{moddev}, it follows that for any $\lambda,\mu \in \RR$ are such that $|\lambda|,|\mu|\leq  (\log n)^{-\mathfrak{h}} n^{1/6}$ then 
\begin{align} \log \EE_{k_p} \big[ e^{\lambda \sum_{k=k_p+1}^{k_{p+1}} H_k+\mu \sum_{k=k_p+1}^{k_{p+1}} S_k}\big] &\leq \sum_{k=k_p+1}^{k_{p+1}}  \Big(\frac{\lambda^2}{2}  \frac{v}{4(k-k_0)} +  \frac{\mu^2}{2}  s_\ell^2 + \mathfrak{C} |\lambda \mu| c_\ell \Big) \nonumber \\
&  +(|\lambda|\vee|\lambda|^3\vee|\mu|^3)O(p^{-9/8} +pn^{-1/3}), \label{estimblock}
\end{align}
where $\mathfrak{C}>0$ depends on the model parameters, $s_\ell^2 := \|\Var_{\ell-1}(S_\ell)\|_{L^\infty}$ and $c_\ell:=  \sqrt{\|\Var_{\ell-1}(S_\ell)\|_{L^\infty}/(\ell-k_0)}$. On the other hand, note that as $s_{\ell-1,\mathfrak{h}'}(S_\ell)\lesssim 1/(\ell-k_0)$ for any $\ell$, $\sum_{k=k_i+1}^{k_{j+1}} S_\ell$ satisfies the same a priori exponential moment estimate as $\psi_{k_{j+1}}(z)-\psi_{k_i}(z)$, see
\eqref{subgauss} from
Lemma \ref{tailpropH}. Hence, using Proposition \ref{aprioridom} and the Cauchy-Schwarz inequality, we obtain that there 
exist $\mathfrak{h},\mathfrak{C}>0$
%depending on the model parameters and $\mathfrak{h}'$ only} 
such that for any $|\lambda|,|\mu|\leq (\log n)^{-\mathfrak{h}'}n^{1/6}$, 
\begin{equation} \label{aprioripsiS} \log \EE_{k_i} \big[ e^{\lambda (\psi_{k_{j+1}}-\psi_{k_i})  +\mu \sum_{k=k_i+1}^{k_{j+1}} S_k}\big] \leq \mathfrak{C}'( |\lambda|\vee \lambda^2\vee \mu^2) \log j,  \end{equation}
where we used in addition the fact that $\sum_{\ell = k_0+\ell_0}^{k_{j+1}} 1/(\ell-k_0) \lesssim \log j$. 
Combining \eqref{aprioripsiS} with the fact that $\PP(\mathscr{G}_j^\complement) \leq e^{-\mathfrak{c} j^{1/2}}$ for some $\mathfrak{c}>0$,
% depending on the model parameters 
by Proposition \ref{probagoodblock}, we get that by Cauchy-Schwarz inequality that for any $|\lambda|,|\mu|\leq (\log n)^{-\mathfrak{h}'} n^{1/6}/2$, 
\begin{equation} \label{aprioripsiS2} \log \EE_{k_i} \big[\Car_{\mathscr{G}_j^{\complement}}e^{\lambda (\psi_{k_{j+1}}-\psi_{k_i})  +\mu \sum_{k=k_i+1}^{k_{j+1}} S_k}\big] \leq \mathfrak{C}'( |\lambda|\vee \lambda^2\vee \mu^2) \log j - \mathfrak{c} j^{1/2}, \end{equation}
where $\mathfrak{C},\mathfrak{c}$ changed values without changing name. 
Now, one can find $\mathfrak{d}>0$ such that for $\kappa$ large enough and $|\lambda|\vee |\mu|\leq \kappa^{\mathfrak{d}}$, $\mathfrak{C}'( |\lambda|\vee \lambda^2\vee \mu^2) \log j - \mathfrak{c} j^{1/2} \leq  -(\mathfrak{c}/2) j^{1/2}$ for any $j\geq j_o$. Hence, assuming $|\lambda|,|\mu|\leq \kappa^{\mathfrak{d}}$ and 
 putting together \eqref{goodblockeincrem}, \eqref{estimblock} and \eqref{aprioripsiS2} and using the fact that $\sum_{p = i_0}^{i_1}(p^{-9/8}+pn^{-1/3}) \lesssim 1$, we deduce that 
 \begin{align}  \log \EE_{k_i} \big[\Car_{\tau = j}e^{\lambda ( \psi_{k_{i'}}(z) - \psi_{k_{i}}(z)) +\mu \sum_{k=k_i+1}^{k_{i'}} S_k}\big] & \leq \sum_{\ell=k_i+1}^{k_{i'}}  \Big(\frac{\lambda^2}{2}  \frac{v}{4(\ell-k_0)} +  \frac{\mu^2}{2}  s_\ell^2 +\mathfrak{C} |\lambda \mu | c_\ell\Big) \nonumber \\
 &+ \mathfrak{C} (|\lambda|\vee |\lambda|^3\vee |\mu|^3) - \frac{\mathfrak{c}}{2}j^{1/2}.\label{momentexpotauj}
 \end{align}
As the above estimate holds for any $j \in \llbracket i,i'-1\rrbracket \cup \{0\}$ and $\sum_{j\geq 1} e^{-\mathfrak{c} j^{1/2}/2} \lesssim 1$, this ends the proof of the lemma.
\end{proof}

Combining the results of  Lemmas \ref{expomomentpreciseperturbh} and \ref{expomomentpreciseperturbe} and the a priori estimate of Proposition \ref{apriori} to handle the negligible hyperbolic regime and parabolic regime, we get the upper bound on the joint exponential moments of $\psi(z)$ and of a weakly correlated process claimed in Proposition \ref{expomomentpreciseperturb}. Moreover, taking $S=0$ in Lemma \ref{expomomentpreciseperturbe}, and combining Lemma \ref{expomomentpreciseperturbh}  with the a priori estimate of Proposition \ref{apriori}, we get the upper bound on the exponential moments of $\psi(z)$ claimed in Proposition \ref{expopreciseincre}.

Next, we prove the lower bound on the exponential moments of increments of $\psi(z)$ in the elliptic regime.

\begin{lemma}
\label{lowerboundexpoe}
There exists $\mathfrak{C},\mathfrak{d}>0$ depending on  the model parameters only so that for any $z\in I_\eta$ the following holds. For any $k_{0,z} +\ell_0 \leq k \leq k' \leq n$, $| \lambda |\leq \kappa^{\mathfrak{d}}$,
\[  \log  \EE_k\big[ e^{\lambda(\psi_{k'}(z)-\psi_k(z))} \big] \geq  \frac{\lambda^2}{2}  (\widehat{\Sigma}_{k',z}^2-\widehat{\Sigma}_{k,z}^2) - \mathfrak{C}  (|\lambda| \vee |\lambda|^3).\]
\end{lemma}

Note that $\widehat{\Sigma}_{k_{0,z}+\ell_0}^2=\widehat{\Sigma}_{k_{0,z}-\ell_0}^2$ and $\sum_{\ell = k_{0,z}-\ell_0}^{k_{0,z}+\ell_0} \|\EE_{\ell-1}(\Delta \psi_\ell(z))\|_{L^\infty} \lesssim _\kappa 1$ by \eqref{aprioridom}. Hence, assuming that Lemma \ref{lowerboundexpoe} holds, we obtain the lower bound of Proposition \ref{expopreciseincre}  by combining Lemma \ref{expomomentpreciseperturbh}   together with Jensen's inequality to handle the parabolic regime $\llbracket k_{0,z}-\ell_0,k_{0,z} +\ell_0\rrbracket$.    

%Finally, we give a proof of Lemma \ref{lowerboundexpoe}.

\begin{proof}[Proof of Lemma \ref{lowerboundexpoe}] In the following, the constants $\mathfrak{h},\mathfrak{C}>0$ will change from line to line but will only depend 
% two constants depending 
on the model parameters.
%and $\mathfrak{h}'$.} 
With the same argument as in the proof of Lemma \ref{expomomentpreciseperturbh}, we deduce that it is sufficient to prove the lower bound estimate when $k= k_i$ and $k'=k_{i'}$ for some $i< i'$. Let $\mathscr{G}:=\bigcap_{j=i}^{i'-1} \mathscr{G}_j$, where $\mathscr{G}_j$ is the event that the $j^{\text{th}}$ block is good. By \eqref{goodblockeincrem}, we know that for any $i\leq j <i'$, on the event $\mathscr{G}$,
\begin{equation} \label{resprespsiblock}\Delta \psi_{k_{j+1}} = \sum_{\ell=k_j+1}^{k_{j+1}} \big(w_\ell(\zeta_{pj,\ell-1}) +\mathcal{P}_\ell\big) +O(j^{-5/4}),\end{equation}
where $\mathcal{P}_\ell$ is a centered $\mathcal{F}_\ell$-random variable satisfying \eqref{momentestimPe}. Let $i\leq j <i'$ and $H_\ell = w_\ell(\zeta_{j,\ell-1})+\mathcal{P}_\ell $ for any $\ell$. Since $s_{\ell-1}(w_\ell(\zeta_{j,\ell-1})+\mathcal{P}_\ell) \lesssim n^{-1/3}$ for any $\ell$, we deduce from Lemma \ref{moddev} that for appropriate constants  $\mathfrak{C},\mathfrak{h}>0$
and for any $|\lambda|\leq (\log n)^{-\mathfrak{h}} n^{-1/6}$, 
\[ \log \EE_{\ell-1}[e^{\lambda H_\ell}]\geq \frac{\lambda^2}{2} \Var_{\ell-1}(H_\ell) - \frac{\mathfrak{C}}{(\ell-k_0)^{3/2}}| \lambda|^3.\]
%where $\mathfrak{C},\mathfrak{h}>0$  are constants depending on the model parameters which will change value without changing name.
Since for any $k_j<\ell\leq k_{j+1}$ we have by \eqref{varwP} that
%\in (k_j, k_{j+1}]$, 
$\Var_{\ell-1}(H_\ell) =  \Var_{k_j}(w_\ell(\zeta_{j,\ell-1}))  +\veps_{j,\ell}$, 
where $\sum_{\ell=k_j+1}^{k_{j+1}} \|\veps_{j,\ell}\|_{L^\infty}\lesssim jn^{-1/3} + j^{-9/8}$, we obtain by  \eqref{momentestimPe} -\eqref{calcsigmae} that for any  $|\lambda|\leq (\log n)^{-\mathfrak{h}} n^{-1/6}$, 
\[ \log \EE_{k_j}\big[e^{\lambda \sum_{\ell=k_j+1}^{k_{j+1}} H_\ell}\big] \geq \frac{\lambda^2}{4} \sum_{\ell=k_j+1}^{k_{j+1}} \frac{v}{\ell-k_0} -\mathfrak{C} |\lambda|^3 \veps_j, \]
where $\veps_j :=  jn^{-1/3} +j^{-9/8} + \sum_{\ell=k_j+1}^{k_{j+1}} (1/\sqrt{k_0(\ell-k_0)} + (\ell-k_0)^{-3/2})$. As observed in the proof of Lemma \ref{expomomentpreciseperturbh}, $\sum_{j\leq j_1} \veps_j \lesssim 1$.  Hence, for any $|\lambda|\leq (\log n)^{-\mathfrak{h}} n^{-1/6}$, 
\begin{equation} \label{lbexpoH} \log \EE_{k_i}\big[e^{\lambda \sum_{\ell=k_i+1}^{k_{i'}} H_\ell}\big] \geq \frac{\lambda^2}{4} \sum_{\ell=k_i+1}^{k_{i'}} \frac{v}{\ell-k_0} -\mathfrak{C} |\lambda|^3. \end{equation}
In view of \eqref{resprespsiblock}, it remains to prove that the exponential moment of $\sum_{\ell=k_i+1}^{k_{i'}} H_\ell$ restricted to $\mathscr{G}^\complement$ is of smaller order. To this end, we proceed as in the proof of Lemma \ref{expomomentpreciseperturbh} and define, on the event $\mathscr{G}^\complement$, $\tau$ as the largest integer $j \in \llbracket i, i'-1\rrbracket$ such that the $j^{\text{th}}$ block is bad.  Fix $j \in  \llbracket i, i'-1\rrbracket$. 
With the same argument as in the proof of \eqref{momentexpotauj}, we deduce that there exist $\mathfrak{d},\mathfrak{c}>0$ depending on the model parameters such that for any $|\lambda|\leq \kappa^{\mathfrak{d}}$, 
\begin{equation} \label{lbexpoH2}  \log \EE_{k_i} \big[\Car_{\tau = j}e^{\lambda  \sum_{\ell=k_i+1}^{k_{i'}} H_\ell}\big]  \leq \frac{\lambda^2}{4}  \sum_{\ell=k_i+1}^{k_{i'}}  \frac{v}{\ell-k_0}     - \mathfrak{c}j^{1/2}.\end{equation}
 Since $j\geq i\geq  j_o \gtrsim \kappa^{1/4}$, we obtain that for $\kappa$ large enough, $\sum_{j\geq i} e^{-\mathfrak{c} j^{1/2}}\leq  \frac14$. At the price of taking $\mathfrak{d}$ even smaller, it follows by putting together \eqref{lbexpoH} and \eqref{lbexpoH2} that 
 \[   \EE_{k_i} \big[\Car_{\mathscr{G}^\complement}e^{\lambda  \sum_{\ell=k_i+1}^{k_{i'}} H_\ell}\big]  \leq  \frac12  \EE_{k_i}\big[e^{\lambda \sum_{\ell=k_i+1}^{k_{i'}} H_\ell}\big] ,\]
 for any $|\lambda|\leq \kappa^{\mathfrak{d}}$, 
which ends the proof of the claim.
\end{proof}

\section{One ray Gaussian coupling}
\label{sec-OneRay}
We fix in this whole section some $z \in I_\eta$, with the understanding that all implicit constants do not depend on $z$. We construct  a coupling between the process $\{\psi_k(z) : 1\leq k\leq n\}$ (rescaled in time) and a certain Gaussian process. Recall that  
$\sigma_{k,z}^2$ denote the variance of the main noise driving the recursions. More precisely, 
with $g_{k,z}$ as in \eqref{defgk}, let
\[ \sigma_{k,z}^2=\EE[g_{k,z}^2]\ \mbox{\rm when}\ 1\leq k \leq k_{0,z}-\ell_0, \ \sigma_{k,z}^2 = 0
\ \mbox{\rm when}\  |k-k_{0,z}|<\ell_0,\]
 and let $\sigma_{k,z}^2$ be given   for $k\geq k_{0,z}+\ell_0$ by \eqref{calcsigmae} of Proposition \ref{represelliptic}. Denote also by $\Sigma_{k,z}^2:=\sum_{\ell \leq k} \sigma_{\ell,z}^2$ the cumulative variance until time $k$. Now, set $T_z := \lceil \frac{2}{v} \Sigma_{n,z}^2\rceil$ and
define the time-change $n_{t,z}$ as 
\begin{equation} \label{deftimechange} n_{t,z} := \inf\Big\{k :  \Sigma_{k,z}^2 \geq \frac{vt}{2}\Big\}\wedge n, \ 1\leq t \leq T_z.\end{equation}
Note that with this definition $n_{\small{T_z},z} = n$.
 Next, define  the rescaled process $(\Psi_t^z)_{t\in \llbracket 1,T_z\rrbracket}$ by
\begin{equation} \label{defPsi} \Psi_t^z := \psi_{n_{t,z}}(z), \quad t \in \llbracket 1,T_z\rrbracket.\end{equation}
Our aim is to couple the process $\Psi_t^z$ with a certain Gaussian process in such a way that the probability that the coupling fails is quasi-exponentially small in $n$. To achieve this, we  restrict our process to a smaller time-set, excluding a larger window around $k_{0,z}$ of length $(\log n)^{\mathfrak{q}} n^{1/3}$ for some $\mathfrak{q}>0$. As we will see, this corresponds after time-change to a set $\mathcal{T}_{\mathfrak{q},z} $ defined as 
\begin{equation} \label{defI} \mathcal{T}_{\mathfrak{q},z} := \llbracket1, t_\mathfrak{q}^-\rrbracket \cup \llbracket t_\mathfrak{q}^+, T_z\rrbracket,\end{equation}
where $t_\mathfrak{q}^\pm = \lfloor \frac23 \uptau\pm \mathfrak{q} \log \uptau\rfloor$ and $\uptau:=\log n$.  We show that the increments of the processes $(\Psi_{t}^z)_{1\leq t \leq t_\mathfrak{q}^-}$ and
 $(\Psi_{t}^z)_{t_\mathfrak{q}^+\leq t \leq T_z}$ can be coupled with the increments of a Gaussian process. Technically, we need to make sure the additional variables one constructs respect the independence structure of our coefficients $(a_{k-1}^2, b_k)_{1\leq k \leq n}$.  
 More precisely, say that a filtration $(\mathcal{G}_k)_{1\leq k \leq n}$ is {\em good} if for any $1\leq k \leq n$, $\mathcal{F}_k\subset \mathcal{G}_k$ and $(a_{\ell-1},b_\ell)_{\ell >k}$ is independent of $\mathcal{G}_k$. With this terminology, we prove the following result.

%\begin{proposition}\label{coupling}
%Upon enlarging the probability space, there exist $(\mathcal{G}_k)_{1\leq k \leq n}$ a filtration augmenting $(\mathcal{F}_k)_{1\leq k \leq n}$ and a centered Gaussian process $(\Gamma_t)_{t_\delta^- \leq t\leq t_\delta^+}$ such that for any  $t_\delta^- \leq t ,s \leq t_\delta^+$,
%\begin{enumerate}
%\item  $\mathrm{Cov}(\Gamma_t,\Gamma_s) = \frac{v}{2} (t\wedge s)$.
%\item  $\Gamma_t$ is $\mathcal{G}_{n_{t,z}}$-measurable.
%\item  $(\Gamma_{t'}-\Gamma_t)_{t'\geq t}$ is independent of $\mathcal{G}_{n_{t,z}}$.
%\end{enumerate}
%Further, for any $ t_\delta^- \leq t \leq t_o^-$ and $t_o^+\leq t' \leq  t_\delta^+$, 
%\[ \PP_{n_{t,z}}\big(  \max_{t \leq s\leq t_o^-}\big|\Psi_{s}^z - \Psi_{t}^z - \big(\Gamma_{s}-\Gamma_{t}\big)\big| \geq \mathfrak{c}^{-1} \Big)\leq e^{-\mathfrak{c}T^2}\]
%\[
%\PP_{n_{t',z}}\big(\max_{t'\leq s \leq t_\delta^+} \big| \Psi_s^{z} - \Psi_{t'}^{z} -\big(\Gamma_{s}-\Gamma_{t'} \big)\big|\geq \mathfrak{c}^{-1} \big) \leq e^{-\mathfrak{c}T^2}.\]
%where $t_o^\pm = \lfloor \frac23 T\pm\mathfrak{q} \log T\rfloor $, $t_\delta^- = \lfloor \log(\kappa/\delta)\rfloor$, $t_\delta^+ = T-t_\delta^-$ and $\mathfrak{c}$ is a positive constant depending on the model parameters.
%\end{proposition}

\begin{proposition}\label{coupling} There exist a positive constant
$\mathfrak{c}$ depending on the model parameters only, a good filtration $(\mathcal{G}_k)_{1\leq k \leq n}$  augmenting $({\mathcal{F}}_k)_{1\leq k \leq n}$ and a centered Gaussian process $(\Gamma_t)_{1 \leq t\leq T_z}$ such that for any  $1\leq t ,s \leq T_z$,
\begin{enumerate}
\item  $\mathrm{Cov}(\Gamma_t,\Gamma_s) = \frac{v}{2} (t\wedge s)$.
\item  $\Gamma_t$ is $\mathcal{G}_{n_{t,z}}$-measurable.
\item  $(\Gamma_{t'}-\Gamma_t)_{t'\geq t}$ is independent of $\mathcal{G}_{n_{t,z}}$.
\end{enumerate}
In addition, for any $1 \leq t \leq t_\mathfrak{q}^-$, on the event where $W_{n_{t,z}}^z \leq \eta_{n_{t,z},z}$,
\[ \PP_{n_{t,z}}\big(  \max_{t \leq s\leq t_\mathfrak{q}^-}\big|\Psi_{s}^z - \Psi_{t}^z - \big(\Gamma_{s}-\Gamma_{t}\big)\big| \geq \mathfrak{c}^{-1} \Big)\leq e^{-\mathfrak{c}\uptau^2}.\]
Further, for any  $t_\mathfrak{q}^+\leq t' \leq  T_z$, 
\[
\PP_{n_{t',z}}\big(\max_{t'\leq s \leq T} \big| \Psi_s^{z} - \Psi_{t'}^{z} -\big(\Gamma_{s}-\Gamma_{t'} \big)\big|\geq \mathfrak{c}^{-1} \big) \leq e^{-\mathfrak{c}\uptau^2}.\]
%where 
%$t_\mathfrak{q}^\pm = \lfloor \frac23 T\pm\mathfrak{q} \log \corO{T}\rfloor $ and $\mathfrak{c}$ is a positive constant depending on the model parameters.
\end{proposition}

To prove the existence of such a coupling, we will use a  strong approximation theorem for sums of independent but non-identically distributed random variables due to Sakhanenko \cite[Theorem 3.1]{Lifshits}.  The heterogeneity of the distributions of the variables is addressed through a sort of uniform boundedness assumption expressed  in terms of the {\em Sakhanenko parameter}. Following \cite{Lifshits}, denote by $\lambda(X)$ the   Sakhanenko parameter of a centered random variable $X$ as 
\begin{equation} \label{deflambda} \lambda(X) := \sup\big\{\lambda : \lambda \EE[ |X|^3 \exp(\lambda|X|)] \leq \EE(X^2)\big\}.\end{equation}
For a variable $X$ bounded by some $a>0$, it is easy to estimate $\lambda(X)$
 (see \cite[p.8]{Lifshits}):
\begin{equation} \label{compSakhanenko1}  \sqrt{\Var(X)} \leq \lambda(X)^{-1} \leq 2a,\end{equation}
while it is always possible to lower bound $\lambda(X)$ in terms of inverse exponential moments:
\begin{equation} \label{compSakhanenko2}   \lambda(X)^{-1}    \leq 14\inf_{h>0} \frac{\EE[e^{h|X|}]}{h^3\EE[X^2]}. \end{equation}
Before going into the proof of Proposition \ref{coupling}, we collect in the next lemma several properties of  the time-change  $n_{t,z}$ and of the final time $T_z$.
 \begin{lemma}\label{approxtimechange}Let $z\in I_\eta$.
 \begin{enumerate}
 \item[ $(i).$ ]  If $t\leq \frac{2}{3} \uptau-2\log \kappa$, $k_{0,z}- n_{t,z} \asymp n e^{-t}$.
  \item[ $(ii).$ ]  If  $t>\frac23 \uptau+2\log \kappa$, $n_{t,z} - k_{0,z} \asymp n^{-2}e^{2t}$.
  \item[ $(iii).$ ] For any $1\leq t\leq T_z$, $ {\Sigma}_{n_{t,z}}^2 = (v/2)t +O(n^{-1/3})$. 
  \item[ $(iv).$ ] For any $1\leq t\leq T_z$, $\widehat{\Sigma}_{n_{t,z}}^2 = (v/2)t +O(1)$. 
  \item[ $(v).$ ] $T_z = \uptau + O(1)$.
  \end{enumerate}
 \end{lemma}

\begin{proof}We begin with the proof of  $(i)$. Let $t \leq \frac23 \uptau- 2\log \kappa$, $s\leq \log \kappa$  and set $k = k_0-n e^{-(t+s)}$.  Recall the definitions of $\widehat{\sigma}_{k}^2$ and $\widehat{\Sigma}_{k}^2$ from \eqref{defSigma}.
By Propositions \ref{incremhn}, \ref{increm} and \ref{represelliptic}, we know that the difference between $\sigma_{\ell}^2$ and $\widehat{\sigma}_{\ell}^2$ is summable. Hence, $\Sigma_{k}^2$ and $\widehat{\Sigma}_{k}^2$ differ only by a constant depending on the model parameters. We get that
\[ \Sigma_k^2 = \widehat{\Sigma}_k^2 +O(1) = \frac{v}{2} (t+s) +O(1).\]
We deduce that for $s$ large enough depending on the model parameters $\Sigma_k^2 \geq (v/2)t$ and for $-s$ large enough depending on the model parameters $\Sigma_k^2 \leq (v/2) t  $. Using the monotonicity of the time change, this shows that $k_0 - n_{t} \asymp n e^{-t}$. 

The proof of $(ii)$ is similar to the proof of $(i)$. Turning to $(iii)$, by the definition of the time-change, for any $1\leq t \leq T_z$,
\[ (v/2) t \leq \Sigma_{n_t}^2 \leq (v/2) t + \sigma_{n_t}^2.\]
As $\sigma_k^2 \lesssim 1/|k_{0,z}-k| \lesssim n^{-1/3}$ for any $|k_0-k|>\ell_0$ by Propositions \ref{incremhn}, \ref{increm} and \ref{represelliptic}, this yields $(iii)$. 
Next, as observed above, $\Sigma^2$ and $\widehat{\Sigma}^2$ are within a  constant depending on the model parameters. Since $ {\Sigma}_{n_{t,z}}^2 = (v/2)t +O(n^{-1/3})$ for any $t$, we get $(iv)$.

Finally, as $\Sigma_{n}^2(z)$ and $\widehat{\Sigma}_{n}^2(z)$ differ only by a constant and $T_z = \lceil \frac{2}{v} \Sigma_n(z)^2\rceil$, it suffices to prove that $\widehat{\Sigma}_{n}^2(z) = \frac{v}{2} T+ O(1)$, which one can check  by using the explicit expression of $\widehat{\Sigma}_n^2(z)$ from \eqref{defSigma}.
\end{proof}
 
Since we are only looking at increments before $t_\mathfrak{q}^-$ and after $t_\mathfrak{q}^+$, we can separate the proof  of Proposition \ref{coupling} into the hyperbolic and elliptic regime. We start with the hyperbolic regime.

\begin{proof}[Proof of Proposition \ref{coupling} - The hyperbolic regime]  
By Lemma \ref{approxtimechange}, we know that $k_0- n_{t_\mathfrak{q}^-} \asymp (\log n)^{\mathfrak{q}} n^{1/3}$. 
With $\eta_\ell$ as in \eqref{defeta} and $\mathfrak{q}$ as in Corollary \ref{controlbadblockhall},
%Proposition \ref{controlbadblockh}, 
define $\mathscr{H}$ as
the event 
\[ \mathscr{H} := \big\{\forall n_t \leq  \ell < n_{t_\mathfrak{q}^-}, W_{\ell} \leq 2\eta_{\ell}\big\}.\]
%where 
From Propositions  \ref{incremhn} and \ref{increm}, we know that there exists a constant $\mathfrak{h}$ so that on the event 
$\mathscr{H}$, 
\[ \Psi_{t'}-\Psi_{t}=  \sum_{\ell=n_t+1}^{n_{t'}} g_{\ell} + \sum_{\ell=n_t+1}^{n_{t'}} \widetilde{\mathcal{P}}_\ell + \frac12 (W_{n_{t'}} - W_{n_t}),\]
where $g_{k}$ is defined in \eqref{defgk} and $\widetilde{\mathcal{P}}_\ell$ is a $\cF_\ell$-measurable random variable such that for $\ell
\leq k_{\delta}$,
\[   \big|\EE_{\ell-1}(\widetilde{\mathcal{P}}_{\ell})\big|  \lesssim  \frac{1}{n}, \quad s_{\ell-1,\mathfrak{h}}(\widetilde{\mathcal{P}}_{\ell}) \lesssim  \frac{(\log n)^\mathfrak{C}}{n\ell},\]
and for $k_{\delta} \leq \ell \leq n_{t_\mathfrak{q}^-}$, 
\[ \big|\EE_{\ell-1}(\widetilde{\mathcal{P}}_{\ell})\big| \lesssim  \frac{1}{\sqrt{k_0(k_0-\ell)}}  + \frac{1}{(k_0-\ell)^{3/2}} + \frac{\sqrt{\eta_{\ell}}}{k_0- \ell}+ \eta_{\ell}^2 \sqrt{\frac{k_0-\ell}{k_0}},\]
\[ s_{\ell-1,\mathfrak{h}}(\widetilde{\mathcal{P}}_{\ell})\lesssim \frac{1}{(k_0-\ell)^2} + \frac{\sqrt{\eta_{\ell}}}{k_0-\ell}.\]
From these estimates, the explicit expression of $\eta_k$ and the fact that $k_0-n_{t_\mathfrak{q}^-} \asymp n^{1/3} (\log n)^{\mathfrak{q}}$ we find that 
\[ \sum_{1\leq \ell \leq n_{t_\mathfrak{q}^-}} \|\EE_{\ell-1} (\widetilde{\mathcal{P}}_\ell)\|_{L^\infty} \lesssim 1, \ \sum_{1\leq \ell \leq n_{t_\mathfrak{q}^-}} \|s_{\ell-1,\mathfrak{h}} (\widetilde{\mathcal{P}}_\ell)\|_{L^\infty} \lesssim (\log n)^{-\mathfrak{q}/2},\]
and moreover $  \|s_{\ell-1}(\widetilde{\mathcal{P}}_\ell)\|_{L^\infty} \lesssim n^{-1/3}$ for any $1\leq \ell\leq n_{t_\mathfrak{q}^-}$. 
Together with the fact that $W$ is bounded, this entails that there exists $\mathfrak{C}>0$ depending on the model parameters such that 
\[ \PP_{n_{t}}\Big(\Big\{\max_{t\leq t' \leq t_\mathfrak{q}^-} \big|\Psi_{t'} - \Psi_t - \sum_{\ell= n_t+1}^{n_{t'}} g_{\ell}\big|> \mathfrak{C}\Big\} \cap \mathscr{H}\Big) \leq \PP_{n_t}\Big(\max_{n_t \leq k' \leq n_{t_\mathfrak{q}^-}}\big|\sum_{\ell=n_t+1}^{k'} \widetilde{\mathcal{P}}_\ell \big|> \mathfrak{C}/2\Big). \]
By Lemma \ref{tailpropH}, we deduce that there exist $\mathfrak{c},\mathfrak{h}>0$ depending on the model parameters such that
\[ \PP_{n_{t}}\Big(\Big\{\max_{t\leq t' \leq t_\mathfrak{q}^-} \big|\Psi_{t'} - \Psi_t - \sum_{\ell= n_t+1}^{n_{t'}} g_{\ell}\big|> \mathfrak{C}\Big\} \cap \mathscr{H}\Big) \leq \exp\Big(-\mathfrak{c} \min\big\{t^2(\log n)^{\mathfrak{q}/4},t n^{1/6} (\log n)^{-\mathfrak{h}}\big\}\Big).\]
Since on the event where $W_{n_t} \leq \eta_{n_t}$, $\PP_{n_t}(\mathscr{H}^\complement) \leq e^{- (\log n)^2}$  by Corollary \ref{controlbadblockhall} , it follows that for $\mathfrak{q}$ large enough,
\begin{equation} \label{process3} \PP_{n_{t}}\Big(\max_{t\leq t' \leq t_\mathfrak{q}^-} \big|\Psi_{t'} - \Psi_t - \sum_{\ell= n_t+1}^{n_{t'}} g_{\ell}\big|> \mathfrak{C}\Big) \leq e^{-\mathfrak{c}' (\log n)^2},\end{equation}
where $\mathfrak{c}'>0$ is a constant depending on the model parameters which will change value without changing name. Next, since $\{ g_{\ell} : 1\leq \ell\leq n_{t_\mathfrak{q}^-}\}$ is a sequence of independent centered random variables  with finite second moment, it follows by Sakhanenko' Theorem \ref{Sakhanenko} that there exists a sequence $\{\gamma_\ell : \ell\leq  n_{t_\mathfrak{q}^-}\}$ of independent centered Gaussian random variables such that $\EE[\gamma_\ell^2] = \EE[g_\ell^2]$ for all $\ell$ and 
\begin{equation} \label{expodiff1} \EE\big[e^{\mathfrak{h} \lambda \Delta}\big] \leq 1+\mathfrak{h}^{-1} \lambda \log n,\end{equation}
where $\mathfrak{h}>0$ depends on the model parameters, $\Delta := \max\{ |\sum_{\ell=1}^k g_ \ell -\sum_{\ell=1}^k \gamma_\ell| : 1\leq k \leq n_{t_\mathfrak{q}^-}  \}$ and $\lambda = \min\{\lambda(g_k) : 1\leq k\leq n_{t_\mathfrak{q}^-} \}$. Clearly, we can assume without loss of generality that $\{ \gamma_\ell : 1\leq \ell \leq   n_{t_\mathfrak{q}^-} \}$ are independent from $(a_{k-1}^2,b_k)_{k>k_0}$. 
Now, note that by Propositions  \ref{incremhn} and \ref{increm} we have that $\EE[g_k^2] \asymp_\delta 1/(k_0-k)$ and from the
noise boundedness  assumption \eqref{boundnoiseass}, that  $|g_k|\leq (\log n)^{\mathfrak{h}'} /\sqrt{k_0-k}$ for some $\mathfrak{h}'>0$ depending on the model parameters.  Hence, we deduce from \eqref{compSakhanenko1} that $(\log n)^{-\mathfrak{C}'}\sqrt{k_0-k}\lesssim \lambda(g_k) \lesssim  \sqrt{k_0-k}$. 
Thus, for $\mathfrak{q}$ large enough, $n^{1/6} \lesssim \lambda \lesssim \sqrt{n}$. By \eqref{expodiff1} it follows that $\EE[e^{\mathfrak{h}' n^{1/6} \Delta} ] \leq \sqrt{n} \log n/\mathfrak{h}'$ for some $\mathfrak{h}'>0$ depending on the model parameters. Using Chernoff's inequality, we get that 
\begin{equation} \label{process2} \PP\Big(\max_{1\leq k \leq  n_{t_\mathfrak{q}^-} } \big|\sum_{\ell=1}^k g_\ell - \sum_{\ell=1}^k \gamma_\ell\big|>1\Big) \leq e^{-\mathfrak{c}' n^{1/6}}. \end{equation}
Finally, define the Gaussian processes $\Upsilon$ and $\Gamma$  by
%\textcolor{blue}{What is $\Delta \Gamma, \Delta \Upsilon$? do you mean $\Delta \Gamma_s=\Gamma_{s+1}-\Gamma_s$?}
\begin{equation}
\label{eq-upsilon}
\Upsilon_t := \sum_{\ell\leq n_t} \gamma_\ell, \quad \Gamma_t :=\sqrt{\frac{v}{2}} \sum_{s=1}^t \frac{\Upsilon_s-\Upsilon_{s-1}}{\sqrt{\Var(\Upsilon_{s}-\Upsilon_{s-1})}}, \quad 1\leq t \leq t_\mathfrak{q}^-.\end{equation}
By construction, we have that $\Var(\Upsilon_t) = \Sigma_{n_t}^2$  for any $1\leq t \leq t_\mathfrak{q}^-$, and the process $\Gamma$ satisfies the claimed  properties of Proposition \ref{coupling}.
Using Lemma \ref{approxtimechange} $(iii)$, we get that 
 $\Sigma_{n_s}^2- \Sigma_{n_{s-1}}^2= v/2+ O(n^{-1/3})$ for any $s \leq t_\mathfrak{q}^-$. This implies that $\Var(\Upsilon_s -\Upsilon_{s-1}) \asymp 1$ for any $s \leq t_\mathfrak{q}^-$, and moreover
\[ \Var(\Upsilon_s -\Upsilon_{s-1}- ( \Gamma_s-\Gamma_{s-1})) =O(n^{-1/3}).\]
Since $t_\mathfrak{q}^- \lesssim \log n$, we deduce that 
\[ \Var\big(\Upsilon_{t_\mathfrak{q}^-} - \Upsilon_t - \Gamma_{t_\mathfrak{q}^-} - \Gamma_t \big)\lesssim n^{-1/3}\log n, \]
where we used the independence of the increments of $\Upsilon-\Gamma$. Using Lévy's maximal inequality, it follows that 
\begin{equation} \label{process1} \PP_{n_t}\Big(\max_{t\leq t' \leq t_\mathfrak{q}^-} \big| \Upsilon_{t'} -\Upsilon_t - \Gamma_{t'} - \Gamma_t\big|>1)\leq e^{-\mathfrak{c} n^{1/3}(\log n)^{-1}}.\end{equation}
 Putting together \eqref{process1}, \eqref{process2} and \eqref{process3}, we get the claim. 
\end{proof}

\begin{proof}[Proof of Proposition \ref{coupling} - The elliptic regime] By Lemma \ref{approxtimechange}, we know that $n_{t_\mathfrak{q}^+} \asymp (\log n)^{2\mathfrak{q}} n^{1/3} $. Let $i_t$ be the smallest index of the block after  $n_t$, meaning that $i_t = \min \{ j :k_j \geq n_t\}$.  Let $\mathscr{G}$ be the event where all the blocks $i\geq i_i$ are good, that is,  $\mathscr{G} =  \bigcap_{i\geq i_t+1} \mathscr{G}_i$, 
where   $\mathscr{G}_i$ is defined in \eqref{goodblock}. Note that $i_t \gtrsim (\log n)^{\mathfrak{q}/4}$. Using Corollary \ref{goodblocse}, we get that $\PP_{n_t}(\mathscr{G}^\complement) \leq e^{-(\log n)^2}$ for $\mathfrak{q}$ large enough. 

%By Corollary \ref{goodblocse} we know that $\PP(\mathscr{G}_z^\complement) \leq e^{-\mathfrak{c} (\log n)^2}$. 
Let $t_\mathfrak{q}^+ \leq t \leq N$. By Proposition \ref{represelliptic}, we have that  for any $i\geq i_t$, on the event $\mathscr{G}$,
%\textcolor{blue}{Again, what is $\Delta \psi$? Do you mean $\psi_{k_{i+1}}-\psi_{k_i}$?}
\[  \psi_{k_{i+1}}-\psi_{k_i}= \sum_{k=k_i+1}^{k_{i+1}} \big(w_k(\zeta_{k-1,i}) + \mathcal{P}_k\big) + O(i^{-5/4}),\]
where $\mathcal{P}_k$ is a centered $\cF_k$-measurable random variable satisfying the moment estimate of \eqref{momentestimPe}. Let $i_\mathfrak{q} = (\log n)^{\mathfrak{q}/4}$ and denote by $\widehat{\psi}_{k_i} := \sum_{i_\mathfrak{q} \leq i'\leq i}\sum_{k=k_{i'}+1}^{k_{i'+1}} w_k(\zeta_{k-1,i'})$ for any $i\geq i_\mathfrak{q}$. By \eqref{momentestimPe}, we have that for some $\mathfrak{h}>0$,
\[ \sum_{i\geq i_\mathfrak{q}} \sum_{k=k_i+1}^{k_{i+1}}\|s_{k-1,\mathfrak{h}}(\mathcal{P}_k)\|_{L^\infty} \lesssim \sum_{i\geq i_{\mathfrak{q}}}i^{-5/4}\lesssim(\log n)^{-\mathfrak{q}/16}. \]
Moreover, $\max_k \|s_{k-1,\mathfrak{h}}(\mathcal{P}_k)\|_{L^\infty}\lesssim n^{-1/6}$. Fix $i\geq i_\mathfrak{q}$.
Hence, for $\mathfrak{q}$ large enough, we obtain by Lemma \ref{tailpropH} that for any $i\geq i_\mathfrak{q}$, 
\[ \PP_{k_i}\Big(\max_{i' \geq i} \big|\psi_{k_{i'}}-\psi_{k_i} - \widehat{\psi}_{k_{i'}} -\widehat{\psi}_{k_i}\big|>1\Big) \leq e^{-  (\log n)^{2}}.\]
To be able to use Sakhanenko's coupling theorem of sum of independent random variables, we need to aggregate further, within each of the blocks $(k_{i},k_{i+1}]$, the variables $w_k(\zeta_{k-1,i})$ over sub-blocks $(k_{ij})_j$. Indeed, the variables $w_k(\zeta_{k-1,i})$ are too heterogeneous since the variance of some of them may vanish. 
Thus, in a similar fashion as in the proof of Lemma \ref{approxint}, we define $(k_{ij})_j$ by $k_{i0} = k_i$  and $k_{i(j+1)} : = k_{ij} + \lfloor 2\pi /\theta_{k_{ij}}\rfloor$ for any $j\geq 0$. Let $j_i = \sup\{j : k_{ij} <k_{i+1}\}$ and redefine $k_{ij_i} = k_{i+1}$.  Since $\theta_k \asymp  \sqrt{(k-k_0)/k_0}$ by \eqref{diffRI} for any $k \geq k_0+\ell_0$, one can check that $k_{ij} -k_0\asymp n^{1/3}(i^6+j)^{2/3}$  and that $k_{ij}-k_{i(j-1)} \asymp  (i^6+j)^{1/3} n^{-1/3}$ for any $j\leq j_i$. Define $H_i$ the process that results from aggregating the variables $w_k(\zeta_{k-1,i})$ over sub-blocks:
\[ H_{ij} :=\sum_{k= k_{i(j-1)}+1}^{k_{ij}} w_k(\zeta_{k-1,i}), \ 1\leq j\leq j_i.\]
We claim that there exists $\mathfrak{h}>0$ depending on the model parameters such that
\begin{equation} \label{SakhanenH} (\log n)^{-\mathfrak{h}} (i^6+j) \lesssim \lambda(H_{ij}) \lesssim (i^6+j), \ 1\leq j\leq j_i.\end{equation}
From Lemma \ref{represellipticg}, we have that 
\begin{equation} \label{varHj} \Var_{k_{i}}(H_j)\asymp  \frac{1}{i^6+j}, \ 1\leq j \leq j_i,\end{equation}
where we used the fact that $k_{i(j-1)}-k_0\asymp n^{1/3} (i^6+j)^{2/3}$ and $k_{ij}- k_{i(j-1)} \asymp (i^6+j)^{1/3}n^{-1/3}$. Using \eqref{compSakhanenko1}, the upper bound on the Sakhanenko parameter of $H_{ij}$ follows from \eqref{varHj}. To prove the lower bound, we will rely on \eqref{compSakhanenko2}.  
Let $f_{k} := \|s_{k-1}(w_k(\zeta_{k-1,i}))\|_{L^\infty}$ for any $k\in (k_i,k_{i+1}]$.  
From the definition of $w_k$ in \eqref{defw} and the fact that $s_{k-1}(c_k)\lesssim 1/(k-k_0)$, $s_{k-1}(d_k) \lesssim 1/n$, it follows that  $f_{k,i} \lesssim 1/(k-k_0)$ for any $k\in(k_i,k_{i+1}]$. A similar computation as in \eqref{varHj} yields that 
\[ \sum_{k_{i(j-1)} < k \leq k_{ij}} f_k \lesssim \frac{1}{i^6+j}, \ 1\leq j\leq j_i.\]
 By Lemma \ref{tailpropH} -- \eqref{subgauss} there exists $\mathfrak{h}>0$ depending on the model parameters such that 
\[ \EE_{k_i}[\exp\Big(-  (\log n)^{-\mathfrak{h}} \sqrt{i^6+j} H_{ij} \Big)\big] \leq 2.\]
Thus, using \eqref{compSakhanenko2} and \eqref{varHj} we deduce that $\lambda(H_{ij})\gtrsim (\log n)^{-3\mathfrak{h}} (i^6+j)^{1/2}$, which ends the proof of \eqref{SakhanenH}. Since $j_i \lesssim i^5$, an immediate consequence of \eqref{SakhanenH} is that 
\[ (\log n)^{-3\mathfrak{h}} i^3 \lesssim \inf_{1\leq j\leq j_i} \lambda(H_{ij})\lesssim  i^3.\]
Now,  since the variables $H_{ij}$, $1\leq j \leq j_i$ are independent and with sum $\widehat{\psi}_{k_{i+1}}-\widehat{\psi}_{k_i}$, 
%\textcolor{blue}{Again, $\Delta$ undefined} 
we can apply Sakhanenko' Theorem \ref{Sakhanenko} conditionally on $\widetilde{\cF}_{k_{i}}$ for each $i\geq i_\mathfrak{q}$. As a result, we can find a good filtration $\mathcal{G}$ augmenting $\widetilde{\cF}$ and random variables $\{\Upsilon_{i+1}, i\geq i_\mathfrak{q}\}$
(not to be confused with the variables $\Upsilon_i$ in \eqref{eq-upsilon}!) such that $\Upsilon_{i+1}$ is $\mathcal{G}_{k_{i+1}}$-mesurable and conditionally on $\mathcal{G}_{k_{i}}$, $\Upsilon_{i+1}$ 
%\textcolor{blue}{earlier, $\Upsilon$ was a running sum, should it be $\Delta \Upsilon$?}
is a centered Gaussian random variable with the same variance as $\Delta \widehat{\psi}_{k_{i+1}}$, and  
\[ \EE_{k_{i}}\Big[\exp\Big( \frac{ i^3\big|\widehat{\psi}_{k_{i+1}} -\widehat{\psi}_{k_i}- \Upsilon_{i+1}\big|}{(\log n)^{\mathfrak{h}'} } \Big) \Big] \leq 1+ \mathfrak{C}' i^3,\]
where $\mathfrak{h}'$ and $\mathfrak{C}'>0$ depend on the model parameters and we used the fact that $\Var_{k_{i}}(\Delta \widehat{\psi}_{k_{i+1}}) \lesssim i^6$. 
 Using Chernoff's inequality, this yields that there exists $\mathfrak{h}''>0$ depending on the model parameters such that 
\begin{equation} \label{errorcouplinge} \PP_{k_{i}}\big(|\widehat{\psi}_{k_{i+1}} -\widehat{\psi}_{k_i}
%\Delta \widehat{\psi}_{k_{i+1}}
-\Upsilon_{i+1}|> \frac{(\log n)^{\mathfrak{h}''}}{i^3}\big) \leq e^{- (\log n)^2},\end{equation}
where we used the fact that $i^3\leq i_1^3\lesssim n^{1/2}$. Since $\sum_{i\geq i_\mathfrak{q}}  i^{-3}\leq (\log n)^{-\mathfrak{h}''}$ for $\mathfrak{q}$ large enough, we deduce by using a union bound that for any $i\geq i_\mathfrak{q}$, 
\begin{equation} \label{probacouplf} \PP_{k_{i}}\Big(\max_{i\leq i'}\big|\widehat{\psi}_{k_{i'}}-\widehat{\psi}_{k_{i}} -\sum_{j=i+1}^{i'} \Upsilon_{j}\big|>1\Big) \leq e^{-\frac12(\log n)^2}.\end{equation}
Now, denote by $\widehat{\Upsilon}_t$ the aggregation of variables $\Upsilon_i$ corresponding to blocks included in $(n_{t-1},n_{t}]$, that is, 
\[ \widehat{\Upsilon}_t := \sum_{i : (k_{i-1},k_{i}] \subset (n_{t-1},n_{t}]} \Upsilon_i, \quad t_\mathfrak{q}^+ \leq t \leq T_z.\]
Let  $\widehat{\Gamma}$ be the process with increments $\widehat{\Upsilon}$ and $\Gamma$ the normalised process so that all increment have the same conditional variance. More precisely, set 
\[ \widehat{\Gamma}_{t} := \sum_{s= t_\mathfrak{q}^++1}^{t} \widehat{\Upsilon}_s, \ \Gamma_{t} := \sqrt{\frac{v}{2}}\sum_{s=t_\mathfrak{q}^++1}^{t}  \frac{\widehat{\Upsilon}_s}{\sqrt{ \Var_{n_{s-1}}(\widehat{\Upsilon}_s) }} , \ \quad t_\mathfrak{q}^+ \leq t \leq T_z.\]
We claim that the process $\Gamma$ satisfies the statements of the lemma.
%has the claimed properties. 
Indeed, $\Gamma$ is clearly adapted to the filtration $(\mathcal{G}_{n_s})_{s}$. Moreover, given $\mathcal{G}_{n_{s-1}}$, $\Gamma_{s}-\Gamma_{s-1}$ is a centered Gaussian random variable with deterministic variance $(v/2)  $ for any $t_\mathfrak{q}^+< s \leq  T_z$. Thus, $\Gamma_{s}-\Gamma_{s-1}$ is actually independent of $\mathcal{G}_{n_{s-1}}$ and is a centered Gaussian of variance $v/2$  for any $t_\mathfrak{q}^+< s \leq T_z$. It now remains to prove that $\Gamma$ constitutes a good coupling. To this end fix $t\geq t_\mathfrak{q}^+$. In a first step, we will show that the increments of $\widehat{\Gamma}$ and $\Gamma$ are within a  distance of order $1$. To this end, we will estimate the variance of difference of the increments of $\widehat{\Gamma}$ and $\Gamma$, and show that it is at most $(\log n)^{-4}$, from which the claim follows by using a union bound.  First, we prove  that for any $t_\mathfrak{q}^+<s \leq T_z$, 
\begin{equation} \label{claimvar} \Var_{n_{s-1}}(\widehat{\Upsilon}_s) = \frac{v}{2} + O((\log n)^{-3}).\end{equation}
Since the variables $\Upsilon_i$ are conditionally centered, it follows that
\[ \Var_{n_{s-1}}[\widehat{\Upsilon}_t] = \sum_{i : (k_{i-1},k_{i}] \subset (n_{s-1},n_{s}]} \Var_{n_{s-1}}[\Upsilon_i].\]
Using that $\Var_{k_{i-1}} (\widehat{\psi}_{k_{i}}-\widehat{\psi}_{k_{i-1}}) = \sum_{k=k_{i-1}+1}^{k_i} \sigma_k^2 + O(i^{-7})$ according to Proposition \ref{represelliptic} and the fact that $\Var_{k_{i-1}}(\Upsilon_i) = \Var_{k_{i-1}} (\widehat{\psi}_{k_i}-\widehat{\psi}_{k_{i-1}})$, we find that
\[ \Var_{n_{s-1}}[\widehat{\Upsilon}_s] =  \sum_{i : (k_{i-1},k_{i}] \subset (n_{s-1},n_{s}]} \sum_{k=k_{i-1}+1}^{k_i} \sigma_k^2 + O\Big(\sum_{i\geq i_{\mathfrak{q}}}i^{-7}\Big). \]
Since $\sigma_k^2\lesssim 1/(k-k_0)$, we deduce that the accumulation of the $\sigma_k$'s in a given block $i$ is at most $\Delta k_i/(k_i-k_0)\lesssim 1/i$. For $i\geq i_\mathfrak{q}$ and $\mathfrak{q}$ large enough, this contribution is at most of order $(\log n)^{-4}$. Moreover, $\sum_{i\geq i_\mathfrak{q}} i^{-7}\lesssim( \log n)^{-4}$ for $\mathfrak{q}$ large enough. Hence, for $\mathfrak{q}$ large enough, 
\[ \Var_{n_{s-1}}[\widehat{\Upsilon}_s] =  \sum_{k=n_{s-1}+1}^{n_s} \sigma_k^2 + O((\log n)^{-4}). \]
Besides, we know by Lemma \ref{approxtimechange} $(iii)$ that $ \sum_{k=n_{s-1}+1}^{n_s} \sigma_k^2  = v/2 + O(n^{-1/6})$, which ends the proof of \eqref{claimvar}. As a result of \eqref{claimvar} and the fact that $\Var_{n_{s-1}}(\widehat{\Upsilon}_s)\gtrsim 1$ for any $s$, we get for any $t'> t$, 
\begin{equation} \label{vardeltaGamma} \Var_{n_{t'-1}}\big(\widehat{\Gamma}_{t'}-\widehat{\Gamma}_{t'-1} - ( \Gamma_{t'} -\Gamma_{t'-1})\big)\lesssim (\log n)^{-4}.\end{equation}
Using a union bound, the fact that  $\widehat{\Gamma}_{t'}-\widehat{\Gamma}_{t'-1} - ( \Gamma_{t'} -\Gamma_{t'-1})$
%$\Delta \widehat{\Gamma}_{t'} -\Delta \Gamma_{t'}$ 
is a centered Gaussian random variable given $\mathcal{G}_{n_{t'}}$ for any $t'$, \eqref{vardeltaGamma} and that $T=\log n$, we find that 
\begin{align*}& \PP_{n_{t}}\Big(\max_{t\leq t'\leq T} \big|\widehat{\Gamma}_{t'}-\widehat{\Gamma}_{t} -(\Gamma_{t'}-\Gamma_{t})\big|
>1\Big) \\
& \leq \sum_{t'= t+1}^{T} \PP_{n_t}\big( \big|\widehat{\Gamma}_{t'}-\widehat{\Gamma}_{t'-1} - ( \Gamma_{t'} -\Gamma_{t'-1})
%\Delta \widehat{\Gamma}_{t'} -\Delta \Gamma_{t'}
\big|>(\log n)^{-2} \big) \leq e^{-\mathfrak{c} (\log n)^2},\end{align*}
where $\mathfrak{c}>0$ depends on the model parameters, which conclude the first step of the proof. 

Next, we show that the increments of $\widehat{\Gamma}$ and $\Psi$ are close to each other with overwhelming probability. 
Let $I_t$ be the set of indices $i$ such that the block $(k_{i-1},k_i]$ overlaps some interval $(n_{s-1},n_s]$ for some $s>t$ without being included in $(n_{s-1},n_s]$. 
 If $t\leq t' \leq N$,  then there exists $j_{t'}\geq i_t-1$ such that $n_{t'} \in [k_{j_{t'}}, k_{j_{t'}+1})$ and we can write $\widehat{\Gamma}_{t'} - \widehat{\Gamma}_t = \sum_{i=i_t}^{j_t}\widehat{\Upsilon}_i + O(\sum_{i\in I_t} |\Upsilon_i|)$.
 Thus, we can write 
\begin{align}
 \big|\Psi_{t'} -\Psi_t - (\widehat{\Gamma}_{t'} -\widehat{\Gamma}_t)\big|&\leq |\psi_{n_{t}}-\psi_{k_{i_t}}| + |\psi_{n_{t'}}-\psi_{k_{j_{t'}}}| +  |\psi_{k_{j_{t'}}}-\psi_{k_{i_t}} - (\widehat{\psi}_{k_{j_{t'}}} -\widehat{\psi}_{k_{i_t}})| \nonumber \\
& + | \widehat{\psi}_{k_{j_{t'}}} -\widehat{\psi}_{k_{i_t}}- \sum_{j=i_t}^{j_{t'}}\Upsilon_{j} | +O\Big( \sum_{i\in I_t} |\Upsilon_i|\Big). \label{decoupdist}
\end{align}
Note that  for any $i\geq i_{\mathfrak{q}}$, since $\Delta k_i/(k_i-k_0)\lesssim 1/i\lesssim( \log n)^{-2}$ for $\mathfrak{q}\geq 2$, we have by using the a priori estimate of Proposition \ref{apriori} that for any $k\in [k_i,k_{i+1})$, $i\geq i_\mathfrak{q}$, 
\begin{equation} \label{aprioriblock} \EE_{k_{i}}\big[e^{\theta (\psi_k- \psi_{k_{i})}}\big] \leq e^{\mathfrak{c} \theta^2 (\log n)^{-2}},\end{equation}
for any $1\leq |\theta|\leq \kappa^{\mathfrak{d}}$, where $\mathfrak{d}>0$ is some constant.
Using a union bound, Chernoff inequality and \eqref{aprioriblock}
\begin{equation} \label{errordisc1} \PP_{n_{t}}\Big(|\psi_{n_t} -\psi_{k_{i_t}}|\vee \max_{j\geq i_t} \max_{k\in [k_j,k_{j+1}) } |\psi_k -\psi_{k_{j}}|>1\Big) \leq e^{-\mathfrak{c}'(\log n)^2},\end{equation}
where $\mathfrak{c}'>0$ depends on the model parameters. 
Moreover, again by using a union bound and the fact that $\Upsilon_i$ is a centered Gaussian random variable with variance at most $\Delta k_i/(k_-i-k_0)\lesssim 1/i \leq (\log n)^{-3}$ conditionally on $\mathcal{G}_{k_{i-1}}$ for any $i\geq i_\mathfrak{q}$ and $\mathfrak{q}$ large enough, we get that 
\begin{equation} \label{errordisc2} \PP_{n_{t}}\big(\sum_{i\in I_t} |\Upsilon_i|>1\big) \leq e^{-\mathfrak{c}''(\log n)^2},\end{equation}
where $\mathfrak{c}''>0$ depends on the model parameters and we used the fact that $|I|\lesssim \log n$.  In view of \eqref{decoupdist}, combining \eqref{errordisc1}, \eqref{errordisc2}, \eqref{probacouplf}  we get the claim.
\end{proof}

\section{Establishing a rough barrier via continuity estimates}
\label{sec-Barrier}
We begin in this section the actual proof of the upper bound of Theorem \ref{reduc}. Recall the definition of the field $\Psi^z$  from \eqref{defPsi} and of the restricted
 time-set $\mathcal{T}_{\mathfrak{q},z}$ in \eqref{defI}. We exhibit in this section two results. The first standard result is that one can 
%In the first standard step, we
 reduce attention to a discrete 
set of $z$'s, of cardinality of order $n$, see Lemma \ref{lem-LP}. The second, more involved, result shows that we can  we introduce a rough barrier such that the solution $\Psi_t^z$ of the recursion, for any  $z$ 
in a set of cardinality $O(n)$,
must be below the barrier at all times $t\in \mathcal{T}_{\mathfrak{q},z}$. The a-priori exponential bounds of Section \ref{sec-Precise} allow us to show that claim at a time $t$ for a collection of 
"representatives at level $t$" which are introduced in Definition \ref{representatives}, see Lemma \ref{allgoodh}. Unfortunately, the cardinality of that set of representatives is only of the order of $e^t$, and going from that to all $z$'s in a set of cardinality $O(n)$ utilizes continuity estimates that we develop in this section.

%To prove the upper bound of Theorem \ref{reduc}, 
We now turn to the first step.
%, that is, we  discretize the maximum of the field $\psi_n(z)$ over $z\in I_\eta$.}
 This is done via the following general deterministic statement on the maxima of polynomials over intervals. 

\begin{Lem}[{\cite[Lemma 5.3]{LP}}]
\label{lem-LP}
Let $f$ be a polynomial of degree $n\geq 1$. Then, 
\[ \max_{x \in [-1,1]} |f(x)|\leq 14 \max_{x\in \mathcal{N}} |f(x)|,\]
where $\mathcal{N} := \{\cos\big(\frac{\pi(k-1)}{2n}\big) : 1\leq k \leq 2n+1\}$. 
\end{Lem}

As a consequence, we deduce that there exists $\mathcal{N} \subset I_\eta$, which we will set for the rest of this paper, verifying that $\#\mathcal{N} \lesssim n$, $\min\{|z-z'| : z,z\in \mathcal{N}, z\neq z'\}\gtrsim n^{-1}$ and $\max_{z\in I_\eta} \psi_n(z) \leq  \max_{z\in \mathcal{N}} \psi_n(z) +O(1)$.
Thus, it suffices to prove the upper bound of Theorem \ref{maxnorm1} for the discretized field $\psi_n(z)$, $z\in \mathcal{N}$ with which we will work from now on.

  The second goal of this section, on which we spend most of our effort,  is  the following lemma.
\begin{lemma}\label{barrierloose1}There exists $\mathfrak{C}>0$ depending on the model parameters  such that for $\mathfrak{q}$ large enough,
\[ \PP\Big(\exists z \in \mathcal{N}, \exists t \in \mathcal{T}_{\mathfrak{q},z},  \ \Psi_t^z > \sqrt{v} t  +\mathfrak{C}  \log \uptau\Big) \underset{n\to +\infty}{\longrightarrow} 0.\]
\end{lemma}

Before going into the proof of Lemma \ref{barrierloose1}, we introduce the notion of  ``representatives''   which we will use throughout the remainder of this work. 

\begin{Def}[Representatives at level $t$]\label{representatives}
In the sequel, we set $T^* := \max_{z\in \mathcal{N}} T_z$. By Lemma \ref{approxtimechange} we know that $T^* = \uptau+O(1)$. Fix  for any  $t\in \llbracket 1, T^*\rrbracket$, a subset $\mathcal{N}_t \subset \mathcal{N}$ such that $\# \mathcal{N}_t \lesssim e^{t}$, $\min\{|z-z'| : z,z' \in \mathcal{N}_t, z\neq z'\}\gtrsim e^{-t}$ and such that for any $z'\in \mathcal{N}_t$ there exists $z\in\mathcal{N}_t$, $|z-z'|\leq e^{-t}$, $|z|\leq |z'|$. In other words, $\mathcal{N}_t$ is a $e^{-t}$-net of $\mathcal{N}$ of size $O(e^t)$ that is $\Omega(e^{-t})$-separated and such that one can approximate every element of $\mathcal{N}$ by an element of $\mathcal{N}_t$ with smaller absolute value, a fact that will enable us to exploit some monotonicity properties. 
The elements of $\mathcal{N}_t$ will be called the  ``representatives at level $t$''. 

\end{Def}

We separate the  proof Lemma \ref{barrierloose1} into the hyperbolic and elliptic regime. 

\subsection{Barrier estimate in the hyperbolic regime} \label{barrierhyper} We start with the hyperbolic regime, meaning that  the potential crossing of the barrier happens in the hyperbolic regime. The first step towards Lemma \ref{barrierloose1} is to show that  all the representatives at any level $t \leq t_\mathfrak{q}^-$ stay below the barrier. Recall that $\uptau=\log n$.
\begin{lemma}[``All representatives are good'']\label{allgoodh}
There exists $\mathfrak{C}>0$ such that 
\[ \PP\Big(\exists t \leq t_\mathfrak{q}^-, \exists z \in \mathcal{N}_t, {\Psi}_t^z \geq \sqrt{v} t + \mathfrak{C} \log \uptau\Big) \underset{n\to+\infty}{\longrightarrow} 0.\]
\end{lemma}

\begin{proof}
This follows from the precise exponential moment estimate of Proposition \ref{expopreciseincre} together with a union bound. Indeed, using Proposition \ref{expopreciseincre} and the fact that $\Sigma_{n_t,z}^2 = \frac{v}{2} t + O(1)$ by Lemma \ref{approxtimechange} $(iii)$, we get that for any $t \leq t_\mathfrak{q}^-$ and $z\in \mathcal{N}_t$, 
\[ \EE\big[e^{\frac{2}{\sqrt{v}} \Psi_t^{z}}\big] \lesssim e^{t }.\]
Hence, using Chernoff inequality, we get that for any $t\leq t_\mathfrak{q}^-$ and $z\in \mathcal{N}_t$,
\[ \PP\big(\Psi_t^z \geq \sqrt{v} t + \mathfrak{C} \log \uptau\big)\lesssim e^{-2t - \frac{2}{\sqrt{v}}\mathfrak{C} \log \uptau} e^{t} = e^{- t - \frac{2}{\sqrt{v}}\mathfrak{C}\log \uptau}.\]
Using a union bound and the fact that  $\#\mathcal{N}_t\lesssim e^t$, we get the claim.
\end{proof}

%Since we are only aiming at putting a barrier up to an additive $O(\log T)$ term, we can reduce ourselves to show a continuity estimate on the martingale part of the field $\Psi^z$, which will simplify the computations. More precisely, 
%for any $z \in \mathcal{N}$, let $\widehat{\psi}(z)$ be the process such that $\widehat{\psi}_0(z) =0$ and 
%\[ \Delta\widehat{\psi}_k(z) := \widehat{\psi}_k(z) - \EE_{k-1}(\widehat{\psi}_k(z)),\  k\leq k_{0,z}-\ell_0,\]
%where $\mathcal{P}_k^z$ is as in Lemma \ref{incremhn} for $k\leq \corO{k_{\delta,z}}$ and Lemma \ref{increm} for $k\geq \corO{k_{\delta,z}}$. Further, define the rescaled process $\widehat{\Psi}^z$ as 
%\[ \widehat{\Psi}_t^z := \widehat{\psi}_{n_{t,z}}(z), \ t\in \mathcal{T}.\]
%The two fields $\Psi^z$ and $\widehat{\Psi}^z$ are close up an order $1$ term with high probability as shown in the following.

The second step consists in showing the following continuity estimate on the field $\Psi^z$.  

\begin{lemma}[Continuity estimate]\label{jointprobamax}
 For  $\mathfrak{q}$ large enough, there exists $\mathfrak{C}, \mathfrak{c}>0$  depending on the model parameters such that for any $z\in \mathcal{N}$, $t\leq t_\mathfrak{q}^-$, $0\leq r\leq \sqrt{v} t$ and $0\leq s \leq (\log n)^{\mathfrak{q}/24}$, 
\[ \PP_{\mathscr{H}_\mathfrak{q}}\Big( {\Psi}_t^z \geq \sqrt{v} t - r, \max_{|z'-z|\leq  e^{-t}\atop z' \in \mathcal{N}, |z|\leq |z'|}\big|{\Psi}_t^{z'} -  {\Psi}_t^{z} \big| \geq s \Big) \leq \mathfrak{C}e^{-t-\mathfrak{c}s^2 +\frac{2}{\sqrt{v}}r}, \]
and 
\[ \PP_{\mathscr{H}_\mathfrak{q}}\Big( \max_{|z'-z|\leq  e^{-t}\atop z' \in \mathcal{N}, |z|\leq |z'|}\big|{\Psi}_t^{z'} -  {\Psi}_t^{z} \big| \geq s \Big) \leq \mathfrak{C}e^{-\mathfrak{c}s^2},\]
where $\PP_{\mathscr{H}_\mathfrak{q}} = \PP(\cdot \cap \mathscr{H}_\mathfrak{q})$.
\end{lemma}
Before going into the proof Lemma  \ref{jointprobamax}, we show how Lemmas \ref{allgoodh} and \ref{jointprobamax} entail the first half of Lemma \ref{barrierloose1}.

\begin{proof}[Proof of Lemma \ref{barrierloose1} - the case where $t\leq t_\mathfrak{q}^-$] 
%Write $\PP_{\mathscr{H}_\mathfrak{q}}$ for the probability measure restricted to $\mathscr{H}_\mathfrak{q}$.
Let $\mathscr{R}$ denote the event
\[  \mathscr{R} := \Big\{\forall t \leq t_\mathfrak{q}^-, \forall z \in \mathcal{N}_t, {\Psi}_t^{z} \leq \sqrt{v} t + \mathfrak{C} \log 
\uptau\Big\},\]
where $\mathfrak{C}$ is the constant from Lemma \ref{allgoodh}. To use Lemma \ref{jointprobamax}, we discretize the values of   $\sqrt{v}t +\mathfrak{C} \log \uptau -\widehat{\Psi}_t^{z}$ by integers and denote for any $k\geq 1$, $t\leq t_\mathfrak{q}^-$ and $z \in \mathcal{N}_t$  by $\mathscr{E}_{z,t,k}$ the event:
 \[ \mathscr{E}_{z,t,k} :=\Big\{ {\Psi}_t^{z} \geq \sqrt{v}t+\mathfrak{C}\log \uptau - k, \max_{|z-z'|\leq e^{-t} \atop |z|\leq |z'|} \big({\Psi}_t^{z'}-{\Psi}_t^{z}\big) \geq k-1 + \mathfrak{C} \log \uptau\Big\},\]
 and by $\mathscr{E}_{z,t,-}$ the event 
 \[ \mathscr{E}_{z,t,-} := \Big\{\Psi_t^z \leq 0, \max_{|z-z'|\leq e^{-t}\atop |z|\leq |z'|} (\Psi_t^{z'} - \Psi_t^z) \geq \sqrt{v}t +\mathfrak{C}
\log \uptau\Big\}.\]
Assume  $t\leq t_\mathfrak{q}^-$ and $z'\in \mathcal{N}$ are such that $\Psi_t^{z'} >\sqrt{v}t +2\mathfrak{C} \log 
\uptau$ and the event $\mathscr{R}$ happens. Then, there exists $z\in \mathcal{N}_t$ such that $|z|\leq |z'|$ and either $\Psi_t^{z} \geq \sqrt{v}t+\mathfrak{C}\log \uptau - k$ and ${\Psi}_t^{z'}-{\Psi}_t^{z} \geq k-1 + \mathfrak{C} \log \uptau$ for some $1\leq k \leq \sqrt{v} t +\mathfrak{C} \log \uptau$, or  $\Psi_t^{z} \leq 0$ and $\Psi_t^{z'} - \Psi_t^z \geq \sqrt{v}t +\mathfrak{C}\log \uptau$. Thus, by using a union bound we obtain that 
\[ \PP_{\mathscr{H}_\mathfrak{q}}\Big(\big\{\exists t \leq t_\mathfrak{q}^-, \exists z' \in \mathcal{N}, \  {\Psi}_{t}^{z'}> \sqrt{v} t  + 2\mathfrak{C} \log \uptau\big\}\cap \mathscr{R} \Big)   \leq \sum_{t\leq t_\mathfrak{q}^-, z\in \mathcal{N}_t} \Big(\sum_{ k=1}^{m} \PP_{\mathscr{H}_\mathfrak{q}}(\mathscr{E}_{z,k})+ \PP_{\mathscr{H}_\mathfrak{q}}(\mathscr{E}_{z,k,-})\Big),\]
where $m:=\lceil \sqrt{v} t +\mathfrak{C} \log \uptau\rceil$.
 For any $t\leq t_\mathfrak{q}^-$, $z\in\mathcal{N}_t$ and $1\leq k \leq m$,  we know by Lemma \ref{jointprobamax} that 
\[  \label{boundcont} \PP_{\mathscr{H}_\mathfrak{q}}(\mathscr{E}_{z,t,k}) \leq  \mathfrak{c}^{-1}e^{-t -\frac{2}{\sqrt{v}}\mathfrak{C} \log 
\uptau + \frac{2}{\sqrt{v}} k - \mathfrak{c} k^2}, \]
where $\mathfrak{c}>0$ depends on the model parameters. Moreover, again by Lemma \ref{jointprobamax}, for $n$ large enough
\[ \PP_{\mathscr{H}_\mathfrak{q}}(\mathscr{E}_{z,t,k,-}) \leq e^{-t -\mathfrak{c}' (\log \uptau)^2},\]
where $\mathfrak{c}'>0$ depends on the model parameters.
Since $\# \mathcal{N}_t\lesssim e^t$, $\# \mathcal{T}\lesssim \uptau$ and $\sum_{k\geq 1} e^{\frac{2}{\sqrt{v}}k-\mathfrak{c}k^2} \lesssim 1$, we find that at the price of taking $\mathfrak{C}$ large enough,
\[ \PP\Big(\big\{\exists t \in \mathcal{T},  t \leq t_\mathfrak{q}^-, \exists z \in \mathcal{N}, \  {\Psi}_{t}^z> \sqrt{v} t  + 2\mathfrak{C}
 \log \uptau\big\}\cap \mathscr{R} \cap \mathscr{H}_\mathfrak{q}\Big)   \underset{n\to+\infty}{\longrightarrow} 0.\]
Since $\PP(\mathscr{R}\cap \mathscr{H}_\mathfrak{q}) \to 1$ by Lemma \ref{allgoodh} and \eqref{goodeventH} for $\mathfrak{q}$ large enough, this gives the claim.
\end{proof}

%
%\begin{lemma}
%There exists $\mathfrak{h}>0$ depending on the model parameters such that 
%\[ \PP\big(\max_{t\in \mathcal{T},z\in \mathcal{N}} |\Psi_t^z - \widehat{\Psi}_t^z| > \mathfrak{h}\big) \underset{n\to+\infty}{\longrightarrow} 0.\]
%\end{lemma}
%\begin{proof}
%\[ g_k^z + \mathcal{P}_k^z - \EE_{k-1}(\mathcal{P}_k^z), \]
%Recall the event $\mathscr{G} = \big\{\forall z \in \mathcal{N}, \forall k \leq k_\mathfrak{q}^-, W_k^z\leq \eta_{k,z}\big\}$. By Lemma \ref{controlbadblockh}, we know that $\PP(\mathscr{G}) \to 1$ as $n\to+\infty$. On this event,  we know that for any $z\in \mathcal{N}$, 
%\[ \sum_{k\leq k_{0,z}} \|\EE_{k-1}\mathcal{P}_k\|_{L^\infty} \lesssim 1.\]
%Hence, it follows from the definition of $\widehat{\Psi}$ that  on $\mathscr{G}$, 
%\[ \sup_{t\in \mathcal{T}, z\in \mathcal{N}} \big|{\Psi}_t^z-\widehat{\Psi}_t^z \big| \lesssim 1,\]
%which gives the claim.
%\end{proof}

We now move on to prove Lemma \ref{jointprobamax}. Our argument is based on a exponential moment estimate together with a chaining argument. 
To state the exponential moment estimate we will need, define $\widehat{\psi}(z)$ as the martingale part of $\psi(z)$ and $\widehat{\Psi}^z$ the corresponding scaled process. More precisely,  set for $z\in I_\eta$
\[ \widehat{\psi}_k(z) := g_{k,z} +\mathcal{P}_k^z -\EE_{k-1}(\mathcal{P}_k^z), \ 1\leq k \leq k_{0,z}-\ell_0,\]
where $\mathcal{P}_k^z$ is as in Propositions  \ref{incremhn} and \ref{increm}, and 
\[ \widehat{\Psi}_t^z := \widehat{\psi}_{n_{t,z}}(z), \ 1\leq t\leq \frac23\uptau-2\log \kappa,\]
where $n_{t,z}$ is the time-change defined in \eqref{deftimechange}.  Further, define for any $z\in \mathcal{N}$, the event $\mathscr{H}_{\mathfrak{q},z}$ as
\[ \mathscr{H}_{\mathfrak{q},z} := \big\{  \forall k_{\delta,z}\leq k \leq k_{0,z} - (\log n)^{\mathfrak{q}} n^{1/3}, W_k^{z} \leq \eta_{k,z}\big\}, \  \mathscr{H}_\mathfrak{q} := \bigcap_{z\in \mathcal{N}} \mathscr{H}_{\mathfrak{q},z}.\]
Note that at $k= k_{\delta,z}$, $W_{k}^{z} \leq (\log n)^{\mathfrak{h}_0}/n \leq \eta_{k}^{z}/2$ almost surely by Proposition \ref{boundWhn}, where $\mathfrak{h}_0>0$ depends on the model parameters. Thus, by Lemma \ref{controlbadblockh}, $\mathscr{H}_{\mathfrak{q},z}$ happens with overwhelmingly probability provided $\mathfrak{q}$ is large enough and a union gives that 
\begin{equation} \label{goodeventH} \PP(\mathscr{H}_{\mathfrak{q}}^\complement)  \leq e^{-\frac12 (\log n)^2}.\end{equation}
By Propositions \ref{incremhn} and \ref{increm} we know that the processes $\widehat{\psi}(z)$ and ${\psi}(z)$ are at a distance at most of order $1$  on the event $\mathscr{H}_\mathfrak{q}$ until time $k_{0,z} - n^{1/3} (\log n)^{\mathfrak{q}}$, meaning that on $\mathscr{H}_\mathfrak{q}$, 
\[\max_{1\leq k \leq k_{0,z}-n^{1/3} (\log n)^{\mathfrak{q}}}\big|\widehat{\psi}_k(z)- {\psi}_k(z)\big| \lesssim 1.\] 
With this notation, the  exponential estimate reads as follows.

%for any $z \in \mathcal{N}$, let $\widehat{\psi}(z)$ be the process such that $\widehat{\psi}_0(z) =0$ and 
%\[ \Delta\widehat{\psi}_k(z) := \widehat{\psi}_k(z) - \EE_{k-1}(\widehat{\psi}_k(z)),\  k\leq k_{0,z}-\ell_0,\]
%where $\mathcal{P}_k^z$ is as in Lemma \ref{incremhn} for $k\leq \corO{k_{\delta,z}}$ and Lemma \ref{increm} for $k\geq \corO{k_{\delta,z}}$. Further, define the rescaled process $\widehat{\Psi}^z$ as 
%\[ \widehat{\Psi}_t^z := \widehat{\psi}_{n_{t,z}}(z), \ t\in \mathcal{T}.\]
%The two fields $\Psi^z$ and $\widehat{\Psi}^z$ are close up an order $1$ term with high probability as shown in the following.

% of $\widehat{\psi}(x')-\widehat{\psi}(x)$ and $\widehat{\psi}(z)$ where $x,x'$ are $e^{-t}$-close to $z$, together with a chaining argument.

\begin{lemma}\label{jointexpomomenh}
Let  $z\in \mathcal{N}$ and $t \leq T_z$. Let $x',x\in \mathcal{N}$ such that $|z|\leq |x|$, $|z|\leq |x'|$, $|z-x'|\leq e^{-t}$ and $|z-x|\leq e^{-t}$. For $\mathfrak{q}>0$ large enough, there exist $\mathfrak{C},\mathfrak{h}>0$ depending on the model parameters such that for any $k\leq \min(k_{0,z}- (\log n)^{\mathfrak{q}} n^{1/3}, n_{t,z})$ and $\mu\in \RR$ such that $ |\mu|\leq n^{1/6}(\log n)^{-\mathfrak{h}}$, 
\begin{align*} \log \EE\big[\Car_{\mathscr{H}_\mathfrak{q}}e^{ \mu (\widehat{\psi}_k(x') - \widehat{\psi}_k(x))} \big] &\leq   \mathfrak{C} \mu^2   |x-x'|^2e^{2(t\wedge \frac23 \uptau)} +  \mathfrak{C}  (|\mu|^3\vee \mu^2) (\log n)^{-\mathfrak{q}/4}.
\end{align*}
\end{lemma}

\begin{proof} 
 Note that since $|z|\leq |x|$ and $|z|\leq |x'|$, we have that $k_{0,z}\leq k_{0,x'}$ and $k_{0,z} \leq k_{0,x}$. To lighten the notation, write $k_{0}$ instead of $k_{0,z}$ and $k_\delta$ instead of $k_{\delta,z}$.  On the event $\mathscr{H}_\mathfrak{q}$, we have due to Propositions \ref{incremhn} and \ref{increm}  that  for any $1\leq \ell \leq k_0-\ell_0$,
\begin{equation} \label{condnoiseerror1}   \Delta \widehat{\psi}_\ell(x') - \Delta \widehat{\psi}_\ell(x) =g_{\ell,x'} - g_{\ell,x}  + \mathcal{H}_\ell,\end{equation}
 where  $\mathcal{H}_\ell$ are $\cF_\ell$-measurable variables which satisfy that for some $\mathfrak{h}>0$ depending on the model parameters,
\begin{equation} \label{condnoiseerror}  \EE_{\ell-1} \mathcal{H}_\ell =0,  \    \ s_{\ell-1,\mathfrak{h}}(\mathcal{H}_\ell) \lesssim \veps_{\ell,x}+\veps_{\ell,x'},\end{equation}
where $  \veps_{\ell,x}, \veps_{\ell,x'}$ are given by
\begin{equation} \label{errorepsilon} \veps_{\ell,z'} :=\Big( \frac{(\log n)^{\mathfrak{h}}}{n^{3/2}} + \frac{1}{n\ell} \Big) \Car_{\ell \leq k_{\delta,z'}^-} + \Big(\frac{1}{(k_{0,z'}-\ell)^2} + \frac{\sqrt{\eta_{z',\ell}}}{k_{0,z'}-\ell}\Big) \Car_{\ell >k_{\delta,z'}^-},   z' \in \{z,x,x'\},\end{equation}
where $\mathfrak{h}>0$ depends on the model parameters. 
Note that because of the boundedness assumption of the noise \eqref{boundnoiseass} and \eqref{condnoiseerror}, there exists $\mathfrak{h}>0$ depending on the model parameters such that  almost surely, 
\begin{equation} \label{boundpsihat}  |\Delta \widehat{\psi}_\ell(x') - \Delta \widehat{\psi}_\ell(x)|\leq \frac{(\log n)^{\mathfrak{h}}}{\sqrt{k_0-\ell}}\leq (\log n)^{\mathfrak{h}} n^{-1/6}.\end{equation}
 Let $Z_{k,\mu} =   \mu( \Delta \widehat{\psi}_k(x')- \Delta \widehat{\psi}_k(x))$ for any $k \leq k_0-\ell_0$ and $ \mu \in \RR$.
 Using \eqref{boundpsihat} and Proposition \ref{moddev}, we deduce that for any $ \mu \in \RR$, $ |\mu|\leq n^{1/6}( \log n)^{-\mathfrak{h}}/2$, 
\begin{equation} \label{ineqloglaplace}  \log \EE\big[e^{ Z_{k, \mu}} \big] \leq  \frac12 \Sigma_{k, \mu}^2+ \mathfrak{C}  |\mu|^3 E_k, \end{equation}
where $\Sigma_{k, \mu}^2 := \sum_{\ell \leq k} \|\Var_{\ell-1}(Z_{k, \mu})\|_{L^\infty}$, $\mathfrak{C}>0$ depends on the model parameters and 
\begin{equation} \label{boundEkerror}  E_k := \sum_{\ell\leq k} \frac{(\log n)^{3\mathfrak{C}}}{(k_0-\ell)^{3/2}}\leq  \sum_{\ell\leq k_0-\ell_0} \frac{(\log n)^{3\mathfrak{C}}}{(k_0-\ell)^{3/2}}  \lesssim \frac{(\log n)^{3\mathfrak{C}}}{n^{1/6}}. \end{equation}
It now remains to compute $\Sigma_{k,\mu}^2$. We  claim that 
\begin{equation}\label{vardiff} \EE\big[(g_{\ell,x'} - g_{\ell,x})^2\big] \lesssim
|x'-x|^2  \frac{n^2}{(k_0-\ell)^3}, \quad \ell \leq k_0-\ell_0.\end{equation}
In view of \eqref{defgk}, this amounts to computing the derivatives of the coefficients in front of the noise variables $b_\ell$ and $a_{\ell-1}^2$. We find that for any $y\in I_\eta$ and $\ell<  k_{0,y}$, using the explicit expression of $\alpha_{\ell,y}$ from \eqref{defalpha}, 
\[ \Big| \frac{d}{dy}\frac{1}{\alpha_{\ell,y}^2-1}\Big| \vee  \Big| \frac{d}{dy}\frac{\alpha_{\ell,y}}{\alpha_{\ell,y}^{2}-1}\Big| \lesssim \sqrt{\frac{n}{\ell}} \Big(1 + \frac{y_\ell}{\sqrt{y_\ell^2-4}}\Big) \Big[\frac{\alpha_{\ell,y}^{-3}}{(1-\alpha_{\ell,y}^{-2})^2} + \frac{\alpha_{\ell,y}^{-2}}{1-\alpha_{\ell,y}^{-2}}\Big].\]
If  $\ell \leq k_{\delta,y}^-$, then $1-\alpha_{\ell,y}^{-2}$ is bounded away from $0$ and $y_\ell/\sqrt{y_\ell^2-4}$ is bounded. Hence, since $\alpha_{\ell,y}^{-1} \lesssim \sqrt{\ell/n}$, we get that 
\begin{equation} \label{derivcoeff1} \Big| \frac{d}{dy}\frac{1}{\alpha_{\ell,y}^2-1}\Big| \vee  \Big| \frac{d}{dy}\frac{\alpha_{\ell,y}}{\alpha_{\ell,y}^{2}-1}\Big] \lesssim \sqrt{\frac{\ell}{n}}.\end{equation}
If on the other hand $\ell> k_{\delta,y}^-$, then $\alpha_{\ell,y}^{-1}\lesssim 1$ and $(y_\ell^2-4)^{-1/2} \lesssim \sqrt{k_0/(k_0-\ell)}$, $1-\alpha_{\ell,y}^{-2}  \gtrsim \sqrt{(k_0-\ell)/k_0}$ by \eqref{estimalphak}. Thus, 
\begin{equation} \label{derivcoeff2} \Big| \frac{d}{dz}\frac{1}{\alpha_{k,z}^2-1}\Big| \vee  \Big| \frac{d}{dz}\frac{\alpha_{k,z}}{\alpha_{k,z}^{2}-1}\Big] \lesssim  \Big(\frac{k_{0,z}}{k_{0,z}-k}\Big)^{3/2}.\end{equation}
As $\Var(b_\ell) = O(1)$ and $\Var(a_{\ell-1}^2) = O(\ell)$, the claim \eqref{vardiff} follows from the two estimates  \eqref{derivcoeff1}-\eqref{derivcoeff2} on the derivatives of the coefficients in the definition of $g_{k,y}$ (see \eqref{defgk}) and the fact that $k_{0,x}, k_{0,x'} \geq k_{0}$.
Now, write $\ell_\mathfrak{q} = n^{1/3} (\log n)^{\mathfrak{q}}$. Since $k_0- n_t \asymp   ne^{-t}$ for $t\leq (2/3)\uptau-2\log \kappa$ by Lemma \ref{approxtimechange}, we get that  $k_0- (k_0-\ell_\mathfrak{q})\wedge n_t \gtrsim n^{1/3} \vee (n e^{-t})$ for any $t\leq \uptau$. Hence,  
\[ \sum_{\ell = 1}^{k_\mathfrak{q}^-\wedge n_{t}} \frac{n^2}{(k_{0} -\ell)^3} \lesssim e^{2(t\wedge \frac23 \uptau)}.\]
Together with \eqref{vardiff}, this implies that 
\begin{equation} \label{sumvargdiff} \sum_{\ell \leq k} \EE[(g_{\ell,x'}-g_{\ell,x})^2] \lesssim |x-x'|^2 e^{2(t\wedge \frac23 
\uptau)}.\end{equation}
Besides, using the  definitions of $\eta_x,\eta_{x'}$ (see \eqref{defeta})  and the fact that $k_{0,x},k_{0,x'}\leq k_0$, we get that 
\begin{equation} \label{sumvarerror} \sum_{\ell \leq k_{0}-\ell_{\mathfrak{q}}}(\veps_{x,\ell} + \veps_{x',\ell}) \leq \sum_{\ell\leq k_{0,x'}-\ell_{\mathfrak{q}}} \veps_{x,\ell} +  \sum_{\ell\leq k_{0,x'}-\ell_{\mathfrak{q}}} \veps_{x',\ell} \lesssim n^{-1/3} + \sum_{i\geq i_\mathfrak{q}} i^{-4/3},\end{equation}
where $i_\mathfrak{q} \gtrsim (\log n)^{3\mathfrak{q}/2}$. Thus,  $\sum_{\ell \leq k_{0}-\ell_{\mathfrak{q}}}(\veps_{x,\ell} + \veps_{x',\ell}) \lesssim (\log n)^{-\mathfrak{q}/2}$. 
Together with  \eqref{sumvargdiff},  we deduce  that 
\begin{equation} \label{calvardeltapsi} \sum_{\ell \leq k} \big\|\Var_{\ell-1}\big(\Delta \widehat{\psi}_\ell(x') - \Delta \widehat{\psi}_\ell(x') \big)\big\|_{L^\infty} \lesssim  |x-x'|^2 e^{2(t\wedge \frac23 \uptau)}  + (\log n)^{-\mathfrak{q}/2}.\end{equation}
 Combining \eqref{ineqloglaplace}-\eqref{boundEkerror} with  \eqref{calvardeltapsi},  we get the claim.
\end{proof}

Before giving the proof of Lemma \ref{jointprobamax}, we show the following technical lemma which in particular estimates the variance accumulated between time $n_{t,x}$ and $n_{t,x'}$ when $|x-x'|\leq e^{-t}$.
\begin{lemma} \label{varacctimes}
Let $x,x'\in \mathcal{N}$ and $t\leq T_{x} \wedge T_{x'}$ such that $|t-\frac23 \uptau|> 2\log \kappa$. Then, 
\begin{equation} \label{contsigma} \big| \Sigma^2_{n_{t,x'}}(x) - \Sigma^2_{n_{t,x}}(x)\big| \lesssim |x-x'|e^{t\wedge \frac23 
\uptau}.\end{equation}
Moreover, 
\begin{enumerate}
\item[ $(i).$ ] If  $t\leq \frac23 \uptau - 2\log \kappa$ and $|x|\leq |x'|$ then $n_{t,x} \leq n_{t,x'}$.
\item[ $(ii).$ ] If $t\geq \frac23 \uptau+2\log \kappa$ then for $\kappa$ large enough $T_x=T_{x'}$.
\end{enumerate}
\end{lemma}
\begin{proof} Let $x\in \mathcal{N}$, we claim that 
\begin{equation} \label{claimerivsigma} \Big|\frac{d}{dx} \sigma_{k,x}^2\Big| \lesssim \frac{n}{(k_{0,x}-k)^2}, \ k\leq k_{0,x}-\ell_0, \  \Big|\frac{d}{dx} \sigma_{k,x}^2\Big| \lesssim  \frac{n^{1/2}}{(k-k_{0,x})^{3/2}}, \ k\geq k_{0,x}+\ell_0.\end{equation}
Let $k\leq k_{0,x}-\ell_0$.  
Recall that $\sigma_{k}^2(x) = \EE g_{k,x}^2$ where    $g_{k,x}$ is defined by \eqref{defgk}. By independence, we have that 
\begin{equation} \label{expressvarg} \EE g_{k,x}^2 = \frac{\alpha_{k,x}^{2}}{(\alpha_{k,x}^{2}-1)^2} v_k + \frac{1}{(\alpha_{k,z}^{2}-1)^2}r_k,\end{equation}
where $v_k = \Var(b_k/\sqrt{k})$ and $r_k = \Var(a_{k-1}^2/\sqrt{k(k-1)})$. With similar arguments as in the proof of \eqref{derivcoeff2}, we find that for any $k \leq k_{0,x}-\ell_0$, 
\begin{equation} \label{compderivcarre} \Big| \frac{d}{dx} \frac{\alpha_{k,x}^{2}}{(\alpha_{k,x}^{2}-1)^2}\Big| \vee \Big| \frac{d}{dx} \frac{1}{(\alpha_{k,x}^{2}-1)^2}\Big|  \lesssim \frac{n^2}{(k_{0,x}-k)^2}.\end{equation}
Since $v_k =O(1/k)$ and $r_k= O(1/k)$, this proves the first case of the claim \eqref{claimerivsigma}. Now, assume $k\geq k_{0,x}+\ell_0$. According to \eqref{expresigmaelliptic}, $\sigma_{k}^2(x)$ is then given by
\begin{equation}\label{defsigmaelliptic} \sigma_k(x)^2 = (\EE c_{k,x}^2 + \EE d_k^2) \big(1+\frac12 \sin^2(\theta_k)\big) + \EE(c_{k,x} d_k)\big(\sin(2\theta_k) - \frac12 \cos^2(\theta_k)\big),\end{equation}
where $c_{k,x}$ and $d_k$ are as in \eqref{defc}. Using \eqref{estim}, we find that 
\begin{equation} \label{estimderivcd} \Big| \frac{d}{dx} \EE c_{k,x}^2 \Big| \lesssim \frac{n}{(k-k_{0,x})^2}, \ \Big|\frac{d}{dx} \EE(c_{k,x} d_k)\Big|\lesssim \frac{n^{1/2}}{(k-k_{0,x})^{3/2}}.\end{equation}
Moreover,   $\theta_k^x = \arcsin(\frac{\sqrt{4-x_k^2}}{2})$  by definition. Using again \eqref{estim}, one can check that $|(d/dx) \theta_{k}^x|\lesssim \sqrt{n/(k-k_{0,x})}$. As $|\theta_k^x|\lesssim \sqrt{(k-k_{0,x})/n}$ by \eqref{diffRI} and $\EE c_{k,x}^2 \lesssim 1/(k-k_{0,x})$, $|\EE(c_{k,x}d_k)|\lesssim (k_{0_x}(k-k_{0,x}))^{-1/2}$, $\EE d_k^2 \lesssim 1/n$, the claim \eqref{claimerivsigma} follows  from \eqref{estimderivcd} and \eqref{defsigmaelliptic}.
 
 Next, summing the estimates \eqref{claimerivsigma}, it follows that 
 \begin{equation} \label{claimerivSigma}\Big|\frac{d}{dx} \Sigma_{k,x}^2\Big| \lesssim\left\{\begin{array}{ll}
 {n}{(k_{0,x}-k)^{-1}}, & \ k\leq k_{0,x}-\ell_0, \\
%  \Big|\frac{d}{dx} \Sigma_{k,x}^2\Big| \lesssim  
n^{2/3},& \ k\geq k_{0,x}+\ell_0.\end{array}\right. \end{equation}
Now, if $t\leq \frac23\uptau-2\log \kappa$, then for any $y\in \mathcal{N}$ we have by definition of $n_{t,y}$ that 
\begin{equation} \label{encadreprecis}  \frac{v}{2} t \leq \Sigma_{n_{t,y}}^2(y)\leq \frac{v}{2} t + \sigma_{n_{t,y}}^2(y) = \frac{v}{2} t + O(n^{-1}e^t),\end{equation}
where we used the fact that $\sigma_{n_{t,y}}^2(y) \lesssim 1/(k_{0,y}-n_{t,y})$ by Propositions \ref{incremhn} and \ref{increm} and that $k_{0,y}-n_{t,y} \asymp n e^{-t}$ by Lemma \ref{approxtimechange}. Besides,   \eqref{claimerivSigma} and the fact that $k_{0,x'}\leq k_{0,x}$ for $|x|\geq |x'|$ yield that $|\Sigma_{n_{t,x'}}^2(x) - \Sigma_{n_{t,x'}}^2(x')|\lesssim |x-x'| n/(k_{0,x'}-n_{t,x'})$. Using again Lemma \ref{approxtimechange}, we get that $|\Sigma_{n_{t,x'}}^2(x) - \Sigma_{n_{t,x'}}^2(x')|\lesssim |x-x'| e^t$. Together with \eqref{encadreprecis} and the fact that $|x-x'|\gtrsim n^{-1}$ if $x\neq x'$ this ends the proof of the claim  in the case where $t\leq \frac23 \uptau-2\log \kappa$.

Assume now that $t\geq \frac23 \uptau+2\log \kappa$. Using the same argument as in the case where $t\leq \frac23
\uptau-2\log \kappa$, we get that by definition of the time-change that $|\Sigma_{n_{t,x}}(x)^2-\Sigma_{n_{t,x'}}(x')^2|\lesssim ne^{-2t}$, where we used that $n_{t,y} - k_{0,y} \asymp n^{-1}e^{2t}$ for any $y$ by Lemma \ref{approxtimechange}. Together with \eqref{claimerivSigma} and the fact that $|x-x'|\gtrsim n^{-1}$, we get that $\Sigma_{n_{t,x'}}^2(x) - \Sigma_{n_{t,x'}}^2(x')|\lesssim |x-x'| n^{2/3}$.

It remains to prove the last two claims $(i)$ and $(ii)$. Note first that $\alpha_{\ell,z}$ is an increasing function of $|z|$ for any $\ell \leq k_{0,z}$. Hence, by \eqref{expressvarg} it follows that $\sigma_{\ell}^2(x) \leq \sigma_\ell^2(x')$ when $|x|\leq |x'|$ and $\ell \leq k_{0,x}-\ell_0$. Thus, $\Sigma_k^2(x) \leq \Sigma_k^2(x')$ as soon as $|x|\leq |x'|$ and $k \leq k_{0,x}-\ell_0$, which by definition of the time-change \eqref{deftimechange} entails  that $n_{t,x} \leq n_{t,x'}$ for $t\leq \frac23 \uptau-2\log \kappa$. Finally, assume $t\geq \frac23
\uptau+2\log \kappa$. Using \eqref{claimerivSigma} we deduce that $|\Sigma_{n,x}^2-\Sigma_{n,x'}^2|\lesssim n^{2/3} e^{-t} \lesssim \kappa^{-2}$. Since $T_x= \lceil \frac{2}{v}\Sigma_{n}^2(x)\rceil$ and $T_{x'}= \lceil \frac{2}{v}\Sigma_{n}^2(x')\rceil$, it follows that, for $\kappa$ large enough, $T_x=T_{x'}$.
\end{proof}

We are now ready to give a proof of Lemma \ref{jointprobamax}.

\begin{proof}[Proof of Lemma \ref{jointprobamax}] Since  for any $z$, $\Psi^z$ and $\widehat{\Psi}^z$ are within a constant from each other on the event $\mathscr{H}_\mathfrak{q}$, it suffices to prove the claimed estimates for $\widehat{\Psi}$ instead of $\Psi$. 
Moreover, without loss of generality, we can assume that $z\geq 0$. 
We now perform a chaining argument based on the joint exponential moment of Lemma \ref{jointexpomomenh}. Set $j_0 =\lfloor \log \log n\rfloor$. Let $\mathcal{D}_0= \{z\}$ and for any $1\leq j \leq j_0$, let $\mathcal{D}_j \subset[z,z+e^{-t}]\cap \mathcal{N}$ be a $e^{-t-j}$ net   that is $e^{-t-j}$-separated and such that $\#\mathcal{D}_j \lesssim e^j$. Let $z '\in \mathcal{N}$. One can find a sequence $z'_0,\ldots,z'_{j_0}$ such that $z'_j \in \mathcal{D}_j$ for any $j\leq j_0$ and $|z'_j -z'_{j-1}|\leq e^{-t - j+2}$ for any $j\geq 1$. Hence, 
\begin{align*}
 |\widehat{\Psi}_t^{z'} - \widehat{\Psi}_t^z|& \leq | \widehat{\Psi}_t^{z'}  - \widehat{\Psi}_t^{z'_{j_0}} | +  \sum_{j=1}^{j_0}\big| \widehat{\Psi}_t^{z'_{j}}-\widehat{\Psi}_t^{z'_{j-1}}\big|\\
& \leq \max_{(x,x') \in H} \big| \widehat{\Psi}_t^{x'}-\widehat{\Psi}_t^{x}\big|+  \sum_{j=1}^{j_0} \max_{(x,x') \in H_j} \big|\widehat{\Psi}_t^{x'}-\widehat{\Psi}_t^{x}\big|,  
\end{align*}
where $H := \{(x,x') \in \mathcal{N}^2 : |x-x'|\leq e^{-t}(\log n)^{-1}\}$ and $H_j := \{(x,x') \in \mathcal{D}_j \times \mathcal{D}_{j-1}  : |x-x'|\leq e^{-t-j+2}\}$. Using a union bound, we find that 
\begin{align} &\PP\Big(\widehat{\Psi}_t^z \geq \sqrt{v} t - r, \max_{z'\in [z,z+e^{-t}] \cap \mathcal{N}  } \big(\widehat{\Psi}_t^{z'} - \widehat{\Psi}_t^{z}\big) \geq s\big)  \leq \sum_{(x,x') \in H} 
\PP\Big( \widehat{\Psi}_t^z \geq \sqrt{v} t - r, \widehat{\Psi}_t^{x'} - \widehat{\Psi}_t^{x}\geq s/2\Big)\nonumber \\
&\qquad \qquad  \qquad\qquad + 
\sum_{j=1}^{j_0} \sum_{(x,x') \in H_j}  \PP\Big( \widehat{\Psi}_t^z \geq \sqrt{v} t - r, \widehat{\Psi}_t^{x'} - \widehat{\Psi}_t^{x}\geq s j^{-2}/6\Big),\label{probajoint}
\end{align}
where we used the fact that $\sum_{j\geq 1}j^{-2} = \pi^2/6\leq 3$. Let $x,x' \in \mathcal{N}$ such that $|x-x'|\leq e^{-t-j+2}$ for some $j\leq j_0$. Without loss of generality, we can assume that $x' \geq x$.  Then, $n_{t,z} \leq n_{t,x} \leq n_{t,x'}$ by Lemma \ref{varacctimes}. Therefore, for any  $\mu \in \RR$, 
\[ \EE\Big[\Car_{\mathscr{H}_\mathfrak{q}}e^{\frac{2}{\sqrt{v}}\widehat{\Psi}_{t}^z 
 + \mu (\widehat{\Psi}_{t}^{x'} - \widehat{\Psi}_{t}^x) } \Big] \leq \EE\Big[\Car_{\mathscr{H}_\mathfrak{q}^{(1)}} \EE_{n_{t,x}}\Big[\Car_{\mathscr{H}_\mathfrak{q}^{(2)}}e^{\mu (\widehat{\psi}_{n_{t,x'}}(x') - \widehat{\psi}_{n_{t,x}}(x'))}\Big] e^{ \mu (\widehat{\psi}_{n_{t,x}}(x') - \widehat{\psi}_{n_{t,x}}(x))+\frac{2}{\sqrt{v}} \widehat{\psi}_{n_{t,z}}(z)}\Big], \]
 where $\mathscr{H}_\mathfrak{q}^{(1)} := \{\forall y \in \mathcal{N}, \forall k_{\delta,z} \leq k <  n_{t,x}, W_k^y \leq \eta_{k,y}\}$ and $\mathscr{H}_\mathfrak{q}^{(2)} := \{\forall y \in \mathcal{N}, \forall n_{t,x} \leq k <  n_{t,x'}, W_k^y \leq \eta_{k,y}\}$. Note that by \eqref{condnoiseerror1} - \eqref{condnoiseerror} we know that on the event $\mathscr{H}_\mathfrak{q}^{(2)}$, 
 \[ \widehat{\psi}_{n_{t,x'}}(x') - \widehat{\psi}_{n_{t,x}}(x') = \sum_{\ell = n_{t,x}+1}^{n_{t,x'}} H_\ell,\]
 where $\EE_{\ell-1} H_\ell=0$ and $s_{\ell-1}(H_\ell)\lesssim 1/(k_{0,x'}-\ell) \lesssim \sigma_\ell^2(x')$. Using Lemma \ref{tailpropH} - \eqref{subgauss}, we get that there exist $\mathfrak{h},\mathfrak{C}>0$ depending on the model parameters such that for $|\mu|\leq n^{1/6} (\log n)^{-\mathfrak{h}}$, 
 \begin{equation} \label{9expomomentcont}   \EE_{n_{t,x}}\Big[\Car_{\mathscr{H}_\mathfrak{q}^{(2)}}e^{\mu (\widehat{\psi}_{n_{t,x'}}(x') - \widehat{\psi}_{n_{t,x}}(x'))}\Big] \leq e^{\mathfrak{C} \mu^2 (\Sigma_{n_{t,x'}}^2(x') - \Sigma_{n_{t,x}}^2(x))}  \leq e^{\mathfrak{C}' \mu^2 e^{-j}},\end{equation}
 where $\mathfrak{C}'>0$ depends on the model parameters and we used Lemma \ref{varacctimes} in the last inequality.  Using the fact that $\widehat{\psi}(z)$ and ${\psi}(z)$ are within a constant from each other on the event $\mathscr{H}_\mathfrak{q}^{(1)}$, we get that the upper bound of Proposition \ref{expopreciseincre} holds true for the exponential moments of $\widehat{\psi}(z)$  when restricted to $\mathscr{H}_\mathfrak{q}^{(1)}$. Together with  Lemma \ref{jointexpomomenh}, and \eqref{9expomomentcont} we find that 
 \begin{equation} \label{9expomomentjoint} \EE\Big[\Car_{\mathscr{H}_\mathfrak{q}}e^{\frac{2}{\sqrt{v}}\widehat{\Psi}_{t}^z
 + \mu (\widehat{\Psi}_{t}^{x'} - \widehat{\Psi}_{t}^x) } \Big] \leq  e^{\mathfrak{C}' \mu^2 |x-x'|^2 e^{2t}  + t + \mathfrak{C}' + \mathfrak{C}' (|\mu|^3\vee \mu^2) (\log n)^{-\mathfrak{q}/4}},\end{equation}
 for any $|\mu|\leq n^{1/6} (\log n)^{-\mathfrak{h}}$, where $\mathfrak{C}'$ and $\mathfrak{h}$ changed values without changing names. Since $\exp(\mu |\widehat{\Psi}_{t}^{x'} - \widehat{\Psi}_{t}^x|)\leq \exp(\mu (\widehat{\Psi}_{t}^{x'} - \widehat{\Psi}_{t}^x))+\exp(-\mu (\widehat{\Psi}_{t}^{x'} - \widehat{\Psi}_{t}^x))$, at the price of modifying $\mathfrak{C}'$, we deduce that the  estimate \eqref{9expomomentjoint} holds as well with $|\widehat{\Psi}_{t}^{x'} - \widehat{\Psi}_{t}^x|$ instead of $\widehat{\Psi}_{t}^{x'} - \widehat{\Psi}_{t}^x$.

Now, choose $\mu = s e^{j} j^{-2}/(12\mathfrak{C}')$. Using Chernoff inequality with parameters $\lambda = 2/\sqrt{v}$ and $\mu$ we deduce from \eqref{9expomomentjoint} that for $\mathfrak{q}$ large enough
 \begin{align*}
 \log \PP_{\mathscr{H}_\mathfrak{q}}\Big( \widehat{\Psi}_t^z \geq \sqrt{v} t - r, |\widehat{\Psi}_t^{x'} - \widehat{\Psi}_t^{x}|\geq sj^{-2}/6\Big) & \leq -t  -\frac{s^2j^{-4}e^{j}}{144\mathfrak{C}'}  +2 \mathfrak{C}' -\frac{2}{\sqrt{v}} r.
 \end{align*}
Next, if $x,x' \in \mathcal{N}$ are such that $|x-x'|\leq e^{-t} (\log n)^{-1}$, by the same argument we find that for $\mathfrak{q}$ large enough
\[  \log \PP_{\mathscr{H}_\mathfrak{q}}\Big( \widehat{\Psi}_t^z \geq \sqrt{v} t - r, |\widehat{\Psi}_t^{x'} - \widehat{\Psi}_t^{x}|\geq s/2\Big)  \leq -t  -\frac{s^2 (\log n)^2}{16\mathfrak{C}'}  +2 \mathfrak{C}' -\frac{2}{\sqrt{v}} r.\]
Summing these inequalities, using that $\# H_j \lesssim e^j$ and $\# H \lesssim n^2$, we obtain that 
\begin{align*}  \PP_{\mathscr{H}_\mathfrak{q}}\Big(\widehat{\Psi}_t^z \geq \sqrt{v} t - r, \max_{z' \in [z,z+e^{-t}]\cap \mathcal{N}} \big|\widehat{\Psi}_t^{z'} - \widehat{\Psi}_t^{z}\big| \geq s\big) & \lesssim e^{-t - \frac{2}{\sqrt{v}}r } \Big(n^2 e^{- \frac{s^2(\log n)^2 }{16\mathfrak{C}'} } +  \sum_{j\geq 1}e^j e^{- \frac{s^2j^{-4} e^{2j}}{144 \mathfrak{C}'}}\Big)\\
&\lesssim e^{-t - \frac{2}{\sqrt{v}}r-\mathfrak{c} s^2},
\end{align*}
where $\mathfrak{c}>0$ depends on the model parameters. 
 \end{proof}

\subsection{Barrier estimate in the elliptic regime} \label{barrierelliptic}
We turn our attention to the proof of Lemma \ref{barrierloose1} in the case where $t\geq t_\mathfrak{q}^+$, meaning that the crossing of the barrier happens in the elliptic regime.  
The proof relies on a continuity estimate with a similar flavor as the one proved in Lemma \ref{jointprobamax} for the hyperbolic regime except that we won't prove a continuity estimate on the field $\Psi^z$ itself. Rather, we define for any $t\geq t_\mathfrak{q}^+$ some blocks $(t_i)_{1\leq i \leq N}$ partitionning the interval $[t_{\sqrt{\mathfrak{q}}}^-, t]$ and show that on each of the  blocks the products of transition matrices satisfy a certain continuity property.  
More precisely, fix some large constant $\mathfrak{f}>0$ and a small constant $\mathfrak{p}>0$ to be chosen later. Set 
%\corF{Denote by 
$s_{\mathfrak{q}} = \lfloor \frac23 \uptau -\sqrt{\mathfrak{q} }\log \uptau\rfloor$ and,  for any $t \geq t_\mathfrak{q}^+$,
 divide the interval $[s_{\mathfrak{q}},t]$ into $(1+\mathfrak{p})$-adic subintervals as follows. Define $(q_i)_{i\geq 1}$ as 
\[ q_i:=t- \lfloor \mathfrak{f}(1+\mathfrak{p})^{i-1} \log \uptau\rfloor,  \  i\geq 1.\]
Let $N=N_t := \min\{i : q_i \leq s_{\mathfrak{q}}\}$ and set 
\begin{equation} \label{defti} t_0 := t_{\sqrt{\mathfrak{q}}}^-,  \  t_1 := t_{\sqrt{\mathfrak{q}}}^+, \ t_{N} := t,  \  t_i := q_{N-i}, \  1\leq i \leq N-1.\end{equation}
Since $s_{\mathfrak{q}} = \lfloor \frac23 \uptau -\sqrt{\mathfrak{q} }\log \uptau\rfloor$  we 
have that $N \lesssim \log \uptau$ for $t = O(\uptau)$.
Note that even though  not retained in the notation, $N$ and the blocks $t_i$'s depend on $t$. The motivation behind this choice of blocks comes from the constraints that $t-t_i \gg \log \corF{\uptau}$ for $i<N$ and that $(t-t_i) \gg \Delta t_i$, which will be crucial to get good continuity estimates on each blocks. Moreover, we made the choice of partitioning the interval $[s_{\mathfrak{q}},t]$ instead of $[t_\mathfrak{q}^-,t]$ to have some extra margin so that $t-t_1\gg \Delta t_1$ for any $t\geq t_\mathfrak{q}^+$ and large enough $\mathfrak{q}$. 

 To lighten the notation we write for any $z\in I_\eta$ and $t \in \llbracket 1, T_z\rrbracket$, $\widetilde{M}_t (z)= M_{n_{t,z}}(z)$. With this notation, we have the following result. 
\begin{lemma}[Continuity estimates]
\label{contmax} Let $z \in \mathcal{N}$ and $t \in \llbracket t_\mathfrak{q}^+,T_z\rrbracket$. For any $i\in \llbracket 1,N\rrbracket$, define 
\[ \Delta_{i,z} := \sup_{|z'-z| \leq e^{-t} \atop z'\in \mathcal{N}} \big\| \Xi_{n_{t_{i-1},z},n_{t_i,z}}^{z'} - \Xi_{n_{t_{i-1},z},n_{t_i,z}}^{z} \big\|e^{-\Delta \widetilde{M}_{t_i}(z) -\sqrt{v}\Delta t_i},\]
where $\Delta t_i=t_i-t_{i-1}$ and $\Delta \widetilde M_{t_i}(z)=\widetilde M_{t_i}(z)-\widetilde M_{t_{i-1}}(z)$. Then
 there exist $\mathfrak{c},\mathfrak{C}>0$ depending on the model parameters such that  for $\mathfrak{f}$ large enough and $\mathfrak{p}$ small enough, for any $1\leq \zeta \leq \mathfrak{c}/\mathfrak{p}$, and $1\leq i \leq N-1$, 
\[ \EE_{n_{t_{i-1}}}[ \Delta_{i,z}^\zeta] \leq   e^{-\frac{\zeta}{3} (t-t_i)}, \ \EE_{n_{t_{N-1}}}[ \Delta_{N,z}^\zeta] \leq   e^{\mathfrak{C}\mathfrak{f} \zeta^2 \log \uptau}.\]
\end{lemma}

Before going into the proof of Lemma \ref{contmax}, we need a technical lemma estimating the distance between the change of basis matrices for given $x,x' \in I_\eta$.  

\begin{lemma}\label{contchangebase}
Let $x,x'\in I_\eta$  such that  $|x-x'| = o(n^{-2/3})$. Let $k_0^* := k_{0,x}\vee k_{0,x'}$. For any $k\geq k_{0}^*-\ell_0$, 
\[ \|\mathrm{I}_2 - (P_k^{x'})^{-1}P_k^x\| \lesssim |x-x'|\frac{n}{|k-k_{0}^*|\vee n^{1/3}}.\]
\end{lemma}

\begin{proof} We only prove the statement in the case where $|x'|\leq |x|$ so that $k_{0,x'} \leq k_{0,x}$. To lighten the notation, write $k_0$ (respectively $k_0'$), instead of $k_{0,x}$ (respectively $k_{0,x'}$). Denote by $R_k := \mathrm{I}_2 - (P_k^{x'})^{-1}P_k^x$.
Assume first that $k\geq k_{0,x} +\ell_0$. Then, 
\[ R_k  = \begin{pmatrix} 1-\sqrt{\frac{4-x_k^2}{4-(x_k')^2}} & \frac{x_k-x_k'}{\sqrt{4-(x_k')^2}} \\ 0 & 0 \end{pmatrix}.\]
Since $|x_k-x_k'|\lesssim |x-x'|$ and $(4-(x_k')^2)^{-1/2} \lesssim \sqrt{n/(k-k_{0,x'})} \leq \sqrt{n/(k-k_{0,x'})}$ by \eqref{estim}, we deduce that $|R_k(1,2)|\lesssim |x-x'| \sqrt{n/(k-k_{0,x})}$. Moreover, we get that $|(x_k')^2-x_k^2|/(4-(x'_k)^2)\lesssim |x-x'| n/(k-k_{0,x})\leq |x-x'| n^{2/3}$. Since $|x-x'|n^{2/3}\ll 1$ by assumption, we deduce that $|R_k(1,1)|\lesssim |x-x'| n/(k-k_{0,x})$.

Now, assume $k_{0,x}-\ell_0 \leq k < k_{0,x}+\ell_0$. Write $k_0$, respectively $k_0'$, instead of $k_{0,x}$, respectively $k_{0,x'}$.  Since $|x-x'|\ll n^{-2/3}$ we have that $k_{0}-k_{0}' \lesssim |x-x'| n \ll n^{1/3}$. There are two cases: either $k\in [k_{0}-\ell_0, k_{0}'+\ell_0[$ or  $k\notin [k_{0}-\ell_0, k_{0}'+\ell_0[$, which gives sligthly different expressions of $R_k$. We only consider the case  where $k_{0}-\ell_0\leq k < k_{0}'+\ell_0$, the second one being similar.
In this case, 
\[ R_k = \begin{pmatrix} 1-\sqrt{\frac{4-x_{k_0-\ell_0}^2}{4-(x_{k'_0-\ell_0}')^2}} & \frac{x_{k_0-\ell_0}-x_{k'_0-\ell_0}'}{\sqrt{4-(x_{k'_0-\ell_0}')^2}} \\ 0 & 0 \end{pmatrix}.\]
Then, one can check that $|x_{k_0-\ell_0} - x_{k_0'-\ell_0}|\lesssim (k_0-k_0')/n \lesssim |x-x'|$. Hence, $|x_{k_0-\ell_0}-x'_{k_0'-\ell_0}|\lesssim |x-x'|$ which together with \eqref{estim} yields that  $|R_k(1,2)| \lesssim |x-x'| n^{1/3}$. 
Next, since $|(x'_{k_0'-\ell_0})^2 -x_{k_0-\ell_0}^2| \lesssim x_{k_0'-\ell_0} - x_{k_0-\ell_0}\lesssim |x-x'|$, $4-(x'_{k_0'-\ell_0})^2\gtrsim n^{2/3}$ and $|x-x'|\ll n^{-2/3}$ we get that $|R_k(1,1)|\lesssim |x-x'|n^{2/3}$, thus proving the claim. 
% Finally, it remains to consider the case where $k_0'-\ell_0\leq k \leq k_0-\ell_0$ or $k_0'+\ell_0 \leq k \leq k_0+\ell_0$. We only look at the case where $k_0'-\ell_0\leq k \leq k_0-\ell_0$, the second one being similar. Then, 
% \[ R_k =   \begin{pmatrix} 1-\sqrt{\frac{4-x_k^2}{4-(x_{k'_0-\ell_0}')^2}} & \frac{x_k-x_{k'_0-\ell_0}'}{\sqrt{4-(x_{k'_0-\ell_0}')^2}} \\ 0 & 0 \end{pmatrix}.\]
%Using that $k-(k'_0-\ell_0)\leq k_0-k_0' \lesssim |x-x'|n$, we find that $x_{k_0'-\ell_0} - x_k \lesssim |x-x'|$.  Together with \eqref{estim} and the fact that $|x_{k_0'-\ell_0}-x'_{k_0'-\ell_0}|\lesssim |x-x'|$, it follows that  $|R_k(1,2)|\lesssim |x-x'| n^{1/3}$ and $|R_k(1,1)|\lesssim |x-x'|n^{2/3}$.
\end{proof}

Equipped with Lemma \ref{contchangebase}, we are now ready to give a proof of Lemma \ref{contmax}.

\begin{proof}[Proof of Lemma \ref{contmax}] 
Let $i \in \llbracket 1, N\rrbracket$ and $x,x' \in \mathcal{N}$ such that  $|x-z|, |x'-z|\leq e^{-t}$. Note that as $t\geq t_\mathfrak{q}^+$, we have that $|x-x'|\ll n^{-2/3}$ and as a result $|k_{0,z}-k_{0,x}|, |k_{0,z}-k_{0,x'}| \ll n^{1/3}$. To lighten notation write $n_{t_i}$ instead of $n_{t_i,z}$. Denote by $\widetilde{\Delta}_{i,x,x'}$ the random variable 
\[ \widetilde{\Delta}_{i,x,x'} := \big\| \Xi_{n_{t_{i-1}},n_{t_i}}^{x'} - \Xi_{n_{t_{i-1}},n_{t_i}}^{x} \big\| .\]
We claim that there exists $\mathfrak{C}>0$ depending on the model parameters such that 
\begin{equation} \label{claimdelttilde}  \widetilde{\Delta}_{i,x,x'} \leq  \mathfrak{C}  |x-x'| \Big(\sum_{\ell = n_{t_{i-1}}+1}^{n_{t_i}} s_\ell \|\Xi_{\ell+1,n_{t_i}}^{x'}\| \|\Xi_{n_{t_{i-1}},\ell-1}^x\| \\
  + s^2_{n_{t_i}}\|\Xi_{n_{t_{i-1}}, n_{t_i}}^x\| + s^2_{n_{t_{i-1}}}\|\Xi_{n_{t_{i-1}}, n_{t_i}}^{x'}\|\Big),
\end{equation}
 where $s_\ell :=  \sqrt{n/|\ell-k_{0,z}|}\wedge n^{1/3}$ for any $\ell$.  We write
 \begin{align} \Xi_{n_{t_{i-1}},n_{t_i}}^{x'} - \Xi_{n_{t_{i-1}},n_{t_i}}^{x} & = (P_{n_{t_i}}^{x'})^{-1}\big( T_{n_{t_{i-1}, n_{t_i}}}^{x'} - T_{n_{t_{i-1}, n_{t_i}}}^{x}\big) P_{n_{t_{i-1}}}^x   + \Xi_{n_{t_{i-1}},n_{t_i}}^{x'} \big(\mathrm{I}_2-(P_{n_{t_{i-1}}}^x)^{-1} P_{n_{t_{i-1}}}^{x'}\big) \nonumber \\
 & + \big(\mathrm{I}_2-(P_{n_{t_{i}}}^{x'})^{-1} P_{n_{t_{i}}}^{x}\big)  \Xi_{n_{t_{i-1}},n_{t_i}}^{x}. \label{decomdiffprod} 
  \end{align}
 Decomposing $T_{n_{t_{i-1}, n_{t_i}}}^{x'} - T_{n_{t_{i-1}, n_{t_i}}}^{x}$ as a telescopic sum, we get that 
 \[  (P_{n_{t_i}}^{x'})^{-1}\big( T_{n_{t_{i-1}, n_{t_i}}}^{x'} - T_{n_{t_{i-1}, n_{t_i}}}^{x}\big) P_{n_{t_{i-1}}}^x = \sum_{\ell = n_{t_{i-1}}+1}^{n_{t_i}}  \Xi_{\ell,n_{t_i}}^{x'} (P_\ell^{x'})^{-1} (T_\ell^{x'} - T_\ell^x) P_{\ell-1}^{x} \Xi_{n_{t_{i-1}}, \ell-1}^x
 \]
Since $\|T_\ell^{x'}-T_\ell^x\|\lesssim |x-x'|$, $\|(P_\ell^{x'})^{-1}\|\lesssim |x-x'|s_\ell$ and $\|P_{\ell-1}^x\|\lesssim 1$ using Lemma \ref{estimQlemma} and the fact that $|k_{0,x'}- k_{0,z}|\ll n^{1/3}$, we get by the triangular inequality that 
 \[ \big\| (P_{n_{t_i}}^{x'})^{-1}\big( T_{n_{t_{i-1}, n_{t_i}}}^{x'} - T_{n_{t_{i-1}, n_{t_i}}}^{x}\big) P_{n_{t_{i-1}}}^x\big\| \lesssim |x-x'|\sum_{\ell=n_{t_{i-1}}+1}^{n_{t_i}} s_\ell \|\Xi_{\ell,n_{t_i}}^{x'}\| \|\Xi_{n_{t_{i-1}},\ell-1}^x\|. \]
Using in addition Lemma \ref{contchangebase} and again the fact that $|k_{0,z}-k_{0,x}|, |k_{0,z}-k_{0,x'}|\ll n^{1/3}$
 to bound the last terms in \eqref{decomdiffprod} we get the claim \eqref{claimdelttilde}.

 Write $\Delta_{i,x,x'} := \widetilde{\Delta}_{i,x,x'} e^{-\Delta \widetilde{M}_{t_i}(z) -\sqrt{v}\Delta t_i}$.  Let $S_i = \sum_{\ell = n_{t_{i-1}}+1}^{n_{t_i}} s_\ell+s_{n_{t_{i-1}}}^2+ s_{n_{t_{i}}}^2$. 
Using convexity, it follows that for $\zeta \geq 1$,
\begin{align} \EE_{n_{t_{i-1}}}\big[\Delta_{i,x,x'}^\zeta\big] &\leq(\mathfrak{C} |x-x'|)^\zeta e^{-\zeta (\Delta \widetilde M_{t_i}(x) +\sqrt{v}\Delta t_i)} S_i^{\zeta-1}\Big( \sum_{\ell=n_{t_{i-1}}+1}^{n_{t_i}} s_\ell \EE_{n_{t_{i-1}}}\big[\|\Xi_{\ell,n_{t_i}}^{x'}\|^\zeta \|\Xi_{n_{t_{i-1}},\ell-1}^x\|^\zeta\big] \nonumber \\
& + s_{n_{t_{i-1}}}^2\EE_{n_{t_{i-1}}}\big[\|\Xi_{\ell+1,n_{t_i}}^{x'}\|^\zeta\big]+s_{n_{t_{i}}}^2\EE_{n_{t_{i-1}}}\big[\|\Xi_{n_{t_{i-1}},n_{t_i}}^{x}\|^\zeta\big]\Big),\label{boundexpecDelta}
\end{align}  
where $\mathfrak{C}>0$ is a constant depending on the model parameters which will change value without changing name. 
Now, denote for any $i\in \llbracket 1,N\rrbracket$ and $\ell \in [n_{t_{i-1}},n_{t_i}]$, by $Q^i_\ell$ the cumulated approximate variance given by the a priori estimate of Proposition \ref{apriori} on the interval  $[n_{t_{i-1}}+1,\ell-1]$ for the $x$-recursion and  on the interval  $[\ell+1, n_{t_{i}} ]$ for the $x'$-recursion, that is 
\[  Q^i_\ell := \sum_{k=n_{t_{i-1}}+1}^{\ell-1}\frac{1}{|k-k_{0,x}|\vee n^{1/3}} + \sum_{k= \ell+1}^{n_{t_i}} \frac{1}{|k-k_{0,x'}|\vee n^{1/3}} \lesssim  \sum_{k=n_{t_{i-1}}+1}^{n_{t_i}} \frac{1}{|k-k_{0,z}|\vee n^{1/3}},\]
where we used the fact that $|k_{0,z} -k_{0,x}|,|k_{0,z}-k_{0,x'}|\ll n^{1/3}$.  Since $\sigma_{k,z}^2\gtrsim 1/(k-k_{0,z})$ for $ k\geq k_{0,z}+\ell_0$ by Proposition   \ref{represelliptic} it follows that if $i\geq 2$  then  $Q^i_\ell \lesssim \Delta t_i$ by definition of the time-change $n_t$ (see \eqref{deftimechange}). Moreover, if $i=1$, then we find that $Q^1_\ell \lesssim  \sqrt{\mathfrak{q}} \log T \lesssim \Delta t_1$  using the fact that $n_{t_1} - k_{0,z} \asymp n^{1/3} (\log n)^{2\sqrt{\mathfrak{q}}}$, $k_{0,z} - n_{t_0} \asymp n^{1/3} (\log n)^{\sqrt{\mathfrak{q}}}$ by Lemma \ref{approxtimechange} as $t_1 = t_{\sqrt{\mathfrak{q}}}^+$ and $t_0=t_{\sqrt{\mathfrak{q}}}^-$. Hence, for any $i\in \llbracket 1,N\rrbracket$ and $\ell \in [n_{t_{i-1}},n_{t_i}]$, $Q_\ell^i \lesssim \Delta t_i$. We deduce by the a priori estimate of Proposition \ref{aprioridom} and \eqref{boundexpecDelta} that 
\[ \EE_{n_{t_{i-1}}}\big[\Delta_{i,x,x'}^\zeta\big] \leq(\mathfrak{C} |x-x'|)^\zeta  S_i^\zeta e^{\mathfrak{C} \zeta^2 \Delta t_i- \zeta(
\Delta \widetilde M_{t_i} + \sqrt{v} \Delta t_i)},\] 
for any $1\leq \zeta \leq (\log n)^{-\mathfrak{h}}n^{1/6}$, where $\mathfrak{h}>0$ depends on the model parameters.
But, for any $i\geq 2$,
\begin{align}
 S_i& \lesssim \sum_{\ell= n_{t_{i-1}}+1}^{n_{t_i}} \sqrt{\frac{n}{k-k_0}} + \frac{n}{n_{t_i}-k_{0,z}} +\frac{n}{n_{t_{i-1}}-k_{0,z}} \lesssim  \sqrt{n (n_{t_{i}}-k_0) }\lesssim e^{t_i},\label{estimSi}
 \end{align}
 where we used in the last inequality  Lemma \ref{approxtimechange}.   Moreover,  one can check that $S_1 \lesssim T^{\sqrt{\mathfrak{q}}} n^{-2/3}  \lesssim   e^{t_1}$.
Thus, for $1\leq \zeta \leq (\log n)^{-\mathfrak{h}}n^{1/6}$,
\[ \EE_{n_{t_{i-1}}}\big[\Delta_{i,x,x'}^\zeta\big]\leq (\mathfrak{C} |x-x'|e^t)^\zeta  e^{-\zeta [(t-t_i) +\Delta M_{t_i}+\sqrt{v} \Delta t_i] + \mathfrak{C} \zeta^2 \Delta t_i}, \ 1\leq i \leq N.\]
Note that  since $1/(k-k_{0,z}) \lesssim \sigma_{k,z}^2$ for any $k\geq k_{0,z}+\ell_0$ by Proposition \ref{represelliptic}, we have by definition of the time-change that $|\Delta \widetilde M_{t_i}(z)| \lesssim \Delta t_i$ for any $i\geq 2$. Again, using Lemma \ref{approxtimechange}, one can check that  we also have that $|\Delta \widetilde M_{t_1}(z)| \lesssim \Delta t_1$.  Moreover, $(t-t_i)/\Delta t_i \gtrsim \mathfrak{p}^{-1}$ for $2\leq i\leq N-1$ whereas $(t-t_1)/\Delta t_1 \gtrsim \sqrt{\mathfrak{q}}$ since $t\geq t_\mathfrak{q}^+$. Thus, there exists  $\mathfrak{c}>0$ depending on the model parameters such that  for $\mathfrak{p}$ small enough and $\mathfrak{q}, \mathfrak{f}$ large enough, for any  $1\leq i\leq N-1$ and $1\leq \zeta \leq \mathfrak{c}(\mathfrak{p}^{-1}\wedge \sqrt{\mathfrak{q}})$, 
\begin{equation} \label{momentDelta} \EE_{n_{t_{i-1}}}\big[\Delta_{i,x,x'}^\zeta\big]\leq (\mathfrak{C}|x-x'|e^t)^\zeta e^{-\frac12 \zeta (t-t_i)}.\end{equation}
Moreover, as $\Delta t_N \lesssim \mathfrak{f} \log T$, we deduce that for any $1\leq \zeta \leq (\log n)^{-\mathfrak{h}}n^{1/6}$, 
\begin{equation} \label{momentDelta2} \EE_{n_{t_{i-1}}}\big[\Delta_{N,x,x'}^\zeta\big]\leq (\mathfrak{C} |z-z'|e^t)^\zeta e^{\mathfrak{C}\mathfrak{f}\zeta^2\log T}.\end{equation}
 We are now ready to perform our chaining argument to derive the claimed moment estimates of $\Delta_{i,z}$. We only give a proof of the moment estimate of $\Delta_{i,z}$ for $i\leq N-1$, the one for $\Delta_{N,z}$ being similar. 
 Let $m$ be the smallest integer such that $\min_{x,x'\in \mathcal{N}, x\neq x'} |x-x'|> e^{-t-m}$. Since $\min_{x,x'\in \mathcal{N}, x\neq x'} |x-x'|\gtrsim n^{-1}$, it follows that $m\lesssim T$. Set $\mathcal{D}_0=\{z\}$ and let $\mathcal{D}_j$ denote a $e^{-(t+j)}$-net of  $\mathcal{N} \cap [z- e^{-t},z+ e^{-t}]$ that is  $e^{-(t+j)}$-separated and such that  $\#\mathcal{D}_j \lesssim e^j$ for any $1\leq j\leq m$. Note that $\mathcal{D}_m=\mathcal{N}\cap [z-e^{-t},z+e^{-t}]$.  For any $z'\in \mathcal{N}\cap [z-e^{-t},z+e^{-t}]$, there exists a sequence $z=z'_0,\ldots,z'_m=z'$ such that $z'_j \in \mathcal{D}_j$ for any $j$ and $|z'_{j} - z'_{j-1}|\leq e^{-t-j}$ for any $1\leq j \leq m$.  We can write, 
\[ \Delta_{i,z,z'}   \leq \sum_{j=1}^{m} \sup_{|z'-z|\leq e^{-t} \atop z' \in \mathcal{N}} \Delta_{i,z'_{j-1},z'_j}.\]
Therefore, using convexity we deduce that for any $\zeta \geq 1$, 
\[  \Delta_{i,z}^\zeta \leq m^{\zeta-1} \sum_{j=1}^{m} \sup_{|z'-z|\leq e^{-t} \atop z' \in \mathcal{N}} \Delta_{i,z'_{j-1},z'_j}^\zeta. \]
Taking expectation and using a union bound, we find that 
\begin{align*} \EE_{n_{t_{i-1}}}\big[\Delta_{i,z}^\zeta\big]&   \leq m^{\zeta-1}  \sum_{j=1}^m \sum_{(x,x') \in H_j} \EE_{n_{t_{i-1}}} \big[\Delta_{i,x,x'}^\zeta\big],
\end{align*}
where $H_j := \big\{ (x,x') \in \mathcal{D}_{j-1}\times \mathcal{D}_j : |x-x'|\leq e^{-t-j+2}\big\}$ for any $j\leq j_0$ and $H := \big\{ (x,x') \in \mathcal{N}^2   : |x-x'|\leq T^{-1}\big\}$. Since $\mathcal{D}_j$ is $e^{-(t+j)}$-separated and $\# \mathcal{D}_{j-1}\lesssim e^j$, we get that  $\# H_j \lesssim e^j$ for any $j\leq m$. Using \eqref{momentDelta}, we get that for  any $1\leq \zeta \leq \mathfrak{c}(\mathfrak{p}^{-1}\wedge \sqrt{\mathfrak{q}})$, 
\[ \EE_{n_{t_{i-1}}}\big[\Delta_{i,z}^\zeta\big] \leq \mathfrak{C}^\zeta m^{\zeta-1} e^{-\frac{\zeta}{2} (t-t_i)} \sum_{j=1}^m e^{j} e^{-\zeta j}  \leq (\mathfrak{C}m)^\zeta e^{-\frac{\zeta}{2}(t-t_i)}.\]
Since $t-t_i \gtrsim \mathfrak{f} \log T$ for $i\leq N-1$ and $m\lesssim T$, choosing $\mathfrak{f}$ large enough gives the claim. 
\end{proof}

The last piece needed for the proof of  Lemma \ref{barrierloose1} is a continuity in time estimate ensuring  that no significant error on the field is made by approximating $n_{s,z'}$ and $n_{s,z}$ for $s\in \{t,t_{\sqrt{\mathfrak{q}}}^-\}$  when $|z-z'|\leq e^{-t}$ and $t\geq t_\mathfrak{q}^+$. With a slight abuse of notation, we write in the following $\Xi_{k,\ell}^x$ for $x\in I_\eta$ and $k,\ell \leq n$ to denote $\Xi_{k,\ell}^x$ if $k\leq \ell$ and $\Xi_{\ell,k}^x$ if $\ell \leq k$ and we extend the definition of $n_{t,x}$ to any $t\geq 1$ by setting $n_{t,x} := n_{t\wedge T_x,x}$.

\begin{lemma}Let $t\geq t_\mathfrak{q}^+$ and set $s_\mathfrak{q} := t_{\sqrt{\mathfrak{q}}}^-$. 
\label{conttime}Define 
\[ R_t  := \sup_{|z-z'|\leq e^{-t}  } \log \|\Xi^{z'}_{n_{t,z}, n_{t,z'}}\|, \quad Q_t := \sup_{|z-z'|\leq e^{-t}  } -\log\|\Xi^{z'}_{n_{s_\mathfrak{q},z}, n_{s_\mathfrak{q},z'}}\|.\]
Then, for $\mathfrak{q}$ large enough,
\[ \PP( R_t \geq 4) \vee \PP(Q_t \geq 4) \leq e^{-\mathfrak{c}(\log n)^2},\]
where $\mathfrak{c}>0$ depends on the model parameters. 
\end{lemma}

\begin{proof}
Let $z' \in \mathcal{N}$ such that $|z'-z|\leq e^{-t}$. We claim that there exists $\mathfrak{C},\mathfrak{h}>0$ depending on the model parameters such that for any $1\leq \lambda \leq (\log n)^{-\mathfrak{h}} n^{1/6}$, 
\begin{equation} \label{momentconttime}  \EE\big[ \big\| \Xi^{z'}_{n_{t,z},n_{t,z'}}\big\|^\lambda\big]  \leq e^{\mathfrak{C} (\log n)^{-2}\lambda^2+2\lambda}.\end{equation}
Note that since $t\geq t_\mathfrak{q}^+$, $|k_{0,z}-k_{0,z'}|\lesssim e^{-t} n \ll n^{1/3}$. Togther with Lemma \ref{approxtimechange} this yields that  $n_{t,z},n_{t,z'} \geq k_{0,z'} +\ell_0$. Since $\sigma_{k,z'}\gtrsim 1/(k-k_{0,z'})$ by Proposition \ref{represelliptic} for any $k\geq k_{0,z'}+\ell_0$, we get by Lemma \ref{varacctimes} that
\begin{equation} \label{boundvaracc} K_t:=\sum_{k = n_{t,z}\wedge n_{t,z'}+1}^{n_{t,z}\vee n_{t,z'}} \frac{1}{k-k_{0,z'}} \lesssim|\Sigma^2_{n_{t,z'}}(z') - \Sigma^2_{n_{t,z}}(z')| \lesssim n^{2/3}e^{-t}.\end{equation}
Since $t\geq t_\mathfrak{q}^+$, it follows that for $\mathfrak{q}$ large enough,  $K_t\leq (\log n)^{-2}$.  The claim \eqref{momentconttime} then follows by using the a priori estimate of Lemma \ref{aprioridom}.
Using a union bound,  the fact that $\#\mathcal{N}\lesssim n$, and Chernoff inequality with $\lambda =  (\log n)^2/\mathfrak{C}$, we obtain that
\begin{equation} \label{probacontime} \PP\big(R_t \geq 4 \big) \leq n^2 e^{-\frac{(\log n)^2}{\mathfrak{C}}},\end{equation}
which proves the first part of the claim. For the estimate on the tail probability of $Q_t$, we claim again that for $z'\in \mathcal{N}$ such that $|z-z'|\leq e^{-t}$,
\begin{equation}   \EE\big[ \big\| \Xi^{z'}_{n_{t,z},n_{t,z'}}\big\|^\lambda\big]  \leq e^{\mathfrak{C} (\log n)^{-2}\lambda^2 +2\lambda}.\end{equation}
for any $1\leq \lambda \leq (\log n)^{-\mathfrak{h} }n^{1/6}$, where $\mathfrak{h},\mathfrak{C}$ changed values without changing names.  Indeed, using the same argument as for \eqref{boundvaracc}, we find that 
\[ L_t := \sum_{k= n_{s_\mathfrak{q},z}\wedge n_{s_\mathfrak{q},z'}+1}^{n_{s_\mathfrak{q},z}\vee n_{s_\mathfrak{q},z'}} \frac{1}{k_{0,z'} - k} \lesssim| \Sigma^2_{n_{s_\mathfrak{q},z'}}(z') - \Sigma^2_{n_{s_\mathfrak{q},z}}(z')| \lesssim e^{-t}n^{2/3} (\log n)^{-\sqrt{\mathfrak{q}}}.\]
Since $t\geq t_\mathfrak{q}^+$, we deduce that for $\mathfrak{q}$ large enough, $L_t\leq (\log n)^{-2}$. With a similar argument as in the proof of \eqref{probacontime}, we get the second claim on the tail distribution of $Q_t$.
 \end{proof}

Equipped with the results of Lemmas \ref{conttime} and \ref{contmax}, we can now give a proof of Lemma \ref{barrierloose1} in the elliptic case.

\begin{proof}[Proof of Lemma \ref{barrierloose1} - the case $t\geq t_\mathfrak{q}^+$]
Let $\mathscr{E} := \{\exists z' \in \mathcal{N}, t\in [t_\mathfrak{q}^+, T_{z'}] , \Psi_t^{z'}\geq \sqrt{v} t + 4\mathfrak{C}\log \uptau\}$, where  $\mathfrak{C} >0$ is some constant depending on $\kappa$ to be chosen later.  Set $s_\mathfrak{q} := t_{\sqrt{\mathfrak{q}}}^-$. 
For any $z \in I_\eta$, $t\in [ t_\mathfrak{q}^+,T_z]$, $\chi \in \RR$, define the event
\[ \mathscr{E}_{z,t}=\mathscr{E}_{z,t}^\chi := \big\{ \exists z' \in \mathcal{N}, |z-z'|\leq e^{-t}, |z|\leq |z'|, \psi_{n_{t,z}}(z') - \psi_{n_{s_\mathfrak{q},z}}(z')> \sqrt{v} (t-s_{\mathfrak{q}}) + 2\mathfrak{C} \log \uptau +\chi \big\}.\]
We claim that it suffices to show that there exists $\mathfrak{C}>0$ large enough depending on $\kappa$ such that for $\mathfrak{q}$ large enough and for any $z\in \mathcal{N}$, $t\in [ t_\mathfrak{q}^+,T_z]$, $\chi \in \RR$, 
\begin{equation} \label{ubP} \PP_{n_{s_\mathfrak{q},z}}(\mathscr{E}_{z,t}) = o\Big(\frac{e^{-(t-s_\mathfrak{q})}}{\uptau}\Big)e^{-\frac{2}{\sqrt{v}} \chi}. \end{equation}
Indeed, assume for the moment that \eqref{ubP} holds true. Recall the variables $R_z$ and $Q_z$ from Lemma \ref{conttime}.  Observe that $T_z=T_{z'}$ as soon as $|z'-z|\leq e^{-t}$ by Lemma \ref{varacctimes}. Recall the definition of $M_k(z)$ and $\mu_k(z)$ in \eqref{defmean}. For any $z,z'\in \mathcal{N}$ such that $|z'-z|\leq e^{-t}$, we have that 
\[ |\Psi_t^{z'} - \psi_{n_{t,z}}(z')| \leq R_z + M  +|\Psi_{s_\mathfrak{q}}^{z'}-\psi_{n_{s_\mathfrak{q}},z}(z')|\leq Q_z + M,\]
 where $M := |M_{n_{t,z'}}(z') - M_{n_{t,z}}(z')| +|M_{n_{s_\mathfrak{q}},z'}(z') - M_{n_{s_\mathfrak{q}},z}(z')|$.
  Since $\mu_k(z') \lesssim 1/|k-k_{0,z'}| \lesssim \sigma_k^2(z')$ for any $|k-k_{0,z}|>\ell_0$ by  Propositions \ref{increm} and \ref{represelliptic}, it follows that 
 \[ |M|\lesssim |\Sigma_{n_{t,z'}}^2(z') - \Sigma_{n_{t,z}}^2(z')| + |\Sigma_{n_{s_\mathfrak{q}},z'}^2(z') - \Sigma_{n_{s_\mathfrak{q}},z}^2(z')| \lesssim n^{2/3} e^{-t} \leq 1,  \]
 where we used in the last inequality Lemma \ref{varacctimes} and the fact that $t\geq t_\mathfrak{q}^+$. Denote by 
 \[ \mathscr{R} := \{\forall z \in \mathcal{N}, R_z\leq 1,Q_z\leq 1, \Psi_{s_\mathfrak{q}}^z \leq \sqrt{v} s_\mathfrak{q} + \mathfrak{C}\log 
\uptau \}.\]
Write $\Delta_{t,z} :=   \max_{|x-z|\leq e^{-t}} |\psi_{n_{s_\mathfrak{q},z}}(z) - \psi_{n_{s_\mathfrak{q}},z}(x)|$.  On $\mathscr{E}\cap \mathscr{R}$, there exist $z'\in \mathcal{N}$, $t \in [t_\mathfrak{q}^+,T_{z'}]$ and $z\in \mathcal{N}_t$ such that $|z-z'|\leq e^{-t}$, $|z|\leq |z'|$ and  
  \[ \psi_{n_{t,z}}(z') - \psi_{n_{s_\mathfrak{q},z}}(z) \geq \sqrt{v} (t-s_\mathfrak{q}) + 4\mathfrak{C} \log \uptau + (\sqrt{v}s_\mathfrak{q}-\Psi_{s_\mathfrak{q}}^z) -\Delta_{t,z}-4.\]
Let $m:= \lceil \sqrt{v}s_\mathfrak{q} +\mathfrak{C} \log \uptau\rceil$. Denote for any $z\in \mathcal{N}$ by $\mathscr{F}_{k,z} := \{ \Psi_{s_\mathfrak{q}}^z\geq  \sqrt{v} s_\mathfrak{q} +\mathfrak{C}\log  \uptau-k\}$ for any $1\leq k<m$, $\mathscr{F}_{m,z}:= \{\Psi_{s_\mathfrak{q}}^z\leq 0\}$ and $\mathscr{D}_{\ell,z} := \{ \max_{|x-z|\leq e^{-t}, |z|\leq |x|} \Delta_{t,z} \geq \ell\}$ for any $\ell\geq 0$. By discretizing the values of $\sqrt{v}s_\mathfrak{q}+\mathfrak{C} \log \uptau-\Psi_{s_\mathfrak{q}}^z$ and $\Delta_{t,z}$, where $z$ is the representative at level $t$ of $z'$ and using a union bound, we get that 
\begin{equation} \label{unionER} \PP(\mathscr{E}\cap \mathscr{R}) \leq \sum_{z\in \mathcal{N}_t} \sum_{1\leq \ell\leq r \atop 1\leq k \leq m}   \PP\big(\mathscr{R}\cap\mathscr{F}_{k,z}\cap \mathcal{D}_{\ell,z}\cap \mathscr{E}_{z,t}^{k-\ell}\big),\end{equation}
where   $r:= \lfloor \uptau^{\mathfrak{q}/24}\rfloor$. Note that on the event $\mathscr{R}$, $\Delta_{t,z} \leq \max_{|x-z|\leq e^{s_\mathfrak{q}}}|\Psi_{s_\mathfrak{q}}^x-\Psi_{s_\mathfrak{q}}^z| + 1$. Therefore, for any $1\leq k\leq m$ and $1\leq \ell \leq r$, we have by using the continuity estimates of Lemma \ref{jointprobamax} that for $\mathfrak{q}$ large enough,
\[ \PP(\mathscr{R}\cap\mathscr{F}_{k,z} \cap \mathscr{D}_{\ell,z}) \leq \mathfrak{c}^{-1} e^{-s_\mathfrak{q} -\mathfrak{c}\ell^2 + \frac{2}{\sqrt{v}}k},\]
where $\mathfrak{c}>0$ depends on the model parameters. Together with \eqref{ubP}  and \eqref{unionER} this yields that 
\[ \PP(\mathscr{E}\cap \mathscr{R})=  m   \Big(\sum_{1\leq \ell\leq r} e^{\frac{2}{\sqrt{v}} \ell-\mathfrak{c}\ell^2}\Big)o(\uptau^{-1}),\]
where we used the fact that $\#\mathcal{N}_t \lesssim e^t$. Since $m\lesssim \uptau$ this shows that \eqref{ubP} implies that $\PP(\mathscr{E}\cap \mathscr{R})\to1$ as $n\to+\infty$, for $\mathfrak{q}$ large enough. 
 As $\PP(\mathscr{R} )\to 1$  for $\mathfrak{q}$ large enough by Lemmas \ref{conttime} and \ref{allgoodh}, this implies
%would imply the claim 
that $\PP(\mathscr{E}) \to 0$ as $n\to +\infty$.

We now move on to prove \eqref{ubP}. 
Fix $z\in I_\eta$,  $t\in [ t_\mathfrak{q}^+,T_z]$ and $\chi\in\RR$. Recall the definition the blocks $(t_i)_{1\leq i\leq N}$ from \eqref{defti}.
Denote for any $z' \in \mathcal{N}$  by ${L}_{i,z'}$ and  ${K}_i$, $1\leq i\leq N$,  the random variables 
\[ {L}_{i,z'} :=  \log \big\| \Xi_{n_{t_{i-1}}, n_{t_i}}^{z'}\|  - \Delta M_{n_{t_i}}(z)-\sqrt{v} \Delta t_i,\]
\[ {K}_i :=  \log \big\| \Xi_{n_{t_{i-1}}, n_{t_i}}^{z}\| -  \Delta M_{n_{t_i}}(z)- \sqrt{v} \Delta t_i,\]
where we write to lighten the notation $n_{t_i}$ instead of $n_{t_i,z}$ for any $i$. 
On the event $ \mathscr{E}_{z,t} $,  there exists $z'\in \mathcal{N}$ such that $|z'-z|\leq e^{-t}$ and $\sum_{i=1}^N {L}_{i,z'} \geq 2\mathfrak{C} \log \uptau+\chi$. Define $\mathcal{J}_{z'}$ as the set
\[ \mathcal{J}_{z'} := \big\{ i \in \llbracket1,N\rrbracket :  {L}_{i,z'} -{K}_i> N^{-1}\big\}.\]
With this notation, it follows that on the event $\mathscr{E}_{z,t}$, 
% $\mathscr{G}_{z,t}\cap \mathscr{E}_{z,t}$, 
%\textcolor{purple}{The event $\mathscr{G}_{z,t}$ was not defined in the original version....}
there exists $z' \in \mathcal{N}$ such that $|z-z'|\leq e^{-t}$ and 
\[   \sum_{i\in \mathcal{J}_{z'}} {L}_{i,z'} + \sum_{i\notin \mathcal{J}_{z'}} {K}_i \geq 2\mathfrak{C}\log \uptau - 1+\chi.\] 
Note that when $i\in \mathcal{J}_{z'}$, by the inverse triangular inequality,
\[ \big\|\Xi_{n_{t_{i-1}}, n_{t_i}}^{z'} - \Xi_{n_{t_{i-1}},n_{t_i}}^{z} \big\| \geq \big\| \Xi_{n_{t_{i-1}},n_{t_i}}^{z} \big\| \Big(\frac{\big\|\Xi_{n_{t_{i-1}},n_{t_i}}^{z'} \big\|}{\big\|\Xi_{n_{t_{i-1}},n_{t_i} }^{z} \big\|}-1\Big) \geq \frac12 N^{-1} \big\| \Xi_{n_{t_{i-1}},n_{t_i}}^{z'} \big\|,\]
where we used the inequality $e^{x}-1 \geq \frac12N^{-1} e^x$ for any $x\geq N^{-1}$ and $N$ large enough. Thus, on the event $ \mathscr{G}_{z,t}\cap \mathscr{E}_{z,t} $, there exists $z'\in \mathcal{N}$, such that $|z-z'|\leq e^{-t}$ and 
\[   \sum_{i\in \mathcal{J}_{z'}} \log [2N\Delta_{i,z,z'}] + \sum_{i\notin \mathcal{J}_{z'}} K_i \geq  \mathfrak{C}\log \uptau+\chi,\]
where $\Delta_{i,z,z'} := \big\|\Xi_{n_{t_{i-1}}, n_{t_i}}^{z'} - \Xi_{n_{t_{i-1}},n_{t_i}}^{z} \big\| e^{-\Delta M_{n_{t_i}}(z) -\sqrt{v} \Delta t_i}$ for any $i$. 
As a consequence, on $\mathscr{G}_{z,t}\cap \mathscr{E}_{z,t}$, 
\[ \sup_{ J\subset \llbracket 1,N\rrbracket \atop J\neq \emptyset} \Big(\sum_{i\in J} \log[2N \Delta_{i,z}] + \sum_{i\notin J} K_i \Big) \geq \mathfrak{C} \log \uptau+\chi,\]
where $\Delta_{i,z} := \sup_{z'\in \mathcal{N}: |z-z'|\leq e^{-t}}  \Delta_{i,z,z'}$.
Let  $\zeta = 2\vee \frac{2}{\sqrt{v}}$ and $\lambda = 2/\sqrt{v}$. Using a union bound over the possible sets $J$ and Chernoff inequality, we get that 
%
%\textcolor{purple}{Here the intersection with $\mathscr{G}_{z,t} $ seems to be missing!}
\begin{equation} \label{ChernoffP}  \PP (\mathscr{E}_{z,t}) \leq \sum_{J \subset \llbracket 1,N\rrbracket \atop J\neq \emptyset} \EE\Big(\prod_{i\in J} [2N \Delta_{i,z}]^\zeta e^{\frac{2}{\sqrt{v}}\sum_{i\notin J} K_i}\Big) e^{-m_{\zeta}}, \end{equation}
where $m_{\zeta} = \min\{ \zeta x +\frac{2}{\sqrt{v}} y : x+y \geq \mathfrak{C} \log \uptau, x,y\geq 0\}$. Since $\zeta \geq \frac{2}{\sqrt{v}}$, we have that $m_\zeta = (2/\sqrt{v})\mathfrak{C} \log \uptau$. By Proposition \ref{expopreciseincre} (see Remark \ref{exponorm}), we know that there exists $\mathfrak{h}>0$ depending on the model parameters such that for any $i \in \llbracket 1, N\rrbracket$, 
\[ \EE_{n_{t_{i-1}}}[e^{ \frac{2}{\sqrt{v}}K_i}] \leq e^{-\Delta t_i + \mathfrak{h} }.\]
Together with Lemma \ref{contmax}, this yields that for $\mathfrak{f}$ large enough and $\mathfrak{p}$ small enough, for any $J\subset \llbracket 1, N\rrbracket$,
\[ \log \EE\Big(\prod_{i\in J} [2N \Delta_{i,z}]^\zeta e^{\frac{2}{\sqrt{v}} \sum_{i\notin J} K_i}\Big)  \leq  \sum_{i\in J\setminus \{N\}} \zeta \big[\log (2N) - \frac13 (t-t_i)\big]  +\mathfrak{h}  \log \uptau\Car_{N\in J} - \sum_{i\notin J} \Delta t_i  + \mathfrak{h} N,\]
where $\mathfrak{h}>0$ changed value  without changing name 
%\textcolor{purple}{not sure what you mean here}
and depends on the model parameters and on  $\mathfrak{f}$.
But choosing $\mathfrak{f}$ large enough and $\mathfrak{p}$ small enough, we have for any $i\in \llbracket 1,N-1\rrbracket$ that $\zeta(2\log N - \frac{1}{3}(t-t_i)) \leq - \Delta t_i$ since $N\lesssim \uptau$. Thus, 
\begin{align*} \log \EE\Big(\prod_{i\in J} [N^{2} \Delta_{i,z}]^\zeta e^{\frac{2}{\sqrt{v}} \sum_{i\notin J} K_i}\Big) & \leq - (t-s_\mathfrak{q}) + [\Delta t_N + \mathfrak{h}  \log \uptau] \Car_{N\in J} + \mathfrak{h} N\\
& \leq - (t-s_\mathfrak{q}) + \mathfrak{h}' \log \uptau,
\end{align*}
where $\mathfrak{h}' >0$ depends on the model parameters and on $\mathfrak{f}$ and we used that $N\lesssim \uptau$ and $\Delta t_N \lesssim \mathfrak{f} \log \uptau$. Therefore, coming back to \eqref{ChernoffP}, we find that 
\[ \PP(\mathscr{E}_{z,t}) \leq 2^N e^{-(t-s_\mathfrak{q})+\mathfrak{h}' \log T - \frac{2}{\sqrt{v}} \mathfrak{C} \log \uptau }.\]
Since $N\lesssim \uptau$, choosing $\mathfrak{C}$ large enough depending on $\kappa$ gives the claim. 
\end{proof}

\section{Proof of the upper bound of Theorem \ref{reduc}}
We first describe the strategy of the proof. Thanks to Lemma \ref{barrierloose1} we can put a rough
 barrier on the process $\Psi_t^z$ for any $z\in \mathcal{N}$. Next, we will replace the process $\Psi^z$ by a Gaussian process using the coupling constructed in Proposition \ref{coupling}. By a union bound, we are then reduced to computing the one-ray probability that a stationary Gaussian process with independent increments stays under an affine curve and ends  its trajectory at a certain level below
its endpoint.  Finally, invoking a ballot theorem for such Gaussian processes 
%random walks with Gaussian increments 
(Theorem  \ref{ballot}) will  conclude the proof.

We now provide the details.
%prove in details the upper bound of Theorem \ref{reduc}. 
In the following, $\mathfrak{C}$ denotes a constant depending on the model parameter  which will be taken as large as needed. By Lemma \ref{barrierloose1}, we know that for $\mathfrak{q}$ large enough,
\[ \PP\big( \exists z \in \mathcal{N}, \exists t \in \mathcal{T}_{\mathfrak{q},z}, \Psi_{t}^z \geq \sqrt{v}t + \mathfrak{C} \log T \big) \underset{n\to+\infty}{\longrightarrow} 0,\]
where $\mathcal{T}_{\mathfrak{q},z}=\llbracket 1,T_z\rrbracket \setminus \llbracket t_\mathfrak{q}^-,t_\mathfrak{q}^+\rrbracket$.
Since $\# \mathcal{N}\lesssim e^{\uptau}$ and $T_z= \uptau+O(1)$,  we can reduce ourselves by a union bound to proving
 that for any fixed $z\in I_\eta$ and $\veps>0$,
\begin{equation}\label{claimub} \PP\Big(\Psi_{T_z}^z \geq \sqrt{v} T_z - \Big( \frac{3\sqrt{v}}{4} -\veps\Big) \log T, \forall t \in \mathcal{T}_{\mathfrak{q},z} , \Psi_{t}^z \leq \sqrt{v} t + \mathfrak{C} \log \uptau\Big) = o(e^{-\uptau}).\end{equation}
Now, Proposition \ref{coupling} enables us to replace the processes $(\Psi_{t}^z-\Psi_{t_\mathfrak{q}^+}^z)_{t_\mathfrak{q}^+\leq t \leq T_z}$ and $(\Psi_t^z)_{1 \leq t\leq t_\mathfrak{q}^-}$ by some Gaussian processes,  up to an error of order $1$, with overwhelming probability.
Denote by $\mathscr{B}^+_\veps$, $\mathscr{B}^-_\veps$  and $\mathscr{A}$ the events: 
\[ \mathscr{B}^+_\veps := \Big\{\Psi_{t_\mathfrak{q}^+}^z + \Gamma_{T_z}-\Gamma_{t_\mathfrak{q}^+} \geq \sqrt{v} T_z - \Big( \frac{3\sqrt{v}}{4} -\veps\Big) \log T , \forall t \in \llbracket t_\mathfrak{q}^+,T_z\rrbracket, \Psi_{t_\mathfrak{q}^+}^z + \Gamma_{t}-\Gamma_{t_\mathfrak{q}^+} \leq \sqrt{v} t + 2\mathfrak{C} \log T\Big\},\]
\[ \mathscr{A} := \big\{\Psi_{t_\mathfrak{q}^+}^z \leq \sqrt{v}t_\mathfrak{q}^+ +\mathfrak{C} \log \uptau, \ \big|\Psi_{t_\mathfrak{q}^-}^z -\Gamma_{t_\mathfrak{q}^-}\big| \leq \mathfrak{C}, \  \Psi_{t_\mathfrak{q}^-}^z \leq \sqrt{v}t_\mathfrak{q}^- +\mathfrak{C} \log \uptau, W^z_{r_\mathfrak{q}}\leq \eta_{r_\mathfrak{q},z} \big\},\]
\[ \mathscr{B}^-_\veps := \Big\{\Gamma_{t} \leq \sqrt{v} t + \mathfrak{C} \log \uptau, \forall t\in \llbracket 1,t_\mathfrak{q}^-\rrbracket\Big\},\]
where $r_\mathfrak{q}:= n_{t_\mathfrak{q}^-,z}$ and $\Gamma$ is the Gaussian process given by Proposition \ref{coupling}. By Propositions \ref{coupling} and \ref{controlbadblockhall},
to prove \eqref{claimub}, it suffices to show that 
\[ \PP\big(\mathscr{B}_\veps^- \cap \mathscr{A}\cap \mathscr{B}_\veps^+\big) = o(e^{-\uptau}).\]
Recall the enlarged filtration $(\mathcal{G}_k)_{1\leq k \leq n}$ given by Proposition \ref{coupling}. To lighten the notation, 
write $\widetilde{\mathcal{G}}_t =\mathcal{G}_{n_{t,z}}$, and for any $t>0$, let
$\PP_t,\EE_t$ denote the conditional probability and expectation given $\widetilde{\mathcal{G}}_{t}$. Using the Ballot Theorem  \ref{ballot3} conditionally on $\widetilde{\mathcal{G}}_{t_\mathfrak{q}^+}$, we find that on the event where $\Psi_{t_\mathfrak{q}^+}^z \leq \sqrt{v}t_\mathfrak{q}^+ +\mathfrak{C} \log \uptau$, 
\begin{align}
 \PP_{t_\mathfrak{q}^+}(\mathscr{B}^+) &\lesssim \frac{(2\mathfrak{C}\log \uptau + \sqrt{v} t_\mathfrak{q}^+- \Psi_{t_\mathfrak{q}^+}^z\big) \log \uptau}{(\uptau-t_\mathfrak{q}^+)^{3/2}} e^{-(\uptau-t_\mathfrak{q}^+)+\frac{2}{\sqrt{v}}(\frac{3\sqrt{v}}{4}-\veps) \log \uptau- \frac{2}{\sqrt{v}}(\sqrt{v}t_\mathfrak{q}^+-\Psi_{t_\mathfrak{q}^+}^z)}\nonumber \\
& \lesssim (2\mathfrak{C} \log \uptau + \sqrt{v} t_\mathfrak{q}^+- \Psi_{t_\mathfrak{q}^+}^z\big) \log \uptau e^{-\uptau-t_\mathfrak{q}^+ -\frac{2\veps}{\sqrt{v}} \log \uptau} e^{\frac{2}{\sqrt{v}} \Psi_{t_\mathfrak{q}^{+}}^z},\label{10expo1}
\end{align}
where we used the fact that $\uptau-t_\mathfrak{q}^+\gtrsim \uptau$. On the event where $\Psi_{t_\mathfrak{q}^-}^z \leq \sqrt{v}t_\mathfrak{q}^- +\mathfrak{C} \log \uptau$, using in addition that $(x+y)_+ \leq x_+ +y_+$ and that $x_-e^{x} \leq \exp(|x|)$, for any $x,y\in \RR$, we obtain that
\begin{align}
 \EE_{t_\mathfrak{q}^-} \big[  \big(2\mathfrak{C}\log \uptau+ \sqrt{v}t_\mathfrak{q}^+&-\Psi_{t_\mathfrak{q}^+}^z\big)_+  e^{\frac{2}{\sqrt{v}}(\Psi_{t_\mathfrak{q}^+}^z-\Psi_{t_\mathfrak{q}^-}^z)}\Big]  \nonumber \\ 
& \leq \EE_{t_\mathfrak{q}^-} \big[ \big(2\mathfrak{C}\log \uptau+ \sqrt{v}t_\mathfrak{q}^+-\Psi_{t_\mathfrak{q}^-}^z+(\Psi_{t_\mathfrak{q}^+}^z- \Psi_{t_\mathfrak{q}^-}^z)_-\big) e^{\frac{2}{\sqrt{v}}(\Psi_{t_\mathfrak{q}^+}^z-\Psi_{t_\mathfrak{q}^-}^z)}\Big]\nonumber \\
& \lesssim   \big(2\mathfrak{C}\log \uptau + \sqrt{v}t_\mathfrak{q}^+-\Psi_{t_\mathfrak{q}^-}^z\big) \EE_{t_\mathfrak{q}^-} \big[e^{\frac{2}{\sqrt{v}}|\Psi_{t_\mathfrak{q}^+}^z-\Psi_{t_\mathfrak{q}^-}^z|}\Big].\label{10expo2}
\end{align}
Now, by Proposition  \ref{expopreciseincre}, we know that on the event where $W_{k}^z \leq \eta_{k,z}$ at time $k= n_{t_\mathfrak{q}^-,z}$,  $\EE_{t_\mathfrak{q}^-} [\exp\big(\frac{2}{\sqrt{v}}|\Psi_{t_\mathfrak{q}^+}^z-\Psi_{t_\mathfrak{q}^-}^z|\big)] \leq \exp(t_\mathfrak{q}^+-t_\mathfrak{q}^-+\mathfrak{C})$. Thus,  using \eqref{10expo1}-\eqref{10expo2}, we get  that on the event where $\Psi_{t_\mathfrak{q}^-}^z \leq \sqrt{v}t_\mathfrak{q}^- +\mathfrak{C} \log \uptau$,  $\big|\Psi_{t_\mathfrak{q}^-}^z -\Gamma_{t_\mathfrak{q}^-}\big| \leq \mathfrak{C}$ and $W_k^z\leq \eta_{k,z}$ at $k= n_{t_\mathfrak{q}^-,z}$, 
\begin{align}  \PP_{t_\mathfrak{q}^-}(\mathscr{A}\cap \mathscr{B}^+_\veps)  &\lesssim  \big(2\mathfrak{C}\log \uptau + \sqrt{v}t_\mathfrak{q}^+-\Psi_{t_\mathfrak{q}^-}^z\big) \log \uptau e^{-\uptau-t_\mathfrak{q}^- -\frac{2\veps}{\sqrt{v}} \log \uptau +\frac{2}{\sqrt{v}} \Psi_{t_\mathfrak{q}^-}^z}\nonumber \\
& \lesssim \big(3\mathfrak{C}\log \uptau + \sqrt{v}t_\mathfrak{q}^--\Gamma_{t_\mathfrak{q}^-}\big) \log \uptau e^{-\uptau-t_\mathfrak{q}^- -\frac{\veps}{\sqrt{v}} \log \uptau +\frac{2}{\sqrt{v}} \Gamma_{t_\mathfrak{q}^-}}\label{boundup}
\end{align}
where we used the fact that $t_\mathfrak{q}^+-t_\mathfrak{q}^- \lesssim \log \uptau$. Next, we decompose the event $\mathscr{B}_\veps^-$ according to the values of $\Gamma_{t_\mathfrak{q}^-}$. To this end, define the events $\mathscr{B}_h^-$, for $-\uptau^{1/2} \log \uptau \leq  h \leq \mathfrak{C}\log \uptau$, and $\mathscr{E}$ by 
\[ \mathscr{B}_{h}^- := \mathscr{B}^-_\veps \cap \big\{\Gamma_{t_\mathfrak{q}^-} -\sqrt{v}t_\mathfrak{q}^- \in [h,h+1)\big\}, \ \mathscr{E} := \big\{ \Gamma_{t_\mathfrak{q}^-} -\sqrt{v}t_\mathfrak{q}^- \leq -\uptau^{1/2} \log \uptau\big\}.\]
By performing a Gaussian change of measure, we deduce that 
\begin{equation} \label{estimBh1}\EE\big( \Car_{\mathscr{B}_{h}^-} e^{\frac{2}{\sqrt{v}} \Gamma_{t_\mathfrak{q}^-}}\big) = \PP(\widetilde{\mathscr{B}}_{h}^-)e^{t_\mathfrak{q}^-}, \end{equation}
where $\widetilde{\mathscr{B}}_h^- = \{\Gamma_t \leq \mathfrak{C} \log \uptau, \forall t \leq t_\mathfrak{q}^-, \ \Gamma_{t_\mathfrak{q}^-} \in [h,h+1)\}$. Thus, by the Ballot Theorem \ref{ballot}, we deduce that 
\begin{equation} \label{estimtBh}\EE\big( \Car_{\mathscr{B}_{h}^-} e^{\frac{2}{\sqrt{v}} \Gamma_{t_\mathfrak{q}^-}}\big)   \lesssim \frac{(\mathfrak{C}\log \uptau-h)\log \uptau}{\uptau^{3/2}}e^{t_\mathfrak{q}^-},\end{equation}
where we used the fact that $t_\mathfrak{q}^-\gtrsim \uptau$.  Similarly, using a Gaussian change of measure, we get that 
\begin{equation} \label{boundtailgauss} \EE\Big(\Car_{\mathscr{E}} (\sqrt{v}t_\mathfrak{q}^- - \Gamma_{t_\mathfrak{q}^-} \big)e^{\frac{2}{\sqrt{v}} \Gamma_{t_\mathfrak{q}^-}}\Big) = \EE\big( \Car_{\Gamma_{t_\mathfrak{q}^-} \geq \uptau^{1/2} \log \uptau} \Gamma_{t_\mathfrak{q}^-}\big)e^{t_\mathfrak{q}^-} \lesssim \sqrt{\uptau} e^{-\mathfrak{c}(\log \uptau)^2 + t_\mathfrak{q}^-},\end{equation}
where we used Cauchy-Schwarz inequality and the fact that $\Gamma_{t_\mathfrak{q}^-}$ is a centered Gaussian random variable with variance of order at most $\uptau$, and where $\mathfrak{c}>0$ depends on the model parameters.
Combining \eqref{estimtBh}, \eqref{estimBh1} with \eqref{boundup} we conclude that 
\[ \PP(\mathscr{B}_h^-\cap \mathscr{A}\cap \mathscr{B}^+_\veps)  \lesssim \frac{(\log \uptau)^4}{\uptau^{1/2}} e^{-\uptau -\frac{\veps}{\sqrt{v}} \log \uptau},\]
for any $-\uptau^{-1/2} \log \uptau \leq h\leq \mathfrak{C}\log \uptau-1$. Summing over integers $h$ in $[ -\uptau^{-1/2} \log \uptau, \mathfrak{C} \log -1]$ and using a union bound, this yields that
\[ \PP(\mathscr{B}^-_\veps \cap \mathscr{E}^\complement \cap \mathscr{A} \cap \mathscr{B}_\veps^+) \leq  (\log \uptau)^5  e^{-\uptau -\frac{\veps}{\sqrt{v}} \log \uptau}.\]
On the other hand, using \eqref{boundtailgauss} and \eqref{boundup} , we find that 
\[ \PP(\mathscr{E}\cap \mathscr{A} \cap \mathscr{B}^+_\veps) \lesssim e^{-\uptau-\mathfrak{c}(\log \uptau)^2},\]
where $\mathfrak{c}>0$ changed value without changing name. Hence, $\PP(\mathscr{B}_\veps^-\cap \mathscr{A}\cap \mathscr{B}_\veps^+) = o(e^{-\uptau})$, which ends the proof.

\section{Truncation of the recursion}
We now turn our attention to the proof of the lower bound, for which we will use a classical second moment method.
There are two ingredients for this. First, for points $z,z'$ which are roughly macroscopically separated,  we would like to create some independence of the variables $\Psi_{T_z}^z$ and $\Psi_{T_{z'}}^{z'}$,
which unfortunately are still correlated (the estimates in Section \ref{decorrsection}  can be used to actually quantify the covariance, which turns out to be  of order $1$). This step is carried out in the current section. The second ingredient, which quantifies the correlation for points $z,z'$ which are mesoscopically separated, is carried out in Section \ref{decorrsection}.

To address the first issue, we truncate the recursion as follows. Recall that $\uptau=\log n$.
Let $\uptau_\veps = \uptau- \lfloor \veps \log \uptau\rfloor$, $t_\veps =\lfloor \veps \log \uptau\rfloor$ and let $\overline{\Psi}_{\uptau_\veps}$ be the field defined as   
\begin{equation} \label{defpsitrunc} \overline{\Psi}_{\uptau_\veps}^z = \log \Big\| \prod_{\ell= n_{t_\veps,z}+1}^{n_{\uptau_\veps,z}} \widehat{\Xi}_\ell^z e_1\Big\|, \quad z\in I_\eta. \end{equation}
where  $\widehat{\Xi}_\ell^z $ is the normalised transition matrix $\widehat{\Xi}_\ell^z := \Xi_\ell^z/ \mu_\ell(z)$, with $\mu(z)$  defined in \eqref{defmean} and $\Xi^z$ in \eqref{defXi} and $n_{t,z}$ is the time-change defined in \eqref{deftimechange}.  Note that since $T_z=\uptau+O(1)$, $n_{t_\veps,z}$ and $n_{\uptau_\veps,z}$ are well-defined.  In other words, $\overline{\Psi}_{\uptau_\veps}^z $ is the vector obtained by initializing the recursion at time $n_{t_\veps,z}$ at the first coordinate vector $e_1$ and by running the recursion until time $n_{\uptau_\veps ,z}$. In particular, $\overline{\Psi}_{\uptau_\veps}^z$ has the same distribution as  $\Psi_{\uptau_\veps }^z - \Psi_{t_\veps }^z$ under the conditional distribution given that $W_{n_{t_\veps},z}^z = 0$. 
We will see that with this truncation of the recursion, if $z,z'$ are such that $|z-z'| \gg (\log n)^{-\veps}$ then the intervals $[n_{t_\veps,z},n_{\uptau_\veps,z}]$ and $[n_{t_\veps,z'},n_{\uptau_\veps,z'}]$ are disjoint so that $\overline{\Psi}_{\uptau_\veps}^z$ and $\overline{\Psi}_{\uptau_\veps}^{z'}$ are independent.

 In the next lemma, we show that with high probability, if there exists some $z\in \mathcal{N}$ such that $\overline{\Psi}_{\uptau_\veps}^z$ is high then $ {\Psi}_{T_z}^z$ is also high. In particular, this will allow us  to reduce proving the lower bound on the maximum of the field of $\Psi_{T_z}^z$ claimed in Theorem \ref{reduc},  to showing a lower bound on the maximum of the field $\overline{\Psi}_{\uptau_\veps}^z$.

\begin{lemma}\label{truncation}There exists  a positive constant $\mathfrak{C}$ depending on the model parameters such that for any $\veps>0$ small enough,
\[ \PP\Big( \exists z \in \mathcal{N}, \overline{\Psi}_{\uptau_\veps}^z  \geq \sqrt{v} (\uptau_\veps-t_\veps) -\Big(\frac{3\sqrt{v}}{4}+{\veps}\Big)\log 
\uptau, \   \overline{\Psi}_{\uptau_\veps}^z - \Psi_{T_z}^z  >\mathfrak{C}\sqrt{\veps} \log \uptau \Big) \underset{n\to+\infty}{\longrightarrow} 0.\]
\end{lemma}

\begin{proof}We will show that for any $z\in I_\eta$, 
\[ \PP\Big( \overline{\Psi}_{\uptau_\veps}^z  \geq \sqrt{v} (\uptau_\veps-t_\veps) -\Big(\frac{3\sqrt{v}}{4}+{\veps}\Big)\log \uptau, \   \overline{\Psi}_{\uptau_\veps}^z - \Psi_{T_z}^z  >\mathfrak{C}\sqrt{\veps} \log \uptau \Big) =o(e^{-\uptau}).\]
The claim then will immediately follows by using a union bound. Fix $z\in I_\eta$. For sake of clarity, we drop the $z$-dependency in the notation.
 We decompose $\Psi_{T_z}^z-\overline{\Psi}_{\uptau_\veps}^z$ into three parts $\Delta_i$, $i=1,2,3$ defined by 
\[ \Delta_1 = \log \Big\| \prod_{\ell=2}^{n_{t_\veps}} \widehat{\Xi}_\ell Y_1 \Big\|, \ \Delta_3 = \log \Big\| \prod_{\ell = n_{\uptau_\veps}}^n \widehat{\Xi}_\ell u_{\uptau_\veps}\Big\|,\]
\[ \Delta_2 =  \log \Big\| \prod_{\ell=n_{t_\veps}+1}^{n_{\uptau_\veps}} \widehat{\Xi}_\ell e_1\Big\|- \log \Big\| \prod_{\ell=n_{t_\veps}+1}^{n_{\uptau_\veps}} \widehat{\Xi}_\ell u_{{t_\veps}} \Big\|,\]
where $u_{t} :=  Y_{n_{t}} / \|Y_{n_{t}}\|$ for any $t\leq T_z$. Write $m_{n,\veps} :=  \sqrt{v} (\uptau_\veps-t_\veps) -(\frac{3\sqrt{v}}{4}+ 
{\veps})\log \uptau$. 
Note that
 $\overline{\Psi}_{\uptau_\veps }^z- \Psi_{T_z}^z = \Delta_2 - \Delta_1 - \Delta_3$. We will show that  there exists $\mathfrak{C}>0$ such that 
\begin{equation} \label{Delta1} \PP\Big( \mathscr{E}_2 \Big)  =o(e^{-\uptau}), \qquad \mbox{\rm where  $\mathscr{E}_2=\{\overline{\Psi}_{{\uptau_\veps}}^z \geq m_{n,\veps}, \   \Delta_2  > 1 \}$,}
\qquad 
\end{equation}
%and for $i=1,3$,
\begin{equation} \label{Delta12} \PP\Big( \mathscr{E}_i\Big)=o(e^{-\uptau}), \qquad \mbox{\rm where $\mathscr{E}_i=\{ \overline{\Psi}_{\uptau_\veps}^z \geq m_{n,\veps}, \    \Delta_i<-\mathfrak{C} \sqrt{\veps} \log \uptau\}, \quad i=1,3$.}
\end{equation}
%Denote by $\mathscr{E}_i$, $i=1,2,3$, the corresponding events on the left-hand side of \eqref{Delta1} and \eqref{Delta12}, where $\mathfrak{C}>0$ is some constant to be \corO{chosen} later. 
We start by proving \eqref{Delta1}.
Observe  that if $\log(B) - \log (A) \geq  s$ then  by concavity of the logarithm, we have that $B-A \geq  s A$, which implies that
\[ B-A \geq \frac{s}{s+1} B.\]
Applying this inequality with $B = \big\| \prod_{\ell=n_{t_\veps}+1}^{n_{\uptau_\veps}} \widehat{\Xi}_\ell e_1\big\|$, $A = \big\| \prod_{\ell=n_{t_\veps}+1}^{n_{\uptau_\veps}} \widehat{\Xi}_\ell u_{t_\veps}\big\|$
 and $s = 1$, we find that 
on the event $\mathscr{E}_2$, we have for $n$ large enough,
\begin{equation} \label{conclog} \Big\| \prod_{\ell=n_{t_\veps}+1}^{n_{\uptau_\veps}} \widehat{\Xi}_\ell (e_1-u_{t_\veps})\Big\| \geq    e^{\sqrt{v} 
\uptau - \frac{3\sqrt{v}}{8} \log \uptau}.\end{equation}
Note that $\eta_k \asymp_\delta n^{-2/3}$ at time $k=k_{\delta}^-$.
Hence,  using Proposition \ref{incremhn}, $W_k\leq \eta_k/2$ almost surely at $k=k_{\delta}^-$. Since  $\eta_{n_{t_\veps}} \lesssim (\log n)^{\veps}n^{-2/3} \leq  n^{-1/3}$ as $k_{0} - n_{t_\veps} \gtrsim n (\log n)^{-\veps}$ by Lemma \ref{approxtimechange}, it follows  by Corollary \ref{controlbadblockhall} that for $\mathfrak{q}$ large enough,   
\begin{equation} \label{condini} \PP\big(\|e_1 - u_{{t_\veps}}\|> n^{-1/3}\big)\leq e^{- (\log n)^2}.\end{equation}
Using  Proposition \ref{expomomentpreciseperturb} with $\lambda = 2/\sqrt{v}$ and the fact that $\widehat{\Sigma}_{n_{\uptau_\veps}}^2 - \widehat{\Sigma}_{n_{t_\veps}}^2 \leq (v/2)
\uptau +O(1)$ by Lemma \ref{approxtimechange} $(iv)$, we get for any $v\in \mathbb{S}^1$ that 
\begin{equation} \label{tailnorm} \PP_{n_{t_\veps}}\Big(\Big\| \prod_{\ell=n_{t_\veps}+1}^{n_{\uptau_\veps}} \widehat{\Xi}_\ell v \Big\| \geq   e^{(\sqrt{v}+\frac16) \uptau }\Big) \lesssim  e^{-\uptau-\mathfrak{c} \uptau},\end{equation}
 where $\mathfrak{c}>0$ depends on the model parameters. 
We  now deduce \eqref{Delta1} by putting together \eqref{conclog}, \eqref{condini} and \eqref{tailnorm}.
\iffalse
 that 
\[ \PP(\mathscr{E}_2) = o\big(e^{-\uptau}\big).\]
\fi

Next, since   $\widehat{\Sigma}^2_{n_{t_\veps}} = \veps \log \uptau+O(1)$ by  Lemma \ref{approxtimechange} $(iv)$, we deduce by using Proposition \ref{apriori} for  $\lambda = -1/\sqrt{\veps}$  that 
\[ \PP\big(\Delta_1< -\mathfrak{C}\sqrt{\veps} \log \uptau \big)  \lesssim e^{-  \mathfrak{C} \log \uptau + \mathfrak{C'}  \sqrt{\veps} \log \uptau},\]
where $\mathfrak{C}'>0$ depend on the model parameters.  
Thus, for $\veps$ small enough, we get that 
\begin{equation} \label{probaDel1}  \PP\big(\Delta_1< -\mathfrak{C} \sqrt{\veps} \log \uptau \big)  \leq e^{- \frac{\mathfrak{C}}{2}  \log \uptau}.\end{equation}
A similar argument for the increment between time $n_{\uptau_\veps}$ and $n$ yields that for $\mathfrak{C}$ large enough,
\begin{equation} \label{probaDel2} \PP_{n_{\uptau_\veps}}\big(\Delta_2<- \mathfrak{C} \sqrt{\veps} \log \uptau \big)  \leq e^{- \frac{\mathfrak{C}}{2}  \log \uptau}.\end{equation}
Moreover, using Proposition \ref{expopreciseincre} for $\lambda = 2/\sqrt{v}$, we have that 
\begin{equation} \label{probpsi} \PP_{n_{t_\veps}}\big( \overline{\Psi}_{\uptau_\veps}^z\geq m_{n,\veps} \big) \leq \uptau^{\frac{3}{2} +O(\sqrt{\veps})} e^{-\uptau},\end{equation}
where we used that $\widehat{\Sigma}_{n_{\uptau_\veps}}^2 \leq (v/2) \uptau +O(1)$ by Lemma \ref{approxtimechange} $(iv)$. 
Now, note that $\overline{\Psi}_{\uptau_\veps}^z$ and $\Delta_1$ are independent. Thus, combining \eqref{probaDel1}-\eqref{probpsi} and  choosing
$\mathfrak{C}$ large enough, we get that 
\[ \PP(\mathscr{E}_1) \leq \uptau^{-\frac{\mathfrak{C}}{4}}e^{-\uptau}.\]
Similarly, putting together \eqref{probaDel2}, \eqref{probpsi}, it follows that  $\PP\big(\mathscr{E}_3) =o(e^{-\uptau})$, provided that $\mathfrak{C}$ is large enough, which ends the proof.
\end{proof}

\section{Decorrelation estimates}
\label{decorrsection}
The last key ingredient in the proof of the lower bound is a decorrelation result showing that  if  $z,z'$ are at distance $e^{-t}$ then the increments of the  processes $\psi(z)$ and $\psi(z')$ decorrelate after a certain ``branching time'' corresponding to the moment
 when both recursions accumulated a variance of order $t$. Formally, we prove this decorrelation at the level of exponential moments, as follows.

\begin{Pro}[Joint exponential moment upper bound estimate]\label{decore}
Let  $z,z'\in I_\eta$ and $1\leq t \leq T_z$ such that  $|z-z'|\leq e^{-t}$.  For any $|\lambda|,|\mu|\leq \kappa$ and $n_{t,z}\leq k \leq  k'$, $n_{t,z} \leq \ell \leq \ell'$, 
\[ \EE_{k\wedge \ell} \Big[ e^{\lambda (\psi_{\ell'}(z)-\psi_\ell(z)) +\mu( (\psi_{k'}(z')-\psi_k(z'))}\Big] \leq e^{\frac{\lambda^2}{2}(\widehat{\Sigma}_{\ell',z}^2-\widehat{\Sigma}_{\ell,z}^2)  +\frac{\mu^2}{2} (\widehat{\Sigma}_{k',z'}^2-\widehat{\Sigma}^2_{k,z'})   + \mathfrak{C}_\kappa (|\lambda|^3\vee |\mu|^3\vee 1)},\]
where $\mathfrak{C}_\kappa>0$ depends on  $\kappa$. 
\end{Pro}
%Note that the ``branching time'' is actually rather a ``branching window'' that can be taken either as $[n_{t-\mathfrak{m},z}, n_{t+\mathfrak{m},z}]$ or as $[n_{t-\mathfrak{m},z'}, n_{t+\mathfrak{m},z'}]$ where $\mathfrak{m}$ is order $1$, since these two intervals are commensurable by Lemma \ref{approxtimechange} when $|z-z'|\leq e^{-t}$. 

Since $\psi(z)$ and $\psi(z')$ are both adapted processes and we have precise one-ray exponential moment estimates by Proposition \ref{expopreciseincre} it suffices to prove this joint exponential moment estimate in the case where $\ell'= k'$ and $\ell=k$. 

To prove Proposition \ref{decore} we distinguish two cases, according to whether the branching time happens in the hyperbolic regime (meaning that $t\leq \frac23\uptau-O(\log \kappa)$ and $n_{t,z}\leq n_{t,z'} \leq k_{0,z'}-\ell_0$) or parabolic/elliptic regime (meaning that $t\geq \frac23 \uptau-O(\log \kappa)$). We first start with the case where the branching time lies in the parabolic/elliptic regime.

To do so, we will need a technical lemma that will be instrumental in showing that  the covariance of increments after the branching time vanish.

\begin{lemma}\label{approxintcorr} Let  $\frac23\uptau-O_\kappa(1)\leq t \leq T^*$ and $z,z'\in I_\eta$ such that $e^{-(t+1)}\leq |z-z'|\leq e^{-t}$ and $|z'|\leq |z|$. For  any $k_{0,z'} + 2(k_{0,z}-k_{0,z'}) \leq k \leq k' \leq n $ such that $n^{-1}e^{2t}\lesssim k-k_{0,z}$, $k'-k \lesssim k-k_{0,z}$ and $\zeta,\zeta'\in \RR$, 
\begin{align*} \Big|\sum_{\ell=k+1}^{k'} \frac{\sin(\zeta+\theta^z_{k, \ell})\sin(\zeta'+\theta^{z'}_{k,\ell})}{\sqrt{(\ell - k_{0,z})(\ell-k_{0,z'})}} \Big |&\lesssim_\kappa \frac{e^t n^{-\frac12}}{\sqrt{k-k_{0,z}}}+\frac{\sqrt{k'-k_{0,z}} -\sqrt{k-k_{0,z}}}{\sqrt{n}}\\
& + \frac{e^{-t}n^{\frac12}  (k'-k)^3}{(k-k_{0,z})^{5/2}}+ \Big( \frac{k'-k}{k-k_{0,z}}\Big)^2,\end{align*}
where  $\theta^z_{k,\ell} = \sum_{j=k}^\ell \theta^z_j$ and $\theta^{z'}_{k,\ell} = \sum_{j=k}^\ell \theta^{z'}_j$ for any $\ell\geq k$. 
\end{lemma}

\begin{proof}Let $k_{0,z'} + 2(k_{0,z}-k_{0,z'}) \leq k \leq k' \leq n$. Note that since $k_{0,z'}\leq k_{0,z}$, the map $\ell \mapsto  1/\sqrt{(\ell-k_{0,z})(\ell-k_{0,z'})}$ is $O(1/(k-k_{0,z})^2)$-Lipschitz when $\ell \geq k_{0,z'} + 2(k_{0,z}-k_{0,z'})$. Hence, 
\begin{align*}  \Big|\sum_{\ell=k+1}^{k'} \frac{\sin(\zeta+\theta^z_{k, \ell})\sin(\zeta'+\theta^{z'}_{k,\ell})}{\sqrt{(\ell - k_{0,z})(\ell-k_{0,z'})}} \Big|  \lesssim \frac{1}{k-k_{0,z}} &\big|\sum_{\ell=k+1}^{k' } \sin(\zeta+\theta^z_{k, \ell})\sin(\zeta'+\theta^{z'}_{k,\ell})\big|  
   + \Big( \frac{k'-k}{k-k_{0,z}}\Big)^2,
\end{align*}
where we used that $k_{0,z'}\leq k_{0,z}$. We write 
\[  \sin(\zeta+\theta^z_{k, \ell})\sin(\zeta'+\theta^{z'}_{k,\ell}) = \frac12\big(\cos(\zeta-\zeta' +2\theta_{k,\ell}^z+\psi_{k,\ell}) - \cos(\zeta+\zeta'+\psi_{k,\ell})\big),\]
 where $\psi_{k,\ell} := \sum_{j=k+1}^\ell \psi_j$, with $\psi_j := \theta_j^{z'}-\theta_j^z$, and compute separately both resulting sums.  
Now, by definition $\theta_\ell^z = \arcsin(\sqrt{4-z_\ell^2}/2)$ where $z_\ell = z \sqrt{n/\ell}$ 
 for any $\ell>k_{0z,}$. Since 
  \[ \partial_z  \theta_\ell^z= -\frac12 \sqrt{\frac{n/\ell}{4-z_\ell^2}}, \quad  \partial_\ell\partial_z\theta_\ell^z = \frac{2\sqrt{n/\ell}}{\ell(4 - z_\ell^2)^{3/2}}, \ z\in (-2,2), \ell>k_{0,z}, \]
denoting by $\psi_\ell = \theta_\ell^z-\theta_\ell^{z'}$ we find that
%for any $\ell >k_{0,z}$, 
 \begin{equation} \label{eqpsiestim} \psi_\ell \asymp \frac{|z-z'|}{\sqrt{4-z_\ell^2}} \asymp \frac{e^{-t} n^{\frac{1}{2}}}{\sqrt{\ell-k_{0,z}}}, \quad |\psi_\ell- \psi_k| \lesssim \frac{k-\ell}{k-k_{0,z}} \psi_k, \ k_{0,z}<k \leq \ell,  \end{equation}
 where we used \eqref{estim1}. Write $\psi_{k,\ell} := \sum_{j = k+1}^\ell \psi_j$ for any $\ell\geq k$. Then, 
 \begin{align} \sum_{\ell=k+1}^{k'} \cos(\zeta-\zeta'+ &\psi_{k,\ell}) = \sum_{\ell=k+1}^{k'} \cos(\zeta-\zeta'+ (\ell-k) \psi_{k}) +  O\Big( (k'-k) \psi_{k} \frac{(k'-k)^2}{k-k_{0,z}}\Big)\nonumber \\
& =  \sum_{\ell=k+1}^{k'} \cos(\zeta-\zeta'+ (\ell-k) \psi_{k}) +   O\Big(  e^{-t}n^{\frac12}  \frac{(k'-k)^3}{(k-k_{0,z})^{3/2}}\Big),\label{linepsi}
\end{align}
where we used \eqref{eqpsiestim} in the last equality. 
Next, comparing the last sum to an integral, we find that
\begin{align}
\sum_{\ell=k+1}^{k'} \cos(\zeta-\zeta'+ (\ell-k) \psi_{k})& = \int_{0}^{k'-k} \cos(\zeta-\zeta' + \psi_{k} x) dx + O((k'-k) \psi_k) \\
& =   O\Big(\frac{1}{\psi_{k}}\Big)   +O((k'-k) \psi_k)  \nonumber\\
 & = O\big(e^{t}n^{-\frac12} \sqrt{k-k_{0,z}}\big) + O\Big(e^{-t}n^{\frac12} \frac{k'-k}{\sqrt{k-k_{0,z}}}\Big),\label{intcomp}
\end{align}
where we used again \eqref{eqpsiestim} in the last equality.
Putting \eqref{linepsi} and \eqref{intcomp} together,  using the facts that $k'-k\lesssim k-k_{0,z}$ and that $t\geq \frac23 \uptau-O_\kappa(1)$, we deduce that 
\begin{equation} \label{sum1}\frac{1}{k-k_{0,z}} \sum_{\ell=k+1}^{k'} \cos(\zeta-\zeta'+ \psi_{k,\ell})  = O\Big(\frac{e^{t}n^{-\frac12}}{\sqrt{k-k_{0,z}}}\Big) +   O\Big(  e^{-t}n^{\frac12}  \frac{(k'-k)^3}{(k-k_{0,z})^{3/2}}\Big).\end{equation}
We now turn our attention to the second sum involving the sum of the angles $\theta^z$ and $\theta^{z'}$.
Define blocks $(\ell_{j,z})_{j\geq 1}$ defined by $\ell_{j,z} = k_{0,z} + \lfloor n^{1/3} j^{2/3}\rfloor$ for any $j\geq 1$. To ease the notation, write $\ell_j$ instead of $\ell_{j,z}$ for any $j$. Note that $[\ell_j,\ell_{j+1}]\subset [k+1,k']$ implies that $j\gtrsim (k-k_{0,z})^{3/2} n^{-1/2}$. Thus, we can throw the two $\ell_j$'s blocks at the beginning and the end of  $[k+1,k']$ which may overlap this interval without being included up to a small error:
\begin{equation} \label{smallblock} \frac{1}{k-k_{0,z}} \sum_{\ell=k+1}^{k'} \cos(\zeta +\zeta'+ 2\theta^z_{k,\ell} + \psi_{k,\ell}) = \sum_{ j : [\ell_j,\ell_{j+1}] \subset [k+1,k']} S_{j} + O\big( n^{1/2} (k-k_{0,z})^{-3/2}\big), \end{equation}
 where we used that for $j$ such that $[\ell_j,\ell_{j+1}]\subset [k+1,k']$, $\Delta \ell_j/(k-k_{0,z}) \lesssim 1/j \lesssim n^{1/2} (k-k_{0,z})^{-3/2}$, and where 
 \[ S_{j} := \frac{1}{k-k_{0,z}} \sum_{\ell=\ell_j+1}^{\ell_{j+1}} \cos( \zeta_j  +2\theta^z_{\ell_j,\ell} + \psi_{\ell_j,\ell}),\]
 with $\zeta_j := \zeta+\zeta' + 2\theta^z_{k,\ell_i} + \psi_{k,k_j}$. 
Now, fix some $j$ such that $[\ell_j,\ell_{j+1}]\subset [k+1,k']$.  From \eqref{eqpsiestim} and the fact that $\Delta \ell_{j+1}/(\ell_j-k_{0,z}) \lesssim 1$ we deduce that $|\psi_{\ell_j,\ell}|\lesssim |\psi_{\ell_j} \Delta \ell_{j+1}|$ for any $\ell\in [\ell_j,\ell_{j+1}]$. Using that  $j\gtrsim (k-k_{0,z})^{3/2}n^{-1/2}$ and again \eqref{eqpsiestim}, one can check that $|\psi_{\ell_j,\ell}| \lesssim e^{-t} n/(k-k_{0,z})$ for any $\ell\in [\ell_j,\ell_{j+1}]$.
Thus, we can neglect the contribution of $\psi_{k_j,\ell}$ to the sum $S_{j}$:
 \begin{equation} \label{errorpsi}     S_{j} = \frac{1}{k-k_{0,z}} \sum_{\ell=\ell_j+1}^{\ell_{j+1}} \cos(\zeta_{j} + 2\theta^z_{k_j,k}) + O\Big(ne^{-t} \frac{\Delta \ell_{j+1}}{(k-k_{0,z})^2}\Big).\end{equation}
 Applying  Lemma \ref{approxint}, it follows that 
  \begin{equation} \label{sumsmall}  \frac{1}{k-k_{0,z}} \sum_{\ell=\ell_j+1}^{\ell_{j+1}} \cos(\zeta_{j} + 2\theta^z_{\ell_j,\ell})= O\Big(\frac{1}{j^2}\Big) +O\Big(\frac{1}{n^{1/3}j^{2/3}}\Big).\end{equation}
Since $\max\{j : [\ell_j, \ell_{j+1} \subset [k+1,k']\} \lesssim n^{1/2}(k'-k_{0,z})^{-3/2}$ and $\min\{j : [\ell_j, \ell_{j+1} \subset [k+1,k']\} \gtrsim n^{1/2}(k-k_{0,z})^{-3/2}$,  we deduce by summing \eqref{sumsmall} over $j$ and using \eqref{smallblock}-\eqref{errorpsi}  that
\begin{align}&\Big|\frac{1}{k-k_{0,z}} \sum_{\ell=k+1}^{k'} \cos(\zeta +\zeta'+ 2\theta^z_{k,\ell} + \psi_{k,\ell})\Big| \nonumber \\
&\lesssim \frac{n^{1/2}}{(k-k_{0,z})^{3/2}} +  \frac{\sqrt{k'-k_{0,z}}-\sqrt{k-k_{0,z}}}{\sqrt{n}}+  \frac{ne^{-t}(k'-k)}{(k-k_{0,z})^2} \nonumber\\
& \lesssim_\kappa \frac{e^tn^{-\frac12}}{\sqrt{k-k_{0,z}}}+  \frac{\sqrt{k'-k_{0,z}}-\sqrt{k-k_{0,z}}}{\sqrt{n}},\label{sum2}
\end{align}
where we used in the last equality the fact that $k'-k\lesssim k-k_{0,z}$, $k-k_{0,z}\gtrsim n^{-1}e^{2t}$ and that $t\geq \frac23\uptau-O_\kappa(1)$.
Therefore, putting together \eqref{sum1} and \eqref{sum2}, this ends the proof of the claim. 
 \end{proof}

Equipped with Lemma \ref{approxintcorr},  we can now give a proof of Proposition \ref{decore} in the case where $t\geq \frac23 
\uptau-5\log \kappa$.

\begin{lemma}\label{decorelliptic}
Let  $z,z'\in I_\eta$ and $\frac23\uptau-5\log \kappa\leq t \leq T_z$  such that  $|z-z'|\leq e^{-t}$.  For any $|\lambda|,|\mu|\leq \kappa$ and $n_{t,z} \leq k  \leq  k' $, 
\[ \EE_{k} \Big[ e^{\lambda (\psi_{k'}(z)-\psi_k(z)) +\mu( \psi_{k'}(z')-\psi_k(z'))}\Big] \leq e^{\frac{\lambda^2}{2}(\widehat{\Sigma}_{k',z}^2-\widehat{\Sigma}_{k,z}^2)  +\frac{\mu^2}{2} (\widehat{\Sigma}_{k',z'}^2-\widehat{\Sigma}^2_{k,z'})   + \mathfrak{C}_\kappa (|\lambda|^3\vee |\mu|^3\vee 1)},\]
where $\mathfrak{C}_\kappa>0$ depends on  $\kappa$.

\end{lemma}

\begin{proof}
Note that by Lemma \ref{approxtimechange}, there exists $\mathfrak{m}>0$ depending on the model parameters such that $n_{t-\mathfrak{m},z'}\leq n_{t,z}$  if $|z-z'|\leq e^{-t}$. Thus, it suffices to prove the claim when $|z'|\leq |z|$, which we will assume from now on.

Next, we claim that it is enough to prove the statement
for $k\geq n_{t,z} \vee( k_{0,z}+\ell_0)$. Indeed, if $n_{t,z} \leq k_{0,z} +\ell_0$ then as $t \geq \frac23 \uptau-O_\kappa(1)$, $|k_{0,z} -n_{t,z}|\lesssim_\kappa n^{1/3}$  by Lemma \ref{approxtimechange}. Moreover, $k_{0,z}-k_{0,z'}\lesssim_\kappa n^{1/3}$ so that we have as well that $|n_{t,z}-k_{0,z'}|\lesssim_\kappa n^{1/3}$.  
As a result, we deduce from the a priori estimate of Proposition \ref{apriori} and the Cauchy-Schwarz inequality that 
\[ \log \EE_{k} \Big[ e^{\lambda (\psi_{k'}(z)-\psi_k(z)) +\mu( \psi_{k'}(z')-\psi_k(z'))}\Big] \leq  \mathfrak{C}_\kappa (\lambda^2\vee \mu^2\vee1),\]
for any $ |\lambda|,|\mu| \leq (\log n)^{-\mathfrak{h}} n^{1/6}$, where $\mathfrak{C}_\kappa>0$ depends on $\kappa$ and $\mathfrak{h}>0$ on the model parameters. Hence, we can assume $k\geq n_{t,z} \vee (k_{0,z}+\ell_0)$.

Set $m_n := (n_{t,z}-k_{0,z})\vee \ell_0$. Now, define blocks $(L_i)_{1\leq i\leq i_*}$ as 
\begin{equation} \label{defblockcorr} L_i := k_{0,z} + \lfloor i^{3} m_n\rfloor, \ 1\leq i< i_*, \ L_{i_*} := n,\end{equation}
where $i_* :=  \max\{i : k_{0,z} +\lfloor i^5 m_n\rfloor <  n\}$ .
Note that $\widehat{\Sigma}^2_{L_{i+1},z} - \widehat{\Sigma}^2_{L_i,z} \lesssim 1$ and $\widehat{\Sigma}^2_{L_{i+1},z'} - \widehat{\Sigma}^2_{L_i,z'} \lesssim 1$ for any $i\geq 1$,  since $k-k_{0,z}\asymp k-k_{0,z'}$ for $k\geq k_{0,z}+m_n$. Hence, using Cauchy-Schwarz inequality and the exponential moment estimate of Proposition \ref{expopreciseincre}, it is actually enough to consider the case where $k' = L_{j'}$ and $k= L_j$ for some $1\leq j\leq j'$. Denote by $Z_{p,q}^{\lambda,\mu}$ the random variable
\[ Z_{p,q}^{\lambda,\mu} := \lambda (\psi_{L_{q}}(z) - \psi_{L_p}(z)) + \mu (\psi_{L_{q}}(z) - \psi_{L_p}(z)), \ \lambda,\mu\in \RR, \ p\leq q.\]
Recall that $\delta_i =i^{-1/4}$. Say that the {\em $i^{\text{th}}$ block is good} if the following event holds:
%for \corO{all}  $L_i < k \leq L_{i+1}$,
\begin{equation} \label{defgoodblockcor} \mathscr{G}_{i,z,z'}:=\Big\{\big|\zeta^z_k - \zeta^z_{L_i} - \sum_{\ell= L_i+1}^k \theta^z_\ell\big| \leq \delta_i, \quad \big|\{\zeta^{z'}_k - \zeta^{z'}_{L_i} - \sum_{\ell= L_i+1}^k \theta^{z'}_\ell\big| \leq \delta_i,
\quad 
\forall L_i < k \leq L_{i+1}\Big\}.\end{equation}
%and denote by $\mathscr{G}_{i,z,z'}$ this event. 
Using Lemma \ref{probagoodblockgene} and a union bound, we find that for any $1\leq i\leq i_*$, 
\begin{equation} \label{probagoodblockcorr} \PP_{L_i} (\mathscr{G}_{i,z,z'}^\complement) \leq e^{-\mathfrak{c}i^{1/2}},\end{equation}
where $\mathfrak{c}$ is a constant depending on the model parameters and where we used that $i_*\lesssim n^{2/15}$ as $m_n \gtrsim_\kappa n^{1/3}$. Applying Lemma \ref{represellipticg}, we deduce that on the event $\mathscr{G}_{i,z,z'}$, for $y\in \{z,z'\}$
\begin{equation} \label{repres} \psi_{L_{i+1}}(y) - \psi_{L_{i}}(y) = \sum_{\ell= L_{i}+1}^{L_{i+1}} (w_{\ell-1}^y(\widehat{\zeta}_{\ell-1}^y) +\mathcal{P}_\ell^y) + O(i^{-5/4}),\end{equation}
where $\widehat{\zeta}_{\ell-1}^y = \zeta_{L_{i}}^y + \sum_{j= L_{i}+1}^{\ell-1} \theta_j^y$ for any $\ell\in (L_{i},L_{i+1}]$ and  $\mathcal{P}^y$ is a martingale satisfying that 
\begin{equation} \label{estimerrorP} s_{\ell-1}(\mathcal{P}_\ell^y) \lesssim \frac{i^{-1/4}}{\ell- k_{0,z}} + \frac{1}{(\ell-k_{0,z})^2},\end{equation}
where we used that $L_{i+1}-L_i \lesssim j^4 m_n$,  $m_n\asymp_\kappa n^{-1}e^{2t}$ and that $t\geq \frac23 T- O_\kappa(1)$.
Set $H_\ell^y = w_\ell^y(\widehat{\zeta}_{\ell-1}^y) + \mathcal{P}_\ell^y$ for $y\in \{z,z'\}$ and $\ell \geq k_{0,z}+m_n$ and define 
$\tau$ to be the largest index $i \in \llbracket j,j'-1 \rrbracket$ such that the $i^{\text{th}}$ block is bad, if there are any, and set $\tau=0$ otherwise. Fix  $i \in \llbracket j,j'-1 \rrbracket$. From \eqref{estimerrorP}, we  can find $\mathfrak{h}>0$ depending on the model parameters such that $|w_{\ell-1}^y(\widehat{\zeta}_{\ell-1}^y) +\mathcal{P}_\ell^y|\lesssim (\log n)^{\mathfrak{h}}/\sqrt{\ell-k_{0,z}}$ for any $\ell\geq L_1$ and $y \in \{z,z'\}$. We deduce from \eqref{repres} and Lemma \ref{moddev}, that for $|\mu|,|\lambda| \leq  n^{1/6} (\log n)^{-2\mathfrak{h}}$ and $p\geq 1$, 
 \begin{align} \label{expomomentjointgood} \log \EE_{L_p}\big[\Car_{\mathscr{G}_{p,z,z'}}e^{\lambda \Delta \psi_{L_{p+1}}(z)+\mu \Delta \psi_{L_{p+1}}(z') }\big]& \leq \frac12\mathrm{Var}_{L_p}\Big( \sum_{\ell =L_{p}+1}^{L_{p+1}} (\lambda H^z_\ell + \mu H^{z'}_\ell)\Big)   +\upsilon_p (1\vee |\lambda|^3\vee |\mu|^3), \end{align}
  where $(\upsilon_p)_p$ is a sequence of positive real numbers such that $\sum_{p\geq 1}\upsilon_p \lesssim 1$ and we used the fact that $\sum_{i\geq 1} i^{-5/4} \lesssim 1$ and $\sum_{\ell \geq k_{0,z}+\ell_0} (\log n)^{\mathfrak{h}}/(\ell-k_{0,z})^{3/2}\lesssim 1$.   We claim that for any $p\geq 1$, 
  \begin{equation} \label{claimvarHy} \mathrm{Var}_{L_p}\Big( \sum_{\ell =L_{p}+1}^{L_{p+1}} (\lambda H^z_\ell + \mu H^{z'}_\ell)\Big) \leq  \lambda^2 (\widehat{\Sigma}^2_{L_{p+1},y} - \widehat{\Sigma}^2_{L_{p},y} )+ \mu^2 (\widehat{\Sigma}^2_{L_{p+1},y} - \widehat{\Sigma}^2_{L_{p},y}) + \veps_p(\mu^2\vee \lambda^2),\end{equation}
  where $(\veps_p)_{p\geq 1}$ is a deterministic sequence of positive real numbers such that $\sum_{p\geq 1}\veps_p\lesssim1$.
Let $y \in \{z,z'\}$. Using   \eqref{varclaimeq}-\eqref{varclaimeqw}, we find that 
  \[ \mathrm{Var}_{L_p}\Big(\sum_{\ell= L_p+1}^{L_{p+1}} w_\ell^y(\widehat{\zeta}_{\ell-1})\Big)  = \widehat{\Sigma}^2_{L_{p+1},y} - \widehat{\Sigma}^2_{L_{p},y} + O_\kappa\big(p^{-15/2}\big) + O_\kappa\Big(\frac{\sqrt{L_{p+1}}-\sqrt{L_p}}{\sqrt{n}}\Big),\]
  where we used that $m_n \asymp_\kappa n^{-1}e^{2t}$ and $L_{p+1}-k_{0,z'} \asymp L_{p+1}- k_{0,z}$ in the case where $y=z'$. 
  In particular, this implies that $\mathrm{Var}_{L_p} (\sum_{\ell= L_p+1}^{L_{p+1}} w_\ell^y(\widehat{\zeta}_{\ell-1}))  \lesssim 1/p$. Moreover, we find using \eqref{estimerrorP} and the fact that $\mathcal{P}^y$ is a martingale that   $\mathrm{Var}_{L_p}(\sum_{\ell=L_p+1}^{L_{p+1}} \mathcal{P}_\ell^y) \lesssim p^{-5/4}$. Using Cauchy-Schwarz inequality to bound the covariance,  it follows that  
  \begin{equation} \label{estimvarH}\mathrm{Var}_{L_p}\Big( \sum_{\ell = L_p+1}^{L_{p+1}} H_\ell^y\Big)  = \widehat{\Sigma}^2_{L_{p+1},y} - \widehat{\Sigma}^2_{L_{p},y}  + O_\kappa(p^{-9/8})+O_\kappa\Big(\frac{\sqrt{L_{p+1}}-\sqrt{L_p}}{\sqrt{n}}\Big).\end{equation}
 We now compute the conditional covariance between $\sum_{\ell= L_p+1}^{L_{p+1}} H_\ell^z$ and $\sum_{\ell= L_p+1}^{L_{p+1}}  H_\ell^{z'}$. First, we note that the contributions of $\mathcal{P}^z$ and $\mathcal{P}^{z'}$ are negligible. Indeed, using Cauchy-Schwarz inequality,  the fact that $\mathrm{Var}_{L_p}(\sum_{\ell= L_p+1}^{L_{p+1}} w_\ell^y(\widehat{\zeta}^y_{\ell-1})) \lesssim 1/p$,  $\mathrm{Var}_{L_p}(\sum_{\ell=L_p+1}^{L_{p+1}} \mathcal{P}_\ell^y) \lesssim p^{-5/4}$, $y\in \{z,z'\}$, and that $(w_\ell^z(\widehat{\zeta}_{\ell-1}^z) w_\ell^z(\widehat{\zeta}_{\ell-1}^{z'}))_{L_p <\ell \leq L_{p+1}}$ are independent conditionally on $\mathcal{F}_{L_p}$, we get that 
 \begin{equation} \label{estimcovH} \EE_{L_p}\big[ \sum_{\ell=L_p+1}^{L_{p+1}}H_\ell^z H_\ell^{z'}\big] =\sum_{\ell=L_p+1}^{L_{p+1}} \EE_{L_p}[w_\ell^z(\widehat{\zeta}_{\ell-1}^z) w_\ell^z(\widehat{\zeta}_{\ell-1}^{z'}) ] + O(p^{-9/8}). \end{equation}
 Next, recall the definition of $w_\ell^y$ from \eqref{defw}. Using the fact that for $y\in \{z,z'\}$, $\theta_k^y \lesssim \sqrt{(\ell-k_{0,y})/n}$ by \eqref{diffRI}, $\EE[d_\ell^2] \lesssim 1/n$, $\EE[c_{\ell,y}^2] \lesssim 1/(\ell-k_{0,y})$, and that $\EE[(c_{\ell,y} - g_\ell/\sqrt{\ell-k_{0,y}})^2] \lesssim 1/n +1/(\ell-k_{0,z})^2$ by Lemma \ref{noisec}, where $g_\ell$ is a centered random variable independent of $z$ of variance $2v$, we deduce that 
 \[ \EE_{L_p}[w_\ell^z(\widehat{\zeta}_{\ell-1}^z) w_\ell^z(\widehat{\zeta}_{\ell-1}^{z'}) ] =  \frac{v\sin(2\widehat{\zeta}^z_{\ell-1}) \sin(2\widehat{\zeta}^{z'}_{\ell-1})}{2\sqrt{(\ell-k_{0,z})(\ell-k_{0,z'})}}  +  O\Big(\frac{1}{\sqrt{n(\ell-k_{0,z})}}\Big) + O\Big(\frac{1}{(\ell-k_{0,z})^{3/2}}\Big).\]
Summing over $\ell$ in the $p^{\text{th}}$-block it follows that 
 \begin{equation} \label{covH1} \EE_{L_p}\big[ \sum_{\ell=L_p+1}^{L_{p+1}}H_\ell^z H_\ell^{z'}\big] = \sum_{\ell=L_p+1}^{L_{p+1}}\frac{v\sin(\widehat{\zeta}^z_{\ell-1}) \sin(\widehat{\zeta}^{z'}_{\ell-1})}{2\sqrt{(\ell-k_{0,z})(\ell-k_{0,z'})}}+  O(p^{-9/8})+O\Big(\frac{\sqrt{L_{p+1}}-\sqrt{L_p}}{\sqrt{n}}\Big). \end{equation}
 But Lemma \ref{approxintcorr} gives us that 
 \begin{equation} \label{covH2}  \Big| \sum_{\ell=L_p+1}^{L_{p+1}}\frac{2v\sin(\widehat{\zeta}^z_{\ell-1}) \sin(\widehat{\zeta}^{z'}_{\ell-1})}{\sqrt{(\ell-k_{0,z})(\ell-k_{0,z'})}}\Big| \lesssim_\kappa p^{-3/2}.\end{equation}
Putting together \eqref{estimvarH}, \eqref{estimcovH}, \eqref{covH1}, \eqref{covH2}, noting that all errors are summable, we get the claim \eqref{claimvarHy}. Coming back to \eqref{expomomentjointgood}, it follows that 
\begin{equation} \label{estimgoodjointsim}  \log \EE_{L_i}\big[e^{Z_{i+1,j'}^{\lambda,\mu}}\Car_{\bigcap_{p= i+1}^{j'-1} \mathscr{G}_{p,z,z'}}\big] \leq \frac{\lambda^2}{2} (\widehat{\Sigma}^2_{L_{j'},z} - \widehat{\Sigma}^2_{L_{i+1},z} )+ \frac{\mu^2}{2} (\widehat{\Sigma}^2_{L_{j'},z'} - \widehat{\Sigma}^2_{L_{i+1},z'}) + \mathfrak{C}_\kappa(1\vee |\lambda|^3\vee |\mu|^3), \end{equation}
where $\mathfrak{C}_\kappa>0$ depends on the model parameters and on $\kappa$. Further, using Cauchy-Schwarz inequality, the a priori exponential moment estimate of Proposition \ref{apriori} and \eqref{probagoodblockcorr}, we get that for any $|\lambda|,|\mu|\leq (\log n)^{-\mathfrak{h}} n^{1/6}$, 
\[ \log \EE_{L_j}\big[e^{Z_{j,i+1}^{\lambda,\mu}} \Car_{\mathscr{G}_{i,z,z'}^{\complement}}\big] \leq \mathfrak{C}(\lambda^2\vee \mu^2) \log i -\mathfrak{C}^{-1} i^{1/2},\]
where $\mathfrak{C},\mathfrak{h}>0$ depend on the model parameters and changed values without changing names and where we used the fact that $\sum_{\ell = L_j+1}^{L_i} 1/(\ell-k_{0,z}) \lesssim \log i$ and $\sum_{\ell = L_j+1}^{L_i} 1/(\ell-k_{0,z'}) \lesssim \log i$. Assuming that $|\lambda|,|\mu|\leq \kappa$, there exists $\mathfrak{C}'_\kappa$ depending on $\kappa$ such that 
\[ \log \EE_{L_j}\big[e^{Z_{j,i+1}^{\lambda,\mu}} \Car_{\mathscr{G}_{i,z,z'}^{\complement}}\big] \leq - \frac{\mathfrak{C^{-1}}}{2} i^{1/2} + \mathfrak{C}'_\kappa.\]
Hence, combined with \eqref{estimgoodjointsim} the above estimate gives that 
\[  \log \EE_{L_j}\big[e^{Z_{j,j'}^{\lambda,\mu}}\Car_{\tau = i}\big] \leq \frac{\lambda^2}{2} (\widehat{\Sigma}^2_{L_{j'},z} - \widehat{\Sigma}^2_{L_{j},z} )+ \frac{\mu^2}{2} (\widehat{\Sigma}^2_{L_{j'},z'} - \widehat{\Sigma}^2_{L_{j},z'}) - \frac{\mathfrak{C}^{-1}}{2} i^{1/2}+ \mathfrak{C}_\kappa(1\vee |\lambda|^3\vee |\mu|^3),\]
for any $i \in \llbracket j,j'-1\rrbracket$, $|\lambda|,|\mu|\leq \kappa$ and where $\mathfrak{C}_\kappa>0$ depends on $\kappa$ and changed value without changing name. Clearly, the same estimate remains true for $i=0$, meaning that all blocks are good. Exponentiating the above estimate and summing over $i$ finally gives the claim.
\end{proof}

We are now ready to give a proof Proposition \ref{decore}.

\begin{proof}[Proof of Proposition \ref{decore}]  With the same argument as in the proof of Lemma \ref{decorelliptic}, we can assume that $|z|\leq |z'|$. Observe that using Lemma \ref{approxtimechange} and the fact that $k_{0,z'}-k_{0,z} \asymp n e^{-t}$,  there exists $\mathfrak{m}>0$  depending on the model parameters  such that
\begin{equation} \label{chainn1} n_{t-\mathfrak{m},z'} \leq n_{t,z} \leq n_{\uptau-\frac{t}{2}-\mathfrak{m},z} \leq n_{t+\mathfrak{m},z'} \leq k_{0,z'} \leq k_{0,z'} + \ell_0 \leq n_{\uptau-\frac{t}{2},z'} \leq n_{\uptau-\frac{t}{2} +\mathfrak{m},z}.\end{equation}
Note that since $\uptau-t/2\geq \frac23\uptau +2\log \kappa$, we get from Lemma \ref{decorelliptic} that for any $|\lambda|,|\mu|\leq \kappa$, 
 \[ \EE_{\ell} \Big[ e^{\lambda (\psi_{{\ell}'}(z)-\psi_\ell(z)) +\mu( \psi_{\ell'}(z')-\psi_\ell(z'))}\Big] \leq e^{\frac{\lambda^2}{2}(\widehat{\Sigma}_{\ell',z}^2-\widehat{\Sigma}_{\ell,z}^2)  +\frac{\mu^2}{2} (\widehat{\Sigma}_{\ell',z'}^2-\widehat{\Sigma}^2_{\ell,z'})   + \mathfrak{C}_\kappa (|\lambda|^3\vee |\mu|^3\vee 1)},\]
where $\ell = n_{\uptau-\frac{t}{2},z'}\vee k$, $\ell' = k' \wedge n_{\uptau-\frac{t}{2},z'}$ and $\mathfrak{C}_\kappa>0$ depends on  $\kappa$. Hence, we only need to compute the joint exponential moment of the increments between time $k$ and $k'$ such that $n_{t,z} \leq k \leq k' \leq n_{\uptau-\frac{t}{2},z'}$.
Now, \eqref{chainn1} and the fact that $\sigma_{k,y}^2\gtrsim 1/|k-k_{0,z}|$ for $|k-k_{0,y}|>\ell_0$ imply that 
 \begin{equation} \label{chainn}\sum_{j= n_{t,z}+1}^{n_{t+\mathfrak{m},z'}} \frac{1}{ k_{0,z'}-j } \lesssim 1, \sum_{j = n_{t+\mathfrak{m},z'}+1}^{n_{T-\frac{t}{2},z'}} \frac{1}{j-k_{0,z} }  \lesssim 1. \end{equation}
 In view of \eqref{chainn},  the joint exponential moment of the increments of $\psi(z)$ and $\psi(z')$ between time $k\wedge n_{t+\mathfrak{m},z'}$ and $k' \wedge n_{t+\mathfrak{m},z'}$ can be reduced to an exponential moment of a perturbation of the increment of $\psi(z')$, and symmetrically the joint exponential moment of the increments of $\psi(z)$ and $\psi(z')$ between time $(k\vee n_{t+\mathfrak{m},z'})\wedge k'$ and $k'$ boils down to an exponential moment of a perturbation of the increment of $\psi(z')$.
Indeed, by Lemma \ref{aprioridom}, we know that for any $1\leq j \leq n$,  
\[ \Delta \psi_j(z') \leq \xi_j^{z'}, \ \Delta \psi_j(z') \leq \xi_j^{z'},\]
where $\xi^y$ is an adapted process such that for $y\in \{z,z'\}$, 
\begin{equation} \label{estimtQ} |\EE_{j-1}(\xi^y_{j})|\lesssim \frac{\kappa}{|j-k_{0,y}|\vee n^{1/3}} +\veps_j, \ s_{j-1}(\xi^y_\ell) \lesssim \frac{\kappa}{|j-k_{0,y}|\vee n^{1/3}},\end{equation}
where $(\veps_j)_{j\geq 1}$ is a deterministic sequence of positive numbers such that $\sum_{j\geq 1} \veps_j \lesssim 1$. Since we are looking at an exponential moment with positive parameters, it suffices to compute the joint exponential moment of $\psi_{\ell'}(z) - \psi_k(z)$ and $\sum_{j=k+1}^{\ell'} \xi_j^{z'}$ and the conditional joint exponential moment of $\psi_{k'}(z') - \psi_{\ell'}(z')$ and $\sum_{j=\ell'+1}^\ell \xi_j^{z}$, where $\ell' = (n_{t+\mathfrak{m},z'}\wedge k')\vee k$. The inequalities \eqref{chainn} and \eqref{estimtQ} already tell us that $\xi^{z'}$ accumulates a variance of order at most $1$ in the interval $[k, \ell']$ and that $\xi^z$ as well in the interval $(\ell',k']$. Thus, to apply Lemma \ref{expomomentpreciseperturb} we only need to check that the sum of the conditional covariances of $\psi_j(z)-\psi_{j-1}(z)$ and $\xi^{z'}_j$ in $[k, \ell']$ vanishes and similarly for  $\psi_j(z)-\psi_{j-1}(z)$ and $\xi^{z'}_j$ in the interval $(\ell', k']$. 
Now, since $k_{0,z'}-n_{t+\mathfrak{m},z'} \gtrsim n e^{-t}$, $k_{0,z} -n_{t,z} \asymp n e^{-t}$ and $n_{t+\mathfrak{m},z'} -k_{0,z} \asymp n e^{-t}$ by Lemma \ref{approxtimechange} and   $k_{0,z'} - k_{0,z} \asymp n e^{-t}$, we get that 
\[ \sum_{j=n_{t,z}+1}^{n_{t+\mathfrak{m},z'}} [(k_{0,z'}-j)(|j-k_{0,z}|\vee n^{1/3})]^{-\frac12} \lesssim n^{-1/2}e^{t/2} \sum_{j=n_{t,z}+1}^{n_{t+\mathfrak{m},z'}} [|j-k_{0,z}|\vee n^{1/3}]^{-\frac12}  \lesssim 1.\]
Similarly, we find that 
\[ \sum_{j=n_{t+\mathfrak{m},z'}+1}^{n_{\uptau-\frac{t}{2}+\mathfrak{m},z'}} [(k_{0,z}-j)(|j-k_{0,z'}|\vee n^{1/3})]^{-\frac12}  \lesssim 1.\]
Invoking Lemma \ref{expomomentpreciseperturb} ends the proof of the proposition.
\end{proof}

\section{Proof of the lower bound of Theorem \ref{reduc}}
\label{sec-LBproof}
The goal of this section is to prove the lower bound in Theorem \ref{reduc}.
%show the lower bound part of the asymptotic expansion of Theorem \ref{reduc}.
By Lemma \ref{truncation}, it suffices to prove  the same statement where $\psi_n(z)$ is replaced by $\overline{\Psi}_{\uptau_\veps}^z$ defined in \eqref{defpsitrunc} (recall that $\uptau_\veps = \uptau-t_\veps$ where $t_\veps=\lfloor \veps \log \uptau\rfloor$). Further, as $\mathcal{N}_{\uptau_\veps} \subset \mathcal{N}$, 
it is enough to prove that  for any $\veps>0$, 
\begin{equation} \label{condequivlb} \lim_{n\to+\infty} \PP\Big( \max_{z \in \mathcal{N}_{\uptau_\veps}} \overline{\Psi}_{\uptau_\veps}(z) \leq \sqrt{v}(\uptau_\veps-t_\veps) - \Big(\frac{3\sqrt{v}}{4}+ \veps\Big)\log \uptau\Big) \underset{n\to+\infty}{\longrightarrow} 0.\end{equation}
(Note the extra subtraction of factors of the form $\veps \log \uptau$ in the right hand side of the inequality in 
\eqref{condequivlb}; tightening our results would require working with a less generous slack.) To prove \eqref{condequivlb}, we will use a classical second moment method, for which the decorrelation estimates of Section \ref{decorrsection} are crucial.   First,
we consider $\overline{\Psi}_{\uptau_\veps}$ as the endpoint of the trajectory of a process  $\overline{\Psi}_t^z$ defined as 
\[ \overline{\Psi}_t^z := \log \Big\| \prod_{\ell= n_{t_\veps,z}+1}^{n_{t,z}} \widehat{\Xi}_\ell^z e_1\Big\|, \ t_\veps \leq t \leq \uptau_\veps,\]
%where $t_\veps = \lfloor\veps \log \uptau\rfloor$, $\uptau_\veps = T-t_\veps$, and 
where $\widehat{\Xi}_\ell^z:= \Xi_\ell^z/\alpha_{\ell,z}$ if $\ell \leq k_{0,z}-\ell_0$ while  $\widehat{\Xi}_\ell^z:= \Xi_\ell^z$ otherwise. 
Note that $(\overline{\Psi}_t^z)_{t\in \llbracket t_\veps,\uptau_\veps\rrbracket}$ has the same distribution as 
$(\Psi_t^z - \Psi_{t_\veps}^z)_{t \in \llbracket t_\veps,\uptau_\veps\rrbracket }$ under the conditional distribution given that $ W_{n_{t_\veps,z}}^z = 0$. Let $\mathcal{T}_{\mathfrak{q},\veps} :=   \llbracket t_\veps,\uptau_\veps\rrbracket \setminus \llbracket t_\mathfrak{q}^-,t_\mathfrak{q}^+\rrbracket$. Define for any $\veps>0$ and $t \in\mathcal{T}_{\mathfrak{q},\veps}$ the interval $\mathcal{B}_{t}$ as 
\[ \mathcal{B}_{t} :=   \begin{cases}
 \big(-\infty , -(t-t_\veps)^{\mathfrak{d}} +\mathfrak{f}   \big] & \text{ if } t\leq t_\mathfrak{q}^-,\\   
\big[-(\uptau_\veps-t )^{\mathfrak{e}} - \veps \log \uptau -\mathfrak{f}, - (\uptau_\veps-t )^{\mathfrak{d}} +\mathfrak{f}\big] & \text{ if } t\geq t_\mathfrak{q}^+,
  \end{cases}
  \]
where $0<\mathfrak{d}<\frac12 < \mathfrak{e}$  and $\mathfrak{f}\geq 1$ are  constants depending on the model parameters, to be chosen below.
Next, set $\alpha_\uptau:= 1-  \frac{3}{4} \frac{\log \uptau}{\uptau_\veps-t_\veps}$ and for any 
$z\in \mathcal{N}_{\uptau_\veps}$, introduce the event $\mathscr{A}_{z,\veps}$ as 
\begin{align*} \mathscr{A}_{z,\veps}:=    \Big\{ & \overline{\Psi}_{t}^z \in  \sqrt{v}\alpha_\uptau(t-t_\veps) + \mathcal{B}_{t},  \forall t\in \mathcal{T}_{\mathfrak{q},\veps}, \ \overline{\Psi}_{\uptau_\veps}^z \in \sqrt{v}\alpha_\uptau(\uptau_\veps-t_\veps)  -  \big[\veps \log \uptau,
 \veps \log \uptau+1\big], \\
&  \max_{|t-s|\leq \mathfrak{u} \atop t_\veps \leq s,t \leq t_\mathfrak{q}^- } |\overline{\Psi}_t^z - \overline{\Psi}_s^z|\leq \mathfrak{A}\sqrt{\log \uptau},  \max_{|t-s|\leq \mathfrak{u} \atop t_\mathfrak{q}^+ \leq s,t \leq \uptau_\veps } |\overline{\Psi}_t^z - \overline{\Psi}_s^z|\leq \mathfrak{A}\sqrt{\log \uptau}\Big\},
\end{align*}
where $\mathfrak{p},\mathfrak{u},\mathfrak{A}>0$ are constants to be chosen later.
In other words, $\mathscr{A}_{z,\veps}$ represents the event that the process $\overline{\Psi}^z$ stays in a certain region below  a
linear barrier delineated by curved barriers (nicknamed "bananas")
% the linear barrier 
and in addition ends close to the conjectured value of the maximum, up to an error $\varepsilon \log \uptau$, and that the values of $\overline{\Psi}^z$ over intervals of length $\mathfrak{u}$ are at most at distance $\mathfrak{A}\sqrt{\log \uptau}$ from one another. The shape of the region is chosen so that it contains typical trajectories of a Brownian motion constrainted to remain below a linear curve. 
The  points $z$ for which the event $\mathscr{A}_{z,\veps}$ holds are the \textit{good} points discussed in Section \ref{highlevel}.

Let $Z_\veps := \sum_{z \in \mathcal{N}_{\uptau_\veps}} \Car_{\mathscr{A}_{z,\veps}}$, which counts the number of good net points.
Clearly,  \eqref{condequivlb} is equivalent to showing that $\PP(Z_\veps>0) \to 1$ as $n\to+\infty$. 
The second-moment method relies on the  Paley-Zigmund inequality \cite[Lemma 3.1]{Kallenberg}, which asserts 
(through an application of Cauchy-Schwarz) that
\begin{equation} \label{PZ}  \PP(Z_\veps>0) \geq  \frac{(\EE Z_\veps)^2}{\EE (Z_\veps^2)}.\end{equation} 
Therefore, it is enough to show that 
%the \corO{the right hand side of \eqref{PZ} tends to $1$ as $n\to\infty$. second moment is comparable to the first moment square, meaning that
\begin{equation} \label{2momentmethod} \frac{(\EE Z_\veps)^2}{\EE (Z_\veps^2)} \underset{n\to+\infty}{\longrightarrow} 1.\end{equation}
The rest of this section is devoted to the proof of \eqref{2momentmethod}.
%prove this asymptotics. 
We start by computing a lower bound on the first moment of $Z_\veps$.
\begin{lemma}[First moment lower bound]\label{firstmoment}
\[ \EE(Z_\veps) \gtrsim_\veps (\log \uptau) \uptau^{\frac{2\veps}{\sqrt{v}}-\frac32}  e^{\uptau_\veps -\alpha_\uptau^2 (\uptau_\veps-t_\veps)} \gtrsim (\log \uptau) \uptau^{\frac{2\veps}{\sqrt{v}}+ \veps}. \]
\end{lemma} 
\begin{proof}
Since $\#\mathcal{N}_{\uptau_\veps} \lesssim e^{\uptau_\veps}$, it suffices to prove that for any $z\in I_\eta$,
$\PP(\mathscr{A}_{z,\veps}) \gtrsim_\veps  (\log \uptau) \uptau^{\sqrt{v \veps}-\frac32}  e^{-\alpha_\uptau^2 (\uptau_\veps-t_\veps)}$.
To prove this lower bound we will use the Gaussian coupling of Lemma \ref{coupling} and a Ballot Theorem for Gaussian processes. 
%Let $z \in I_\eta$. 
Since $(\overline{\Psi}_t^z)_{t\in \llbracket t_\veps,\uptau_\veps\rrbracket}$ has the same distribution as $(\Psi_t^z - \Psi_{t_\veps}^z)_{t\in \llbracket t_\veps,\uptau_\veps\rrbracket }$ under the conditional distribution given that $W_{n_{t_\veps,z}}^z =0$, it follows by Lemma \ref{coupling} that there exists a centered Gaussian process $\Gamma$ with covariance $\mathrm{Cov}(\Gamma_t,\Gamma_s) =  \frac{v}{2} (t \wedge s)$ for any $1\leq s,t \leq T_z$, whose increments outside of the window $[t_\mathfrak{q}^-,t_\mathfrak{q}^+]$ are at most at distance $\mathfrak{C}>0$ from the increments of $\overline{\Psi}^z$ with overwhelmingly probability. Denote by $\mathscr{C}$ the event where the coupling succeeds, that is, 
\[ \mathscr{C} := \big\{ \sup_{t_\veps \leq t \leq t_\mathfrak{q}^-} |\overline{\Psi}_t^z - \Gamma_t-\Gamma_{t_\veps}| \leq \mathfrak{C},  \sup_{t_\mathfrak{q}^+ \leq t \leq \uptau_\veps} |\overline{\Psi}_t^z - \Gamma_t - \Gamma_{t_\mathfrak{q}^+}|\leq \mathfrak{C}\big\}. \]
By Lemma \ref{coupling}, the probability of failure of the coupling is smaller than our target lower bound. Hence, it suffices to prove the same lower bound for the probability of a certain barrier event for the Gaussian process $\Gamma$. More precisely,  define the event $\mathscr{G}_\veps^-$ where   $\Gamma-\Gamma_{t_\veps}$ stays below a curved barrier until time $t_\mathfrak{q}^-$ and ends below the linear barrier at a distance of order $T^{1/2}$, 
\begin{align*} \mathscr{G}_\veps^- := \big\{&\Gamma_t - \Gamma_{t_\veps} \leq  \sqrt{v} \alpha_\uptau (t-t_\veps) - (t-t_\veps)^{\mathfrak{d}} , \forall t\in \llbracket t_\veps,t_\mathfrak{q}^-\rrbracket , \\
&\Gamma_{t_\mathfrak{q}^-}-\Gamma_{t_\veps} \in \sqrt{v} \alpha_\uptau (t_\mathfrak{q}^--t_\veps) + [-\mathfrak{b} \uptau^{1/2},-\mathfrak{b}^{-1} \uptau^{1/2}]\big\},\end{align*}
where  $\mathfrak{b}>0$ is to be chosen later,
and let $\mathscr{G}_\veps^+$ denote the event
\begin{align*} \mathscr{G}_\veps^+ := \big\{&\Gamma_t -\Gamma_{t_\mathfrak{q}^+} +\overline{\Psi}_{t_\mathfrak{q}^+} \in  \sqrt{v}\alpha_\uptau(t-t_\veps) + \widetilde{\mathcal{B}}_t, \forall t \in \llbracket t_\mathfrak{q}^+,\uptau_\veps \rrbracket,  \\
&\Gamma_{\uptau_\veps} -\Gamma_{t_\mathfrak{q}^+} +\overline{\Psi}^z_{t_\mathfrak{q}^+}  -\sqrt{v}\alpha_\uptau(\uptau_\veps-t_\veps) \in [- \frac{\veps}{2} \log \uptau, -  \frac{\veps}{2} \log \uptau+1]\big\},
\end{align*}
where  $\widetilde{\mathcal{B}}_t := [- (\uptau_\veps-t )^{\mathfrak{e}}-\frac{\veps}{2} \log T-\mathfrak{f},-(\uptau_\veps-t )^{\mathfrak{d}}]$ for any $t\geq t_\mathfrak{q}^+$. That is, on  $\mathscr{G}_\veps^+$,
 the process $\Gamma-\Gamma_{t_\mathfrak{q}^+}+ \overline{\Psi}_{t_\mathfrak{q}^+}$ stays in a certain region below the linear 
barrier and ends at time $\uptau_\veps$ in a window of order $1$ at a distance $\veps \log \uptau$ below the linear curve.
%, that is,  
Further, define the event $\mathscr{C}'$ that the coupling between $\Gamma_{t_\mathfrak{q}^-}-\Gamma_{t_\veps}$ and $\overline{\Psi}_{t_\mathfrak{q}^-}$ succeeds and by  $\mathscr{E}$ the event where $\overline{\Psi}^z$ jumps at most of order $t_\mathfrak{q}^+-t_\mathfrak{q}^-$ between times $t_\mathfrak{q}^-$ and $t_\mathfrak{q}^+$:
   \[ \mathscr{C}' := \big\{|\Gamma_{t_\mathfrak{q}^-} - \Gamma_{t_\veps}-\overline{\Psi}^z_{t_\mathfrak{q}^-}|\leq \mathfrak{C}\big\}, \quad \mathscr{E}:= \big\{\overline{\Psi}^z_{t_\mathfrak{q}^+}-\overline{\Psi}^z_{t_{\mathfrak{q}^-}} \leq \mathfrak{h}(t_\mathfrak{q}^+ -t_{\mathfrak{q}}^-)\big\},\]
   where $\mathfrak{h}>0$ is to be chosen later.
  Finally, denote by $\mathscr{D}^-,\mathscr{D}^+$ the events
\[ \mathscr{D}^- := \big\{ \max_{|t-s|\leq \mathfrak{u} \atop t_\veps \leq s,t \leq t_\mathfrak{q}^- } |\Gamma_t - \Gamma_s|\leq \mathfrak{A}\sqrt{\log \uptau}-2\mathfrak{C}\big\}, \quad \mathscr{D}^+:=\big\{  \max_{|t-s|\leq \mathfrak{u} \atop t_\mathfrak{q}^+ \leq s,t \leq \uptau_\veps } |\Gamma_t - \Gamma_s|\leq \mathfrak{A}\sqrt{\log \uptau} -2\mathfrak{C}\big\}.\]
Let $\mathscr{G}_\veps = \mathscr{G}_\veps^-\cap \mathscr{D}^-\cap \mathscr{C}'\cap \mathscr{E}\cap \mathscr{G}_\veps^+\cap \mathscr{D}^+$. One can check that provided $\mathfrak{f}$ is large enough depending on $\mathfrak{C}$, we have that  $\mathscr{G}_\veps \cap \mathscr{C} \subset \mathscr{A}_{\veps,z}\cap \mathscr{C}$. 
Since   $\PP(\mathscr{C}^\complement)\leq e^{-\Omega(\uptau^2)}$ by Lemma \ref{coupling},  it suffices to prove that 
\[ \PP(\mathscr{G}_{\veps}) \gtrsim_\veps (\log \uptau) \uptau^{\frac{\veps}{\sqrt{v}} -\frac32} e^{-\alpha_\uptau^2(\uptau_\veps-t_\veps)}.\]
Recall the augmented filtration $(\mathcal{G}_k)_{1\leq k \leq n}$ from Lemma \ref{coupling}. To ease the notation, write $\PP_t$ and $\EE_t$ the conditional probability and expectation given $\mathcal{G}_{n_{t,z}}$ for any $t$. Now, using a Gaussian change of measure, we find that
\begin{equation} \label{lowerbfirst} \PP_{t_\mathfrak{q}^+}\big(\mathscr{G}_\veps^+\cap \mathscr{D}^+\big) \gtrsim \uptau^{\frac{\veps}{\sqrt{v}} } e^{-\alpha_\uptau^2(\uptau_\veps-t_\mathfrak{q}^+)- 2\alpha_\uptau^2 (t_\mathfrak{q}^+-t_\veps) +\frac{2\alpha_\uptau}{\sqrt{v}} \overline{\Psi}_{t_\mathfrak{q}^+}^z} \PP_{t_\mathfrak{q}^+}\big(\widetilde{\mathscr{G}}_\veps^+\cap \widetilde{\mathscr{D}}^+\big),\end{equation}
where $\widetilde{\mathscr{D}}^+:= \big\{\max_{|t-s|\leq \mathfrak{u}\atop t_\mathfrak{q}^+ \leq s,t \leq \uptau_\veps} |\Gamma_t-\Gamma_s|\leq \mathfrak{A}\sqrt{\log \uptau}/2\}$ and $\widetilde{\mathscr{G}}_\veps^+ $ is the event:
\begin{align*}
\widetilde{\mathscr{G}}_\veps^+  := \big\{&\Gamma_t -\Gamma_{t_\mathfrak{q}^+} +\overline{\Psi}_{t_\mathfrak{q}^+} -  \sqrt{v}\alpha_\uptau(t_\mathfrak{q}^+-t_\veps) \in  \widetilde{\mathcal{B}}_t, \forall t \in \llbracket t_\mathfrak{q}^+,\uptau_\veps \rrbracket,  \\
&\Gamma_{\uptau_\veps} -\Gamma_{t_\mathfrak{q}^+} +\overline{\Psi}^z_{t_\mathfrak{q}^+}  -\sqrt{v}\alpha_\uptau(t_\mathfrak{q}^+-t_\veps) \in [- \frac{\veps}{2} \log \uptau, - \frac{\veps}{2} \log \uptau+1]\big\},
\end{align*}
Using a union bound, we find  that for $\mathfrak{A}$ large enough depending on $\mathfrak{u}$, 
\begin{equation} \label{complemDtilde} \PP_{t_\mathfrak{q}^+}\big((\widetilde{\mathscr{D}}^+)^{\complement}\big) \leq \uptau^{-\mathfrak{A}^2/16\mathfrak{u}}.\end{equation}
Next, note that on the event $\mathscr{G}_\veps^-\cap \mathscr{C}' \cap \mathscr{E}$, $\sqrt{v} \alpha_\uptau (t_\mathfrak{q}^+-t_\veps) -\overline{\Psi}_{t_\mathfrak{q}^+}^z \asymp_\mathfrak{b} \uptau^{1/2}$. Hence, 
applying the Ballot Theorem \ref{ballot} to the process $\Gamma_{\uptau_\veps-t}-\Gamma_{\uptau_\veps}$, $t\in \llbracket 1,\uptau_\veps-t_\mathfrak{q}^+\rrbracket$,  we get that on the event $\mathscr{G}_\veps^-\cap \mathscr{C}' \cap \mathscr{E}$, 
  \begin{equation} \label{lbBplus} \PP_{t_\mathfrak{q}^+}(\widetilde{\mathscr{G}}_{\veps}^+) \gtrsim_{\veps } \frac{\log \uptau}{\uptau},\end{equation}
  where we used the fact that $\uptau_\veps-t_\mathfrak{q}^+\gtrsim \uptau$. 
Combining \eqref{lbBplus} and \eqref{complemDtilde}, we conclude that $\PP_{t_\mathfrak{q}^+}(\widetilde{\mathscr{G}} _\veps^+\cap \widetilde{\mathscr{D}}^+) \gtrsim_{\veps} \log \uptau/\uptau$ provided $\mathfrak{A}$ is large enough. Plugging this estimate in \eqref{lowerbfirst} and using that $\overline{\Psi}_{t_\mathfrak{q}^-}^z \geq  \Gamma_{t_\mathfrak{q}^-} -\Gamma_{t_\veps}-\mathfrak{C}$ on the event $\mathscr{G}_\veps$, we get that 
    \begin{align} \PP_{t_\mathfrak{q}^-}\big(\mathscr{G}_\veps\big) & \gtrsim_{\veps} \frac{\log \uptau}{\uptau}  \uptau^{\frac{\veps}{\sqrt{v}} } e^{-\alpha_\uptau^2(\uptau_\veps-t_\mathfrak{q}^+)-{2} \alpha_\uptau^2 (t_\mathfrak{q}^+-t_\veps)}e^{\frac{2\alpha_\uptau}{\sqrt{v}}(\Gamma_{t_\mathfrak{q}^-} -\Gamma_{t_\veps})}\Car_{\mathscr{G}_\veps^-\cap \mathscr{D}^-\cap \mathscr{C}'} \nonumber \\
&\times \EE_{t_\mathfrak{q}^-}\Big[\Car_{\mathscr{E}} e^{\frac{2\alpha_\uptau}{\sqrt{v}} (\overline{\Psi}_{t_\mathfrak{q}^+}^z-\overline{\Psi}^z_{t_\mathfrak{q}^-}) }\Big] \label{lbmoment1} . \end{align}
We claim that 
\begin{equation}\label{claimmomentlb1} \EE_{t_\mathfrak{q}^-}\Big[\Car_{\mathscr{E}} e^{\frac{2\alpha_\uptau}{\sqrt{v}} (\overline{\Psi}_{t_\mathfrak{q}^+}^z-\overline{\Psi}^z_{t_\mathfrak{q}^-}) }\Big] \gtrsim   e^{\alpha_\uptau^2(t_\mathfrak{q}^+-t_\mathfrak{q}^-)}.\end{equation}
Now, using Proposition \ref{expopreciseincre}  we know that the above estimate holds true for the unrestricted exponential moment.  Thus, it only remains to show that the exponential moment restricted to $\mathscr{E}^\complement$ is of smaller order. Indeed, using Chernoff's inequality with $\lambda \geq (4/\sqrt{v} c_z) \vee 1$ and the a priori estimate of Proposition \ref{apriori}, we find that 
\begin{align*} \EE_{n_{t_\mathfrak{q}^-,z}}\big[\Car_{\mathscr{E}^\complement}e^{\frac{2\alpha_\uptau}{\sqrt{v}} (\overline{\Psi}_{t_\mathfrak{q}^+}^z-\overline{\Psi}^z_{t_\mathfrak{q}^-}) }\big]& \leq e^{-\lambda \mathfrak{h} (t_\mathfrak{q}^+-t_\mathfrak{q}^-)}\EE_{n_{t_\mathfrak{q}^-,z}} \big[ e^{ (\lambda + \frac{2\alpha_\uptau}{\sqrt{v}})(\overline{\Psi}_{t_\mathfrak{q}^+}^z-\overline{\Psi}^z_{t_\mathfrak{q}^-})} \big]  \\
&\lesssim e^{-(\lambda \mathfrak{h} -\lambda^2 \mathfrak{C'}) (t_\mathfrak{q}^+-t_\mathfrak{q}^-)},\end{align*}
where $\mathfrak{C}'>0$ depends on the model parameters. Choosing $\lambda = \mathfrak{h}/(2\mathfrak{C}')$ and $\mathfrak{h}$ large enough, we get the claim \eqref{claimmomentlb1}. Coming back to \eqref{lbmoment1}, this yields that 
\begin{align} \PP(\mathscr{G}_\veps) &\gtrsim_\veps  \frac{\log \uptau}{\uptau} \uptau^{\frac{\veps}{\sqrt{v}}}e^{-\alpha_\uptau^2(\uptau_\veps -t_\mathfrak{q}^-)-2\alpha_\uptau^2(t_\mathfrak{q}^--t_\veps)} \EE\Big[e^{\frac{2\alpha_\uptau}{\sqrt{v}}(\Gamma_{t_\mathfrak{q}^-} -\Gamma_{t_\veps})}\Car_{\mathscr{G}_\veps^-\cap \mathscr{D}^-\cap \mathscr{C}'}\Big].  \label{lbinterm} \end{align}
Since $\PP((\mathscr{C}')^\complement)\leq e^{-\mathfrak{c} \uptau^2}$, where $\mathfrak{c}>0$ depends on the model parameters by Proposition \ref{coupling} and $t_\mathfrak{q}^- - t_\veps \lesssim \uptau$, we deduce by using the
Cauchy-Schwarz inequality   that 
\begin{equation} \label{badeventexpo} \EE\Big[e^{\frac{2}{\sqrt{v}}(\Gamma_{t_\mathfrak{q}^-} -\Gamma_{t_\veps})}\Car_{(\mathscr{C}')^\complement}\big]\leq e^{-\frac{\mathfrak{c}}{2} \uptau^2}.\end{equation}
Further, using a Gaussian change of measure we find that 
\begin{equation} \label{changeballot} \EE\Big[e^{\frac{2\alpha_\uptau}{\sqrt{v}}(\Gamma_{t_\mathfrak{q}^-} -\Gamma_{t_\veps})}\Car_{\mathscr{G}_\veps^-\cap \mathscr{D}^-}\big] \gtrsim \PP\big(\widetilde{\mathscr{G}}_\veps^-\cap \widetilde{\mathscr{D}}^-\big)e^{\alpha_\uptau^2(t_\mathfrak{q}^--t_\veps)},\end{equation}
where $\widetilde{\mathscr{G}}^-_\veps := \big\{\Gamma_{t-t_\veps} \leq - (t-t_\veps)^{\mathfrak{d}} , \forall t\in  \llbracket t_\veps,t_\mathfrak{q}^-\rrbracket,\ \Gamma_{t_\mathfrak{q}^+-t_\veps} \in [-\mathfrak{b} \uptau^{1/2},-\mathfrak{b}^{-1} 
\uptau^{1/2}] \big\}$ and $\widetilde{\mathscr{D}}^-:=\{\max_{|t-s|\leq \mathfrak{u}\atop t_\veps \leq s,t \leq t_\mathfrak{q}^-} |\Gamma_t-\Gamma_s|\leq \mathfrak{A}\sqrt{\log \uptau}/2\}$. Similarly as in \eqref{complemDtilde}, we have on the one hand that 
\begin{equation} \label{probaDminuscompl} \PP\big((\widetilde{\mathscr{D}}^-)^\complement\big) \leq \uptau^{-\mathfrak{A}^2/16\mathfrak{u}},\end{equation}
for $\mathfrak{A}$ large enough. On the other hand, using a union bound to discretize the values of $\Gamma_{t_\mathfrak{q}^+-t_\veps}$ and the Ballot Theorem \ref{ballot}, we find that 
\begin{equation} \label{probaGminus}   \PP( \widetilde{\mathscr{G}}_\veps^- ) \gtrsim  \uptau^{-1/2},\end{equation}
where we used that $t_\mathfrak{q}^+-t_\veps\gtrsim \uptau$.
Putting together \eqref{probaDminuscompl} and \eqref{probaGminus} entails that $\PP(\widetilde{\mathscr{G}}_\veps^- \cap \widetilde{\mathscr{D}}^-) \gtrsim \uptau^{-1/2}$. Plugging this estimate back into \eqref{changeballot} and combining this lower bound with \eqref{lbinterm}, \eqref{badeventexpo} ends the proof of the claim. \end{proof}

We now turn our attention to proving an upper bound on the second moment of $Z_\veps$ matching the square of its first moment. Expanding $Z_\veps^2$ and gathering pairs $(z,z')$ according to their distances, we can write that 
\begin{equation} \label{moment2} \EE[Z_\veps^2] = S_o +S_1+S_2+S_3,\end{equation}
where 
\[ S_i := 
 \sum_{(z,z') \in  \mathcal{H}_i} \PP(\mathscr{A}_{z,\veps} \cap \mathscr{A}_{z',\veps}),\quad i\in \{o,1,2,3\},\]
 with $\mathcal{H}_i$ denotes the subset of pairs $(z,z')\in \mathcal{N}_{\uptau_\veps}$ such that 
 \begin{itemize}
 \item $i=o$: $-\log |z-z'|< t_\veps +\mathfrak{p}$ (Very early branching),
  \item $i=1$: $t_\veps+\mathfrak{p}\leq -\log |z-z'|< t_\veps +(\mathfrak{h} \log \uptau)^{1/\mathfrak{d}}$ (Early branching),
   \item $i=2$: $\uptau_\veps -(\mathfrak{h} \log \uptau)^{1/\mathfrak{d}}  \leq -\log |z-z'|< t_\veps +(\mathfrak{h} \log \uptau)^{1/\mathfrak{d}}$ (Bulk branching),
    \item $i=3$: $-\log |z-z'|< \uptau_\veps -(\mathfrak{h} \log \uptau)^{1/\mathfrak{d}} $ (Late branching),
 \end{itemize}
  where $\mathfrak{h},\mathfrak{p}>0$ are some constants to be chosen later. 
  The terminology of very early, early, bulk and late branching coming from the fact that  if  $|z-z'|\asymp e^{-t}$ then the recursions for $z$ and $z'$ ``branch'' at time $n_{t,z}\simeq n_{t,z'}$ in the sense that they essentially become independent at the level of exponential moments. 
  
Below, we analyse separately the four regimes and show that $S_o\leq (\EE[Z_\veps])^2$ provided $\mathfrak{p}>0$ is well-chosen, and that $S_1,S_2,S_3 = o_\veps(\EE[Z_\veps])^2$ for $\mathfrak{h}$ large enough, hence proving \eqref{2momentmethod}.

\subsubsection*{Very early branching} We claim that there exists $\mathfrak{p}>0$ depending on the model parameters such that for $n$ large enough depending on $\veps$, for any $z,z'\in I_\eta $ such that $  |z-z'|\geq e^{\mathfrak{p}-t_\veps}$, we have that $[n_{t_\veps,z},n_{\uptau_\veps,z}]\cap [n_{t_\veps,z'},n_{\uptau_\veps,z'}]=\emptyset$. Indeed, $k_{0,x} - n_{t_\veps,x} \lesssim n (\log n)^{-\veps}$ and $n_{\uptau_\veps,x} -k_{0,x} \lesssim n(\log n)^{-\veps}$ for $x\in \{z,z'\}$ according to Lemma \ref{approxtimechange} while $|k_{0,z}-k_{0,z'}| \gtrsim |z-z'| n$. Recalling that $t_\veps =\lfloor \veps \log \log n\rfloor$, it follows that if $\mathfrak{p}$ is large enough and $|z-z'|\geq e^{\mathfrak{p}-t_\veps}$, $[n_{t_\veps,z},n_{\uptau_\veps,z}]\cap [n_{t_\veps,z'},n_{\uptau_\veps,z'}]=\emptyset$. 
Therefore, for any  $z,z' \in \mathcal{N}_{\uptau_\veps}$ and $|z-z'|\geq e^{-t_\veps + \mathfrak{p}}$ the events $\mathscr{A}_{z,\veps}$ and $\mathscr{A}_{z',\veps}$ are independent, so that 
\begin{equation} \label{S1} S_o  = \sum_{(z,z') \in  \mathcal{H}_o} \PP(\mathscr{A}_{z,\veps})\PP(\mathscr{A}_{z',\veps})\leq (\EE[Z_\veps])^2.\end{equation}
\subsubsection*{Early Branching} Let $t_\veps  -\mathfrak{p} \leq t <t_\veps+ (\mathfrak{h} \log \uptau)^{1/\mathfrak{d}}$ and $z,z'\in \mathcal{N}_{\uptau_\veps}$ such that $|z-z'| \asymp e^{-t}$ and $|z|\leq |z'|$. 
From \eqref{chainn1}, we know that 
\[  n_{t-\mathfrak{m},z'} \leq n_{t,z } \leq n_{\uptau-\frac{t}{2}-\mathfrak{m},z} \leq n_{t+\mathfrak{m},z'} \leq k_{0,z'}  \leq n_{\uptau-\frac{t}{2},z'} \leq n_{\uptau-\frac{t}{2}  +\mathfrak{m},z},\]
where $\mathfrak{m}>0$ depends on the model parameters. Assume for the moment that $T-t/2+\mathfrak{m} \leq \uptau_\veps$. Then, on the event $\mathscr{A}_{z,\veps}\cap \mathscr{A}_{z',\veps}$, we have that
\[ \overline{\Psi}_{\uptau_\veps}^z - \overline{\Psi}_{\uptau-\frac{t}{2}+\mathfrak{m}}^z \geq  \alpha_\uptau \sqrt{v} (\uptau_\veps -t_\veps) +  \veps \log \uptau -  \overline{\Psi}_{\uptau-\frac{t}{2}-\mathfrak{m}}^z-\mathfrak{A}\sqrt{\log \uptau}, \] 
and $\overline{\Psi}_{\uptau_\veps}^{z'} -\overline{\Psi}_{\uptau-\frac{t}{2}}^{z'}   \geq  \alpha_\uptau \sqrt{v} (\uptau_\veps -t_\veps)-\overline{\Psi}_{\uptau-\frac{t}{2}}^{z'} +  \veps \log \uptau$. Since  $n_{\uptau-\frac{t}{2}-\mathfrak{m},z}\leq n_{\uptau-\frac{t}{2},z'}$, we get by using Chernoff inequality with parameter $2\alpha_\uptau/\sqrt{v}$, the joint exponential moment estimate of Proposition \ref{decore} and Lemma \ref{approxtimechange} $(iv)$ that 
\begin{align}\PP_{n_{\uptau-\frac{t}{2},z'}}\big( \mathscr{A}_{z,\veps} \cap \mathscr{A}_{z',\veps}\big)&\lesssim   e^{\frac{2\mathfrak{A}}{\sqrt{v}}\sqrt{\log \uptau}}\uptau^{\frac{4\veps}{\sqrt{v}}} e^{-4\alpha_\uptau^2 (\uptau_\veps-t_\veps) + \frac{2\alpha_\uptau}{\sqrt{v}} \overline{\Psi}_{\uptau-\frac{t}{2}-\mathfrak{m}}^z +  \frac{2\alpha_\uptau}{\sqrt{v}}\overline{\Psi}_{\uptau-\frac{t}{2}}^{z'}}\nonumber\\
& \times e^{\alpha_\uptau^2(t-2t_\veps) }\Car_{\widetilde{\mathscr{A}}_{z,\veps}\cap \widetilde{\mathscr{A}}_{z',\veps}}, \label{chernoffjoint} 
\end{align}
where we used the fact that $\uptau_\veps- (\uptau-\frac{t}{2}) = \frac{t}{2} -t_\veps$, $\alpha_\uptau =1+O(\log \uptau/\uptau)$ and where $ \widetilde{\mathscr{A}}_{z,\veps}, \widetilde{\mathscr{A}}_{z',\veps} $ are the following simplified barrier events
\begin{align*} \widetilde{\mathscr{A}}_{z,\veps} :=\big\{ &\overline{\Psi}_s^z \leq \sqrt{v}\alpha_\uptau  (s-t_\veps)+\mathfrak{f}, \forall s \in \llbracket t_\veps,\uptau-\frac{t}{2} -\mathfrak{m}\rrbracket \setminus \llbracket t_\mathfrak{q}^-,t_\mathfrak{q}^+\rrbracket, \\
& \overline{\Psi}_{\uptau-\frac{t}{2}-\mathfrak{m}}^z \in \sqrt{v}\alpha_\uptau \big(\uptau-\frac{t}{2}-t_\veps\big) + \big[-t^{\mathfrak{e}'}-\veps \log \uptau ,-t^{\mathfrak{d'}} ]\big\},  
\end{align*}
\begin{align*} \widetilde{\mathscr{A}}_{z',\veps}: =\big\{ 
&\overline{\Psi}_{t+\mathfrak{m}}^{z'} \geq \sqrt{v} \alpha_\uptau(t+\mathfrak{m}-t_\veps) - t^{\mathfrak{d}'}, \\
&\overline{\Psi}_s^{z'} \leq \sqrt{v}\alpha_\uptau  (s-t_\veps) +\mathfrak{f}, \forall s \in \llbracket t+\mathfrak{m},\uptau-\frac{t}{2}\rrbracket \setminus \llbracket t_\mathfrak{q}^-,t_\mathfrak{q}^+\rrbracket, \\
& \overline{\Psi}_{T-\frac{t}{2}}^{z'} \in \sqrt{v}\alpha_\uptau \big(\uptau-\frac{t}{2}-t_\veps\big) + \big[-t^{\mathfrak{e}'}-\veps \log \uptau,-t^{\mathfrak{d}'} \big],   \end{align*}
with $0<\mathfrak{d}'<\mathfrak{d}<\frac12<\mathfrak{e}<\mathfrak{e}' $ and where we used that $t\geq t_\veps-\mathfrak{p}$.
By Proposition \ref{coupling}, we know that there exists a centered Gaussian process  whose increments outside $[t_\mathfrak{q}^-,t_\mathfrak{q}^+]$ are within a constant $\mathfrak{C}$ from the ones of $\overline{\Psi}^{z'}$ with overwhelming probability.  Since the probability of failure of the coupling is  $e^{-\Omega(\uptau^2)}$ by Lemma \ref{coupling} whereas exponential moments  of $\overline{\Psi}_t^{z'}$ for bounded parameters are at most $e^{O(\uptau)}$ according to Proposition \ref{expopreciseincre}, it follows  by using the Cauchy-Schwarz inequality that  
\begin{equation} \label{couplz} \EE_{n_{t+\mathfrak{m},z'}}\big[\Car_{\widetilde{\mathscr{A}}_{z',\veps}} e^{\frac{2\alpha_\uptau}{\sqrt{v}}(\overline{\Psi}_{\uptau-\frac{t}{2}}^{z'} - \overline{\Psi}_{t+\mathfrak{m}}^{z'})}\big] \lesssim \EE_{n_{t+\mathfrak{m},z'}}\big[\Car_{\mathscr{B}} e^{\frac{2\alpha_\uptau}{\sqrt{v}} \Gamma_{\uptau-\frac{3t}{2} - \mathfrak{m}}}\big] + e^{-\mathfrak{c} \uptau^2},\end{equation}
where $\mathfrak{c}>0$ depends on the model parameters, and under $\PP_{n_{t+\mathfrak{m},z'}}$, $\Gamma$ is a centered Gaussian process with covariance given by $\mathrm{Cov}(\Gamma_s,\Gamma_{s'}) = \frac{v}{2}   (s\wedge s')$ for any $s,s' \in \llbracket 1,T_{z'}\rrbracket$  and 
\begin{align*} \mathscr{B}:=\big\{ &\Gamma_{s - t-\mathfrak{m}} + \overline{\Psi}_{t+\mathfrak{m}}^{z'} \leq \sqrt{v} \alpha_\uptau (s-t_\veps)+2\mathfrak{f}, \forall s \in \llbracket t+\mathfrak{m},\uptau-\frac{t}{2}\rrbracket \setminus \llbracket t_\mathfrak{q}^-,t_\mathfrak{q}^+\rrbracket, \\
&  \Gamma_{\uptau -\frac{3t}{2}+\mathfrak{m}} + \overline{\Psi}_{t+\mathfrak{m}}^{z'} \leq \sqrt{v} \alpha_\uptau \big(\uptau-\frac{t}{2} -t_\veps\big) + [-t^{\mathfrak{e}'}-\veps \log \uptau,-t^{\mathfrak{d}'}] \big\},
\end{align*}
where $\mathfrak{d}',\mathfrak{e'}$  changed values without changing name and $\mathfrak{f}>\mathfrak{C}$ is the constant appearing in Proposition \ref{coupling}.
Using a Gaussian change of measure, we find that 
\[ \EE_{n_{t+\mathfrak{m},z'}}\big[\Car_{\mathscr{B}} e^{\frac{2\alpha_\uptau}{\sqrt{v}} \Gamma_{\uptau-\frac{3t}{2} - \mathfrak{m}}} \big]  \lesssim  e^{\alpha_\uptau^2 (\uptau-\frac{3}{2}t)} \PP_{n_{t+\mathfrak{m},z'}}\big(\widetilde{\mathscr{B}}\big),\]
where $\widetilde{\mathscr{B}}$ is the event defined by 
\begin{align*} \widetilde{\mathscr{B}} := \big\{ &\Gamma_{s - t+\mathfrak{m}} + \overline{\Psi}_{t+\mathfrak{m}}^{z'} \leq \sqrt{v} \alpha_\uptau (t+\mathfrak{m}-t_\veps)+\mathfrak{f}, \forall s \in \llbracket t+\mathfrak{m},\uptau-\frac{t}{2}\rrbracket \setminus \llbracket t_\mathfrak{q}^-,t_\mathfrak{q}^+\rrbracket, \\
& \Gamma_{\uptau-\frac{3t}{2} +\mathfrak{m}} + \overline{\Psi}_{t+\mathfrak{m}}^{z'} \in \sqrt{v} \alpha_\uptau (t+\mathfrak{m}-t_\veps) + [-t^{\mathfrak{e}'}-\veps\log \uptau,-t^{\mathfrak{d}'}]\big\}.
\end{align*}
On the event where $ \sqrt{v} \alpha_\uptau (t+\mathfrak{m}-t_\veps)- \overline{\Psi}_{t+\mathfrak{m}}^{z'} \geq 1 $, we get by discretizing the values of $\Gamma_{\uptau-\frac{3t}{2}+\mathfrak{m}}$ and using a union bound together with the Ballot Theorem \ref{ballot2} that 
\[  \PP_{n_{t+\mathfrak{m},z'}}\big(\widetilde{\mathscr{B}}\big) \lesssim (\sqrt{v} \alpha_\uptau (t+\mathfrak{m}-t_\veps) - \overline{\Psi}_{t+\mathfrak{m}}^{z'}) (t^{2\mathfrak{e}'}+ (\log \uptau)^2)\uptau^{-3/2},\]
where we used the fact that $\uptau-\frac32 t-\mathfrak{m} \gtrsim \uptau$.
Write $\mathscr{E}_{z'} := \{ \sqrt{v} \alpha_\uptau (t+\mathfrak{m}-t_\veps)- \overline{\Psi}_{t+\mathfrak{m}}^{z'} \geq  t^{\mathfrak{d}'} \}$. Coming back to \eqref{couplz}, we deduce that  on $\mathscr{E}_{z'}$, 
\begin{align*} \EE_{n_{t+\mathfrak{m},z'}}\big[\Car_{\widetilde{\mathscr{A}}_{z',\veps}} e^{\frac{2\alpha_\uptau}{\sqrt{v} }\overline{\Psi}_{\uptau-\frac{t}{2}}^{z'}}\big]& \lesssim  (\sqrt{v} \alpha_\uptau (t+\mathfrak{m}-t_\veps) - \overline{\Psi}_{t+\mathfrak{m}}^{z'}) e^{\alpha_\uptau^2(\uptau-\frac32t) + \frac{2\alpha_\uptau}{\sqrt{v} } \overline{\Psi}_{t+\mathfrak{m}}^{z'}}(t^{2\mathfrak{e}'}+(\log \uptau)^2)\uptau^{-3/2}\\
& + e^{-\mathfrak{c}T^2 + \frac{2\alpha_\uptau}{\sqrt{v} } \overline{\Psi}_{t+\mathfrak{m}}^{z'}}. \end{align*}
Using the inequality $x e^{-x}\leq 2e^{-x/2}$ for $x = \sqrt{v} \alpha_\uptau(t+\mathfrak{m}-t_\veps) -\overline{\Psi}_{t+\mathfrak{m}}^{z'} \geq t^{\mathfrak{d}'}$, this yields that on $\mathscr{E}_{z'}$, 
\begin{equation} \label{estimzexp} \EE_{n_{t+\mathfrak{m},z'}}\big[\Car_{\widetilde{\mathscr{A}}_{z',\veps}} e^{\frac{2\alpha_\uptau}{\sqrt{v} }\overline{\Psi}_{T-\frac{t}{2}}^{z'}}\big] \lesssim  e^{-t^{\mathfrak{d}'}/3} e^{\alpha_\uptau^2(\uptau-\frac32t)+ 2\alpha_\uptau^2 (t-t_\veps)} \uptau^{-3/2},\end{equation}
where we used the fact that $t\geq t_\veps$.
With a similar argument, we find that 
\begin{equation} \label{estimzexp2} \EE\big[\Car_{\widetilde{\mathscr{A}}_{z,\veps}} e^{\frac{2\alpha_\uptau}{\sqrt{v} }\overline{\Psi}_{\uptau-\frac{t}{2}-\mathfrak{m}}^{z} }\big] \leq (t^{2\mathfrak{e}'}+(\log \uptau)^2)e^{\alpha_\uptau^2 (\uptau-\frac{t}{2}-t_\veps)}\uptau^{-3/2}.\end{equation}
Hence, coming back to \eqref{chernoffjoint}, we deduce by putting together the two above estimates \eqref{estimzexp} -\eqref{estimzexp2}  and using the the facts that $n_{\uptau-\frac{t}{2}-\mathfrak{m},z} \leq n_{t+\mathfrak{m},z'}$ and  $t\geq t_\veps-\mathfrak{p}$ that, 
\begin{align*}
 \PP\big( \mathscr{A}_{z,\veps} \cap \mathscr{A}_{z',\veps}\big)  
& \lesssim \uptau^{\frac{4\veps}{\sqrt{v}}-3} e^{-t^{\mathfrak{d}'}/4} e^{-2\alpha_\uptau^2(\uptau_\veps-t_\veps)} e^{\alpha_\uptau^2(t-t_\veps)}.
\end{align*}
Now, note that the number of pairs $(z,z')$ such that $z,z' \in \mathcal{N}_{\uptau_\veps}$ and $|z-z'| \in (e^{-(t+1)},e^{-t}]$ is at most of order $e^{2\uptau_\veps -t}$. Indeed, the size of $\mathcal{N}_{\uptau_\veps}$ is of order $e^{\uptau_\veps}$ and for each $z\in \mathcal{N}_{\uptau_\veps}$, the number of $z' \in \mathcal{N}_{\uptau_\veps}$ such that $|z-z'|\leq e^{-t}$ is at most of order $e^{\uptau_\veps-t}$ since the distance between consecutive points in $\mathcal{N}_{\uptau_\veps}$ is at least of order $e^{-\uptau_\veps}$. Hence, 
\begin{align*} S_1 &\lesssim \sum_{t = t_\veps -\mathfrak{p}}^{t_\veps + (\mathfrak{h} \log \uptau)^{1/\mathfrak{q}}} e^{2\uptau_\veps-t} \uptau^{\frac{4\veps}{\sqrt{v}}-3} e^{-2\alpha_\uptau^2(\uptau_\veps-t_\veps)} e^{\alpha_\uptau^2(t-t_\veps)} e^{-t^{\mathfrak{d}'}/4}\\
& \lesssim \uptau^{\frac{4\veps}{\sqrt{v}}-3}  e^{2\uptau_\veps - 2\alpha_\uptau^2(\uptau_\veps-t_\veps)} e^{-t_\veps},
 \end{align*}
where we used the fact that $\alpha_\uptau^2 = 1 + O(\log \uptau/\uptau)$ and $t -t_\veps \lesssim (\log \uptau)^{1/\mathfrak{d}}$. In view of the lower bound on the first moment proven in Lemma \ref{firstmoment}, we conclude that 
\[ S_1 \lesssim e^{-t_\veps} (\EE Z_\veps)^2.\]

\subsubsection*{Bulk Branching}We  now estimate $S_2$. Let $t$ be such that $[(t-t_\veps)\wedge (\uptau_\veps-t)]^\mathfrak{d}\geq \mathfrak{h}\log \uptau$ and 
let $z,z' \in \mathcal{N}_{\uptau_\veps}$ such that $e^{-(t+1)}<|z-z'|\leq e^{-t}$. Without loss of generality, we can assume that $n_{t,z} \leq n_{t,z'}$. We distinguish two cases. Assume first that $t\notin [t_\mathfrak{q}^-,t_\mathfrak{q}^+]$. To compute $\PP(\mathscr{A}_{z,\veps}\cap\mathscr{A}_{z',\veps})$ we simply use Chernoff inequality together with our decorrelation estimate of Proposition \ref{decore}, taking advantage of the curved barrier. More precisely, define the events $ {\mathscr{C}}_{z,\veps},  {\mathscr{C}}_{z',\veps}$  as 
\[  {\mathscr{C}}_{z,\veps} :=\big\{ \overline{\Psi}_{\uptau_\veps}^z \geq \sqrt{v} \alpha_\uptau (\uptau_\veps-t_\veps) - {\veps} \log \uptau\big\}  \]
\[  {\mathscr{C}}_{z',\veps}:=\big\{ \overline{\Psi}_{\uptau_\veps}^{z'} - \overline{\Psi}_{t}^{z'} \geq   \sqrt{v}\alpha_\uptau (\uptau_\veps- t) + [(t-t_\veps)\wedge (\uptau_\veps-t)]^\mathfrak{d} +\mathfrak{f}- {\veps}\log \uptau\big\} \]
Clearly $\mathscr{A}_{z,\veps} \cap \mathscr{A}_{z',\veps} \subset {\mathscr{C}}_{z,\veps}\cap  {\mathscr{C}}_{z',\veps}$.  Using Proposition \ref{decore} with $\lambda =\mu =  2/\sqrt{v}$ and the facts that $\widehat{\Sigma}_{n_{\uptau_\veps,x}}^2 - \widehat{\Sigma}_{n_{t,x}}^2 = \uptau_\veps -t +O(1)$ for any $x$ by Lemma \ref{approxtimechange} $(iv)$ and that $n_{t,z} \leq n_{t,z'}$, we deduce that 
\begin{align} \PP_{n_{t,z}}\big(  {\mathscr{C}}_{z,\veps} \cap  {\mathscr{C}}_{z',\veps}\big) & \lesssim \uptau^{\frac{4\veps}{\sqrt{v}}}
e^{-2\alpha_\uptau (\uptau_\veps-t_\veps) + \frac{2}{\sqrt{v}}\overline{\Psi}_{t}^{z} }e^{-2\alpha_\uptau(\uptau_\veps-t) - \frac{2}{\sqrt{v}} [(t-t_\veps)\wedge (\uptau_\veps-t)]^\mathfrak{d} }e^{2(\uptau_\veps-t)  }\nonumber \\ 
& \lesssim e^{- \frac{1}{\sqrt{v}} [(t-t_\veps)\wedge (\uptau_\veps-t)]^\mathfrak{d} } e^{-2\uptau + \frac{2}{\sqrt{v}}\overline{\Psi}_{t}^z},\label{base}
\end{align}
for $\veps$ small enough (in this computation, the term with exponent $\mathfrak{d}$ comes from the banana and greatly simplifies the analysis).
Now, using Proposition \ref{expopreciseincre} and the fact that $\widehat{\Sigma}_{n_{t,z},z}^2- \widehat{\Sigma}_{n_{t_\veps,z},z}^2 \leq t $ by Lemma \ref{approxtimechange} $(iii)$   we find that 
\[ \EE\big[ e^{\frac{2}{\sqrt{v}} \overline{\Psi}_{t}^{z} } \big] \lesssim e^{t},\]
so that integrating  \eqref{base} we get that
\begin{equation} \label{boundparaout}  \PP\big(\mathscr{A}_{z,\veps} \cap \mathscr{A}_{z',\veps}\big)\leq  \PP\big(  {\mathscr{C}}_{z,\veps} \cap  {\mathscr{C}}_{z',\veps}\big)  \lesssim e^{- \frac{1}{\sqrt{v}} [(t-t_\veps)\wedge (\uptau_\veps-t)]^\mathfrak{d} } e^{-2\uptau+t}. \end{equation}
If now $t\in [t_\mathfrak{q}^-,t_{\mathfrak{q}}^+]$, then we use the same argument as in \eqref{boundparaout} but with $t_\mathfrak{q}^+$ instead of $t$, given that $n_{t,z} \leq n_{t_\mathfrak{q}^+,z}$ and $n_{t,z} \leq n_{t,z'}\leq n_{t_\mathfrak{q}^+,z'}$ so that it is still possible to apply the decorrelation estimate of Proposition \ref{decore}. We find that 
\[   \PP\big(\mathscr{A}_{z,\veps} \cap \mathscr{A}_{z',\veps}\big)\lesssim e^{- \mathfrak{c} \tau^{\mathfrak{d}}} e^{-2\uptau + t_\mathfrak{q}^+},\]
where $\mathfrak{c}>0$ depends on the model parameters and we used that $(t_\mathfrak{q}^+-t_\veps) \wedge (\uptau_\veps-t_\mathfrak{q}^+) \gtrsim \uptau$. Since $t_\mathfrak{q}^+-t \lesssim \mathfrak{q} \log \uptau$, we get that 
\begin{equation} \label{boundparain}  \PP\big(\mathscr{A}_{z,\veps} \cap \mathscr{A}_{z',\veps}\big)\lesssim e^{- \frac{\mathfrak{c}}{2} \tau^{\mathfrak{d}}} e^{-2\uptau + t}.\end{equation}
But the number of pairs $(z,z')$ in $\mathcal{N}_{\uptau_\veps}^2$ such that $e^{-(t+1)}< |z-z'|\leq  e^{-t}$ is at most of order $e^{2\uptau-t}$ (with $t\leq \uptau$). Summing over such pairs with $t$ such that $[(t-t_\veps)\wedge (\uptau_\veps-t)]^\mathfrak{d}\geq \mathfrak{h}\log \uptau$, we obtain from \eqref{boundparaout} and \eqref{boundparain} that 
\begin{equation} \label{S2} S_2 \lesssim  \uptau^{-\frac{ \mathfrak{h}}{\sqrt{v}} }.\end{equation}
In view of the lower bound proven in Lemma \ref{firstmoment}, this implies that $S_2 = o( [\EE(Z_\veps)]^2)$ for $\mathfrak{h}$ large enough. 

\subsubsection*{Late Branching} We now move on to compute $S_3$. Fix some $t$ such that $\uptau_\veps -t \leq (\mathfrak{h} \log \uptau)^{1/\mathfrak{d}}$.  Let $z,z' \in \mathcal{N}$ such that  $e^{-(t+1)}< |z-z'|\leq e^{-t}$. Without loss of generality, we can assume that $n_{t,z}\leq n_{t,z'}$. 
We decompose the event $\mathscr{A}_{z,\veps} \cap \mathscr{A}_{z',\veps}$ according to weather $ \overline{\Psi}^{z}_{t}$ is at distance larger   than $\mathfrak{h} \log \uptau$ from the barrier or not. We start with the first case. We observe that if $\overline{\Psi}^{z}_{t}\leq \sqrt{v} (t-t_\veps) - \mathfrak{h} \log \uptau$ then for $\mathfrak{h}$ large enough, on the event $\mathscr{A}_{z,\veps}$, $\overline{\Psi}^{z}_{\uptau_\veps}- \overline{\Psi}^{z}_t$ climbs at least $ (\mathfrak{h}/2) \log \uptau$ in addition to the growth of  the barrier. Denoting by $\mathscr{E}_{z}$ the event that $\overline{\Psi}_t^{z} \leq \sqrt{v} (t-t_\veps)-\mathfrak{h} \log \uptau$, we get that
\begin{align*}
 \PP\big(\mathscr{A}_{z,\veps} &\cap \mathscr{A}_{z',\veps} \cap \mathscr{E}_{z}\big) \\
&\leq  \PP\big(\overline{\Psi}^{z'}_{\uptau_\veps}- \overline{\Psi}^{z'}_t \geq \sqrt{v}\alpha_\uptau(\uptau_\veps-t) +(\mathfrak{h}/2) \log \uptau, \ \overline{\Psi}^{z}_{\uptau_\veps}\geq \sqrt{v}\alpha_\uptau (\uptau_\veps-t_\veps) -\sqrt{v}\veps \log \uptau \big).
\end{align*}
Using Chernoff's inequality and the precise joint exponential moment of Lemma \ref{decore} with $\lambda = \mu =2\alpha_\uptau/\sqrt{v}$, we find that
\begin{align*} \PP_{n_{t,z}}\big(\mathscr{A}_{z,\veps} \cap \mathscr{A}_{z',\veps} \cap \mathscr{E}_{z}\big) &\lesssim \uptau^{-2\alpha_\uptau\veps- \frac{\alpha_\uptau}{\sqrt{v}}\mathfrak{h}} e^{-2\alpha_\uptau^2(\uptau_\veps-t) -2\alpha_\uptau^2(\uptau_\veps-t_\veps) + \frac{2\alpha_\uptau}{\sqrt{v}} \overline{\Psi}_{t}^z }e^{2\alpha_\uptau^2(\uptau_\veps-t)},
\end{align*}
Since $\alpha_\uptau = 1- O(\log T/T)$ and $\uptau_\veps - t_\veps = \uptau- O(\veps \log \uptau)$, we deduce that by taking $\mathfrak{h}$ large enough,
\[  \PP_{n_{t,z}}\big(\mathscr{A}_{z,\veps} \cap \mathscr{A}_{z',\veps} \cap \mathscr{E}_{z}\big) \lesssim \uptau^{-\frac{\mathfrak{h}}{2\sqrt{v}}}  e^{-2T + \frac{2\alpha_\uptau}{\sqrt{v}} \overline{\Psi}_t^z}.\]
Using the precise exponential estimate of Proposition  \ref{expopreciseincre}, we obtain that 
\begin{equation}\label{latebranch1}  \PP\big(\mathscr{A}_{z,\veps} \cap \mathscr{A}_{z',\veps} \cap \mathscr{E}_{z}\big) \lesssim 
\uptau^{-\frac{\mathfrak{h}}{2\sqrt{v}}}  e^{-2\uptau  +t},\end{equation}
where we used  that $\alpha_\uptau\leq 1$. We  now turn our attention to the second case where $\overline{\Psi}^{z'}_{t}$  is at distance smaller than $\mathfrak{h} \log \uptau$ from the barrier. Define the event $\widetilde{\mathscr{A}}_{z,\veps}$ as the simplified barrier event for the trajectory of $\overline{\Psi}^z$ until time $t$:
\[ \widetilde{\mathscr{A}}_{z,\veps} := \big\{\overline{\Psi}_s^z \leq \sqrt{v} \alpha_\uptau (s-t_\veps)+\mathfrak{f}, \forall s \in \llbracket t_\veps,T_{\veps}\rrbracket \setminus \llbracket t_\mathfrak{q}^-,t_\mathfrak{q}^+\rrbracket,  \ \overline{\Psi}_t^z \geq \sqrt{v} \alpha_\uptau(t-t_\veps) -\mathfrak{h} \log \uptau\}.\]
Using Chernoff inequality, the joint exponential moment estimate of Proposition \ref{decore} with $\lambda = 2\alpha_\uptau/\sqrt{v}$ and $\mu = 2\alpha_\uptau/\sqrt{v}$ and  that $\uptau_\veps-t \leq (\mathfrak{h}\log \uptau)^{1/\mathfrak{d}}$, we find that 
\begin{align}  \PP_{n_{t,z}}\big(\mathscr{A}_{z,\veps} \cap \mathscr{A}_{z',\veps}\cap\mathscr{E}_{z}^\complement\big) & \lesssim \Car_{\widetilde{\mathscr{A}}_{z,\veps}}\uptau^{\frac{4\veps}{\sqrt{v}}} e^{-2\alpha_\uptau^2(\uptau_\veps-t_\veps) -2\alpha_\uptau^2(\uptau_\veps-t)- \frac{2}{\sqrt{v}} (\uptau_\veps-t)^{\mathfrak{d}} + \frac{2\alpha_\uptau}{\sqrt{v}} \overline{\Psi}_t^{z}} e^{2\alpha_\uptau^2(\uptau_\veps-t)},\nonumber \\
& \lesssim \Car_{\widetilde{\mathscr{A}}_{z,\veps}} \uptau^{\frac{4\veps}{\sqrt{v}}} e^{-2\alpha_\uptau^2(\uptau_\veps-t_\veps)-\frac{2}{\sqrt{v}} (\uptau_\veps-t)^{\mathfrak{d}} + \frac{2\alpha_\uptau}{\sqrt{v}} \overline{\Psi}_t^{z}}. \label{latebranch}
\end{align}
(Again, the term with exponent $\mathfrak{d}$ is due to the banana and greatly simplifies the analysis.)
 Next, we use our Gaussian coupling to compute the expectation of $\exp(2\alpha_\uptau\overline{\Psi}_t^{z}/\sqrt{v})$ on the event $\widetilde{\mathscr{A}}_{z,\veps}$. Using the same argument as in the proof of inequality \eqref{couplz}, we get that 
\begin{equation}\label{latebranch2}
 \EE\big[\Car_{\widetilde{\mathscr{A}}_{z,\veps}} e^{\frac{2\alpha_\uptau}{\sqrt{v}} \overline{\Psi}_t^{z}}\big]\leq e^{-\mathfrak{c} T^2} + \EE\big[e^{\frac{2\alpha_\uptau}{\sqrt{v}} \Gamma_{t -t_\veps} }\Car_{\widetilde{\mathscr{B}}_t}\big],
\end{equation}
where $\widetilde{\mathscr{B}}_{t} := \big\{  \Gamma_{s-t_\veps} \leq \sqrt{v} \alpha_\uptau (s-t_\veps)+2\mathfrak{f}, \forall s\in [t_\veps,\uptau_\veps]\setminus [t_\mathfrak{q}^-,t_\mathfrak{q}^+], \Gamma_{t-t_\veps} \geq \sqrt{v} \alpha_\uptau t -2\mathfrak{h} \log T\big\}$ and $\Gamma$ is a centered Gaussian process such that $\EE[\Gamma_s\Gamma_t] = (v/2) (s\wedge t)$ for any $s,t$.
But a Gaussian change of variable yields that 
\begin{equation}\label{latebranch3} \EE\big[e^{\frac{2\alpha_\uptau}{\sqrt{v}}( \Gamma_{t-t_\veps})}\Car_{\widetilde{\mathscr{B}}_{t}}\big] = \PP( \mathscr{B}_{t}) e^{\alpha_\uptau^2(t-t_\veps)},\end{equation}
where $\mathscr{B}_t := \{  \Gamma_{s-t_\veps} \leq \mathfrak{f}, \forall s\in [t_\veps , t]\setminus [t_\mathfrak{q}^-,t_\mathfrak{q}^+], \ \Gamma_{t-t_\veps} \geq  -2\mathfrak{h} \log \uptau\big\}$.
Using a union bound to discretize the possible values of the endpoint and the Ballot Theorem \ref{ballot2}, we find that 
\begin{equation}\label{latebranch4} \PP\big( \mathscr{B}_{t}\big) \lesssim \frac{(\log \uptau)^2}{\uptau^{3/2}}. \end{equation}
Putting together \eqref{latebranch}, \eqref{latebranch2}, \eqref{latebranch3} and \eqref{latebranch4}, we get that 
\[   \PP_{n_{t,z}}\big(\mathscr{A}_{z,\veps} \cap \mathscr{A}_{z',\veps}\cap\mathscr{E}_{z'}^\complement\big) \lesssim (\log \uptau)^2 \uptau^{ \frac{4\veps}{\sqrt{v}}-\frac32} e^{-2\alpha_\uptau^2(\uptau_\veps-t_\veps)+\alpha_\uptau^2(t-t_\veps) -\frac{2}{\sqrt{v}}(\uptau_\veps-t)^\mathfrak{d}}\]
But since $\uptau_\veps -t \leq (\mathfrak{h} \log \uptau)^{1/\mathfrak{d}}$, $-2\alpha_\uptau^2(\uptau_\veps-t_\veps) +\alpha_\uptau^2(t-t_\veps) = -2(\uptau_\veps-t_\veps) + t-t_\veps +\frac32 \log \uptau+o(1)$. Hence, 
\[   \PP_{n_{t,z}}\big(\mathscr{A}_{z,\veps} \cap \mathscr{A}_{z',\veps}\cap\mathscr{E}_{z'}^\complement\big) \lesssim (\log \uptau)^2 \uptau^{\frac{4\veps}{\sqrt{v}}} e^{-2(\uptau_\veps-t_\veps) + t-t_\veps-\frac{2}{\sqrt{v}}(\uptau_\veps-t)^{\mathfrak{d}}}.\] 
Combining this estimate with \eqref{latebranch1} and summing over all possible pairs $(z,z') \in \mathcal{N}_{\uptau_\veps}$ such that $|z-z'|\in (e^{-(t+1)},e^{-t}]$ for $\uptau_\veps -t \leq (\mathfrak{h} \log \uptau)^{1/\mathfrak{d}}$, we find that 
\[ S_3 \lesssim (\log \uptau)^2 \uptau^{\frac{4\veps}{\sqrt{v}}+\veps},\]
where we used the fact that there is at most $O(e^{2\uptau_\veps-t})$ pairs of points in $\mathcal{N}_{\uptau_\veps}$ at distance of order $e^{-t}$ and  that  $\sum_{\uptau_\veps -t \leq (\mathfrak{h} \log \uptau)^{1/\mathfrak{d}}} \exp(-2(\uptau_\veps-t)^{\mathfrak{d}}/\sqrt{v})\lesssim 1$. In view of Lemma \ref{firstmoment}, this shows that $S_3 =o((\EE Z_\veps)^2)$. This ends the proof of \eqref{2momentmethod} and thus of the lower bound of Theorem \ref{reduc}.

\section{From the norm to the coordinates} \label{anti}
Finally, we show in this section by using an anticoncentration argument how the result of  Theorem \ref{maxnorm} on the maximum of the norm of $X_n^z$ yields the same result stated in Theorem \ref{maintheo} on its first coordinate $\phi_n(z)$. 

To implement our anticoncentration argument, we will consider the end of the recursion, between time $(1-\delta)n$ and $n$ for some small $\veps>0$. In this regime, the angle (taken say in $(\pi,\pi]$) between the eigenvectors of  $A_k^z$ is bounded away from $0$ so that $A_k^z$ is conjugated to a rotation matrix through a change of basis matrix whose norm, as well as the norm of its inverse, are bounded.
% that is bounded and its inverse as well. 
In the following lemma, we show that as a result, the norm of products of $A_k^z$ between time $(1-\delta)n$ and $n$ is bounded, and we compute precisely the $(1,1)$ entry of this product of matrices.

\begin{lemma}\label{prodA}For any $(1-\delta)n \leq k\leq \ell \leq n$ and $z\in I_\eta$,  
\[ \|A^z_{k,\ell} \| \vee \|(A^z_{k,\ell})^{-1}\| \leq \mathfrak{C},\]
where $\mathfrak{C}$ is a positive constant depending on the model parameters. Moreover, for any $k\in \NN$ such that $(1-\delta)n \leq k\leq n$,  
\begin{equation} \label{inner} \Big|\langle e_1, A^z_{k, n}e_1\rangle -  \frac{2}{\sqrt{4-z^2}}  \sin\Big( \sum_{j=k+1}^n \theta_j^z +\theta_n^z\Big) \Big|\lesssim \delta.\end{equation} 
\end{lemma}

\begin{proof}To ease the notation, we drop the $z$-dependence in the notation. Fix $(1-\delta)n \leq k \leq \ell \leq n$ and $z \in I_\eta$.  Recall that $A_j = P_j R_j P_j^{-1}$ for any $j\geq k_0+\ell_0$, where $P_j,R_j$ are defined in \eqref{changebaseP}-\eqref{defRk2}. 
With this notation, we have that $A_j = P_j R_j P_j^{-1}$ for any $j\geq k_0+\ell_0$, so that we can write 
\begin{equation} \label{prodA1} A_{k,\ell} = P_\ell\Big[ \prod_{j=k+1}^\ell R_{j} P_j^{-1} P_{j-1}\Big]P_k^{-1}. \end{equation}
Let $j\geq (1-\delta)n$. We know by \eqref{estimatePk}-\eqref{e3} that $\|P_\ell\|,\|P_k^{-1}\|\lesssim 1$ and $\|P_{j}^{-1}P_{j-1}\| = 1+ O (1/n)$ for any $j\gtrsim 1$.  Thus, using that $\| \cdot \|$ is sub-multiplicative we find that $\|A_{k,\ell}\|\lesssim 1$. Replacing $R_j$ by $R_j^{-1}$ in the above argument entails that  $\|A_{k,\ell}^{-1}\|\lesssim 1$. We now prove \eqref{inner}. Writing $Q_j = R_j(P_jP_{j-1}^{-1}-\mathrm{I}_2)$ for any $j\geq k$, we get from \eqref{prodA1},
\begin{equation} \label{eqA}  A_{k,n}  = P_n\prod_{j= k+1}^n \big( R_j + Q_j\big)P_{k}^{-1}.\end{equation}
Expanding the product, we see that the $0^{\text{th}}$ order term (in the $Q_j$'s) is the rotation matrix of angle $\sum_{j= k+1}^n \theta_j$. Further, note that as $P_j^{-1}P_{j-1}= \mathrm{I}_2+O(1/n)$   and $\|P_{j-1}^{-1}\|\lesssim 1$ by  \eqref{estimatePk}-\eqref{e3} for any $j\gtrsim n$, we deduce  that $\|P_k^{-1} - P_n^{-1}\|\lesssim \delta$. Using moreover that $\|P_n\|\lesssim1$, this implies that 
\[ \Big| \langle e_1, P_n \prod_{j=k+1}^n R_j P_k^{-1} e_1\rangle -\langle e_1, P_n \prod_{j=k+1}^n R_j P_n^{-1}e_1 \rangle \Big| \lesssim \delta.\]
But, we find explicitly that
\[ \langle e_1, P_n \prod_{j=k+1}^n R_j P_n^{-1}e_1 \rangle  = \frac{2}{\sqrt{4-z^2}} \sin\Big( \sum_{j=k+1}^n \theta_j +\theta_n\Big).\]
It remains to estimate the norm of the remainder terms of order greater or equal than $1$  in \eqref{eqA}. Using the fact that $\|Q_j\|\lesssim 1/ n$, while $\|R_j\|=1$, for any $j\geq n(1-\delta)$, we obtain that
\[ 
 \Big\|P_n\Big( \sum_{i\geq 1} \sum_{M_j \in \{R_j, P_j\} \atop \#\{ j : M_j = P_j\} = i} M_n M_{n-1}\ldots M_{k}\Big)P_k^{-1}\Big\| \leq \sum_{i\geq 1} \binom{n-k}{i} \Big(\frac{\mathfrak{C}}{n}\Big)^{i},\]
 where $\mathfrak{C}$ is a positive constant depending on the model parameters. By convexity, we find that
 \[  \sum_{i\geq 1} \binom{n-k}{i} \Big(\frac{\mathfrak{C}}{n}\Big)^{k}  = \Big(1 + \frac{\mathfrak{C}}{n}\Big)^{n-k}-1   \lesssim \frac{n-k}{n},\]
which ends the proof of the claim.
\end{proof}

 Using Lemma \ref{prodA} we show  that with probability going to $1$ as $\veps \to 0$, the first coordinate of $X_n^z$ captures a non trivial fraction of the norm of $X_{\lfloor (1-\delta)n\rfloor}^z$. 
\begin{lemma}\label{coord} Let $ m =\lfloor (1-\delta)n \rfloor$ for any $\delta>0$. For $\delta>0$ small enough and any $z \in I_\eta$,
\[ \PP_m\big( |\langle e_1, X^z_n \rangle |\leq \delta^{5/6}\|X^z_m\|\big)  \lesssim \delta^{1/3}, \quad \text{a.s.}\]
\end{lemma}

To prove this result, we will use the following anti-concentration inequality.  
\begin{lemma}\label{anticonc}
Assume that $\xi_1,\xi_2,\ldots,\xi_m$ are independent centered random variables such that for some  $M>0$,  $\EE |\xi_i|^3 \leq M (\EE \xi_i^2)^{3/2}$ for $i=1,\ldots,m$. Let $S = \xi_1+\xi_2+\ldots+\xi_m$. Then there is a constant $C=C(M)$ so that
for any $t \geq \max_i \sqrt{\EE(\xi_i^2)}/8$,
$$ Q_S(t) \leq \frac{Ct}{\sqrt{ \sum_{i=1}^m \EE \xi_i^2}}.$$
\end{lemma}

The proof of this lemma can be found in \cite[Lemma A.2]{ABZ}. Equipped with this inequality, we are now ready to give a proof of Lemma \ref{coord}.

\begin{proof}[Proof of Lemma \ref{coord} ]We drop again the dependence in $z$ in the notation.
By \cite[Lemma 3.8]{ABZ}, we know that 
\[ X_ n = A_{m,n} X_m  + G_n X_m + \mathcal{Y}_n,\]
where $\mathcal{Y}_n$ is such that $\EE_{m}\| \mathcal{Y}_n\|^2 \lesssim \delta^2 \|X_m\|^2$, 
\[ G_n = \sum_{k = m+1}^n A_{k,n} V_k A_{m,k-1},\]
where $V_k = T_k-A_k$. Denote by $u_m = X_m/\|X_m\|$. We will show that $\langle e_1, G_n X_m\rangle/\|X_m\|$ satisfies the assumptions of Lemma \ref{anticonc}. Write for any $m \leq k\leq n$, $\xi_k = \langle e_1,A_{k,n} (V_k - \EE V_k) A_{m,k-1} u_m\rangle$. The variables $(\xi_k)_{m\leq k\leq n}$ are independent variables. From the assumptions \eqref{ass} on the coefficients of our Jacobi matrices, we have  for any $m\leq k\leq n$, $V_k-\EE V_k = \beta_k e_1e_1^{\sf T} +\gamma_k e_1e_2^{\sf T}$, where $(\beta_k,\gamma_k)$ are independent, centered random variables such that $\EE (\beta_k^2) = \EE(\gamma_k^2) = (v/n)(1+O(\veps))$ and such that there exists $\mathfrak{h} >0$ depending on the model parameters satisfying 
\[ \EE[e^{\mathfrak{h} \sqrt{n} | \beta_k| }] \leq 2, \ \EE[e^{\mathfrak{h}\sqrt{n} |\gamma_k| }] \leq 2, \quad m\leq k\leq n.\]
In particular, there exists $\mathfrak{c}>0$ such that $\EE[|\sqrt{n} \beta_k|^{3/2}], \EE[|\sqrt{n}\beta_k|^{3/2}]\leq \mathfrak{c}$. Next, we compute a lower bound on the variance of $\sum_{k=m+1}^n \xi_k$, that is of $\langle e_1,G_n u_m\rangle$. Using independence and the fact that $\EE(\beta_k^2), \EE(\gamma_k^2)\geq v/(2n)$ for any $k\geq m$, when $\veps$ is small  enough, we get that
\begin{align*}
 \Var\big[\langle e_1,G_n u_m\rangle^2\big] &= \sum_{k=m+1}^n \langle e_1, A_{k,n} e_1\rangle^2\big( \EE(\beta_k^2) \langle e_1, A_{m,k-1} u_m \rangle^2+ \EE(\gamma_k^2)  \langle e_2, A_{m,k-1} u_m\rangle^2\big)\\
& \geq \frac{v}{2n}  \sum_{k=m+1}^n \langle e_1, A_{k,n} e_1\rangle^2 \|A_{m,k-1} u_m\|^2 \geq \frac{\mathfrak{c}v}{2n} \sum_{k=m+1}^n \langle e_1, A_{k,n} e_1\rangle^2,
\end{align*}
where we used in the last inequality the fact from Lemma \ref{prodA} that $\|A_{k,n}^{-1} \| \lesssim 1$. By Lemma \ref{prodA}, we deduce that at the price of taking $\mathfrak{c}$ smaller, 
\[  \Var\big[(\langle e_1,G_n u_m\rangle)^2\big] \geq  \frac{\mathfrak{c}}{n} \sum_{k=m+1}^n \sin^2\Big(\sum_{j=k+1}^n \theta_j +\theta_n\Big) - \mathfrak{c}^{-1} \delta^2.\]
Since $z \in I_\eta$, there exists $\zeta_\eta >0$ such that for $\delta$ small enough, $\theta_k \in [\zeta_\eta, \frac{\pi}{2}-\zeta_\eta]$ for any $k \in [m+1,n]$.  Let $\zeta>0$ and define the set 
\[ J := \big\{k\in [m+1,n] : \sum_{j=k+1}^n \theta_j  + \theta_n \in [0,\zeta] \cup [\pi - \zeta] +\pi \ZZ\big\}.\]
Since all the $\theta_j$'s are bounded away from $0$ and $\pi/2$, choosing $\zeta$ small enough compared to $\zeta_\eta$, we obtain that $J$ has the following property: if $k\in J\setminus\{n\}$, then $k+1 \notin J$. This property implies that $\# J\leq (n-m+1)/2 \lesssim \zeta n$.  As a result, for $\delta$ small enough,
\[  \Var\big[(\langle e_1,G_n u_m\rangle)^2\big] \gtrsim \zeta \delta.\]
Besides, note that $\max_{m< k \leq n}\sqrt{ \EE_m(Z_k^2) }\lesssim 1/\sqrt{n}$.
Using Lemma \ref{anticonc},  we conclude that there exist positive constants $\mathfrak{c}, \mathfrak{C}$ such that for any $t\geq \mathfrak{c}/\sqrt{n}$. 
\begin{equation} \label{anticoncG}  \PP_m\Big( \big| \langle e_1, A_{m,n}X_m +  G_n X_m \rangle \big| \leq t \Big) \leq \frac{\mathfrak{C} t}{\sqrt{\delta}}.\end{equation}
Using a union bound,  \eqref{anticoncG} and the fact that $\EE_m(\|\mathcal{Y}_n\|^2) \lesssim \delta^2 \|X_m\|$, we deduce that for any $t\gtrsim 1/\sqrt{n}$,
\begin{align*} \PP_m\Big(|\langle e_1,X_n \rangle | \leq t \|X_m\|\Big)& \leq \PP_m\Big( |\langle e_1, A_{m,n} X_m + G_n X_m\rangle| \leq t/2\Big) + \PP_m\big(\|\mathcal{Y}_n\| \geq t/2\Big)\\
& \lesssim \frac{t}{\sqrt{\delta}}   + \Big(\frac{\delta}{t}\Big)^2,
\end{align*}
where the last inequality made use of Chebychev inequality. Choosing $t= \delta^{5/6}$ gives the claim.
\end{proof}

Finally, we are now ready to prove that Theorem \ref{maxnorm} implies Theorem \ref{maintheo}.

\begin{proof}[Proof of Theorem \ref{maintheo}]Denote  $D_n(z) := \sum_{k=1}^{k_{0,z}} \log \alpha_{k,z}$ for any $z\in (-2,2)$. By  \cite[Remark 2.3]{ABZ}, we know that \begin{equation} \label{computationconst} D_n(z) = \frac{z^2n}{4} - \frac{1}{4} \log n +O(1),\end{equation} 
where $O(1)$ is uniform in $z\in I_\eta$. 
Note that for any $z\in (-2,2)$,  $p_n(z) = \sqrt{n!}n^{-n/2} \phi_n(z) = \sqrt{n!}n^{-n/2} X_n^z(1)$.   By Stierling's formula, we know that $n^{n/1}/\sqrt{n!} = (2\pi n)^{1/4} e^{-n/2}(1+o(1))$.  Together with \eqref{computationconst}, this entails that for $z\in I_\eta$, 
\[ \sum_{k=1}^{k_{0,z}} \log \alpha_{k,z} - \log\Big( \frac{n^{n/2}}{\sqrt{n!}}\Big)  = n\Big(\frac{z^2}{4}- \frac{1}{2}\Big) + O(1).\]
Thus, it suffices to prove that
\[ \frac{\max_{z \in I_\eta} \big( \log |\phi_n(z)|-D_n(z) \big) - \sqrt{v}\log n }{\log \log n} \underset{n\to+\infty}{\longrightarrow} - \frac{3\sqrt{v}}{4},\quad \mbox{\rm in probability.}\]
Since $|\phi_n(z)|\leq \|X_n^z\|$, the upper bound of Theorem \ref{maintheo} follows immediately from the upper bound for $\log \|X^z_n\|$. We now turn our attention to the lower bound. Denote for any $\ell \geq k_{0,z}$, 
\[ \widehat{X}_\ell^z :=  \frac{\|X_\ell^z\|}{\prod_{k=1}^{k_{0,z}}\alpha_{k,z}},\quad \widehat{\phi}_\ell(z) := \widehat{X}_\ell^z (1).\]
Fix $\veps>0$ and 
set $\mathfrak{h}_{n,\veps} =  \sqrt{v} \log n -\sqrt{v}(\frac{3}{4} +\veps) \log \log n$ for any $n\geq 2$. Let $\delta>0$ and  $m = \lfloor (1-\delta)n\rfloor$.  Define $\mathscr{C} := \{ z \in I_\eta : \log \|\widehat{X}_m^z\| \geq \mathfrak{h}_{n,\veps}\}$ and $\mathscr{E}_{n,\delta}$ the event
\[ \mathscr{E}_{n,\delta}:= \big\{ \mathscr{C}\neq \emptyset, \  \exists z \in \mathscr{C}, |\phi_n(z)| \geq \delta^{5/6} \|X^z_m\|\big\}.\]
Now, note that Lemma \ref{truncation} holds as well for $\log \|\widehat{X}_m^z\|$ instead of $\Psi_{T_z}^z$, meaning that with probability going to $1$, if there exists $z$ at which the field coming from the truncated recursion is high then it true also for $\log \|\widehat{X}_m^z\|$.  Therefore, together with the lower bound for the maximum of the field resulting from truncating the recursion proven in \eqref{condequivlb} it follows that  for any $\delta>0$, $\PP(\mathscr{C}\neq \emptyset)  \to 1$ as $n\to+\infty$.
 Since $\{\mathscr{C}\neq \emptyset\}$ is $\mathcal{F}_m$-measurable, we conclude by Lemma \ref{coord} that 
\begin{equation} \label{probE} \lim_{\delta \downarrow 0} \liminf_{n\to+\infty} \PP(\mathscr{E}_{n,\delta}) = 1.\end{equation}
We claim that for any $\delta>0$ and $n$ large enough,  
\begin{equation} \label{inclu} \mathscr{E}_{n,\delta} \subset \big\{\exists z \in I_\eta : \log |\widehat{\phi}_n(z)|  \geq \mathfrak{h}_{n,\veps/2}\big\}.\end{equation}
Once this inclusion proven, the lower bound of Theorem \ref{maintheo} follows from \eqref{probE}. It now remains to prove \eqref{inclu}. On the event $\mathscr{E}_{n,\delta}$, there exists $z\in I_\eta$ such that 
\[  \log \|\widehat{X}_m^z\| \geq \mathfrak{h}_{n,\veps}, \text{ and } |\phi_n(z)|\geq \delta^{5/6} \|X^z_m\|,\]
which implies that for such $z$, $ \log |\widehat{\phi}_n(z)|  \geq \mathfrak{h}_{n,\veps} - \frac56 \log \delta \geq \mathfrak{h}_{n,\veps/2}$, 
for $n$ large enough. This ends the proof of the claim \eqref{inclu}.
\end{proof}

\appendix

\section{Moderate deviation estimates}

\begin{lemma}
\label{tailpropH} Let  $(\xi_k)_{m\leq k \leq p}$ be a process adapted to the filtration $(\mathcal{F}_k)_{m\leq k \leq p}$. Assume that there exist $\mathfrak{q}>0$ and a deterministic sequence $(\sigma_k^2)_{m\leq k \leq p}$ such that 
\[ \Var_{k-1}(\xi_k) \leq \sigma_k^2, \ |\xi_k- \EE_{k-1}(\xi_k)|\leq (\log n)^{\mathfrak{q}}\sigma_k.\]
There exists a numerical constant $\mathfrak{c}>0$  such that for any $t>0$,
\[ \PP_{m-1}\Big(\sup_{m\leq k\leq \ell  \leq p} \sum_{j=k}^\ell(\xi_j-\EE_{j-1}\xi_j)> t\Big) \leq 2\exp\Big(-\mathfrak{c}\min\Big\{\frac{t^2}{\sum_{k=m}^p \sigma_k^2}, \frac{t(\log n)^{-\mathfrak{q}}}{\max_{m\leq k \leq n} \sigma_k}\Big\}\Big).\]

\end{lemma}
\begin{proof} Denote by ${\zeta}_k := \sigma_k^{-1}(\xi_k-\EE_{k-1}(\xi_k))$ and let $\Lambda_k(\theta) : =  \log \EE_{k-1}[e^{\theta \zeta_k}]$ for any $\theta \in \RR$. By assumption, there exists $\mathfrak{q}>0$ such that $|\zeta_k|\leq (\log n)^{\mathfrak{q}}$ and $\Var_{k-1}(\zeta_k) \leq 1$ for any $m\leq k\leq p$. Thus, for $|\lambda|\leq (\log n)^{-\mathfrak{q}}$, we have that
\[ 0\leq  \Lambda_k''(\lambda) \leq    \EE_{k-1}\big[\zeta_k^2 e^{\lambda \zeta_k} \big] \leq  e.\]
As a result, by Taylor expanding $\Lambda_k$ to the second order around $0$ we find that for any $|\theta|\leq (\log n)^{-\mathfrak{q}}$, 
\[ \EE_{k-1}[e^{ \theta \zeta_k}] \leq e^{\mathfrak{C}\theta^2},\]
where $\mathfrak{c}>0$ is some numerical constant. 
Therefore,  for any $|\theta|\leq (\log n)^{-\mathfrak{q}}/\max_{m\leq k\leq p}\sigma_k$ and $\sigma^2\geq \sum_{k=m}^n \sigma_k^2$, we get that
\begin{equation} \label{subgauss} \EE_{m-1}\big[e^{\theta \sum_{k=m}^p \sigma_k \zeta_k}\big] \leq e^{\mathfrak{c}\theta^2\sum_{k =m}^p \sigma_k^2}\leq e^{\mathfrak{C} \theta^2 \sigma^2}.\end{equation}
Using Doob maximal inequality to the submartingale $(e^{\sum_{j=m}^k \sigma_j\zeta_j})_{m\leq k \leq m}$ and Chernoff inequality we find that for any $\theta\geq 0$, 
\[ \PP\Big(\max_{m \leq k \leq p} \sum_{j=m}^k \sigma_j \zeta_j \geq t \Big) \leq \EE\big(e^{\theta \sum_{j=m}^p \sigma_j\zeta_j} \big) e^{-\theta t}.\]
Using \eqref{subgauss} and optimising over $\theta$ gives that 
\begin{equation} \label{boundsubg} \PP\big(\max_{m\leq k \leq p} \sum_{j=m}^k \sigma_j \zeta_j \geq t \big) \leq \exp\Big(-\frac{t^2}{4\mathfrak{C}\sigma^2}\Big).\end{equation}
for any $0\leq t \leq 2\mathfrak{C}\sigma^2(\log n)^{-\mathfrak{q}}/\max_{j\leq n}\sigma_k$ and $\sigma^2\geq \sum_{k=m}^n \sigma_k^2$. Choosing $\sigma^2 = \max\{\sum_{k=m}^n \sigma_k^2, \frac{t}{2\mathfrak{C}} (\log n)^{\mathfrak{q}}\max_{m\leq k \leq n} \sigma_k\}$ and taking $\mathfrak{C}\geq1$ yields that 
\[  \PP\big(\max_{m\leq k \leq p} \sum_{j=m}^k \sigma_j \zeta_j \geq t \big) \leq \exp\Big(-\frac{1}{4\mathfrak{C}} \min\Big\{\frac{t^2}{\sum_{k=m}^n \sigma_k^2}, \frac{t(\log n)^{-3\mathfrak{q}}}{\max_{m\leq k \leq n} \sigma_k}\Big\}\Big).\]
 To end the proof note that 
 \[ \max_{k\leq \ell} \sum_{j=k}^\ell \sigma_j \zeta_j =\max_{k\leq\ell}\big\{ \sum_{j=m}^\ell \sigma_j\zeta_j - \sum_{j=m}^k \sigma_j\zeta_j\big\} \leq \max_{m\leq k \leq p} \sum_{j=1}^\ell \sigma_j\zeta_j  + \max_{m\leq \ell \leq p} \sum_{j=1}^k -\sigma_j\zeta_j.\]
  Thus, applying \eqref{boundsubg} to $-\xi_k$ and using a union bound ends the proof of claim.
\end{proof}

\begin{lemma}\label{moddev}
Let $\xi$ be a  random variable such that $\|\xi\|_{L^\infty} <\infty$ and $\EE(\xi)=0$. 
For any $|\beta| \|\xi\|_{L^\infty}\leq 1$, 
\[ \big|\log \EE (e^{\beta \xi})- \frac{\beta^2}{2} \Var(\xi)\big|\leq \mathfrak{C}|\beta|^3 \|\xi\|_{L^\infty}^{3} ,\]
where $\mathfrak{C}>0$ is some numerical constant. 
\end{lemma}

\begin{proof}
Denote by $\Lambda(\beta) = \log \EE(e^{\beta \xi})$ for any $\beta \in \RR$.  Since $\EE(\xi)=0$ and $\xi$ is bounded, we have that $\Lambda(0)=\Lambda'(0)=0$, $\Lambda''(0) = \Var(\xi)$ and $|\Lambda^{(3)}(\beta)|\lesssim |\beta|^3\|\xi\|_{L^\infty}$ for any $|\beta|\|\xi\|_{L^\infty}\leq 1$. Taylor's Theorem yields the desired inequality. 
\end{proof}

\section{Sakhanenko's strong approximation Theorem}
Sakhanenko's Strong Approximation Theore  is an analogue of KMT Theorem in the context of non-identically distributed random variables. The heterogeneity of the possible distributions is handled here through a uniform boundedness assumption that is quantified in terms of the {\em Sakhanenko parameters}. For a real centered random variable $X$, its  Sakhanenko parameter is defined as:
\[\lambda(X) := \sup\big\{\lambda >0 : \lambda^2 \EE[|X|^3e^{\lambda|X|}] \leq \EE[X^2]\big\}.\]
With this definition, we have the following result (see \cite[Theorem 3.1]{Lifshits}, \cite{Sakhanenko}).

\begin{theorem} \label{Sakhanenko}
Let $(X_1,\ldots,X_n)$ be a sequence of independent centered real random variables. Upon enlarging the probability space, there exists $(Y_1,\ldots,Y_n)$ a sequence of independent centered Gaussian random variables such that $\Var(X_i) = \Var(Y_i)$ for all $1\leq i \leq n$, and 
\[ \EE\big[e^{\mathfrak{c} \lambda \Delta_n(X,Y)}\big] \leq 1+ \lambda \sqrt{\sum_{i=1}^n \Var(X_i)},\]
where $\mathfrak{c}>0$ is a numerical constant, $\Delta_n(X,Y) = \max_{1\leq k \leq n}\big|\sum_{i=1}^k (X_i -Y_i)\big|$ and $\lambda = \min_{1\leq i \leq n} \lambda(X_i)$.
\end{theorem}
%Note that the statement of the above theorem is empty if $\lambda=0$, and that $\lambda>0$ implies that the variance of the $X_i$'s is finite. 

\section{Ballot Theorems}
\label{sec-Ballot}
We collect here the versions of the Ballot Theorem needed in the present work. 
In all the statements in this section we consider $(\Gamma_t)_{t\in \NN}$ a centered Gaussian process with covariance given by $\mathrm{Cov}(\Gamma_s,\Gamma_t) = s\wedge t$ for any $s,t\in \NN$ and we denote by $\PP_x$ the conditional probability given that $\Gamma_0=x$, where $x\in \RR$. We start with a first Ballot Theorem giving a lower bound on the probability that the process $\Gamma$ starts at some negative point and ends somewhere negative while staying in between two polynomial curves.

\begin{proposition}\label{ballot}Let  $0<\mathfrak{d} < \frac12 <\mathfrak{e}$. There exist $\mathfrak{c},\mathfrak{f}>0$ such that for any $N\geq 1$ and $x,y \geq 1$ such that $xy < \mathfrak{d}^{-1} N$,
\[ \PP_{-x}\big(-x-t^\mathfrak{e}-\mathfrak{f}\leq \Gamma_t \leq -t^{\mathfrak{d}} , t\in\llbracket 1,N\rrbracket, \Gamma_N \in [-y-1,-y]\big)  \geq \mathfrak{c} \frac{xy}{N^{3/2}}.\] 
\end{proposition}

\begin{proof}
Let $B$ be a standard Brownian motion and continue to denote by $\PP_{-x}$ the conditional probability given that $B_0=-x$. Clearly, it suffices to prove that 
\[ \PP_{-x}\big(-x-t^\mathfrak{e}-\mathfrak{f}\leq B_t \leq -t^{\mathfrak{d}} , t\in [0,N], B_N \in [-y-1,-y]\big) \geq \mathfrak{c} \frac{xy}{N^{3/2}},\]
where $\mathfrak{c}>0$ is some constant. Divide the event on the left-hand side above into two events $\mathscr{B}$ and $\mathscr{E}$ defined as 
\[ \mathscr{B}:=\big\{ B_t \leq -t^{\mathfrak{d}} , t\in [0,N], B_N \in [-y-1,-y]\big\},\]
\[ \mathscr{E}:= \big\{B_t\geq -x-t^\mathfrak{e}-\mathfrak{f}, t\in [0,N]\big\}.\]
From \cite[Proposition 2.1]{CHL} we know that $\PP(\mathscr{B})\gtrsim xy N^{-3/2}$. It remains to prove that $\PP(\mathscr{B}\setminus \mathscr{E})$ is of smaller order. To this end, define the stopping time $\tau :=\inf\{t : B_t \leq -x-t^\mathfrak{e}-\mathfrak{f}\}$. Conditioning on $\mathcal{F}_\tau$ and  relaxing  the barrier on ;$[\tau, N]$ into the barrier $ -( (t-\tau)\wedge (N-t))^\mathfrak{d}$ we get that 
\[ \PP_{-x}\big(\mathscr{B}\setminus \mathscr{E} \mid \mathcal{F}_\tau\big) \leq  \PP_{-x-t^\mathfrak{e}-\mathfrak{f}}\Big(B_t \leq -( (t-\tau)\wedge (N-t))^\mathfrak{d}, t\in \llbracket \tau, N\rrbracket, B_N \in [-y-1,-y]\Big).\]
Reversing time and applying the upper bound in \cite[Proposition 2.1]{CHL}, we find that 
 \[ \PP_{-x}\big(\mathscr{B}\setminus \mathscr{E} \mid \mathcal{F}_\tau\big) \lesssim \frac{\tau^{\mathfrak{e}}y}{N^{3/2}}e^{-(\tau^{\mathfrak{e}}+\mathfrak{f})^2/2}\leq \mathfrak{c}' \frac{y}{N^{3/2}} e^{-(\tau^\mathfrak{e}+\mathfrak{f})^2/4},\]
 where $\mathfrak{c}'$ is some positive constant depending on $\mathfrak{d}$ and  $\mathfrak{f}$ is large enough. 
Taking $\mathfrak{f}$ even  larger   so that $\mathfrak{c}'\sum_{k\geq 0} e^{-(k^{\mathfrak{e}}+\mathfrak{f})^2/4} \leq \mathfrak{c}/2$, we get that $\PP_{-x}(\mathscr{B}\setminus \mathscr{E})\leq (\mathfrak{c}/2) yN^{-3/2}$. Since $x\geq 1$, this ends the proof of the claim.
\end{proof}

The second Ballot Theorem we make use of is an upper bound on the probability that the random walk $\Gamma$ starts somewhere negative and ends somewhere negative while staying negative  except possibly on an interval $\llbracket \ell+1,m\rrbracket$. If the gap in the barrier is not too large, meaning that its size is at most of the order of its left endpoint, and if there is enough of the barrier remaining on its right-hand side then this probability is of the same order as the one with the full barrier.  
\begin{proposition}\label{ballot2}Let $1\leq \ell < m\leq N$ such that $N-m\geq \mathfrak{p} N$ and $m-\ell \leq \mathfrak{q} \ell$ for some $\mathfrak{p},\mathfrak{q}>0$. There exists $\mathfrak{c}=\mathfrak{c}(\mathfrak{p},\mathfrak{q})>0$ such that for any $x,y \geq 1$,
\[ \PP_{-x}\big(\Gamma_t \leq 0, t\in \llbracket 1,N \rrbracket \setminus \llbracket \ell+1,m \rrbracket, \Gamma_N \in [-y-1,-y]\big)  \leq \mathfrak{c} \frac{xy}{N^{3/2}}.\] 

\end{proposition}

 Using a Gaussian change of measure,  discretizing the possible values of the endpoint, and performing a union bound, one gets immediately the following corollary.

\begin{corollary}\label{ballot3}Let $1\leq \ell < m\leq N$ such that $N-m\geq \mathfrak{p} N$ and $m-\ell \leq \mathfrak{q} \ell$ for some $\mathfrak{p},\mathfrak{q}>0$. There exists $\mathfrak{c}=\mathfrak{c}(\mathfrak{p},\mathfrak{q})>0$ such that for any $x,y \geq 1$,
\[ \PP_{-x}\big(\Gamma_t \leq \sqrt{2}t, t\in \llbracket 1,N \rrbracket \setminus \llbracket \ell+1,m \rrbracket, \Gamma_N -\sqrt{2}N\geq  -y\big)  \leq \mathfrak{c} \frac{xy}{N^{3/2}}e^{-N+\sqrt{2}(y-x)}.\] 

\end{corollary}

To prove  proposition \ref{ballot2}, we will rely on the following version of the Ballot Theorem.
\begin{theorem}\label{ballotup}
There exists $\mathfrak{c}>0$ such that for any $N\geq 1$, $x,y \geq 1$,
\[ \PP_{-x}\big(\Gamma_t \leq 0, t\in \llbracket 1,N \rrbracket, \Gamma_N \in [-y-1,-y]\big) \leq \mathfrak{c}\frac{xy}{N^{3/2}}e^{-y^2/2N}.\]
\end{theorem}
\begin{proof}
We will prove that there exists $\mathfrak{f}\geq 1$, $\mathfrak{c}>0$ such that for any $N\geq 1$, $x,y \geq \mathfrak{f}$, 
\begin{equation} \label{claimballotup} \PP_{-x}\big(\Gamma_t \leq 0, t\in \llbracket 1,N \rrbracket, \Gamma_N \in [-y-1,-y]\big) \leq \mathfrak{c}\frac{xy}{N^{3/2}}e^{-y^2/2N}.\end{equation}
Assuming for the moment that the above statement is true, we show how it implies Theorem \ref{ballotup}. Indeed, if $1\leq x,y \leq \mathfrak{f}$, then 
\[ \PP_{-x}\big(\Gamma_t \leq 0, t\in \llbracket 1,N \rrbracket, \Gamma_N \in [-y-1,-y]\big) \leq  \PP_{-x-\mathfrak{f}}\big(\Gamma_t \leq 0, t\in \llbracket 1,N \rrbracket, \Gamma_N \in [-y-\mathfrak{f}-1,-y-\mathfrak{f}]\big).\]
Using \eqref{claimballotup} to bound the probability  above on the right-hand side of the inequality, we get the claim by adjusting the constant $\mathfrak{c}$. It now remains to prove \eqref{claimballotup}.
Assume $x,y\geq \mathfrak{f}$. Let $B$ be a standard Brownian motion and denote still by $\PP_{-x}$ the conditional distribution given that $B_0=x$. Let $\mathscr{B}$ and $\widetilde{\mathscr{B}}$ be the barrier events 
\[ \mathscr{B} :=\big\{B_t\leq 0, t \in  \llbracket 1,N \rrbracket, B_N \in [-y-1,-y]\big\},\]
\[ \widetilde{\mathscr{B}} :=\big\{B_t\leq u_t,    t\in [0,N], B_N \in [-y-1,-y]\big\},\]
where $u_t = \mathfrak{d}^{-1}(t\wedge (N-t))^{\mathfrak{d}}$ and $0<\mathfrak{d}<\frac12$ is to be chosen later. By \cite[Proposition 2.1]{CHL}, we know that 
\begin{equation}\label{ballotupcont} \PP_{-x}(\widetilde{\mathscr{B}}) \leq \mathfrak{c} \frac{xy}{N^{3/2}}e^{-y^2/2N},\end{equation}
where $\mathfrak{c}>0$ is some constant. On the event $\mathscr{B}\setminus \widetilde{\mathscr{B}}$, there exists $t\in  \llbracket 0,N -1\rrbracket$ such that $\max_{s\in (t,t+1)} B_s \geq \min(u_t,u_{t+1})$. Now, if  $B_t\leq 0$ and $B_{t+1}\leq 0$, then 
\begin{equation} \label{monotonicty} \PP\big(\max_{s\in (t,t+1)} B_s \geq \min(u_t,u_{t+1})\mid B_t, B_{t+1}\big) \leq \PP_{-x}\big(\max_{s\in (0,1)} X_{s} \geq \min(u_t,u_{t+1})\big),\end{equation}
where $X$ is distributed as a Brownian bridge on $[0,1]$. Indeed, the conditional law of $(B_s)_{s\in [t,t+1]}$ given $B_t = v ,B_{t+1}=w$ for some $v,w\in \RR$ is the same as the law of $(B_{s-t}-(s-t)B_{t+1}+v-(s-t)(v-w))_{s\in [t,t+1]}$. Since $v-(s-t) (v-w) \leq 0$ when $v,w\leq 0$ this gives the claim \eqref{monotonicty}. Similarly, we find that on the event where $B_1\leq 0$, 
\begin{equation} \label{monotonicty3} \PP_{-x}\big(\max_{s\in (0,1)} B_s \geq 0 \mid B_t, B_{t+1}\big) \leq \PP\big(\max_{s\in (0,1)} (X_{s} -(1-s)\mathfrak{f}) \geq 0 \big),\end{equation}
and on the event where $B_{N-1}\leq 0$ and $B_N\leq -\mathfrak{f}$,
\begin{align}  \PP_{-x}\big(\max_{s\in (N-1,N)} B_s \geq 0 \mid B_t, B_{t+1}\big) &\leq \PP\big(\max_{s\in (0,1)} (X_{s} -s)\mathfrak{f}) \geq 0 \big)\nonumber\\
& = \PP\big(\max_{s\in (0,1)} (X_{s} -(1-s)\mathfrak{f}) \geq 0 \big),\label{monotonicty4}
\end{align}
where we used in the last equality that $(X_s)_{s\in [0,1]}$ and $(X_{1-s})_{s\in [0,1]}$ have the same law.
Hence, using \eqref{monotonicty},\eqref{monotonicty3},\eqref{monotonicty4} and the fact that the conditional distribution of $(B_s)_{s\in (t,t+1)}$ given $B_t,B_{t+1}$ and  given $(B_s)_{s\notin (t,t+1)}$ are the same,  we find that 
\begin{align*} \label{barevent} \PP_{-x}\big(\mathscr{B}\setminus \widetilde{\mathscr{B}}\big) &\leq \PP_{-x}(\mathscr{B}) \Big( \sum_{t=1}^{N-1} \PP\big(\max_{s\in (0,1)} X_{s} \geq \min(u_t,u_{t+1})\big)  +2 \PP\big(\max_{s\in (0,1)} (X_{s} -(1-s)\mathfrak{f}) \geq 0 \big)\Big).
\end{align*}
Note that if $\max_{s\in (0,1)} X_s \geq 0$ then $\max_{s\in (0,1)} X_s \leq \max_{s\in (0,1)} (X_s/(1-s)) = \max_{s\in (0,1)} (1+s) X_{s/(1+s)}$. Since $((1+s) X_{s/(1+s)})_{s \in [0,1]}$ is a standard Brownian motion (see \cite[Lemma 13.6]{Kallenberg}, it follows that 
\begin{equation} \label{boundbridge1} \PP\big(\max_{s\in (0,1)} X_{s} \geq \min(u_t,u_{t+1})\big) \leq \PP\big(|B_1| \geq \min(u_t,u_{t+1})\big) \leq 2e^{-\min(u_t^2,u_{t+1}^2)/2}.\end{equation}
Similarly, 
\begin{equation} \label{boundbridge2} \PP\big(\max_{s\in (0,1)} (X_{s} -(1-s)\mathfrak{f}) \geq 0 \big) \leq  \PP(|B_1|\geq \mathfrak{f}) \leq2 e^{-\mathfrak{f}^2/2}. \end{equation}
Putting together \eqref{boundbridge1} and  \eqref{boundbridge2}, and using that $\min(u_t^2,u_{t+1}^2) = \mathfrak{d}^{-2}\min(t^{2\mathfrak{d}},(N-t-1)^{2\mathfrak{d}})$, we deduce that
\[ \PP_{-x}\big(\mathscr{B}\setminus \widetilde{\mathscr{B}}\big) \leq 4\PP_{-x}(\mathscr{B}) \Big(\sum_{t=1}^{N-1} e^{-\mathfrak{d}^{-2} t^{2\mathfrak{d}}} + e^{-\mathfrak{f}^2/2}\Big)\]
Choosing $\mathfrak{d}$ small enough and $\mathfrak{f}$ large enough, we get that $\PP_{-x}(\mathscr{B}\setminus \widetilde{\mathscr{B}}) \leq (1/2) \PP_{-x}(\mathscr{B})$. Combined with \eqref{ballotupcont}  this ends the proof.
\end{proof}

We can now give a proof of Proposition \ref{ballot2}.

\begin{proof}[Proof of Proposition \ref{ballot2}]
Denote by $\mathscr{B}$ the barrier event in Proposition \ref{ballot2}.
Conditioning on $\mathcal{F}_{m+1}$ and applying the Ballot Theorem \ref{ballotup}, we find that on the event where $\Gamma_{m+1}\leq 0$,
\[ \PP_{m+1}(\mathscr{B}) \lesssim y(N-m)^{-3/2} (|\Gamma_{m+1}|+1).\]
Thus, 
\[ \PP_{\ell}(\mathscr{B}) \lesssim y (N-m)^{-3/2} \big(|\Gamma_{\ell}|+\sqrt{m-\ell}),\]
where we used the fact that $\EE_{\ell}(|\Gamma_m-\Gamma_{\ell}|) \lesssim \sqrt{m-\ell}$ and that $m-\ell \geq 1$. 
Using a union bound to discretize the values of $\Gamma_\ell$ and the Ballot Theorem \ref{ballotup} we get that 
\begin{align*} \PP(\mathscr{B} ) &\lesssim y(N-m)^{-3/2}\sum_{k\geq 1} \frac{xk}{\ell^{3/2}}(k+\sqrt{m-\ell}) e^{-k^2/2\ell}\\
& \lesssim y(N-m)^{-3/2}\Big( x +\frac{\sqrt{m-\ell}}{\sqrt{\ell}}\Big),
\end{align*}
where we used an integral comparaison and the fact that $\int_0^{+\infty} u^2e^{-u^2/2\ell} du \lesssim \ell^{3/2}$ and $\int_0^{+\infty} u e^{-u^2/2} du \lesssim \ell$. Since $N-m\gtrsim N$ and $m-\ell \lesssim \ell$ this gives the claim.
 \end{proof}

\end{document}